\documentclass{amsart}
\usepackage{amscd,amssymb,amsmath,amsthm}
\usepackage[dvips]{graphicx}
\usepackage[all]{xy}
\usepackage[dvips]{color}
\theoremstyle{plain}
\newtheorem{proposition}{Proposition}
\newtheorem{theorem}[proposition]{Theorem}

\newtheorem*{conjecture*}{Conjecture}
\newtheorem{definition}[proposition]{Definition}
\newtheorem{corollary}[proposition]{Corollary}
\newtheorem{lemma}[proposition]{Lemma}
\newtheorem{remark}[proposition]{Remark}
\newtheorem{remarks}[proposition]{Remarks}
\newtheorem{example}[proposition]{Example}
\newtheorem{proposition-definition}[proposition]{Proposition/Definition}
\newtheorem*{proposition*}{Proposition}
\newtheorem*{theorem*}{Theorem}
\newtheorem*{maintheorem*}{Main Theorem}
\newtheorem*{maincorollary*}{Main Corollary}
\newtheorem*{corollary*}{Corollary}
\newtheorem*{lemma*}{Lemma}
\newtheorem*{remark*}{Remark}
\newtheorem*{example*}{Example}

\numberwithin{proposition}{section}
\numberwithin{equation}{section}
\numberwithin{figure}{section}

\def\co{\colon\thinspace}

\newcommand{\N}{\mathbb{N}}
\newcommand{\Z}{\mathbb{Z}}
\newcommand{\Q}{\mathbb{Q}}
\newcommand{\R}{\mathbb{R}}
\newcommand{\C}{\mathbb{C}}
\newcommand{\K}{\mathbb{K}}

\begin{document}

\title[TQFT structure on symplectic cohomology]{Topological quantum field theory structure on symplectic cohomology}

\author{Alexander F. Ritter}

\address{Trinity College, Cambridge, CB2 1TQ, England.}

\email{a.f.ritter@dpmms.cam.ac.uk}

\date{version: \today}


\begin{abstract}
We construct the TQFT on symplectic cohomology and wrapped Floer cohomology, possibly twisted by a local system of coefficients, and we prove that Viterbo restriction preserves the TQFT. This yields new applications in symplectic topology relating to the Arnol'd chord conjecture and to exact contact embeddings. We prove that if a Liouville domain $M$ admits an exact embedding into an exact convex symplectic manifold $X$, and the boundary $\partial M$ is displaceable in $X$, then the symplectic cohomology of $M$ vanishes and the chord conjecture holds for any Lagrangianly fillable Legendrian in $\partial M$. The TQFT respects the isomorphism between the symplectic cohomology of a cotangent bundle and the homology of the free loop space, so it recovers the TQFT of string topology.
\end{abstract}
\maketitle
\setcounter{tocdepth}{1}
\tableofcontents
%
\section*{Motivation and outline of the paper}
\label{Subsection Algebraic operations on symplectic cohomology}
%
Symplectic cohomology has become an important tool in symplectic topology, ever since its introduction in Viterbo's foundational paper \cite{Viterbo1}.
For example, it can be used to prove the existence of closed Hamiltonian orbits and Reeb chords, and it gives rise to obstructions to the existence of exact
Lagrangian submanifolds and of exact contact hypersurfaces. It also has theoretical importance: the \emph{wrapped Fukaya category} (a version of the Fukaya category for non-compact symplectic manifolds introduced in \cite{Fukaya-Seidel-Smith2}), is built using the wrapped Floer cohomologies, which are modules over the symplectic cohomology ring.

Symplectic cohomology is an invariant of exact symplectic manifolds
$(M,d\theta)$ whose boundary $\partial M$ is of contact type. Equivalently, after attaching a conical end, symplectic cohomology is an invariant of non-compact exact symplectic manifolds $(\overline{M},d\theta)$ which are \emph{conical at infinity}, meaning that outside of a bounded domain, $\overline{M}$ is symplectomorphic to a conical end 
$(\Sigma \times [1,\infty), d(R \alpha))
$
where $(\Sigma,\alpha)$ is a contact manifold, and $R$ is the coordinate on $[1,\infty)$. 

For example, a disc cotangent bundle $M=DT^*N$
of a closed Riemannian manifold $N$ with the standard symplectic form $\sum dp_j \wedge dq_j=d(\sum p_j dq_j)$ in local coordinates $(q,p)$, then $\overline{M}=T^*N$ is the cotangent bundle and $R=|p|$. 
In this example, Viterbo \cite{Viterbo1} proved that symplectic cohomology recovers the homology of the free loop space $\mathcal{L}N\! =\! C^{\infty}(S^1,N)$,
$$SH^*(T^*N;\Z/2) \cong H_{n-*}(\mathcal{L}N;\Z/2),$$
and Abbondandolo-Schwarz
\cite{Abbondandolo-Schwarz2} proved that $SH^*(T^*N;\Z/2)$ has a product structure, called ``pair-of-pants product", which  under the above isomorphism corresponds to the Chas-Sullivan loop product on $H_{*}(\mathcal{L}N;\Z/2)$ coming from
string topology \cite{Chas-Sullivan}.

In his influential survey on symplectic cohomology  \cite[Sec.(8a)]{Seidel}, Seidel described how to define a pair-of-pants product for any $M$, and more generally how to define TQFT operations
$$\psi_S:
SH^*(M)^{\otimes q} \to SH^*(M)^{\otimes p}  \qquad (p\geq 1,
q\geq 0).$$
The idea is to count maps $u: S \to \overline{M}$ from a Riemann surface $S$ as
in Figure \ref{Figure Introduction}, carrying $p$ negative punctures, $q$ positive punctures, and satisfying a perturbed
Cauchy-Riemann equation of the form $(du-X\otimes \beta)^{0,1}=0$. Here $X$ is a Hamiltonian vector field, and $\beta$ is a $1$-form on $S$ satisfying $d\beta\leq 0$ such that near the punctures the equation turns into Floer's equation $\partial_s u + J (\partial_t u - cX) = 0$ for constant weights $c>0$ depending on the puncture.

\begin{figure}[ht]
\includegraphics[scale=0.55]{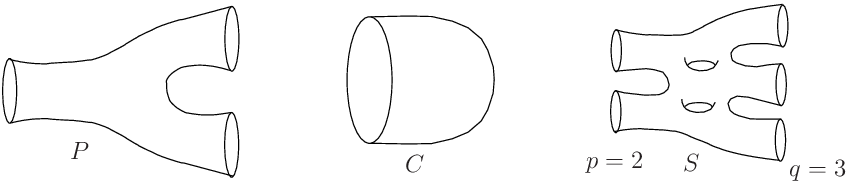}
\caption{Pair of pants $P$; cap $C$; genus $2$ surface $S$ with
$(p,q)\!=\!(2,3)$.} \label{Figure Introduction}
\end{figure}

\begin{figure}[ht]
\input{tqft_sh_comparison2e.pstex_t}
\caption{Summary of the TQFT structure.} \label{Figure TQFT table}
\end{figure}

For closed symplectic manifolds such a TQFT was constructed on Floer
cohomology by Piunikhin-Salamon-Schwarz \cite{PSS}
and in detail by Schwarz \cite{Schwarz}, over fields of characteristic $2$. In this closed setup, one can just require such $u$ to be holomorphic away from the punctures, so $du^{0,1}=0$, and then one interpolates this equation with Floer's equation near the punctures. Since symplectic cohomology is defined as a direct limit of Floer cohomologies for Hamiltonians $H:\overline{M} \to \R$ which are linear in the $R$-coordinate at infinity, one might at first imagine that the construction of the TQFT easily carries over. Unfortunately, this is not the case: this construction would run into the danger that a sequence of curves may escape to infinity. So the use of $\beta$ and the condition $d\beta \leq 0$ are fundamental in preventing this problem, and this somewhat complicates the construction of the TQFT.

For cotangent bundles, Abbondandolo-Schwarz
\cite[Sec.3.2]{Abbondandolo-Schwarz2} construct the product while circumventing the use of $\beta$, by using a clever construction of a pair-of-pants surface $S$ which has explicit time coordinates. For general $M$, the product is discussed by Seidel \cite[Sec.(8a)]{Seidel} and by McLean \cite[Sec.2.3]{McLean}, who proved in \cite[Sec.10]{McLean} how the product behaves under Viterbo restriction and under boundary connected sums. A key ingredient, namely preventing that solutions of $(du-X\otimes \beta)^{0,1}=0$ escape to infinity, is due to Abouzaid-Seidel \cite[Lemma 7.2]{Abouzaid-Seidel}, and we recall this in the appendix Section \ref{Appendix Maximum principle and no escape lemma}.

Applications in symplectic topology typically arise from proving a vanishing result for $SH^*(M)$. For example, Viterbo \cite{Viterbo1} proved that the vanishing of $SH^*(M)$ implies that there is a closed Reeb chord in $\partial M$, which is the \emph{Weinstein Conjecture}. It is for this reason that the most important part of the TQFT
is the unit and the product. Indeed
$SH^*(M)$ vanishes if and only if the unit vanishes, since then
$y=y\cdot 1 = y\cdot 0 =0$ for any $y\in SH^*(M)$. Secondly, if
$W\subset \overline{M}$ is a subdomain with contact type boundary then there is a unital ring homomorphism called  \emph{Viterbo restriction},
$$\varphi:SH^*(M)\to SH^*(W).$$
So if $SH^*(M)=0$ then $1=\varphi(1) = \varphi(0)=0$ in $SH^*(W)$, and so
$SH^*(W)=0$. This circle of ideas is due to
Viterbo, Seidel and McLean \cite{Viterbo1,Seidel,McLean}. We will recall this in \textbf{Sections \ref{Section Viterbo Functoriality}-\ref{Section Vanishing criteria}}.\\[2mm]
\indent We now briefly outline the structure of the paper. Since a lot of the analysis and applications are based on previous work, this will give us an opportunity to properly credit the existing literature. A detailed Introduction will follow in \textbf{Section \ref{Section Introduction}}.

One of the goals of this paper is to put in written form a detailed constrution of the TQFT structure on symplectic cohomology. Since a large part of this material was either known or expected within the circle of specialists in symplectic cohomology, we decided to isolate this part in the first Appendix (\textbf{Section \ref{Section Floer solutions}}). For the convenience of the reader, we summarized this Appendix at the start of \textbf{Section \ref{Section TQFT overview}}.

We will recall the basic definitions and conventions in \textbf{Sections \ref{Section Liouville domains}-\ref{Section Canonical map from ordinary cohomology}}: the definition of exact symplectic manifolds conical at infinity (\textbf{Sec.\ref{Section Liouville domains}}); the construction of symplectic cohomology $SH^*(M)$ and symplectic homology $SH_*(M)$ (\textbf{Sec.\ref{Section Symplectic cohomology}}); the relative analogue of $SH^*$ with Lagrangian boundary conditions, called the \emph{wrapped Floer cohomology} $HW^*(L)$, for Lagrangians $L\subset M$ intersecting $\partial M$ in a Legendrian submanifold (\textbf{Sec.\ref{Section Wrapped Floer cohomology}}); and the maps
$$c^*:H^*(M) \to SH^*(M) \; \textrm{ and }\; c^*:H^*(L) \to HW^*(L)$$
from the ordinary cohomology to the symplectic and wrapped cohomologies (\textbf{Sec.\ref{Section Canonical map from ordinary cohomology}}).

The initial motivation of the paper was to construct algebraic structures on a deformation of symplectic cohomology called \emph{twisted symplectic cohomology} introduced by the author \cite{Ritter}, with the purpose of continuing the study of exact contact hypersurfaces initiated by Cieliebak-Frauenfelder \cite{Cieliebak-Frauenfelder}. 
The twisted symplectic cohomology $SH^*(M)_{\alpha}$ is an invariant of $M$ associated to a class $\alpha \! \in  \! H^1(\mathcal{L}M)$ on the free loop space, and for $\alpha  \! =  \! 0$ it recovers $SH^*(M)$. In fact it is the Novikov cohomology theory
applied to $SH^*$, which involves using twisted coefficients in a bundle of Novikov rings defined in Section \ref{Subsection Novikov bundles of coefficients}.
This twisting yields very concrete applications in symplectic topology, for example in \cite{Ritter} we used it to prove that there are no exact Klein bottles in $T^*S^2$.

There is no TQFT on $SH^*(M)_{\alpha}$ in general because there is no canonical way of viewing surfaces in $M$ (such as pairs of pants) as chains in $\mathcal{L}M$.
Even when a TQFT can be defined on $SH^*(M)_{\alpha}$, it may not possess a unit for the same reasons that the Novikov cohomology of
a manifold usually does not. However, in \textbf{Section \ref{Section Applicatoin 2 Exact contact hypersurf}} we prove that if the form $\alpha \in
H^1(\mathcal{L}M)$ is the transgression of a class $\eta\in
H^2(M)$, then the TQFT structure exists and possesses a unit.
In \textbf{Section \ref{Section Viterbo Functoriality}} we prove that the twisted Viterbo restriction maps 
$\varphi:SH^*(M)_{\eta}\to SH^*(W)_{\eta|_W}$
from \cite{Ritter} are TQFT maps, in particular they are unital ring
homomorphisms. Our interest in the twisted theory stems from the fact that for $\pi_1(N)=1$ we have a vanishing result \cite{Ritter}:
$$SH^*(T^*N;\Z/2)_{\eta} \cong H_{n-*}(\mathcal{L}N;\Z/2)_{\eta} = 0 \, \textrm{ for non-zero
}\eta\in H^2(N).
$$
The applications to exact contact hypersurfaces in \textbf{Section \ref{Section Applicatoin 2 Exact contact hypersurf}} are based on this result.

In \textbf{Section \ref{Section Application Arnol'd Chord conjecture}} we give a family of examples where the \emph{Arnol'd chord conjecture} is satisfied, namely the existence of Reeb chords in $\partial M$ with ends on certain Legendrian submanifolds. This application follows from the module structure of $HW^*(L)$ over $SH^*(M)$ (\textbf{Section \ref{Subsection Wrapped HW is an SH module}}) and it relies on vanishing criteria for these cohomology groups (\textbf{Section \ref{Section Vanishing criteria}}). Some of these examples, but not all, exploit the twisted theory.

The application of the TQFT to prove that $SH^*(M)$ vanishes if $\partial M$ has certain displaceability properties (\textbf{Section \ref{Section Displacement implies vanishing}}) is independent of the twisted theory. Rather, it relies on foundational work due to Cieliebak-Frauenfelder \cite{Cieliebak-Frauenfelder} and Cieliebak-Frauenfelder-Oancea \cite{Cieliebak-Frauenfelder-Oancea}. Namely, \cite{Cieliebak-Frauenfelder} constructs an invariant of $M$ called \emph{Rabinowitz Floer cohomology} $RFH^*(M)$, and proves that it vanishes under certain displaceability properties of $\partial M$; and \cite{Cieliebak-Frauenfelder-Oancea} constructs a long exact sequence relating $SH_*(M)$, $SH^*(M)$ and $RFH^*(M)$. Our result then follows from investigating the role played by the unit in the long exact sequence and then exploiting the vanishing criteria in \textbf{Section \ref{Section Vanishing criteria}}. This argument is slightly trickier when one makes no simplifying assumption on the first Chern class, since then one loses the $\Z$-grading on $SH^*$: in this case we need to exploit the fact that  $c^*:H^*(M) \to SH^*(M)$ respects the product structure, which is interesting in its own right. In fact we prove in \textbf{Section \ref{Section c* maps preserve TQFT}} that the $c^*$ maps are compatible with a TQFT structure on $H^*(M)$. Section \ref{Section c* maps preserve TQFT} relies on generalizing the PSS-maps from \cite{PSS} to the non-compact setting required by $SH^*$ by using the analytical machinery in the foundational work of Salamon-Zehnder \cite{Salamon-Zehnder}, which we will briefly recall.

\textbf{Section \ref{Section Application Loop space homology and string topology}}
 proves that the isomorphism $SH^*(T^*N;\Z/2) \cong H_{n-*}(\mathcal{L}N;\Z/2)$ preserves the TQFT structure and the deformed TQFT structure. We include this Section mainly for theoretical interest. The result is actually a rather formal consequence of the TQFT axioms which follows from the foundational work of Abbondandolo-Schwarz \cite{Abbondandolo-Schwarz2} which proved that the isomorphism preserves the pair-of-pants product. There is one technical aspect worth mentioning: in the work of Abbondandolo-Schwarz, symplectic cohomology is defined as the Floer cohomology of one Hamiltonian of quadratic growth in $R$ using a certain non-contact type almost complex structure $J$, whereas in our work we always use contact type $J$ and we use linear growth Hamiltonians and then take a direct limit of the resulting Floer cohomologies. We reconcile the two approaches for general $\overline{M}$ in the appendix \textbf{Section \ref{Appendix Using non-linear Hamiltonians}}.

Appendix \textbf{Section \ref{Appendix Coherent orientations}} discusses orientation signs for $SH^*(M)$. Orientation signs were constructed for Floer cohomology by Floer-Hofer \cite{Floer-Hofer-Coherent-Orientations} (which immediately generalizes to $SH^*(M)$) but they have not yet been written up for the TQFT structure. However, for the TQFT signs one can mimic the construction of orientation signs used in Symplectic Field Theory, which was carried out by Eliashberg-Givental-Hofer \cite{Eliashberg-Givental-Hofer} and Bourgeois-Mohnke \cite{Bourgeois-Mohnke}.

\textbf{A word about cohomological conventions.} We explain \emph{(1)} why we preferred to use symplectic \emph{cohomology} and \emph{(2)} why we define $SH^*$ as the direct limit of Floer cohomologies (rather than the inverse limit, which for example is the convention of \cite{Cieliebak-Frauenfelder-Oancea}):

\emph{1)} The symplectic cohomology $SH^*$ is much better behaved algebraically because it possesses a unit, which is crucial in applications. While $SH_*$ only has a counit. Moreover, $SH_*$ can be recovered from $SH^*$ by dualization, but not vice-versa. 

\emph{2)} The conventions we use to define Floer \emph{cohomology} actually go back historically to Floer's original papers. But this convention also happens to be consistent with the fact that there is a map $c^*: H^*(M) \to SH^*(M)$ which respects the unit and product structures, all of which naturally respect our grading conventions when $SH^*$ is $\Z$-graded. In this convention, Viterbo's map $SH^*(M) \to SH^*(W)$ can be interpreted as a restriction map generalizing the ordinary restriction map $H^*(M) \to H^*(W)$ induced by the inclusion $W\subset M$ (Viterbo \cite{Viterbo1} originally interpreted it as a \emph{transfer map} $FH^*(W) \to FH^*(M)$ because for disc cotangent bundles $W,M$ it defines a transfer map on the cohomology of the free loop spaces).

We point out however that the other conventions are also very reasonable, for example the conventions of \cite{Cieliebak-Frauenfelder-Oancea} are particularly well-suited for studying cotangent bundles since in their conventions one has $SH_*(T^*N;\Z/2) \cong H_*(\mathcal{L}N;\Z/2)$.\\[1mm]
\noindent \textbf{Acknowledgements:}
I thank Alberto
Abbondandolo, Mohammed Abouzaid, Gabriel Paternain, Matthias
Schwarz, and Ivan Smith for many stimulating conversations. I also thank the anonymous referee for many useful suggestions which improved the exposition.
\section{Introduction}
\label{Section Introduction}
%
%
\subsection{Topological quantum field theory}
\label{Subsection TQFT axioms}
%
%
Formally speaking, a \emph{TQFT} \cite{Atiyah} is a tensor functor from the category
of $2$-dimensional orientable cobordisms between $1$-manifolds to the
category of vector spaces. We make this explicit in our setup: to
a circle we associate the vector space $SH^*(M)_{\eta}$, the
symplectic cohomology of $M$ computed over a field $\K$ of
coefficients (when twistings are present, we keep track of $\eta\in H^2(M)$, and we work over the Novikov field $\Lambda$, which is a $\K$-algebra defined in \ref{Subsection Novikov bundles of coefficients}). To the disjoint union of $p$ circles we associate the
tensor product $SH^*(M)_{\eta}^{\otimes p}$ of $p$ copies of
$SH^*(M)_{\eta}$, and to the empty set we associate the base field
$\K$. On morphisms, the functor is
$$
(\textrm{Punctured Riemann surface } S)\mapsto (\textrm{Operation
} \psi_S: SH^*(M)_{\eta}^{\otimes q} \to SH^*(M)_{\eta}^{\otimes
p}),$$
where $S$ represents a cobordism between $p\geq 1$ and $q\geq 0$
circles (Figure \ref{Figure Introduction}). Functoriality means
that $\psi_Z=\textrm{id}$ for cylinders $Z:p=q=1$, and that
compositions are respected: gluing surfaces along the cylindrical ends near the punctures
yields compositions of $\psi_S$ maps. The QFT is topological
since $\psi_S$ only depends on $p,q$ and the genus of $S$.

\begin{theorem*} $SH^*(M)_{\eta}$ is a $(1\!+\!1)$-dimensional TQFT,
except it does not possess operations for $p=0$ $($so no
counit$)$.
\end{theorem*}

\noindent The $\psi_S$ are not defined for $p=0$ due to a
non-compactness issue. Indeed if $p=0$ were allowed, then the TQFT axioms would imply $SH^*(M)$ is finite dimensional,\footnote{\emph{Proof:} let $a_{\nu}$ be a basis for the $\K$-vector space $SH^*(M;\K)$. Denote $S_{pq}$ the surface of genus $0$ with $p$ negative and $q$ positive punctures. Observe that $S_{20}$ arises from gluing $Q=S_{21}$ onto $C=S_{10}$. This determines a \emph{finite} sum $\psi_{S_{20}}(1) = \sum k_{ij} \, a_i\otimes a_j$ where $k_{ij}\in \K$. Since $S=(Z \sqcup S_{02}) \# (S_{20}\sqcup Z)$ is actually a cylinder, it acts by the identity. So, for any $x\in SH^*$, $x=\psi_{S}(x)=\psi_{Z\sqcup S_{02}}(\sum k_{ij} \, a_i\otimes a_j \otimes x)= \sum (k_{ij}\, \psi_{S_{02}}(a_j,x))\, a_i$. But this means $SH^*$ is spanned by just finitely many vectors $a_i$, so $SH^*$ is finite-dimensional. $\qed$} which is false for simply connected cotangent bundles. This differs from the Floer
cohomology of a \emph{closed} symplectic manifold $M$: $p=0$ is
legitimate, and $FH^*(M)\cong H^*(M)$ is finite dimensional.

The untwisted TQFT is summarized in Figure \ref{Figure TQFT table} (in Section \ref{Section Symplectic cohomology}
we will define $SH_*(M)$ and prove that it is canonically the dual of $SH^*(M)$, but not vice-versa).
\begin{corollary*}
There is a graded-commutative associative unital ring structure on
$SH^*(M)_{\eta}$ with product $\psi_P$ and unit $e=\psi_C(1)$ $($this structure can also be defined
if we replace $\K$ by $\Z)$.
\end{corollary*}
%
\subsection{TQFT structure on ordinary cohomology}
%
A TQFT exists also on $H^*(M)\cong H^*(\overline{M})$
if we use $1$-dimensional oriented cobordisms (we allow cobordisms which are not strictly speaking $1$-manifolds, but which are oriented graphs). The ordinary
cohomology can be identified with the Morse cohomology $MH^*(f)$ of any Morse
function $f: \overline{M}\to \R$ for which $-\nabla f$ points inwards along the conical end of $\overline{M}$. Replace surfaces $S$ by directed graphs
$S'$. Consider for example the Y-shaped
 graph $P'$ in Figure \ref{Figure Graphs}.
Assign a generic Morse function $f_i$ to each edge. The count of
isolated negative gradient flow lines along the graph defines
$$\psi_{P'}:MH^*(f_2)\otimes MH^*(f_3) \to
MH^*(f_1).$$

After identifications with ordinary cohomology, this is the cup
product on $H^*(M)$. Indeed $H^*(M)$ is a $(0+1)$-dimensional
TQFT: to a point associate $H^*(M)$, and to a cobordism between
$p+q$ points represented by a directed graph $S'$ associate
$$\psi_{S'}:H^*(M)^{\otimes q} \to H^*(M)^{\otimes p}.$$
This TQFT is well-known, see for example Betz-Cohen
\cite{Betz-Cohen} and Fukaya \cite{Fukaya}.
%
%
%
\subsection{Wrapped Floer cohomology}
\label{Subsection Intro Wrapped Floer cohomology}
%
%
\emph{Wrapped Floer cohomology} $HW^*(L)$ is an invariant of an
exact Lagrangian $L \subset (M,d\theta)$ having Legendrian
intersection $\partial L = L \cap \partial M$ with $\partial M$. Special cases for $\overline{M}=T^*N$ were
introduced by Abbondandolo-Schwarz \cite{Abbondandolo-Schwarz},
the general definition arises in Fukaya-Seidel-Smith
\cite{Fukaya-Seidel-Smith2}, and a detailed construction is in
Abouzaid-Seidel \cite{Abouzaid-Seidel}. The construction in
\cite{Abouzaid-Seidel} is more complicated than ours because the
authors' aim was to construct an $A_{\infty}$-structure at the
chain level.

The construction of $SH^*(M)$ involves closed Hamiltonian orbits
whereas $HW^*(L)$ involves open Hamiltonian orbits with ends on
$L$. This open-closed string theory analogy is a dictionary to
pass from $SH^*$ to $HW^*$: we now use half the surface obtained
after cutting Figure \ref{Figure Introduction} with a vertical
plane, so $S$ has boundary components. We count maps $u:S\to M$
satisfying the Lagrangian boundary condition $u(\partial S)
\subset \overline{L}$ (where one has extended $L$ to the conical end of $\overline{M}$), converging to open Hamiltonian orbits at the $p+q$
punctures, yielding:
$$
\mathcal{W}_S: HW^*(L)^{\otimes q} \to HW^*(L)^{\otimes p}. \qquad
(p\geq 1, q\geq 0)
$$
One can also define open-closed operations by making $p_{\mathrm{int}}+q_{\mathrm{int}}$ punctures in the interior of $S$. For a disc $S=D$ with $(p,q;p_{\mathrm{int}},q_{\mathrm{int}})=(1,1;0,1)$ this yields a module structure
$$
\mathcal{W}_D: SH^*(M)\otimes HW^*(L) \to HW^*(L).
$$

This paper will introduce a new group: the \emph{twisted wrapped Floer cohomology},  $HW^*(L)_{\eta}$, which depends on a closed two-form $\eta \in H^2(M,L)$. This recovers $HW^*(L)$ when $\eta=0$. Denote $\overline{\eta}$ the image of $\eta$ under $H^2(M,L) \to H^2(M)$.
\begin{theorem*}
$HW^*(L)_{\eta}$ has a TQFT structure, part of which is a unital ring
structure. It is an $SH^*(M)_{\overline{\eta}}$-module via $\mathcal{W}_D$. The Viterbo restriction maps $HW^*(L)_{\eta} \to HW^*(L\cap W)|_{\eta|W}$ preserve the TQFT and module structure.
\end{theorem*}
\begin{corollary*}
If $SH^*(M)_{\overline{\eta}}\!=\!0$ then $HW^*(L)_{\eta}\!=\!0$.
\end{corollary*}
%
%
%
%
\subsection{The Arnol'd chord conjecture}
\label{Introduction on Arnold conjecture}
%
This conjecture states that a contact manifold containing a
Legendrian submanifold must contain a Reeb chord with ends on the
Legendrian. We always deal with the Legendrian $\partial L$ of
\ref{Subsection Intro Wrapped Floer cohomology}:
\begin{theorem*} If $SH^*(M)_{\overline{\eta}}=0$ or $HW^*(L)_{\eta}=0$ then the chord
conjecture holds, and for a generic contact form there are at
least $\textrm{rank}\; H^*(L)$ chords.
\end{theorem*}

This applies for example to subcritical Stein manifolds $M$ since
$SH^*(M)=0$ by Cieliebak \cite{Cieliebak}. For subcritical Stein $M$, the
existence of one Reeb chord for any Legendrian $K\subset
\partial M$ is due to Mohnke \cite{Mohnke}. We will also deduce that:

\begin{corollary*}
If $M$ admits an exact embedding into an exact convex symplectic
manifold $X$ (such as $\overline{M}$), and $\partial M$ is 
displaceable by a compactly supported Hamiltonian flow in
$X$, then the chord conjecture holds for any $\partial L$, and
generically there are $\textrm{rank}\; H^*(L)$ chords.
\end{corollary*}

\begin{theorem*}
For $L\subset M=DT^*N$, with $N$ closed and simply
connected, such that $H^2(T^*N) \to  H^2(L)$ is not
injective, the chord conjecture holds and generically there are
$\geq \textrm{rank}\; H^*(L)$ chords. It also holds after
attaching subcritical handles to $DT^*N$.
\end{theorem*}
An ALE space $\overline{M}$ is a simply connected
hyperk\"{a}hler $4-$manifold which at infinity looks like $\C^2 / G$ for
a finite subgroup $G\subset SL(2,\C)$. Such $\overline{M}$ arise by attaching a conical end to the plumbing $M$ of copies of $DT^*S^2$ according to ADE Dynkin diagrams \cite{Ritter2}.
\begin{theorem*}
For any ALE space the chord conjecture holds for any $\partial L$
and generically there are at least $\textrm{rank}\; H^*(L)$ Reeb
chords.
\end{theorem*}
%
%
\subsection{Obstructions to exact contact embeddings}
%
%
An embedding $j:\Sigma^{2n-1}\hookrightarrow (M^{2n},d\theta)$ is an
\emph{exact contact embedding} if there is a contact form
$\alpha$ on $\Sigma$ with $\alpha-j^*\theta=\textrm{exact}$. For
example, if $L\subset M$ is a closed exact Lagrangian, then a Weinstein
neighbourhood $W\cong DT^*L$ yields an exact contact hypersurface
$ST^*L\cong \partial W \hookrightarrow M$. Using the deformed TQFT we prove the following theorem: (stronger results are discussed in \ref{Subsection Symplectic cohomology obstructions simply
connected case}, \ref{Subsection Symplectic cohomology
obstructions})

\begin{theorem*}
Let $L,N$ be closed simply connected $n$-manifolds, $n\geq 4$. For
any exact contact embedding $ST^*L \hookrightarrow T^*N$, the
following hold
\begin{enumerate}
\item $H^2(N) \to H^2(L)$ is injective;

\item $\pi_2(L) \to \pi_2(N)$ has finite cokernel;

\item if $H^2(N)\neq 0$ then $H_*(L) \cong H_*(W)$ for the filling
$W$ of $ST^*L\subset T^*N$.
\end{enumerate}
\end{theorem*}

\begin{conjecture*}
For simply connected $L,N$ of dimension $\geq 4$, all exact contact $ST^*L \hookrightarrow T^*N$
always arise as the boundary of a Weinstein neighbourhood of an
exact Lagrangian $L\hookrightarrow T^*N$.
\end{conjecture*}
%
\subsection{Displaceability of contact hypersurfaces}
%
By exploiting recent literature on the Rabinowitz-Floer cohomology $RFH^*(M)$ \cite{Cieliebak-Frauenfelder,Cieliebak-Frauenfelder-Oancea}, and after proving that $RFH^*(M)=0$ if and only if $SH^*(M)=0$, we will deduce:
\begin{corollary*}
If $\partial M$ is displaceable by a compactly supported Hamiltonian flow in $\overline{M}$ then $SH^*(M)\!=\!0$, so there
are no closed exact Lagrangians in $M$. This also holds if $M$ exactly embeds into an exact convex symplectic manifold $X$, and $\partial M\!\subset\! X$ is displaceable.
\end{corollary*}
%
%
\subsection{String topology}
%
The \emph{Pontryagin product} on the homology of the space $\Omega N$ of based loops in $N$
is defined by concatenating loops to form figure-8 loops (see for example \cite{Abbondandolo-Schwarz2}). The analogue on the homology of the space
$\mathcal{L}N$ of free loops in $N$ is called
\emph{Chas-Sullivan loop product} \cite{Chas-Sullivan}: given two
families of loops, one forms the family of all possible figure-8
loops obtained when the two base-points happen to coincide.

Abbondandolo-Schwarz \cite{Abbondandolo-Schwarz2} proved that the
products on $HW^*(T_q^*N\subset T^*N;\Z/2)$ and $SH^*(T^*N;\Z/2)$ agree with
the Pontryagin and Chas-Sullivan products on $H_{n-*}(\Omega N;\Z/2)$
and $H_{n-*}(\mathcal{L}N;\Z/2)$ via the respective isomorphisms. We
extend the result to:

\begin{theorem*}
$SH^*(T^*\!N;\Z/2)_{\eta} \!\cong\! H_{n-*}(\mathcal{L}N;\Z/2)_{\eta}$,
$HW^*(T_q^*\!N\!\!\subset\! T^*\!N;\Z/2)_{\eta}\! \cong\! H_{n-*}(\Omega N;\Z/2)_{\eta}$ respect the
TQFT structures, and the units are $[N]$ and $[\textrm{base-point}]$ respectively.
\end{theorem*}
%
%
%
%
%
%
%
%
\section{Liouville domains}
\label{Section Liouville domains}
%
\subsection{Reeb periods, the contact type condition, Hamiltonians}
\label{Subsection Liouville domains Definition}
%
%
%
%
A \emph{Liouville domain} $(M^{2n},\omega=d\theta)$ is a compact exact symplectic manifold with boundary, such that the Liouville vector field $Z$ defined by $\omega(Z,\cdot) =\theta$ points strictly outwards along $\partial M$. The form $\alpha=\theta|_{\partial M}$ is a contact form on $\partial M$, and it defines the Reeb vector field $\mathcal{R}$ on $\partial M$ by the conditions
$$
\alpha(\mathcal{R}) = 1,\; d\alpha(\mathcal{R},\cdot)=0.
$$
The closed orbits of $\mathcal{R}$ on $\partial M$ are called
\emph{Reeb orbits}, and the (non-zero) periods of the orbits are called \emph{Reeb periods}. For a generic choice of $\alpha$ (e.g. a
generic choice of $\theta$ subject to $d\theta=\omega$) the
Reeb periods form a countable closed subset of
$[0,\infty)$, which we will assume.

The symplectization $\overline{M}=M \cup_{\partial M} [0,\infty)\times \partial M$ of $M$ is obtained by gluing a conical \emph{collar} onto $\partial M$.
The flow of $Z$ for small time $r\leq 0$ parametrizes a neighbourhood $(-\varepsilon,0]\times \partial M$ of $\partial M$, so all the data naturally extends to $\overline{M}$ by $Z=\partial_r$, $\theta=e^r \alpha$, $\omega=d\theta$. 
The flow of $Z$ for time $r\in [-\infty,\infty)$ starting from $\partial M$ defines
the coordinate $R=e^r \in [0,\infty)$ on $\overline{M}$ with $\partial M = \{ R=1 \}$. We
will use $R$ instead of $r$ from now on, so $$\overline{M}=M\cup
[1,\infty)\times \partial M.$$

Let $J$ be an $\omega$-compatible almost complex structure on
$\overline{M}$. Let $g=\omega(\cdot,J\cdot)$ denote the
$J$-invariant metric on $\overline{M}$. We always assume that $J$ is of
\emph{contact type} for large $R$:
$$
J^*\theta = dR \qquad (\textrm{equivalently: }J Z = \mathcal{R}).
$$
Since symplectic cohomology is invariant under deformations of such $J$, the particular choice of $J$ will not matter, because the space of such $J$ is contractible \cite[Prop.2.51]{McDuff-Salamon-IntroToSymplTop}.

For $H \in C^{\infty}(\overline{M},\R)$, define the
\emph{Hamiltonian vector field} $X=X_H$ by
$ \omega(\cdot,X) = dH.$
We call \emph{$1$-orbits of $H$} the $1$-periodic Hamiltonian orbits $x:S^1 \to \overline{M}$, so
$\dot{x}(t) = X(x(t))$.

We always assume $H$ is $C^2$-small and Morse inside $M$. So the
only $1$-orbits inside $M$ are constants: the critical
points of $H$. We assume that for large $R$, $H=h(R)$ depends only on $R$. So $X= h'(R) \mathcal{R}$, so the $1$-orbits $x$ of
$X$ on the collar have constant $R=R(x)$ and they correspond to
the Reeb orbits $y$ of period $T=h'(R)$ via $y(t)=x(t/T)$. We
assume that for large $R$, $h$ becomes linear in $R$ with slope not equal to a Reeb period: hence in the region at infinity where $h$ is linear there are no $1$-orbits.
%
\subsection{Action functional}\label{Subsection Action functional}
%
Let $\mathcal{L}\overline{M} = C^{\infty}(S^1,\overline{M})$ be
the space of free loops in $\overline{M}$. Let ${\mathbb{A}}_H$ denote the
$H$-perturbed action functional for $x\in
\mathcal{L}\overline{M}$,
$$
\textstyle {\mathbb{A}}_H(x) = - \int x^*\theta + \int_0^1 H(x(t)) \, dt.
$$
The differential of ${\mathbb{A}}_H$ at $x\in \mathcal{L}\overline{M}$ is
$ \textstyle d{\mathbb{A}}_H \cdot \xi = - \int_0^1 \omega(\xi, \dot{x} - X)
\, dt
$
in the direction of $\xi \in T_x\mathcal{L}\overline{M} =
C^{\infty}(S^1,x^*T\overline{M})$.
So the critical points of ${\mathbb{A}}_H$ are the $1$-orbits
of $H$. 

For a $1$-orbit $x$ on the collar in $\{
R\}\times \partial M$, the action is ${\mathbb{A}}_H(x)=\mathbb{A}_h(R)=-R h'(R)+h(R)$.
%
%
%
\subsection{Floer trajectories}\label{Subsection Floers Equation}
With respect to the $L^2$-metric $\int_0^1 g(\cdot,\cdot) \, dt$
the gradient is $\nabla {\mathbb{A}}_H = J(\dot{x} - X)$. For $u:\R \to
\mathcal{L}\overline{M}$, or equivalently $u: \R \times S^1 \to
\overline{M}$, the negative $L^2$-gradient flow equation $\partial_s u =
-\nabla {\mathbb{A}}_H(u)$ in the coordinates $(s,t) \in \R \times S^1$ is
$$
\partial_s u + J(\partial_t u - X) = 0 \quad \textrm{(Floer's equation)}.
$$
Let $\widehat{\mathcal{M}}(x_{-},x_{+})$ denote the solutions $u$
converging to $1$-orbits
$x_{\pm}$ of $H$ at the ends $s\to \pm \infty$. Then
$\mathcal{M}(x_{-},x_{+})=\widehat{\mathcal{M}}(x_{-},x_{+})/\R$
denotes the moduli space of \emph{Floer trajectories}, where we
identify $u(\cdot,\cdot)\sim
u(\cdot+\textrm{constant},\cdot)$ (the $\R$-reparametrization freedom).
%
\subsection{Energy}\label{Subsection Energy for Floer traj}
The energy $E(u)=\int |\partial_s u|^2 \, ds \wedge dt$ of a Floer trajectory $u\in \mathcal{M}(x_{-},x_{+})$
satisfies
$$
\begin{array}{lll}
E(u) 
& = & \int d\theta(\partial_s u,
\partial_t u - X)\, ds \wedge dt \\
& = & \int u^*d\theta-dH(\partial_s u)\, ds\wedge dt\\
& = & \int_0^1 (\theta(\dot{x}_+)-\theta(\dot{x}_-) - H(x_+) \, +
 H(x_-) ) \, dt\\
& = & {\mathbb{A}}_H(x_-)-{\mathbb{A}}_H(x_+).
\end{array}
$$
Thus we have an a priori energy estimate (in terms only of the
ends $x_{\pm}$ not of $u$).
%
%
\subsection{Transversality and
compactness}\label{Subsection Transversality and Compactness}
%
Standard Floer theory methods (see Salamon \cite{Salamon}, Floer-Hofer-Salamon \cite[Thm 5.1]{Floer-Hofer-Salamon}) show that for a
generic time-dependent perturbation $(H_t,J_t)$ of $(H,J)$ the
$1$-orbits are non-degenerate and the moduli spaces
$\mathcal{M}(x_{-},x_{+})$ are smooth manifolds. Write $\mathcal{M}_k(x_{-},x_{+})$
for the $k$-dimensional part of $\mathcal{M}(x_{-},x_{+})$.

\textbf{Convention.} We write $(H,J)$ even though one
actually uses a perturbed $(H_t,J_t)$.

\emph{Technical Remarks. A $1$-orbit of $H_t$ is \emph{non-degenerate} if $1$ is not an eigenvalue of the linearization of $\varphi_{H_t}^1$ (the time-$1$ flow of $X_{H_t}$). If the $1$-orbits are non-degenerate then they are isolated and $\mathcal{M}(x_{-},x_{+})$ is the zero set of a Fredholm map. For time-independent $H$ the $1$-orbits are non-degenerate iff $H$ is Morse and the $1$-orbits are critical points of $H$. So given $H$ or $H_t$, one typically needs to make a time-dependent perturbation to ensure non-degeneracy. \\ \indent
Suppose $H_t$ satisfies this non-degeneracy. Then, by \cite[Thm 5.1]{Floer-Hofer-Salamon}, we can ensure that all $\mathcal{M}(x_-,x_+)$ are smooth after either a generic time-dependent perturbation $H_t$ of $H$, or a generic time-dependent perturbation $J_t$ of $J$ (or perturbing both). We do not need to perturb $(H,J)$ in the region $R\gg 0 $ where $H$ is linear since there are no $1$-orbits in this region (recall Section
\ref{Subsection Liouville domains Definition}) and since no Floer trajectories enter this region by the following Lemma.}

\begin{lemma}\label{Lemma Maximum principle} Solutions of $\,\partial_s u + J(\partial_t u - X)=0$
converging to $x_{\pm}$ at the ends are entirely contained in the
region $R\leq \max \{R(x_{\pm}),R_0\}$ when $J$ is of contact type for $R\geq R_0$.
\end{lemma}
Lemma \ref{Lemma Maximum principle} is a consequence of a maximum principle (Lemma \ref{Lemma Maximum principle for Floer solns}). It ensures that all 
$u\in \mathcal{M}(x_-,x_+)$ stay in a compact region of $\overline{M}$.
So we reduce to checking whether the compactness proofs that
hold for closed symplectic manifolds (e.g. Salamon \cite{Salamon})
are applicable. Indeed, we have the two sufficient requirements:
an a priori energy estimate and a reason to exclude the
bubbling-off of $J$-holomorphic spheres (there are no non-constant
$J$-holomorphic spheres by Stokes' theorem since $\omega=d\theta$
is exact).
Thus the $\mathcal{M}(x_{-},x_{+})$ have
natural compactifications, whose boundaries are described by
broken Floer trajectories (see Figure \ref{Figure Breaking}). In
particular $\mathcal{M}_0(x_{-},x_{+})$ is already compact, so it
is a finite set of points called \emph{isolated solutions}.
%
%
%
%
%
%
\section{Symplectic cohomology and symplectic homology}
\label{Section Symplectic cohomology}
%
\subsection{Symplectic chain complex}\label{Subsection
Symplectic chain complex}
%
%
Pick a base field $\K$. Let $SC^*(H)$ denote the $\K$-vector space
generated by the $1$-orbits of $H$,
$$
SC^*(H) =\bigoplus \left\{ \K\, x : x \in
\mathcal{L}\overline{M},\; \dot{x}(t) = X(x(t)) \right\}.
$$
The differential $d$ on $SC^*(H)$ is the $\K$-linear map, which on a generator $y$ is defined by counting incoming isolated
Floer trajectories,
$$
d y = \sum_{u\in \mathcal{M}_0(x,y)} \epsilon_u\, x
$$
where $\epsilon_u\in \{ \pm 1\}$ are orientation signs (see
Section \ref{Appendix Coherent orientations}). By convention, the constant solution $u(s,t)=y(t) \in
\mathcal{M}_0(y,y)$ is not counted. A standard argument \cite{Salamon} shows that $d\circ d =0$ (see \ref{Subsection Using orientation signs to prove dd=0}), so we can define $SH^*(H)=H^*(SC^*(H);d)$.

When
gradings are defined (\ref{Subsection Maslov index and
Conley-Zehnder index}), $\textrm{dim}\,
\mathcal{M}(x,y)=|x|-|y|-1$, so $|x|=|y|+1$ in the above.
%
\subsection{Continuation maps}
\label{Subsection Floer continuation solutions} \label{Subsection
Continuation Maps}
%
Let $(H_{\pm},J_{\pm})$ be two choices of data for which we have defined $SH^*$ (the data may depend on $t\in S^1$).
Let $(H_z,J_z)_{z=(s,t)\in \R\times S^1}$ be an interpolation of the data such that for large $|s|$, $H_z=H_{\pm}$ and
$J_z=J_{\pm}$. Assume $J_z$ is of contact type for $R\gg 0$.

The moduli space
$\mathcal{M}^{H_s}(x_-,x_+)$ of \emph{Floer continuation
solutions} are the $v:\R \times S^1 \to
\overline{M}$ solving
$
\partial_s v + J_z(\partial_t v - X_{H_s}) = 0,
$
converging to 1-orbits $x_{\pm}$ of $H_{\pm}$ as $s\to \pm
\infty$. We make no identifications of solutions (there is no $\R$-reparametrization freedom since $J_z$ depends on $s$).

The
maximum principle (Lemma \ref{Lemma Maximum principle for Floer
solns}) holds for such $v$ if $(H_z,J_z)$ is a \emph{monotone homotopy}:
$$
\boxed{\textrm{for large } R \textrm{ we assume: }\;\; H_z = h_z(R),\;
\partial_s h'_z \leq 0,\; J_z^*d\theta=dR \;(\textrm{contact
type condition}).}
$$
In particular, the slopes $m_{\pm}$ of the $H_{\pm}$ at infinity must satisfy $m_+\leq m_-$.

As usual, we often conceal the $t$-dependence from the notation, so we talk about a homotopy $(H_s,J_s)$ from $(H_-,J_-)$ to $(H_+,J_+)$, instead of writing $(H_z,J_z)$. 

Say the boxed assumption holds for $R\geq R_0$, then all $v\in
\mathcal{M}^{H_s}(x_-,x_+)$ are contained in the compact set $C$
defined by $R\leq \max(R(x_{\pm}),R_0)$. Suppose that $H_s,J_s$
are $s$-dependent only for $a\leq s \leq b$, then we get an a
priori energy estimate for all $v\in \mathcal{M}^{H_s}(x_-,x_+)$,
$$
\begin{array}{lll}
E(v) & = &  \int |\partial_s v|_{g_s}^2 \, ds \wedge dt \\ & = &
{\mathbb{A}}_{H_-}(x_{-}) -
{\mathbb{A}}_{H_+}(x_{+}) + \int (\partial_s H_s)(v) \, ds \wedge dt \\
&  \leq &  {\mathbb{A}}_{H_-}(x_{-}) - {\mathbb{A}}_{H_+}(x_{+}) + (b-a)\cdot \sup_{m\in
C} |\partial_s H_s(m)|.
\end{array}
$$

For generic $(H_z,J_z)$, the $\mathcal{M}^{H_s}(x_-,x_+)$ are
smooth manifolds with compactifications by broken solutions (Floer trajectories for $H_{\pm}$ may break off at the respective ends of a continuation solution, see Figure \ref{Figure Breaking}). More precisely, as in the Technical Remarks in \ref{Subsection Transversality and Compactness}, given $(H_z,J_z)$ for which $(H_{\pm},J_{\pm})$ satisfy non-degeneracy, it suffices to make a $C^{2}$-small $z$-dependent perturbation of $H_z$ or $J_z$ (or both) away from the ends to ensure that $\mathcal{M}^{H_s}(x_-,x_+)$ is smooth.

For monotone $H_s$, counting incoming isolated Floer continuation solutions
defines a \emph{continuation map} $\varphi: SC^*(H_{+}) \! \to \!
SC^*(H_{-})$. On generators
$$
\varphi(x_+) = \sum_{v\in \mathcal{M}_0^{H_s}(x_-,x_+)}
\epsilon_v\, x_-
$$
where $\epsilon_v\in \{ \pm 1\}$ are orientation signs (see
Section \ref{Appendix Coherent orientations}). Then extend
$\varphi$ linearly. A standard argument \cite{Salamon} shows that
$\varphi$ is a chain map (see \ref{Subsection Orientation signs Floer chain maps}).\\
When gradings are defined (\ref{Subsection Maslov index and
Conley-Zehnder index}),
$\textrm{dim}\, \mathcal{M}^{H_s}(x_-,x_+) = |x_-|-|x_+|$, so
 $|x_-|=|x_+|$ in the above.

\begin{lemma}[See for example \cite{Ritter}]\label{Lemma Chain Homotopy}
\strut\begin{enumerate}
\item \label{Item Lemma Chain Homotopy does not matter which hpy} Changing the monotone homotopy $(H_s,J_s)$ changes $\varphi$ by a chain homotopy. So on cohomology $\varphi: SH^*(H_+) \to SH^*(H_-)$ is independent of the choice of $(H_s,J_s)$.
\item \label{Item Lemma Chain Homotopy composing conts is cont} $\varphi$ equals any composite $SH^*(H_{+}) \to SH^*(K) \to
SH^*(H_{-})$ induced by continuation maps for monotone
homotopies from $H_{-}$ to $K$ and from $K$ to $H_{+}$.
\item \label{Item Lemma Chain Homotopy constant H is identity} The constant homotopy $H_s=H$ induces the identity on $SH^*(H)$.
\item \label{Item Lemma Chain Homotopy same slopes imply iso} For $H_{\pm}$ linear at infinity of the same slope,
$\varphi$ is an isomorphism on cohomology.
\end{enumerate}
\end{lemma}
%
%
%
\subsection{Hamiltonians $H^m$ of slope $m$ at
infinity}\label{Subsection Hamiltonians Linear At Infty}
%
Let $H^m$ be any Hamiltonian equal to $mR+\textrm{constant}$ for
large $R$ (for generic $m>0$ so that $m$ is not a Reeb period). By Lemma \ref{Lemma Chain Homotopy},
$SH^*(H^m)$ only depends on $m$ and continuations
$SH^*(H^m)\to SH^*(H^{m'})$ exist for $m\leq m'$. 

On cohomology $SH^*(H^m)$ only depends on the slope $m$ by Lemma \ref{Lemma Chain Homotopy}(\ref{Item Lemma Chain Homotopy same slopes imply iso}), so if one wanted one could always take $H^m=mH$ for a fixed Hamiltonian $H: \overline{M}\to \R$ of slope $1$.
%
\subsection{Symplectic cohomology}
\label{Subsection Symplectic cohomology}
%
%
Define $SH^*(M)$ by taking the direct limit over the above continuation maps between the $SH^*(H)$ groups for Hamiltonians linear at infinity, 
$$
SH^*(M) = SH^*(\overline{M})= \varinjlim SH^*(H).
$$
So $SH^*(M)\cong
\displaystyle\varinjlim SH^*(H^{m_k})$ for any $H^{m_k}$ as in \ref{Subsection Hamiltonians Linear At Infty} with slopes $m_k\to \infty$ as $k\to \infty$.

In Theorem \ref{Theorem characterization of continuation maps} we prove that the continuation maps $SH^*(H^m) \to SH^*(H^{m'})$ are the pair-of-pants product by a special element $e_{H^{\ell}}\in SH^0(H^{\ell})$ for $\ell=m'-m\geq 0$.

In Appendix \ref{Appendix Using non-linear Hamiltonians} we
prove that $SH^*(M)$ can be defined as $SH^*(Q)$ for one Hamiltonian
$Q:\overline{M}\to \R$ growing faster than linearly at infinity. This makes many arguments
unnecessarily complicated (see \ref{Subsection contact type J make
transversality fail}) so for now we assume our Hamiltonians to be
linear at infinity.
%
%
\subsection{Invariance under symplectomorphisms of contact type} \label{Section Invariance under contactomorphisms}
%
For a proof of the invariance of $SH^*(M)$ we refer to Viterbo \cite{Viterbo1}, Cieliebak \cite{Cieliebak}, Seidel \cite{Seidel}. The following formulation is written in detail in \cite[Theorem 8]{Ritter2}, which is based on Seidel's exposition \cite[Sec.(3e)]{Seidel}.

Let $M,N$ be Liouville domains. A symplectomorphism $\varphi:
\overline{M} \to \overline{N}$ is of \emph{contact type} if
$$\varphi^*\theta_N = \theta_M + d(\textrm{compactly
supported function}).$$

It follows that at infinity $ \varphi(e^r,y) = (e^{r-f(y)},
\psi(y))$, where $y\in \partial M$, for a smooth $f: \partial M \to \R$
and a
contact isomorphism $\psi: \partial M \to
\partial N$ with $\psi^*\alpha_N = e^f \alpha_M$.

Under such a map $\varphi: \overline{M} \to \overline{N}$, the
Floer solutions on $\overline{N}$ for $(H,d\theta_N,J_N)$
correspond precisely to the Floer solutions on $\overline{M}$ for
$(\varphi^*H,d\theta_M,\varphi^*J_N)$. However, for $H$ on
$\overline{N}$ linear at infinity, $\varphi^*H(e^r,y) =
h(e^{r-f(y)})$ is not linear at infinity.

So $SH^*(N)$ is isomorphic via $\varphi^*$ to
$SH_f^*(M)=\varinjlim SH^*(H_f)$ calculated for Hamiltonians of
the form $H_f=h(R_f)$, with $R_f=e^{r-f(y)}$, using almost complex
structures $J$ satisfying $J^*d\theta=dR_f$. It still turns out
that $SH^*(M)\cong SH^*_f(M)$. This is proved by a continuation
argument by homotopying $f$ to zero and proving a maximum
principle for $R_f\circ u$ for such homotopies (e.g. see
\cite[Lemma 7]{Ritter2}).

\begin{lemma}\label{Lemma Invariance under Contactomorphs}
A contact type $\varphi: \overline{M} \cong \overline{N}$ induces
$\varphi_*:SH^*(M)\cong SH^*(N)$.
\end{lemma}
%
%
%
%
\subsection{Grading of symplectic cohomology}
\label{Subsection Maslov index and Conley-Zehnder index}
%
%
We refer to Robbin-Salamon \cite{Robbin-Salamon} and Salamon \cite{Salamon} for a
detailed exposition on the Maslov index and gradings in Floer theory.
We now explain that if $c_1(M)=c_1(TM,J)$ is zero, then $SH^*(H)$ is $\Z$-graded. In \ref{Subsection Choice of trivializations} we discuss dimension counts for Floer moduli spaces and we discuss the case $c_1(M)|_{\pi_2 M}=0$.

Since $c_1(M)=0$, we can pick a trivialization of the canonical
bundle $\mathcal{K}=\Lambda^{n,0}T^*M$ (\emph{Explanation: we can view $T^*M$ as a complex vector bundle of rank $n=(\dim_{\R} M)/2$ since we chose an almost complex structure $J$ on $M$. Then $\mathcal{K}$ is a complex line bundle with $c_1(\mathcal{K})=-c_1(M)$. If $H^1(M)=0$, then maps $M \to U(1)$ are contractible, so there is only one homotopy class of trivializations for $\mathcal{K}$, and so the above choice of trivialization will not matter.}).

Over any $1$-orbit $x$ for $H$, trivialize $x^*TM$ so that it induces an
isomorphic trivialization of $\mathcal{K}$. Let $\phi_t$ denote
the linearization $D\varphi^t_{X_H}(x(0))$ of the time $t$
flow for $X_H$ written in the trivializing frame for $x^*TM$.
Let $\textrm{sign}(t)$ denote the signature of the quadratic form
$
\omega(\cdot,\partial_t \phi_t\cdot):
\textrm{ker}(\phi_t-\textrm{id}) \to \R,
$
assuming we perturbed $\phi_t$ relative endpoints to make the
quadratic form non-degenerate and to make
$\textrm{ker}(\phi_t-\textrm{id})=0$ except at finitely many $t$.

The \emph{Maslov index} $\mu(x)$ of $x$ is
$
\mu(x) = \frac{1}{2}\, \textrm{sign}(0) + \sum_{0<t<1}
\textrm{sign}(t) + \frac{1}{2}\, \textrm{sign}(1).
$

The Maslov index is invariant under homotopy relative endpoints,
and it is additive with respect to concatenations. If $\phi_t$ is
a loop of unitary transformations, then its Maslov index is the
winding number of the determinant, $\det \phi_t: \mathcal{K} \to
\mathcal{K}$.

For example $\phi_t = e^{2\pi i t} \in U(1)$ for
$t\in [0,1]$ has Maslov index $1$.

So when $c_1(M)=0$, we can define the following 
$\Z$-grading $|\cdot|$ on generators of $SC^*(H)$:
\\[1mm]
\begin{tabular*}{\textwidth}{l@{\extracolsep{\fill}}cr@{\extracolsep{0pt}}} 
\strut & 
$
|x| = \frac{\textrm{dim}(M)}{2} - \mu(x).
$
 & \strut 
\end{tabular*}
\\[1mm]
We call this the \emph{Conley-Zehnder index} of $x$ (there are various conventions -- we explain in Remark \ref{Remark CZ convention} why we chose this convention). As mentioned in \ref{Subsection
Continuation Maps}, the continuation maps preserve this $\Z$-grading, so also $SH^*(M)$ is $\Z$-graded.

\indent When $c_1(M)\neq 0$ this still defines a $\Z/2$ grading since
the Conley-Zehnder indices change by an even integer in
$2 c_1(M)(\pi_2(M))$ when we change the trivialization over a Hamiltonian
orbit. In particular, Koszul signs $(-1)^{|x|}$ are well-defined (we will need this in
\ref{Subsection Using orientation signs to prove that TQFT maps
are chain maps}).
\begin{remark}\label{Remark CZ convention} Our convention ensures that $|x|$ equals the Morse
index of $x$ when $x$ is a critical point of a $C^2$-small Morse
$H:\overline{M}\to \R$. This ensures that: the unit of $SH^*(M)$ lies in
degree $0$; the product is additive on degrees; and the canonical map $c^*:H^*(M) \to SH^*(M)$ defined in Section \ref{Section Canonical map from ordinary cohomology} is degree-preserving. This differs
from Schwarz's convention \cite{Schwarz},
$\|x\|=n-\textrm{ind}_{\textrm{Morse}}(x)$, which explains why the
index of Theorem \ref{Theorem index of Fredholm operator} is $-\sum \|x_a\|+\sum \|y_b\|+n\cdot \chi(S)$ in
\cite{Schwarz}, with
$\chi(S)=2-2g-p-q$. Our convention actually agrees with Salamon's grading $\mu_{H}$ \cite[Exercise 2.8]{Salamon}, since we use an opposite sign convention for $H$ in the action functional.
\end{remark}
%
%
\subsection{Symplectic homology}
\label{Section symplectic homology}
Using the notation from \ref{Subsection Symplectic chain
complex} and \ref{Subsection Continuation Maps}, define:\\[1mm]
$
\strut\qquad\textstyle SC_*(H)=\prod \left\{ \K x : x \in \mathcal{L}
\overline{M},\; \dot{x}(t) = X(x(t)) \right\}
$
\\[1mm]
$
\strut\qquad\displaystyle \delta: SC_*(H) \to SC_{*-1}(H),\;\; \delta x =
\!\!\!\!\sum_{u\in \mathcal{M}_0(x,y)}\!\!\!\! \epsilon_u\, y
\quad\textrm{(compare \ref{Subsection Symplectic chain complex})}
$
\\[1mm]
$
\strut\qquad\displaystyle \varphi_*: SC_*(H_-) \to SC_*(H_+),\;\;
\varphi_*(x_-) =\!\!\!\!\!\!\!\! \sum_{v\in
\mathcal{M}_0^{H_s}(x_-,x_+)}\!\!\!\!\!\!\!\! \epsilon_v \, x_+
\quad\textrm{(compare \ref{Subsection Continuation Maps})}
$
\\[1mm]
$
\strut\qquad\displaystyle SH_*(H) = H_*(SC_*(H);\delta) \quad\textrm{(compare \ref{Subsection Symplectic chain complex})}.
$
\\[1mm]
$
\strut\qquad\displaystyle SH_*(M) = \varprojlim SH_*(H) \quad\textrm{(taking the inverse limit over the }\varphi_*\textrm{, compare \ref{Subsection Symplectic cohomology})}.
$
\\[1mm]
Notice we count solutions flowing \emph{out} of a generator not
\emph{into}. This reverses all maps, so the symplectic homology is
the \emph{inverse} limit over continuation maps.

\emph{Technical Remark. $SC^*(H)$ and $SC_*(H)$ are canonically identifiable as graded vector spaces, since our choice of $H$ ensures that there are only finitely many generators, so there is no actual difference between $\bigoplus$ and $\prod$. Of course the real difference comes from the fact that the differentials and continuation maps are reversed. If we had used a Hamiltonian of quadratic growth (see Section \ref{Appendix Using non-linear Hamiltonians}) then it is crucial to use $\prod$ since $\delta x$ may involve infinitely many generators (unlike $dx$ for $SC^*(H)$ which is a finite sum for energy/action reasons).
}
%
%
%
%
\subsection{Symplectic homology is the dual of symplectic cohomology}
\label{Subsection Symplectic homology is the dual of symplectic cohomology}
%
Observe that 
\\[0.5mm]
\begin{tabular*}{\textwidth}{l@{\extracolsep{\fill}}cr@{\extracolsep{0pt}}} 
\strut & 
$
SC_*(H)= \textrm{Hom}(SC^*(H),\K)
$
 & \strut 
\end{tabular*}
\\[0.5mm]
canonically and
that the differentials $\delta,d$ are dual to each other, indeed
on generators:
$\textstyle \delta x = \sum A_{x,y}\, y$ and $dy = \sum A_{x,y}\,
x $
where $A_{x,y}\! =\! \# \mathcal{M}_0(x,y)$ (count the elements
$u$ with signs $\epsilon_u$). Thus $\delta,d$ are represented by
the matrices $(A_{y,x}),(A_{x,y})$, which are transpose to each
other. Thus $SC_*(H)$ is canonically the dual of $SC^*(H)$, and so
by universal coefficients $SH_*(H)$ is the dual of $SH^*(H)$ and
comes with a canonical isomorphism
\\[0.5mm]
\begin{tabular*}{\textwidth}{l@{\extracolsep{\fill}}cr@{\extracolsep{0pt}}} 
\strut & 
$
SH_*(H) \to \textrm{Hom}(SH^*(H),\K).
$
 & \strut 
\end{tabular*}
\\[0.5mm]
Similarly, $\varphi_*:SC_*(H_-) \to SC_*(H_+)$ is dual to the
$\varphi^*=\varphi:SC^*(H_+) \to SC^*(H_-)$ from \ref{Subsection Continuation
Maps}.
\begin{theorem}\label{Theorem duality theorem between two SHs}
Symplectic homology is canonically dual to symplectic cohomology:
$$
SH_*(M) \cong SH^*(M)^{\vee}.
$$
\end{theorem}
\begin{proof} Using the above canonical
isomorphisms, we get the commutative diagram
$$
\xymatrix@R=12pt@C=12pt{ SH_*(H_+) \ar@{->}[r]^-{\sim}
\ar@{<-}[d]^-{\varphi_*} &
\textrm{Hom}(SH^*(H_+),\K) \ar@{<-}[d]^-{(\varphi^*)^{\vee}} \\
SH_*(H_-) \ar@{->}[r]^-{\sim}  & \textrm{Hom}(SH^*(H_-),\K) }
$$
By category theory, $\textrm{Hom}(\varinjlim M_i,N) \cong
\varprojlim \textrm{Hom}(M_i,N)$, for any module $N$ and any
directed system of modules $M_i$. Using $SH^*(H), \K$ in place of
$M_i, N$, we deduce
\\[0.5mm]
\begin{tabular*}{\textwidth}{l@{\extracolsep{\fill}}cr@{\extracolsep{0pt}}} 
\strut & 
$
\textstyle SH^*(M)^{\vee}\equiv \textrm{Hom}(SH^*(M),\K) \cong
\varprojlim \textrm{Hom}(SH^*(H),\K) \cong \varprojlim SH_*(H) =
SH_*(M)
$
 & \strut 
\end{tabular*}
\\[0.5mm]
using the commutative diagram to obtain the third identification.
\end{proof}
\begin{corollary}
$SH^*(M)=0$ if and only if $SH_*(M)=0$.
\end{corollary}
%
\section{Wrapped Floer cohomology}
\label{Section Wrapped Floer cohomology}
%
\subsection{Lagrangians inside Liouville domains}
\label{Subsection Lagrangians inside Liouville domains}

Observe Figure \ref{Figure Viterbo restriction picture} (ignore $W$). Let $(M^{2n},d\theta)$ be a Liouville domain. Let $L^n\subset M$
be an exact Lagrangian submanifold with Legendrian boundary
$\partial L=L \cap \partial M$ such that this intersection is
transverse. \emph{Exact Lagrangian} means the pull-back $\theta|_L
= \textrm{exact}$. \emph{Legendrian} means $T\partial L\subset
\ker \alpha$, or equivalently $\theta|_{\partial L}=0$.

We strengthen the last condition: the pull-back $\theta|_L = df$
vanishes near $\partial L$. This stronger condition can always be
achieved for the new data obtained after deforming $L$ by a
Hamiltonian isotopy of $M$ relative to $\partial M$ (see
\cite[Lemma 4.1]{Abouzaid-Seidel}).

This condition ensures that near $\partial M$, $L$ has the form
$\textrm{(interval)}\times
\partial L$ in the coordinates $(0,\infty)\times \partial M \subset
\overline{M}$ of \ref{Subsection Liouville domains Definition}. We
extend $L$ to the non-compact exact Lagrangian
$$\overline{L} = L \cup ([1,\infty) \times
\partial L)\subset \overline{M},$$
with $\theta|_{\overline{L}}=df$, and $f$ locally constant on
$\overline{L}\setminus L$ since $\theta|_{\overline{L}\setminus
L}=0$.

\begin{example}\label{Example conormal bdle}
The fibre $\overline{L}=T^*_q N$ in $\overline{M}=T^*N$, where $M=DT^*N$. More generally, the
conormal bundle $\overline{L}\!=\!\mathcal{N}^*\!K\!=\!\{(q,p)\!:\! q\!\in\! K, p|_{T_q
K}\!=\!0 \} \!\subset T^*N$ of a proper submanifold $K\subset N$.
\end{example}

\subsection{Hamiltonian and Reeb chords}
\label{Subsection Hamiltonian and Reeb chords}

A \emph{Hamiltonian chord} of $H:\overline{M}\to \R$ is a map
$$x:[0,1]\to
\overline{M}, \,\textrm{with } \dot{x}(t)=X(x(t)) \textrm{ and
ends } x(0),x(1)\in \overline{L},$$
where we recall that $X=X_H$ is defined by $\omega(\cdot,X)=dH$. A \emph{Reeb chord} of period
$T$ is a map
$$y:[0,T] \to
\partial M , \,\textrm{with } \dot{y}(t)=\mathcal{R}(y(t)) \textrm{ and
ends } y(0),y(1)\in \partial L.$$
Let $\varphi^t_H$ denote the time $t$ flow of $X$. The
Hamiltonian chords correspond to the intersections
$\varphi^1_H(L)\cap L$. We choose $H$ such that on $R\leq 1$, $H$
is $C^2$-small and Morse, and on $R\geq 1$, $H=h(R)$, $h'>0$, and
$h'=m$ for $R\gg 0$. It follows that:

\begin{lemma}\label{Lemma Reeb chords that can arise}
The Hamiltonian chords $x$ in the
collar have constant $H(x)=h(R)$ and correspond to Reeb chords $y(t)=x(t/T)$ of period $T\leq m$ where $T=h'(R)$. $\qed$
\end{lemma}

\begin{lemma}\label{Lemma Reeb chords that can arise 2}
For generic $m,H,\overline{L}$, there are only finitely many
Hamiltonian chords.
\end{lemma}
\begin{proof}
Unlike the $SH^*(H)$ construction, we will not need to make
time-dependent perturbations of $H$, but we need to allow
perturbations $\psi(L)$ of $L$ by a compactly supported
Hamiltonian flow $\psi$ (one can alternatively view this as
keeping $L$ fixed but perturbing $H$ to $H\circ \psi^{-1}$, at the
cost of losing Lemma \ref{Lemma Reeb chords that can arise}). For generic $m>0$ there are no Reeb chords of
period $m$ (by a Sard's Lemma argument). After a small generic
(time-independent) compactly supported Hamiltonian perturbation of
$H,\overline{L}$ one can ensure the Hamiltonian chords are
non-degenerate, that is $\varphi_H^1(\overline{L})$ is transverse
to $\overline{L}$ (\cite[Lemma 8.1]{Abouzaid-Seidel}). So the
Hamiltonian chords are isolated.
\end{proof}
%
\subsection{Action functional}
\label{Subsection Action functional in wrapped case}
Consider the space of smooth paths with ends in $\overline{L}$:
\\[0.6mm]
\begin{tabular*}{\textwidth}{l@{\extracolsep{\fill}}cr@{\extracolsep{0pt}}} 
\strut & 
$
\Omega(\overline{M},\overline{L})=\{ x\in
\C^{\infty}([0,1],\overline{M}): x(0),x(1)\in \overline{L} \}.
$
 & \strut 
\end{tabular*}
\\[0.5mm]
Define ${\mathbb{A}}_H:\Omega(\overline{M},\overline{L}) \to \R$ analogously
to \ref{Subsection Action functional},
\\[0.4mm]
\begin{tabular*}{\textwidth}{l@{\extracolsep{\fill}}cr@{\extracolsep{0pt}}} 
\strut & 
$
\textstyle {\mathbb{A}}_H(x) = f(x(1))-f(x(0))- \int x^*\theta + \int_0^1 H(x(t)) \, dt.
$
 & \strut 
\end{tabular*}
\\[0.6mm]
The motivation for the first three terms is that they would arise
from $-\int u^*\omega$ by Stokes' theorem if $u$ were a disc with
boundary in $\overline{L}\,\cup\, \mathrm{image}(x)$. This ensures
that
$ \textstyle d{\mathbb{A}}_H \cdot \xi = - \int_0^1 \omega(\xi, \dot{x} -
X) \, dt, $
where $\xi \in T_x\Omega(\overline{M},\overline{L}) = \{\xi\in
C^{\infty}([0,1],x^*T\overline{M}): \xi(0), \xi(1)\in
T\overline{L}\}$.

Therefore the critical points of ${\mathbb{A}}_H$ are the Hamiltonian chords.

For Hamiltonian chords $x$ on the collar, ${\mathbb{A}}_H(x)=f(x(1))-f(x(0))-R h'(R)+h(R)$.

%
\subsection{Wrapped trajectories}
\label{Subsection Wrapped Floer trajectories}

Pick $J,g$ as in \ref{Subsection Liouville domains Definition}.
The solutions $u:\R \to \Omega(\overline{M},\overline{L})$ of
$\partial_s u \!=\! -\nabla {\mathbb{A}}_H$ are the solutions $u: \R \times
[0,1] \!\to\! \overline{M}$ of
\\[0.5mm]
\begin{tabular*}{\textwidth}{l@{\extracolsep{\fill}}cr@{\extracolsep{0pt}}} 
\strut & 
$
\partial_s u + J(\partial_t u - X) \!=\! 0,
$
 & \strut 
\end{tabular*}
\\[0.5mm]
with Lagrangian boundary conditions $u(\cdot,0), u(\cdot,1) \in
\overline{L}$. Let $\widehat{\mathcal{W}}(x_{-},x_{+})$ denote the
solutions converging to $x_{\pm}$ at the ends $s\to \pm \infty$.
Let
$\mathcal{W}(x_{-},x_{+})=\widehat{\mathcal{W}}(x_{-},x_{+})/\R$
denote the moduli space of \emph{wrapped trajectories},
identifying $u(\cdot,\cdot)\sim u(\cdot+\textrm{constant},\cdot)$.

\begin{remark} \label{Remark Wrapped FH is same as Lagr FH}
These moduli spaces define the
Lagrangian Floer cohomology
$HF^*(\varphi_H^1(\overline{L}),\overline{L})$. Indeed wrapped
trajectories $u(s,t)$ correspond precisely to pseudo-holomorphic
strips with boundaries in $\varphi_H^1(\overline{L})$ and
$\overline{L}$, that is solutions  of
$\partial_s v + \tilde{J}_t\partial_t v=0,
$
converging to $v(\pm\infty,\cdot) = x_{\pm} \in
\varphi_H^1(\overline{L}) \cap \overline{L}$, where $\tilde J_t =
d\varphi_H^{1-t}\circ J \circ d\varphi_H^{t-1}$. \textbf{Proof}.
Let $v(s,t)=\varphi_H^{1-t}(u(s,t))$. \qed
\end{remark}
%
\subsection{Energy}\label{Subsection Energy for wrapped traj}
The energy of $u\in \mathcal{W}(x_{-},x_{+})$ is defined as
$E(u)=\int |\partial_s u|^2 \, ds \wedge dt$. The same calculation
as in \ref{Subsection Energy for Floer traj}, using
$\theta|_{\overline{L}} = df$, yields
$
E(u) ={\mathbb{A}}_H(x_-)-{\mathbb{A}}_H(x_+).
$
%

\subsection{Transversality and compactness}
\label{Subsection Wrapped Maximum principle}
Just as in the Floer case, a generic time-dependent perturbation
of $J$ ensures that the $\mathcal{W}(x_-,x_+)$ are smooth manifolds. By Lemma \ref{Lemma Maximum principle for wrapped solns} we get:

\begin{lemma}\label{Lemma Max principle for wrapped traj} All $u\in \mathcal{W}(x_-,x_+)$
lie in $R\leq \max(R(x_{\pm}),R_0)$ for $J$ of contact type on $R\geq R_0$.
\end{lemma}

The energy estimate in \ref{Subsection Energy for wrapped
traj} implies that the $\mathcal{W}(x_-,x_+)$ have compactifications by
broken wrapped trajectories. Indeed the
analysis for $\mathcal{W}(x_-,x_+)$ reduces to the known case of
Lagrangian Floer cohomology by Remark \ref{Remark Wrapped FH is same
as Lagr FH}.
\\
\emph{Technical Remark. In general, in addition to breaking, limits of families of pseudo-holomorphic strips (in Remark \ref{Remark Wrapped FH is same as Lagr FH}) may also carry sphere bubbles and disc bubbles which bound $\overline{L}$ or $\varphi_H^1(\overline{L})$. However, since in our setup $\omega$ is exact and $\overline{L}$ is exact, these bubbling phenomena are ruled out by Stokes' theorem. The bubbling-off analysis in the Lagrangian case is due to Floer \cite{Floer}.
}
%
\subsection{Wrapped Floer cohomology}
\label{Subsection Wrapped Floer complex}

Pick a base field $\K$. The wrapped Floer complex $CW^*(L;H)$ is
the $\K$-vector space generated by the Hamiltonian chords of $H$,
$$
CW^*(L;H) =\bigoplus \left\{ \K\, x : x \in
\Omega(\overline{M},\overline{L}),\; \dot{x}(t) = X(x(t))
\right\},
$$
whose differential $d$ on generators counts isolated wrapped
trajectories,
$$
d y = \sum_{u\in \mathcal{W}_0(x,y)} \epsilon_u\, x,
$$
where $\epsilon_u\in \{ \pm 1\}$ are orientation signs (see
Remark \ref{Remark orienation signs for wrapped}). The constant solution $u(s,t)=y(t)\in
\mathcal{W}_0(y,y)$ is not counted. That $d^2=0$ is a standard
consequence of \ref{Subsection Wrapped Maximum principle}. We call
$HW^*(L;H)=H^*(CW^*(L;H);d)$ the wrapped Floer cohomology for $H$.

\begin{remark}\label{Remark orienation signs for wrapped} For fields $\K$ of characteristic $2$, the orientation signs are not necessary.
In general, if $L$ is spin (the Stiefel-Whitney class $w_2(L)=0 \in
H^2(L;\Z/2)$) then \cite[Sec.(12b)]{Seidel-book} orientation signs can be defined.
If the relative Chern class $2 c_1(M,L)=0 \in
H^2(M,L;\Z)$ then a grading on Hamiltonian chords can be defined and
$\dim \mathcal{W}(x_-,x_+)=|x_-|-|x_+|-1$. For a detailed discussion and further references we refer the reader to Abouzaid-Seidel \cite[Sec.9]{Abouzaid-Seidel}.
%
\end{remark}

\begin{remark}\label{Remark Wrapped FH=HW 2}
Via Remark \ref{Remark Wrapped FH is same as Lagr FH},
$HW^*(L;H)\cong HF^*(\varphi^1_H(\overline{L}),\overline{L})$.
Due to the non-compactness of $\overline{M}$, this heavily
depends on $H$: as the slope $m\to \infty$,
$\varphi^1_H(\overline{L})$ wraps more and more times around
$\overline{L}$ giving rise to more chain generators
(namely intersections $\varphi^1_H(\overline{L})\cap \overline{L}$).
\end{remark}

As in \ref{Subsection Continuation Maps}, a monotone homotopy
$H_s$ defines a wrapped continuation $\varphi: HW^*(H_+;L) \to
HW^*(H_-;L)$ by counting wrapped continuation solutions. These are
solutions $v:\R \times [0,1]\to \overline{M}$ of $\partial_s v +
J_s(\partial_t v - X_{H_s})=0$ with Lagrangian boundary conditions
$v(\cdot,0), v(\cdot,1)\in \overline{L}$, converging to
Hamiltonian chords on the ends $s\to \pm \infty$. The a priori
energy estimate for wrapped continuation solutions is the same as
in \ref{Subsection Continuation Maps} and the maximum principle
holds by Lemma \ref{Lemma Maximum principle for wrapped solns}. This implies the
smoothness and compactifiability of the moduli spaces and the
analogue of Lemma \ref{Lemma Chain Homotopy}. Thus $HW^*(L;H)$
only depends on the slope $m$ of $H$ at infinity, and as in \ref{Subsection Hamiltonians Linear At Infty} for $m\leq m'$ there is a wrapped continuation $HW^*(L;H^m) \to HW^*(L;H^{m'})$. 

\textbf{Definition.} Define the \emph{wrapped Floer cohomology} by
$$
HW^*(L) = \varinjlim HW^*(L;H),
$$
taking the direct limit over wrapped continuation maps
between Hamiltonians linear at infinity (\ref{Subsection Hamiltonians Linear At Infty}). So $HW^*(L)\cong
\displaystyle\varinjlim HW^*(L;H^{m_k})$ for any sequence of slopes $m_k\to
\infty$ as $k\to \infty$.

In Theorem \ref{Theorem characterization of continuation maps in Lagrangian case} we prove that the continuation maps $HW^*(L;H^m) \to HW^*(L;H^{m'})$ are the open pair-of-pants product by a special element $e_{H^{\ell}}\in HW^0(L;H^{\ell})$ for $\ell=m'-m\geq 0$.

The $HW^*(L)$ are invariant under symplectomorphisms of contact type, arguing as in \ref{Section Invariance under contactomorphisms}.
The $HW^*(L;H)$ are invariant under compactly supported Hamiltonian isotopies of $\overline{L}$ (this follows from Remark \ref{Remark Wrapped FH=HW 2} and invariance of
Lagrangian Floer cohomology). So $HW^*(L;H)$ is invariant under compactly supported isotopies of $\overline{L}$ through exact Lagrangians (since these are automatically Hamiltonian isotopies by the Weinstein neighbourhood theorem and the fact that graphs of $1$-forms $\mu$ on $L$ are exact in $T^*L$ precisely when $\mu$ is exact). So $HW^*(L)$ is also invariant under compactly supported isotopies of $\overline{L}$ through exact Lagrangians.
%
%
\section{Canonical $c^*$ maps from ordinary cohomology}
\label{Section Canonical map from ordinary cohomology}
%
%
\label{Subsection The maps c from ordinary cohomology}
\label{Subsection wrapped Canonical map from ordinary cohomology}
%
%
For $\delta>0$ smaller than all Reeb periods, consider
$H^{\delta}$ as in \ref{Subsection Hamiltonians Linear At Infty}.
If $H^{\delta}$ is time-independent, Morse and $C^2$-small on $M$ then
$SC^*(H^{\delta})$ reduces to the Morse complex for $H^{\delta}$,
generated by $\textnormal{Crit}(H^{\delta})$ and whose
differential counts $-\nabla H^{\delta}$ trajectories (for a proof, see
\ref{Subsection Floer trajectories converging to broken Morse
trajectories}). The Morse cohomology is isomorphic to the
ordinary cohomology, so $SH^*(H^{\delta})\cong H^*(M;\K)$. Since 
$H^{\delta}$ is part of the direct limit construction, we automatically obtain a map
$$c^*:H^*(M;\K) \to SH^*(M) \quad (\textrm{and dually } c_*: SH_*(M)\to H_*(M;\K) ).$$
We often write $H^*(M)$ instead of $H^*(M;\K)$.
The analogue for the wrapped case is
$$
c^*:H^*(L;\K) \to HW^*(L).
$$
This arises from
$c^*:HW^*(L;H^{\delta}) \to HW^*(L)$, where $H^{\delta}$ is as above so that all
intersections $\varphi^1_{H^{\delta}}(\overline{L})\cap
\overline{L}$ lie inside $M$. By Lemma \ref{Lemma Max principle for wrapped traj}, all
wrapped trajectories lie in $M$, so we reduce to a compact setup with
$\omega=d\theta$. By Remark \ref{Remark Wrapped FH=HW 2},
$HW^*(L;H^{\delta})\cong HF^*(\varphi^1_{H^{\delta}}(L),L)$, and
by Floer \cite{Floer} this is
isomorphic to the Morse cohomology $MH^*(L;H^{\delta})$. The idea is:
$\varphi^1_{H^{\delta}}(L)$ lies in a Weinstein neighbourhood
$W\cong DT^*L$ of $L$ so it is the $\textrm{graph}(d\ell) \subset
DT^*L$ of a function $\ell\in C^{\infty}(L)$. The intersections
$\varphi^1_{H^{\delta}}(L)\cap L$ correspond to
$\textrm{Crit}(\ell)$, and by an implicit function theorem argument,
if $H^{\delta}$ is $C^2$-small in $M$, the evaluation $u\mapsto
u(s,0)$ sets up a bijection between pseudo-holomorphic strips in
$M$ with boundary in $L$ and $-\nabla \ell$ trajectories in $L$.
%
%
%
\section{TQFT structure on $SH^*(M)$, $H^*(M)$, $HW^*(L)$}
\label{Section TQFT overview}
%
%
\subsection{Summary of the TQFT on $\mathbf{SH^*(M)}$}\label{Subsection TQFT on SH summary}
In the first Appendix (Section \ref{Section Floer solutions}) we carry out the detailed construction of the TQFT structure on $SH^*(M)$, which we summarize here. Suppose we are given:
\begin{enumerate}
 \item\label{TQFTitem1}\label{TQFTitem2} a Riemann surface $(S,j)$ with $p+q$ punctures, with fixed complex structure $j$;
 \item\label{TQFTitem3} \emph{ends}: a cylindrical parametrization $s+it$ near each puncture, with $j\partial_s = \partial_t$;
 \item\label{TQFTitem1b} $p\geq 1$ of the punctures are \emph{negative} (so $s\to -\infty$), they are indexed by $a=1,\ldots, p$; 
 \item\label{TQFTitem1c} $q\geq 0$ of the punctures are \emph{positive} (so $s\to +\infty$), they are indexed by $b=1,\ldots,q$; 
 \item\label{TQFTitem4} \emph{weights}: constants $A_a,B_b>0$ satisfying $\sum A_a - \sum B_b\geq 0$;
 \item\label{TQFTitem5} a $1$-form $\beta$ on $S$ with $d\beta \leq 0$, and on the ends $\beta=A_a\,dt$, $\beta=B_b\,dt$ for large $|s|$.
\end{enumerate}

\noindent \emph{Remarks. Negative/positive parametrizations are modeled on $(-\infty,0]\times S^1$ and $[0,\infty)\times S^1$. In (\ref{TQFTitem5}), $d\beta\leq 0$ means $d\beta(v,jv)\leq 0$ for all $v\in TS$. By Stokes, $\sum A_a - \sum B_b = -\int_S d\beta \geq 0$. This forces $p\geq 1$ and (\ref{TQFTitem4}). Subject to this inequality, such $\beta$ exist (Lemma \ref{Lemma beta exists}).
}

\begin{figure}[ht]
\input{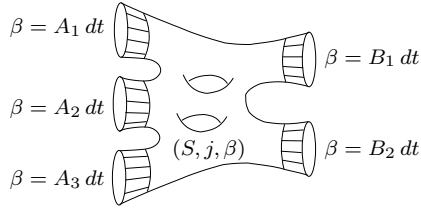}
\caption{A Riemann surface $(S,j)$ of genus $g\!=\!2$, with $p\!=\!3$
negative, $q\!=\!2$ positive punctures, and cylindrical
parametrizations near the punctures.}\label{Figure Riemann surface}
\end{figure}

Fix a Hamiltonian $H:\overline{M}\to \R$ linear at infinity (\ref{Subsection Hamiltonians Linear At Infty}) with $H\geq 0$ (required in \ref{Subsection Energy}), this defines $X=X_H$ (\ref{Subsection Liouville domains Definition}). Fix an almost complex structure $J$ on $M$ of contact type at infinity.

The moduli space $\mathcal{M}(x_a; y_b;S,\beta)$ of \emph{Floer
solutions} consists of smooth maps $u:S\to \overline{M}$ such that 
$du-X\otimes \beta$ is $(j,J)$-holomorphic, and $u$ converges on the ends to
$1$-orbits $x_a,y_b$ of $A_a H$, $B_b H$ which we call the \emph{asymptotics}. After a small generic $S$-dependent perturbation $J_{z}$ of $J$, $\mathcal{M}(x_a; y_b;S,\beta)$ is a smooth manifold. One can ensure that on the ends $J_z$ does not depend on $z\!=\!s+it\!\in\! S$ for $|s|\gg 0$. Just as for Floer continuations (\ref{Subsection Floer continuation solutions}), a maximum principle and an a priori energy estimate $E(u) = \sum \mathbb{A}_{A_a H}(x_a) - \sum \mathbb{A}_{B_b H}(y_b)$ hold, so the $\mathcal{M}(x_a; y_b;S,\beta)$ have compactifications by broken Floer solutions: Floer trajectories for $A_a H,B_bH$ can break off at the respective ends (Figure \ref{Figure compactness for TQFT}). When gradings are defined (\ref{Subsection Maslov index and
Conley-Zehnder index}), 
$$\textstyle \textrm{dim}\,
\mathcal{M}(x_a;y_b;S,\beta) = \sum |x_a| - \sum |y_b|
+2n(1-g-p).$$

Define
$\psi_S: \otimes_{b=1}^q SC^*(B_b H) \to \otimes_{a=1}^p SC^*(A_a
H)$
on generators by counting isolated Floer solutions
$$
\psi_S(y_1 \otimes\cdots \otimes y_q) = \sum_{u\in
\mathcal{M}_0(x_a;y_b;S,\beta)} \epsilon_u
\; x_1\otimes \cdots \otimes  x_p,
$$
where $\epsilon_u \in \{ \pm 1 \}$ are orientation signs (Section
\ref{Appendix Coherent orientations}). Then extend $\psi_S$
linearly. 

\emph{Remark. We use cohomological conventions: the operation $\psi_S$ receives inputs $y_b$ at the positive punctures of $S$, and emits outputs $x_a$ at the negative punctures. So the operation goes ``from right to left'' if we draw the surface $S$ as in Figure \ref{Figure Riemann surface}.}

The $\psi_S$ are chain maps. On cohomology,
$ \psi_S: \otimes_{b=1}^q SH^*(B_b H) \to \otimes_{a=1}^p SH^*(A_a
H) $
is independent of the choices $(\beta,j,J)$ relative to the ends. Taking direct limits:
$$\psi_S: SH^*(M)^{\otimes q}
 \to SH^*(M)^{\otimes p} \qquad (p\geq 1, q\geq 0).$$
So $SH^*(M)$ has a unit $\psi_C(1)$, product $\psi_P$, coproduct $\psi_Q$, but no counit since $p\geq 1$.

For $SH_*$ all arrows are reversed: $\psi^S: \bigotimes_{a=1}^p SC_*(A_a H) \to \bigotimes_{b=1}^q
SC_*(B_b H)$, on generators $\psi^S(x_1 \otimes \cdots \otimes x_p) \!=
\!\sum
\epsilon_u y_1 \otimes \cdots \otimes y_q$ summing over $u\in \mathcal{M}_0(x_a;y_b;S,\beta)$. Take inverse limits:
$$
\psi^S\!:\! SH_*(M)^{\otimes p} \!\to\! SH_*(M)^{\otimes q} \qquad (p\geq 1, q\geq 0).$$
So $SH_*(M)$ has a counit $\psi^C$, product $\psi^Q$,
coproduct $\psi^P$, but no unit since $p\geq 1$.
%
\subsection{The product}\label{Subsection Product}\label{Section Ring structure}
%
The pair of pants surface $P$ (Figure \ref{Figure Introduction})
defines the product
\\[1mm]
\begin{tabular*}{\textwidth}{l@{\extracolsep{\fill}}cr@{\extracolsep{0pt}}} 
\strut & 
$
\psi_P: SH^i(M)\otimes SH^j(M) \to SH^{i+j}(M),\; x \cdot y =
\psi_P(x,y),
$
 & \strut 
\end{tabular*}
\\[1mm]
which is graded-commutative and associative. 

Commutativity follows 
from the construction of the orientation signs in Theorem \ref{Theorem reorder ends}
(this only uses the $\Z/2\Z$-grading on $SH^*$, \ref{Subsection Using
orientation signs to prove that TQFT maps are chain maps}). At a deeper level, this commutativity is really a consequence of the fact that there is no preferred order for three points on a sphere (unlike three points on the boundary of a disc -- which is why the product on $HW^*(L)$ in Theorem \ref{Theorem wrapped TQFT} is typically not graded-commutative).

Associativity is proved by gluing: glue onto one positive end of
$P$ the negative end of another copy of $P$. This yields a surface $S'$ with $p=1$, $q=3$, independent of the positive end
we chose. So $\psi_P(\psi_P(x,y),z)=\psi_{S'}(x,y,z)=\psi_P(x,\psi_P(y,z))$ for $x,y,z\in SH^*(M)$.
%
\subsection{The unit}
\label{Subsection Definition of the unit}
%
Let $C=\C$ with $p=1$, $q=0$. The end is parametrized by
$(-\infty,0]\times S^1$ via $s+it \mapsto e^{ - 2\pi(s+it) }$. On
this end, $\beta=f(s)dt$ with $f'(s)\leq 0$, with $f(s)=1$ for
$s\leq -2$ and $f(s)=0$ for $s\leq -1$. Extend by $\beta=0$ away
from the end. Thus $\psi_C: \K \to SH^*(H)$.

\begin{figure}[ht]
\includegraphics[scale=0.65]{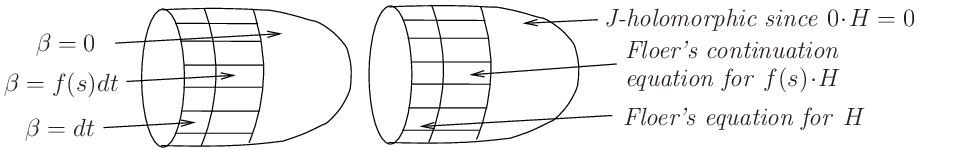}
\caption{A cap $C$, and its interpretation as a continuation cylinder.}\label{Figure Unit}
\end{figure}

\noindent \textbf{Definition.}
Let $e_H\!=\!\psi_{C}(1)\!\in\! SH^0(H)$. Define
$e\! =\! \varinjlim e_H \!\in\! SH^0(M).$
\begin{theorem}\label{Theorem unital ring structure}
$e$ is the unit for the multiplication on $SH^*(M)$.
\end{theorem}
\begin{proof}
By the gluing illustrated in the picture below, $\psi_P(e,\cdot) = \psi_{P\# C}(\cdot) = \psi_Z(\cdot)=\textrm{id}$.\\ \emph{Remark. For ``gluing = compositions'' results, see Theorems \ref{Theorem Homotopy of surfaces}, \ref{Theorem Gluing for compositions}, \ref{Theorem Operations are compatible with limit}. Before taking direct limits, the above is the continuation map $SH^*(H) \stackrel{e_H\otimes \cdot}{\longrightarrow} SH^*(H)^{\otimes 2} \stackrel{\psi_P}{\longrightarrow} SH^*(2H)$.}\\ \strut\hspace{20ex}
\includegraphics[scale=0.8]{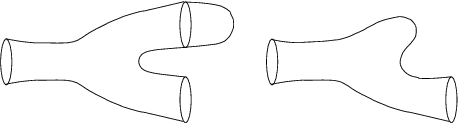} \end{proof}
\begin{lemma}\label{Lemma unit is a count of continuation
solutions} $e_H$ is a count of the isolated finite energy Floer
continuation solutions $u:\R\times S^1 \to \overline{M}$ for the
homotopy $f(s)H$ from $H$ to $0$.
\end{lemma}
\begin{proof}
As $X_H\otimes \beta \! = \!
X_{f(s)H}\otimes dt$, $(du-X\otimes \beta)^{0,1} \! = \! 0$ becomes
$
\partial_s u + J(\partial_t u - X_{f(s)H})  \! = \!  0$. Near $s \! = \! \infty$ such $u$ are $J$-holomorphic, $\partial_s u +
J\partial_t u \! = \! 0$. The finite energy condition implies that $u$
converges to an orbit of $X_H$ at $s \! = \! -\infty$ by Theorem \ref{Theorem
Smoothness and exp convergence} and that the singularity at
$s \! = \! \infty$ is removable. So $u$ extends to a $J$-holomorphic map
over the puncture at $s \! = \! \infty$
(see \cite[Sec.4.2]{McDuff-Salamon}) recovering the solution $\C \to
\overline{M}$ counted by $\psi_C$.
\end{proof}

\begin{lemma}\label{Lemma unit is sum of minima}
For $H\!=\!H^{\delta}$ as in Section \ref{Subsection The maps c from ordinary cohomology}, $e_H = $\,sum of the local minima of $H$.

More generally this holds for any $H\geq 0$ as described at the start of Section \ref{Section SH+} (namely, such that the non-constant $1$-orbits of $H$ have negative action and lie on the collar of $\overline{M}$).
\end{lemma}
\begin{proof}
Pick $J$ so that $g\!=\!\omega(\cdot,J\cdot)$ is Morse-Smale for $H$.
Rescale $H$ by a small constant so that all Floer continuation
solutions for $f(s)H$ are time-independent (this is possible by Claim 5 in
\ref{Subsection Floer trajectories converging to broken Morse
trajectories}: the Floer solutions for $C$ stay inside $M$ by the maximum principle, and we have an a priori energy estimate). Define $\sigma(s) = -\int_s^0 f(\tilde{s})\, d\tilde{s}$, so $\sigma'(s)=f(s)$. Suppose $v$ is a $-\nabla H$ flowline: $v'(s)=-\nabla H$. Define $u(s) = v(\sigma(s))$ so $u'(s) = -f(s) \nabla H = -\nabla (f(s)H)$. Thus $-\nabla (f(s)H)$ flowlines out of $x\in \mathrm{Crit}(H)$ contain a subset parametrized by $W^u(x;H)$, so these are never isolated unless $x$ is a local minimum, in which case the $-\nabla (f(s)H)$ flowline is necessarily constant.
The second claim follows because the Floer solutions counted by $\psi_C: \K \to SH^*(H)$ have energy $E(u) = \mathbb{A}_H(x) \geq 0$ by \ref{Subsection Energy}.
\end{proof}

\begin{theorem}\label{Theorem unit is image of 1}
$e=\varinjlim e_H$ is the image of $1$ under $c^*:H^*(M) \to
SH^*(M)$, and $e_H = c^*_H(1)$ where $c_H^*:H^*(M)\cong SH^*(H^{\delta}) \to SH^*(H)$ is the continuation map.
\end{theorem}
\begin{proof}
By Lemma \ref{Lemma unit is sum of minima}, $e_{H^{\delta}}=c_0^*(1)$ via $c_0^*\!:\!H^*(M) \!\cong\! MH^*(H^{\delta})
\!\equiv\! SH^*(H^{\delta})$. Via the continuation
$\varphi\!:\!SH^*(H^{\delta}) \! \to \!  SH^*(H)$, $e_H\!=\!\varphi(
e_{H^{\delta}}) \!=\! \varphi (c_0^*1) \!=\! c_H^*(1)$ (by Theorem
\ref{Theorem Operations are compatible with limit}). Now take direct limits. The claim also follows by Theorem \ref{Theorem ring structure on ordinary
cohomology}.
\end{proof}

\noindent \emph{Remark. By Section \ref{Section SH+}, one can always construct a Hamiltonian $H=H^{\ell}$ of slope $\ell$ so that $c^*_{H}:MH^*(H^{\delta})\equiv SH^*(H^{\delta}) \to SH^*(H)$ is the inclusion of a subcomplex.
}
%
%
%
%
%
%
\subsection{A minimal set of generating surfaces under gluing}
\label{Section POP decomposition}
\label{Subsection minimal set of gen surfaces}
Write $S_{pq}$ for the surface of genus zero with $p$ negative and $q$ positive punctures, and write $S_{pqg}$ when it has genus $g$. For example, the surfaces $C=S_{10}$, $Z=S_{11}$, $P=S_{12}$, $Q=S_{21}$ are shown in Figure \ref{Figure TQFT table}, and $S_{20}=Q\# C$, $S_{101}=P\# S_{20}$, $S_{111}=P\# Q$ are shown in Figure \ref{Figure POP subdivision}.
\begin{figure}[ht]
\input{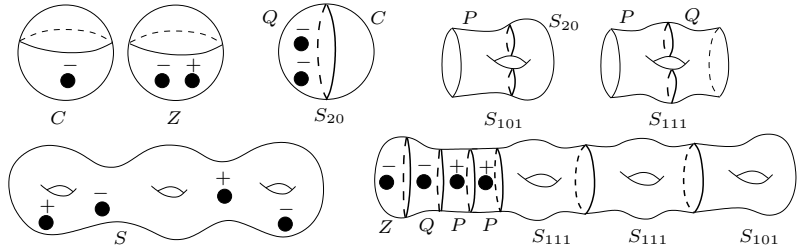}
\caption{Dark dots are punctures with the assigned
sign. A dark circle shows how we construct the surface by gluing two ends in that region.}\label{Figure POP subdivision}
\end{figure}
\begin{theorem}\label{Theorem decomposing S}
Every operation $\psi_S: SH^*(M)^{\otimes q} \to SH^*(M)^{\otimes
p}$ is a composition of operations, chosen among the four basic
ones: the unit $\psi_C: \K \to SH^*(M)$, the identity $\psi_Z=id$,
the product $\psi_P$ and the coproduct $\psi_Q: SH^*(M)\to
SH^*(M)^{\otimes 2}$.
\end{theorem}
\begin{proof} We already showed that $C,Z,P,Q$ generate $S_{20},S_{101},S_{111}$.
As an illustration, consider $S=S_{223}$ in Figure \ref{Figure POP subdivision}: $S$ is isotopic to the last picture, so $S=Z\# Q \# P \# P \# S_{111}\# S_{111} \# S_{101}$. For general $S$, we first isotope all punctures into a small disc $D\subset S$; then the complement $S\setminus D$ has the form $S_{111}\# \cdots \# S_{111}\# S_{101}$ (with genus$(S)-1$ copies of $S_{111}$). Since $S= D \# (S\setminus D)$, we reduce to generating genus zero surfaces such as $D$ (viewing $\partial D$ as a positive puncture -- for example, in the illustration $D=S_{23}$).
By ordering the punctures from left to right as in the last picture in Figure \ref{Figure POP subdivision} (negative punctures on the left) we observe that for $q\geq 1$, $S_{pq}=Z\# Q \# \cdots \# Q \# P \# \cdots \# P$ (with $p-1$ copies of $Q$ and $q-1$ copies of $P$) where we always only glue on the first puncture.
Similarly, for $p\geq 2$, $S_{p0}=Z\# Q \# \cdots \# Q\# S_{20}$ (with $p-2$ copies of $Q$) always gluing only on the first puncture.
\end{proof}
Remark. \emph{This is essentially the statement that $(1+1)$-TQFT's are Frobenius algebras (see Atiyah \cite{Atiyah}), the only subtlety is that we must avoid using surfaces with $p=0$.\\
For $p\geq 1,q\geq 1$ the proof also shows that only $Z,P,Q$ are needed to generate (for genus $\geq 1$ move the last $+$ puncture into the $S_{101}$ region to turn it into $S_{111}\# Z$).
}
%
%
\subsection{The TQFT on $\mathbf{SH^*(M)}$ is compatible with the filtration by $\mathbf{H_1(M)}$}
\label{Subsection TQFT is compatible with filtrations}
%
We can filter $SC^*(H)=\bigoplus_{\gamma} SC_{\gamma}^*(H)$ by the homology classes
$\gamma\in H_1(\overline{M})$ of the generators. The Floer
differential preserves the filtration, and so do Floer operations
on a cylinder and a cap. For the product, $\psi_S: SH_{\gamma_1}^*(M)\otimes SH_{\gamma_2}^*(M) \to
SH_{\gamma_1+\gamma_2}^*(M)$, and for the coproduct, $\psi_Q: SH_{\gamma}^*(M) \to \bigoplus SH_{\gamma_1}^*(M) \otimes
SH_{\gamma_2}^*(M)$ summing over $\gamma_1+\gamma_2=\gamma\in H_1(\overline{M})$.
For general $S$, decompose $S$ as in
\ref{Subsection minimal set of gen surfaces}. Thus:
\\[0.5mm]
\begin{tabular*}{\textwidth}{l@{\extracolsep{\fill}}cr@{\extracolsep{0pt}}} 
\strut & 
$
\psi_S: \otimes_b SH^*_{\gamma_b}(M)\to \otimes_a
SH^*_{\gamma_a}(M) \;\textrm{ is zero unless }\;{\textstyle \sum
\gamma_a = \sum \gamma_b \in H_1(\overline{M})}.
$
 & \strut 
\end{tabular*}
\\[0.5mm]
We can also filter $SH^*(M)=\bigoplus_{\gamma}SH_{\gamma}^*(M)$ by the free homotopy classes
$\gamma\in [S^1,M]$ of the generators. The TQFT operations for genus zero surfaces are compatible with the filtration (the equation above holds after
replacing $\sum$ by concatenation of free loops).

Let $SH^*_0(M)$ denote the summand corresponding to the contractible loops. Considering only contractible loops determines a TQFT with operations $\psi_S:SH^*_0(M)^{\otimes q} \to
SH^*_0(M)^{\otimes p}$ $(p\geq 1,q\geq 0)$. Also $c^*:H^*(M)\to SH^*_0(M)\subset SH^*(M)$ naturally lands in $SH^*_0(M)$. 
%
%
\subsection{Unital ring structure on $\mathbf{SH^*(M)}$ over $\mathbf{\Z}$}
\label{Subsection Choice of coefficients}

Suppose we worked with a (commutative) ring $\K$ instead of a field $\K$ (in the twisted case in Section \ref{Section Twisted symplectic cohomology} this implies $\Lambda$ is a ring, instead of a field). 
We used the assumption that $\K$ was a field in the proof of Theorem \ref{Theorem Homotopy of surfaces} by using the K\"unneth theorem. For a ring $\K$ we can nevertheless define as in \ref{Subsection Operations on SH(H) Definition} the operations
\\[0.5mm]
\begin{tabular*}{\textwidth}{l@{\extracolsep{\fill}}cr@{\extracolsep{0pt}}} 
\strut & 
$
\psi_S: SH^*(\otimes_a B_b H) \to H^*(\otimes_a SC^*(B_b H)) \to H^*(\otimes_a SC^*(A_a H)) 
$
 & \strut 
\end{tabular*}
\\[0.5mm]
where the first map is explicitly $\otimes_b [x_b] \mapsto
[\otimes_b x_b]$. Over a field, this first map is K\"unneth's isomorphism, but over principal ideal domains this map is typically only injective so we cannot in general define a reverse map $H^*(\otimes_a SC^*(A_a H)) \to \otimes_a SH^*(A_a H)$.

For surfaces $S$ with $p=1$ we have $H^*(\otimes_a SC^*(A_a H)) = SH^*(A_1 H)$, so we do get all $\psi_S:\otimes_b SH^*(\otimes_a B_b H) \to SH^*(A_1H)$. This part of the TQFT includes the unital ring structure. 
%
%
%
\subsection{Associated graphs}
\label{Subsection Morse-graph associated to S}
%
Given a decomposition as in \ref{Section POP decomposition}, associate a directed graph $S'$
to the decomposition by replacing $C,Z,P,Q$ by the corresponding
graphs $C',Z',P',Q'$ (Figure \ref{Figure Graphs}) and gluing them
according to the decomposition of $S$.
\begin{figure}[ht]
\includegraphics[scale=0.7]{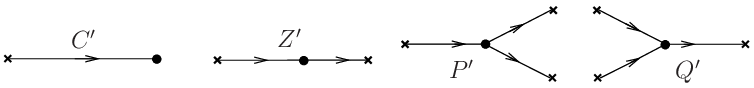}
\caption{Graphs $C',Z',P',Q'$. Dark dots are ``internal''
vertices. Crosses are ``external'' vertices, namely punctures where the edge is infinite.}\label{Figure Graphs}
\end{figure}
%
\subsection{TQFT on $\mathbf{H^*(M)}$ via Morse operations}
\label{Subsection Morse-graph operations}
%
%
Consider an oriented graph $S'$ associated to $S$ in \ref{Subsection Morse-graph
associated to S}. The oriented edges of $S'$ are
parametrized by $(-\infty,0]$, $[0,\infty)$ or $[0,\ell]$ $($for finite $\ell\geq 0)$. We call them respectively \emph{negative ends}
$e_a$, \emph{positive ends} $e_b$ and \emph{internal edges} $e_c$
of length $\ell$. For each edge $e_i$, pick a Morse function $f_i:\overline{M} \to \R$ with $-\nabla f_i$ is inward-pointing along the collar (this condition ensures that 
the Morse cohomology is isomorphic to $H^*(M)$).

For critical points $x_a\!\in\! \textrm{Crit}(f_a)$, $y_b\!\in\!
\textrm{Crit}(f_b)$ define the moduli space of \emph{Morse solutions},
$\mathcal{M}(x_a;y_b;S',f_i),$
consisting of the continuous maps $u\!:\!S'\!\to\! \overline{M}$ which are
$-\nabla f_i$ gradient flows along $e_i$ and which converge to
$x_a,y_b$ at infinity on the edges $e_a,e_b$. The value of $u$ at
vertices where edges meet corresponds to points lying in the
intersection of stable/unstable manifolds of the $f_i$.
The smoothness of $\mathcal{M}(x_a;y_b;S',f_i)$ is guaranteed
by making these intersections transverse by choosing the $f_i$
generically.

After identifications with ordinary cohomology, counting isolated Morse solutions defines
\\[1mm]
\begin{tabular*}{\textwidth}{l@{\extracolsep{\fill}}cr@{\extracolsep{0pt}}} 
\strut & 
$
\psi_{S'}:H^*(M)^{\otimes q} \to H^*(M)^{\otimes p} \quad (p\geq 1,q\geq 0).
$
 & \strut 
\end{tabular*}
\\[1mm]
This map does not depend on the choice of $S'$ associated to $S$,
or the choices $\ell_i$, $f_a$, $f_b$ by continuation and gluing
arguments analogous to the $SH^*(M)$ case. One could also define the operations for $p=0$ using appropriate oriented graphs, so one actually gets the full TQFT.

These TQFT operations are well-known for closed manifolds (Betz-Cohen \cite{Betz-Cohen},
Fukaya \cite{Fukaya}). The cup product on $H^*(M)$ corresponds to the Morse
product operation $\psi_{P'}$ (Figure \ref{Figure Graphs}).

For $C'$ we are counting isolated flowlines $v:(-\infty,0]\to \overline{M}$ with $v'(s)=-\nabla f$, but these are of course never isolated unless $v(-\infty)$ is a local minimum of $f$. So $\psi_{C'}(1)$ is the sum of local minima of $f$, which is the element $1\in H^0(M)$  (compare this to Lemma \ref{Lemma unit is sum of minima} -- indeed the same argument would hold if we defined $\psi_{C'}$ as a time-independent analogue of the construction in \ref{Subsection Definition of the unit}: so using a homotopy $f_s$ from $f$ to $0$ on the negative end of $C'$ and counting isolated finite-energy $-\nabla f_s$ flowlines along $C'$, where the energy is $E(v)=\int |\partial_s v|^2\, ds$).
%
%
\subsection{$\mathbf{c^*: H^*(M) \to SH^*(M)}$ respects the TQFT and $\mathbf{SH^*(M)}$ is an $\mathbf{H^*(M)}$-module}
\label{Subsection H(M) module structure of SH(M)}

By \ref{Subsection Morse-graph operations}, $H^*(M)$ and dually $H_*(M)$ have a TQFT structure via Morse theory analogous to the TQFT structures on $SH^*(M)$ and $SH_*(M)$.
We postpone to Section \ref{Section c* maps preserve TQFT} the following result.

\begin{theorem}\label{Theorem ring structure on ordinary cohomology}
$c^*\!:\!H^*(M) \!\to\! SH^*(M)$ and $c_*\!:\!SH_*(M) \!\to\! H_*(M)$ are TQFT maps. So they are unital ring maps using cup product on $H^*(M)$ and intersection product on $H_*(M)$.
\end{theorem}
Recall $H^{m}$ denotes a Hamiltonian linear at infinity of generic slope $m$ (see \ref{Subsection Hamiltonians Linear At Infty}). The proof of Theorem \ref{Theorem ring structure on ordinary cohomology} in fact shows that $c^*_{H^{\delta}}:H^*(M)\cong SH^*(H^{\delta})$ is a TQFT map, for small $\delta>0$.
So we obtain an $H^*(M)$-module structure on $SH^*(H)$ as follows:
\\[1mm]
\begin{tabular*}{\textwidth}{l@{\extracolsep{\fill}}cr@{\extracolsep{0pt}}} 
\strut & 
$
H^*(M)\otimes SH^*(H) \to SH^*(H), \; \alpha \otimes x \mapsto \psi_P(c^*_{H^{\delta}}(\alpha),x).
$
 & \strut 
\end{tabular*}
\\[0.5mm]
\emph{Remark. This product actually arises as $SH^*(H^{\delta}) \otimes SH^*(H^{m}) \to SH^*(H^{\delta+m})$, but for small $\delta$ no Reeb periods lie in $[m,\delta + m]$ so a monotone homotopy in the region where $H^m$ is linear forces the continuation $SC^*(H^m)\to SC^*(H^{\delta+m})$ to be the identity (no new $1$-orbits appear and, by Lemma \ref{Lemma Maximum principle for Floer solns}, isolated continuation solutions lie in the region where the homotopy is $s$-independent, so by $\R$-symmetry they cannot be isolated unless they are constant).}
%
%

Taking the direct limit, we get a module structure $H^*(M)\otimes SH^*(M) \to SH^*(M)$:
\begin{corollary}
 $SH^*(M)$ is an $H^*(M)$-module via $\alpha \otimes x \mapsto \psi_P(c^*(\alpha),x).$
\end{corollary}

\begin{theorem}\label{Theorem characterization of continuation maps}
The continuation map $\psi_Z:SH^*(H^m) \to SH^*(H^{m'})$ equals the product by the element $e_{H^{\ell}}\in SH^0(H^{\ell})$ for $\ell = m'-m\geq 0$ (defined in Section \ref{Subsection Definition of the unit} and Theorem \ref{Theorem unit is image of 1}).
\end{theorem}
\begin{proof}
Let $e_{H^{\ell}}=\psi_C(1)$ where $\psi_C: \K \to SH^*(H^{\ell})$ (see \ref{Subsection Definition of the unit} for details). Now consider the picture in the proof of Theorem \ref{Theorem unital ring structure} using weights $(m';\ell,m)$ for $P$. Since gluing corresponds to compositions (Theorems \ref{Theorem Gluing for compositions} and \ref{Theorem Homotopy of surfaces}) we deduce:
$\psi_P(e_{H^{\ell}},\cdot)=\psi_{P\#(C\sqcup Z)}(\cdot)=\psi_{Z}(\cdot)$. For cylinders, $\psi_Z$ is the continuation map (Example \ref{Example Continuation cylinder}, Theorem \ref{Theorem Homotopy of surfaces}).
 \end{proof}
%
\subsection{The coproduct}
\label{Subsection Description of the coproduct}
%
Denote $\psi_{Q'}: H^*(M) \to H^*(M)\otimes H^*(M)$ the classical coproduct (see \ref{Subsection Morse-graph operations}). Let $f_2,f_3$ be the Morse functions we use on the negative ends of $Q'$, and use an $f$ on the positive end which has a unique local minimum. By definition $\psi_{Q'}(1)$ is a count of intersection numbers $W^u(x;f_2)\cdot W^u(y;f_3)$ for $x,y$ of complementary Morse indices (for generic $f$, the intersection point at the internal vertex of $Q'$ will always have a unique $-\nabla f$ flowline to $\mathrm{min}\,f$). Since intersection numbers are homotopy invariants, we can in fact write
\\[0.5mm]
\begin{tabular*}{\textwidth}{l@{\extracolsep{\fill}}cr@{\extracolsep{0pt}}} 
\strut & 
$
\psi_{Q'}(1) = \sum \left(  [W^u(x;f)] \cdot [W^u(y;f)] \right) \, x\otimes y \in MH^{2n}(f).
$
 & \strut 
\end{tabular*}
\begin{example}\label{Example coproduct}
 For $M=T^*N$, the coproduct $\psi_{Q'}$ is determined by the Euler characteristic: $\psi_{Q'}(1) = ([N]\cdot [N])\, \mathrm{vol}_N \otimes \mathrm{vol}_N = \pm \chi(N) \,(\mathrm{vol}_N)^{\otimes 2}$  (working over $\Z/2$ if $N$ is not oriented).
\end{example}
\begin{theorem}\label{Theorem description of coproduct}
 The image of $\psi_Q: SH^*(M) \to SH^*(M)^{\otimes 2}$ lies in the image of $(c^*)^{\otimes 2}: H^*(M)^{\otimes 2}\to  SH^*(M)^{\otimes 2}$. Moreover, $\psi_Q$ is determined by $\psi_{Q'}(1)$ and the $H^*(M)$-module structure of $SH^*(M)$ defined in \ref{Subsection H(M) module structure of SH(M)}.
\end{theorem}
\begin{proof}
The fact that $\psi_Q$ lands in $c^*(H^*(M))\otimes SH^*(M)$ follows from the factorization
\\[0.2mm]
\begin{tabular*}{\textwidth}{l@{\extracolsep{\fill}}cr@{\extracolsep{0pt}}} 
\strut & 
$
\psi_{Q} = \psi_{(Z \sqcup Z) \# Q} \co 
SH^*(H) \stackrel{\psi_Q}{\longrightarrow} SH^*(\delta H) \otimes SH^*(H) 
\stackrel{\varphi \otimes \mathrm{id}}{\longrightarrow} SH^*(H) \otimes SH^*(H)
$
 & \strut 
\end{tabular*}
\\[1mm]
where $\varphi: SH^*(\delta H) \to SH^*(H)$ is the continuation, using small enough $\delta>0$ so that $SH^*(\delta H) \cong H^*(M)$. A similar factorization shows $\mathrm{im}(\psi_Q) \subset SH^*(M)\otimes c^*(H^*(M))$.

We can also factorize $\psi_Q: SH^*(H) \to SH^*(2H)^{\otimes 2}$ as 
\\[0.2mm]
$
\psi_{Q} = \psi_{(Z \sqcup P) \# ((Q\# C) \sqcup Z)}(1\otimes \cdot) \co 
SH^*(H) \stackrel{\psi_{Q\# C}(1) \otimes \mathrm{id}}{\longrightarrow} SH^*(\delta H)^{\otimes 2} \otimes SH^*(H) 
\stackrel{\varphi \otimes \psi_P}{\longrightarrow}  SH^*(2H)^{\otimes 2}
$
\\[1mm]
where $\varphi:SH^*(\delta H) \to SH^*(2H)$ is the continuation.

For small $\delta>0$, $\psi_{Q \# C}(1)\in SH^*(\delta H)^{\otimes 2}$ equals $(c^*_{\delta H})^{\otimes 2}(\psi_{Q'}(1))$ (since $c^*_{\delta H}$ is compatible with TQFT by the proof of Theorem \ref{Theorem ring structure on ordinary cohomology}).
 Say $(c^*_{\delta H})^{\otimes 2}(\psi_{Q'}(1))=\sum k_{ij}\, x_i \otimes x_j \in SH^*(\delta H)^{\otimes 2}$, then by the factorization: $\psi_{Q}(y) = \sum k_{ij} \, \varphi(x_i) \otimes \psi_P(x_j, y)$. Finally, in the direct limit, $\psi_P(x_j,y)\in SH^*(M)$ is the $H^*(M)$-module structure of $SH^*(M)$ since $x_j\in \textrm{im}(c^*)$.
\end{proof}

Observe that we proved that if $\psi_{Q'}(1)=\sum k_{ij}\, x_i\otimes x_j$, then $\mathrm{im}(\psi_Q)\subset \mathrm{span}\{c^*(x_i)\} \otimes \mathrm{span}\{c^*(x_j)\}$. In particular, 
if $\psi_{Q'}(1)=0$ then $\psi_Q\equiv 0$.

If the surface $S$ has non-zero genus or has $p\geq 2$, then by Theorem \ref{Theorem decomposing S} the decomposition of $S$ into generators $C,Z,P,Q$ contains at least one $Q$. So Theorem \ref{Theorem description of coproduct} implies:

\begin{corollary}
 If $\psi_{Q'}(1)=0$, then all TQFT operations on $SH^*(M)$ for surfaces of non-zero genus and for surfaces with $p\geq 2$ will vanish. $\qed$
\end{corollary}

\begin{example}
 For $\overline{M}=T^*N$, the only possible non-zero contribution to $\psi_Q$ is a map $SH^0(M) \to \K \, (c^*\mathrm{vol}_N)^{\otimes 2}$ for degree reasons, and the Corollary applies whenever $\chi(N)=0$.
\end{example}
%
%
\subsection{Open model surfaces $S$}
\label{Subsection open model surfaces S}
%
%
Let $(S,j)$ be a Riemann surface isomorphic to the closed unit
disc with $p+q$ boundary points removed. These punctures are
marked in the oriented order as $1,\ldots, p, 1, \ldots q$, the
first $p\geq 1$ are called \emph{negative}, the other $q\geq 0$ \emph{positive}. Fix parametrizations as
\emph{strip-like ends} $(-\infty,0]\times [0,1]$ and
$[0,\infty)\times [0,1]$ near negative and positive
punctures respectively, such that $j$ becomes
standard: $j\partial_s =
\partial_t$.

\begin{figure}[ht]
\input{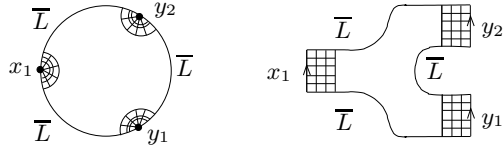}
\caption{Two ways of viewing the open-string pair-of-pants $P$, $(p,q)=(1,2)$, showing the strip-like end parametrizations, the Lagrangian boundary condition and the asymptotics $x_1,y_1,y_2$. The dark circles are punctures.} \label{Figure
wrapped tqft}
\end{figure}

Fix a one-form $\beta$ on $S$ satisfying: $d\beta\leq 0$; on
each strip-like end $\beta$ is a positive constant multiple of
$dt$, these constants $A_a, B_b>0$ are called \emph{weights} and
are generic (not periods of Reeb chords); $\beta$ is exact near $\partial S$; and the
pullback $\beta|_{\partial S}=0$
 (so $\beta|_{\partial
S}({T\partial S})=0$).

By Stokes' theorem, $\sum A_a - \sum B_b \geq -\int_S d\beta \geq 0$.
Arguing as in Lemma \ref{Lemma beta exists}, one shows that this condition on the weights $A_a,B_b$ is the only obstruction to constructing such a form $\beta$, and one can also ensure that $d\beta=0$ everywhere except on a chosen small disc in the interior of $S$. At the cohomology level, the particular choice of $\beta$ will not matter.

Given two such sets of data for $S,S'$, then as in \ref{Subsection
Gluing surfaces} we can glue a positive end of $S$ with a negative
end of $S'$ provided they carry the same weight. The resulting
surface $S\# S'$ carries glued data $j\#j', \beta\#\beta'$
satisfying the above conditions.

More generally, we also allow $S$ to be a disjoint union of
surfaces as above, and we also allow surfaces $S$ which are the
result of gluing any such surfaces along several ends (provided
$p\geq 1$ after gluing), so the surface can have certain holes.

%
\subsection{TQFT structure on $\mathbf{HW^*(L)}$}
\label{Subsection Wrapped TQFT structure}
%
%
Fix a Hamiltonian $H:\overline{M}\to \R$ as in the definition of $HW^*(L;H)$ with $H\geq 0$ and
(possibly non-generic) slope $1$ at infinity. Write $X=X_H$.

In analogy with \ref{Subsection Floer solutions S surface},
define the moduli space of \emph{wrapped solutions}
$\mathcal{W}(x_a;y_b;S,\beta)$
consisting of smooth $u:S \to \overline{M}$ with $u(\partial
S) \!\subset \!\overline{L}$, solving
$ (du - X\otimes \beta)^{0,1}=0, $
converging on the ends to the Hamiltonian chords $x_a,y_b$ for
$A_a H, B_b H$.

The maximum principle holds for these moduli spaces by Lemma \ref{Lemma Maximum principle for wrapped solns}, and the a priori energy estimate $E(u) \leq \textstyle \sum_{a=1}^p {\mathbb{A}}_{A_a H}(x_a) - \sum_{b=1}^q {\mathbb{A}}_{B_b H}(y_b)$ holds by Section \ref{Subsection Energy Appendix Wrapped}. We now make a technical remark about transversality, which proves that for generic $J$ the wrapped moduli spaces are smooth manifolds.

\emph{Technical Remark. For transversality, proceed as in \ref{Subsection
Smoothness of Moduli Spaces}, so \ref{Appendix Genericity of
J} can be applied, except we need to discuss trivial solutions
$du-X\otimes \beta=0$. Let $\gamma$ be a boundary path in
$\partial S$ connecting two consecutive punctures $z,z'$ with
weights $c,c'$. Since the pull-back of $\beta$ to $\partial S$
vanishes and $\beta$ is exact near $\partial S$, we can define a
complex coordinate $s+it$ near $\gamma$ such that: $\partial_s =
\dot{\gamma}$ and $\beta=\ell(t)\,dt$ along $\gamma$, for some
function $\ell(t)$ interpolating $c,c'$. Since $du-X\otimes \beta=0$,
$du(\dot{\gamma})=0$ so $u$ is constant along $\gamma$. Thus
$u(z),u(z')$ are two (possibly equal) Hamiltonian chords such that the
initial point of one is the end point of the other. This gives
rise to a triple intersection in $\overline{L}\cap
\varphi^c(\overline{L})\cap \varphi^{c+c'}(\overline{L})$. As in
\ref{Subsection Hamiltonian and Reeb chords}, after perturbing
$H,\overline{L}$, one can \cite[Lemma 8.2]{Abouzaid-Seidel} rule
out these intersections provided $\dim M \geq 4$, which we tacitly
assume (for $\dim M = 2$ there is a work-around by passing to $M \times DT^*S^2$ \cite[Sec.5b]{Abouzaid-Seidel}).}

By the maximum principle and the energy estimate, $\mathcal{W}(x_a;y_b;S,\beta)$ is
compact up to breaking at the ends (as mentioned in \ref{Subsection Wrapped Maximum principle}, no sphere or disc bubbling can occur since $\overline{M},\overline{L}$ are exact). On the ends, the equation turns into Floer's equation, so the breaking analysis is
the same as for strips. The construction of the TQFT is now analogous to the
$SH^*(M)$ case by counting wrapped solutions. The analogue of Theorem
\ref{Theorem Operations are compatible with limit} is:
\begin{theorem}\label{Theorem wrapped TQFT}
There are TQFT maps $\mathcal{W}_S: \otimes_b HW^*(L;B_b H) \to
\otimes_a HW^*(L;A_a H)$ which in the direct limit become TQFT
operations
$\mathcal{W}_S: HW^*(L)^{\otimes q} \to HW^*(L)^{\otimes p}$, 
$(p\geq 1, q\geq 0)$,
which do not depend on the choices of $\beta,j,J,H$. $\qed$
\end{theorem}

This includes the unital ring structure: $\mathcal{W}_P:
HW^*(L)^{\otimes 2} \to HW^*(L)$ with unit $\mathcal{W}_C:\K \to
HW^*(L)$ where $P$ is a disc with $1+2$ punctures and $C$ is a
disc with $1+0$ punctures. As in \ref{Subsection Product}, the
product is associative. It is not graded-commutative in general (see \ref{Subsection Product}).

\subsection{$\mathbf{c^*: H^*(L) \to HW^*(L)}$ respects the TQFT and $\mathbf{HW^*(L)}$ is an $\mathbf{H^*(L)}$-module}
\label{Subsection H(L) module structure of HW(L)}

We postpone to Section \ref{Subsection c* respects TQFT Lagrangian case} the following result.
\begin{theorem}\label{Theorem c* map for HWL preserve TQFT}
The maps $c^*:H^*(L) \to HW^*(L)$ from Section \ref{Section Canonical map from ordinary cohomology} preserve the TQFT. 
\end{theorem}

Proceeding as in Section \ref{Subsection H(M) module structure of SH(M)}, we deduce the following two results:

\begin{corollary}
 $HW^*(L)$ is an $H^*(L)$-module via 
\\[0.5mm]
\begin{tabular*}{\textwidth}{l@{\extracolsep{\fill}}cr@{\extracolsep{0pt}}} 
\strut & 
$
H^*(L)\otimes HW^*(L) \to HW^*(L), \quad \alpha \otimes x \mapsto \mathcal{W}_P(c^*(\alpha),x).
$
 & \strut 
\end{tabular*}
\end{corollary}

\begin{theorem}\label{Theorem characterization of continuation maps in Lagrangian case}
The continuation map $\mathcal{W}_Z=\mathcal{W}_P(e_{H^{\ell}},\cdot):HW^*(L;H^m) \to HW^*(L;H^{m'})$ equals the product by the element $e_{H^{\ell}}=\mathcal{W}_C(1)\in HW^0(L;H^{\ell})$ where $\ell = m'-m\geq 0$.
\end{theorem}

Just as in Section \ref{Subsection H(M) module structure of SH(M)} and Theorem \ref{Theorem unit is image of 1}, the elements $e_{H^{\ell}}\in HW^0(L;H^{\ell})$ are the image of $1$ under the continuation map $c^*_{H^{\ell}}: H^*(L)\cong HW^*(L;H^{\delta}) \to HW^*(L;H^{\ell})$ (recall 
Section \ref{Section Canonical map from ordinary cohomology}), and by Section \ref{Section SH+} one can construct $H^{\ell}$ so that $c^*_{H^{\ell}}$ is the inclusion of a subcomplex.
%
%
\subsection{$\mathbf{HW^*(L)}$ is a module over $\mathbf{SH^*(M)}$}
\label{Subsection Wrapped HW is an SH module}
%
Pick a surface $S$ as in \ref{Subsection open model surfaces S}.
Introduce $p_{\mathrm{int}}$ negative and $q_{\mathrm{int}}$ positive punctures in the interior of $S$.
Fix cylindrical parametrizations $(-\infty,0]\times S^1$ and
$[0,\infty)\times S^1$ respectively near those punctures, with
$j\partial_s=\partial_t$. Let $\beta$ be a $1$-form on the
resulting surface $S$ with: $d\beta \leq 0$, $\beta$ is a positive
multiple of $dt$ near each puncture, $\beta$ is exact near
$\partial S$, the pull-back $\beta|_{\partial S}=0$.
\begin{figure}[ht]
\input{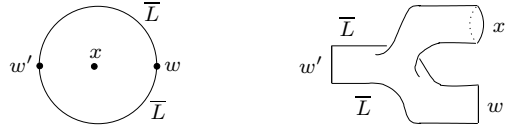}
\caption{The surface $D$ with punctures $(w';w;\,;x)$ viewed in two
ways.} \label{Figure wrapped tqft 2}
\end{figure}
Consider the moduli space of solutions $u:S\to \overline{M}$ of
$(du-X\otimes\beta)^{0,1}=0$, with $u(\partial S)\subset
\overline{L}$, converging at boundary punctures to Hamiltonian
chords and at interior punctures to $1$-periodic Hamiltonian
orbits. The maximum principle and energy estimate still hold, so
transversality and compactness up to breaking on the ends are
proved as before. The count of these moduli spaces define
operations 
$$
SH^*(M)^{\otimes q_{\mathrm{int}}} \otimes HW^*(L)^{\otimes q} \to
SH^*(M)^{\otimes p_{\mathrm{int}}} \otimes HW^*(L)^{\otimes p} \quad (p+p_{\mathrm{int}}\geq 1, q+q_{\mathrm{int}}\geq 0).
$$

Consider a punctured disc $D$ with $(p,q;p_{\mathrm{int}},q_{\mathrm{int}}) = (1,1;0,1)$. 
At the chain level a solution as in Figure \ref{Figure wrapped tqft 2} contributes $\pm w'$ to
$\mathcal{W}_D(x\otimes w)$ where
$$
\mathcal{W}_D: SH^*(M) \otimes HW^*(L) \to HW^*(L).
$$
We deduce the following result, which was known by specialists in symplectic cohomology:

\begin{theorem}\label{Theorem wrapped module over SH}
$\mathcal{W}_D$ defines a module structure of $HW^*(L)$ over
$SH^*(M)$.
\end{theorem}

\begin{proof}
By construction $\mathcal{W}_D$ is linear in $SH^*(M)$ and
$HW^*(L)$. Next we check that $\mathcal{W}_D(e,\cdot)=\textrm{id}$
for the unit $e\in SH^*(M)$. Given $H$ and weights, consider
Figure \ref{Figure wrapped module}: capping off the closed orbit
in Figure \ref{Figure wrapped tqft 2} yields a strip, which we can
homotope to the standard continuation strip $(\R\times
[0,1],\beta=f(s) \,dt,j\partial_s=\partial_t)$ where $f'(s)\leq 0$
and $f$ interpolates the weights. So $\mathcal{W}_D(e_H,\cdot)$ is
chain homotopic to a continuation map. So, in cohomology, taking
direct limits, $\mathcal{W}_D(e,\cdot)$ is the identity, as
required.

\begin{figure}[ht]
\input{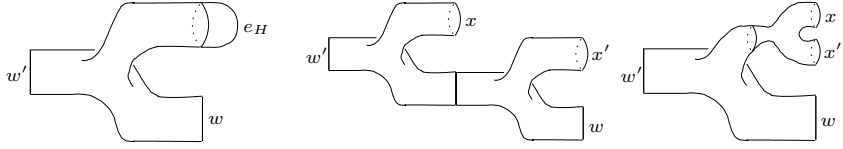}
\caption{Gluing a cap onto Figure \ref{Figure wrapped tqft 2}.
Gluing 2 copies of Figure \ref{Figure wrapped tqft 2}, and finally
the configuration obtained after an isotopy.}\label{Figure wrapped
module}
\end{figure}

Finally we need $\mathcal{W}_D$ to be compatible with the product
on $SH^*(M)$, namely $\mathcal{W}_D \circ (\textrm{id} \otimes
\mathcal{W}_D)=\mathcal{W}_D\circ (\psi_P \otimes \textrm{id}) $.
It suffices to check that for given $H$ and weights, at the chain
level $\mathcal{W}_D(x,\mathcal{W}_D(x',w))$ is chain homotopic to
$\mathcal{W}_D(\psi_P(x,x'),w)$.

Consider the last two configurations in Figure \ref{Figure wrapped
module}: one is the gluing of two copies of Figure \ref{Figure
wrapped tqft 2} corresponding to
$\mathcal{W}_D(x,\mathcal{W}_D(x',w))$, the other is obtained by
an isotopy of this one and it is the gluing corresponding to
$\mathcal{W}_D(\psi_P(x,x'),w)$.
In the last figure, we first pull-back the form $\beta$ via the
isotopy, then we deform it (by linear interpolation) to make it
equal to the $\beta$ obtained from the gluing of the punctured strip
with the pair-of-pants.

Now run the same argument as in the proof of Theorem \ref{Theorem
Gluing for compositions}: combining a stretching argument with the
a priori energy estimate shows that each of the two configurations
arises precisely from gluing broken solutions. Since the two
configurations are homotopic,
$\mathcal{W}_D(x,\mathcal{W}_D(x',w))$ and
$\mathcal{W}_D(\psi_P(x,x'),w)$ are chain homotopic.
\end{proof}
%
\section{Twisted theory}
\label{Section Twisted symplectic cohomology}
%
%
%
\subsection{Novikov bundles}
\label{Subsection Novikov bundles of coefficients}

The Novikov field $\Lambda$ in the
formal variable $t$ is the $\K$-algebra
\\[1mm]
\begin{tabular*}{\textwidth}{l@{\extracolsep{\fill}}cr@{\extracolsep{0pt}}} 
\strut & 
$
\Lambda =\{\, {\textstyle\sum_{j=0}^{\infty}} \, k_{j}\, t^{a_j} \textrm{ for any } \,
k_{j} \in \K, \, a_j \in \R, \textrm{ with } \lim_{j \to \infty}
a_j = \infty \}.
$
 & \strut 
\end{tabular*}
\\[1mm]
\indent Recall
$\mathcal{L} \overline{M}=C^{\infty}(S^1,\overline{M})$ is the free loop space.
Let $\alpha\in C^1_{\textrm{sing}}(\mathcal{L}\overline{M};\R)$ be an $\R$-valued singular cocycle representing a
class $a \in H^1(\mathcal{L} \overline{M}; \R)$. The Novikov bundle
$\underline{\Lambda}_{\alpha}$ is a local system of coefficients
on $\mathcal{L} \overline{M}$: over a loop $x\in \mathcal{L} \overline{M}$ its fibre is a copy
$\Lambda_{x}$ of $\Lambda$, and the parallel translation over
a path $u$ in $\mathcal{L} \overline{M}$ from $x$ to $y$ is
the multiplication isomorphism
\\[1mm]
\begin{tabular*}{\textwidth}{l@{\extracolsep{\fill}}cr@{\extracolsep{0pt}}} 
\strut & 
$
t^{\alpha[u]} \co  \Lambda_{y}
\to \Lambda_{x},
$
 & \strut 
\end{tabular*}
\\[0.5mm]
where $\alpha[\cdot] \co C_1^{\textrm{sing}}(\mathcal{L} \overline{M};\R) \to \R$ denotes the
evaluation of $\alpha$ on singular one-chains.

Up to isomorphism, this local system only depends on the class $a$, so by
abuse of notation we write $\underline{\Lambda}_{a}$ and $a[u]$. \emph{(Proof. changing the representative $\alpha$ to $\alpha+df$, for $f\in C^0_{\textrm{sing}}(\mathcal{L}\overline{M};\R)$, corresponds to an isomorphism of the local systems given by $\lambda_{x} \mapsto t^{-f({x})}\lambda_{x}$ on the $\Lambda_{x}$ fibres.)}

%
%
%
%
%
\subsection{Transgressions}\label{Subsection Transgression}
Let $ev: \mathcal{L} \overline{M} \times S^1 \to \overline{M}$ be the
evaluation map. Define
$$\xymatrix{
\tau=\pi\circ ev^*: H^2(\overline{M};\R)\ar[r]^-{ev^*} & H^2(\mathcal{L} \overline{M}
\times S^1;\R)\ar[r]^-{\pi} & H^1(\mathcal{L} \overline{M};\R) },
$$
where $\pi$ is projection to the K\"{u}nneth summand. This is an
isomorphism when $\overline{M}$ is simply connected. Explicitly, pick a closed de
Rham $2$-form $\eta$ to represent the given $H^2(\overline{M};\R)$ class, then for any path $u\subset\mathcal{L}\overline{M}$ define $\tau \eta\in C_{\textrm{sing}}^1(\mathcal{L}\overline{M};\R)$ by 
$$ \textstyle \tau\eta[u] = \int \eta(\partial_s u, \partial_t u) \,
ds\wedge dt.
$$
\emph{Remark. The path $u:[0,1]\to \mathcal{L}\overline{M}$ defines a cylinder $u:[0,1]\times S^1 \to \overline{M}$, and we can homotope $u$ relative to the ends to make it smooth -- the above integral is then just $\int u^*\eta$. By Stokes' theorem, this does not depend on the choice of homotopy, since $\eta$ is a closed form.%
%
%
}

Observe that $\tau \eta$ vanishes on time-independent paths in
$\mathcal{L} \overline{M}$, so $c^*\underline{\Lambda}_{\tau\eta}$ is trivial
when we pull-back by the inclusion of constant loops $c: \overline{M} \to
\mathcal{L}\overline{M}$.

More generally, $\tau \eta$ evaluated on a Floer solution $u:
S\to \overline{M}$ is defined by
$$
\textstyle \tau\eta[u] = \int_S u^*\eta.
$$
This integral is finite despite the non-compactness of $S$ because
of the exponential convergence of $u$ to the asymptotics (Theorem
\ref{Theorem Smoothness and exp convergence}). Indeed, on an end,
$|\partial_s u|\leq c e^{-\delta|s|}$, $\partial_s u +
J(\partial_t u - CX) = 0$ and $|\eta(u)|\leq C'$ for some
constants $c,\delta,C,C'$. Thus,
$$\textstyle\int_{\textrm{end}} \eta(\partial_s u,
\partial_t u) \, ds\wedge dt \leq C'\cdot \int_{\textrm{end}}
|\partial_s u| \cdot |\partial_t u| \leq C' \cdot
\int_{\textrm{end}} ce^{-\delta |s|}(ce^{-\delta
|s|}+C|X|)<\infty.
$$
%
%
\subsection{Twisted symplectic (co)homology}
\label{Subsection Twisted symplectic chain complex}
\label{Subsection Twisted symplectic cohomology}
Now introduce weights corresponding to $\alpha\in H^1(\mathcal{L}
M;\R) \cong H^1(\mathcal{L} \overline{M};\R)$ in the construction
of \ref{Subsection Symplectic chain complex} (for details, see
\cite{Ritter}):
$$\begin{array}{l}
\displaystyle SC^*(H)_{\alpha}=SC^*(H;\underline{\Lambda}_{\alpha}) =\bigoplus
\left\{ \Lambda x : x \in \mathcal{L} \overline{M},\; \dot{x}(t) =
X_H(x(t)) \right\}
\\[2mm]
\displaystyle d y = \!\!\!\!\sum_{u\in \mathcal{M}_0(x,y)}\!\!\!\!
\epsilon_u\, t^{\alpha[u]}\, x \quad \textrm{(see \ref{Subsection
Symplectic chain complex})}
\\[5mm]
\displaystyle
SH^*(H)_{\alpha}=SH^*(H;\underline{\Lambda}_{\alpha}) =
H^*(SC^*(H;\underline{\Lambda}_{\alpha});d)
\\[2mm]
\displaystyle\varphi: SC^*(H_{+})_{\alpha} \to
SC^*(H_{-})_{\alpha},\; \varphi(x_+)=
\!\!\!\!\sum_{v \in \mathcal{M}^{H_s}_0(x_-,x_+)}\!\!\!\!
\epsilon_v\, t^{\alpha[v]} \, x_- \quad\textrm{(see
\ref{Subsection Continuation Maps})}.
\end{array}
$$
\emph{Remark. These maps are compatible with the isomorphism of local systems described in \ref{Subsection Novikov bundles of coefficients}, since changing the representative $\alpha$ to $\alpha+df$ corresponds to an isomorphism $SC^*(H;\underline{\Lambda}_{\alpha}) \to SC^*(H;\underline{\Lambda}_{\alpha+df})$ given by $x\mapsto t^{-f(x)}x$ on generators.
}

For $(M,d\theta;\alpha)$, define
$
SH^*(M)_{\alpha}=SH^*(M;\underline{\Lambda}_{\alpha}) = \varinjlim
SH^*(H)_{\alpha},
$
taking the direct limit over the twisted continuation maps $\varphi$ above (compare \ref{Subsection Symplectic
cohomology}).

By inserting the local system of coefficients
$c^*\underline{\Lambda}_{\alpha}$ in $MH^*(H^{\delta})$ in Section
\ref{Subsection The maps c from ordinary cohomology} we obtain
$c^*:
H^*(M;c^*\underline{\Lambda}_{\alpha}) \to
SH^*(M;\underline{\Lambda}_{\alpha}).$
For
$\alpha=\tau\eta$, $c^*\underline{\Lambda}_{\alpha}$ is trivial
(\ref{Subsection Transgression}), so
$$c^*:
H^*(M)\otimes \Lambda \to
SH^*(M;\underline{\Lambda}_{\tau\eta}).$$
We sometimes abbreviate $H^*(M;\underline{\Lambda}_{\alpha})$ by $H^*(M)_{\alpha}$, and $H^*(M)\otimes \Lambda$ by $H^*(M)$.

The construction of $SH_*(M)_{\alpha}$ is dual to the above so, in the notation of \ref{Section symplectic homology}, we define $\delta x = \sum \epsilon_u t^{-\alpha[u]} y$ and $\varphi_*(x_-)=\sum \epsilon_v t^{-\alpha[v]} x_+$. The $c_*$ map is $c_*:SH_*(M)_{\alpha} \to H_*(M;c^*\underline{\Lambda}_{\alpha})$.

The dualization result in 
\ref{Subsection Symplectic homology is the dual of symplectic cohomology} now reads: $\mathrm{Hom}_{\Lambda}(SH^*(M)_{\alpha},\Lambda)\cong SH_*(M)_{-\alpha}$.
%
%
\subsection{Twisted TQFT}
\label{Subsection Twisted operations}
%

From now on, use $\alpha=\tau\eta$ (\ref{Subsection
Transgression}). In \ref{Appendix TQFT for other
twistings} we mention other options. Abbreviate
$SC^*(H)_{\eta}=SC^*(H;\underline{\Lambda}_{\tau\eta})$,
$SH^*(H)_{\eta}=SH^*(H;\underline{\Lambda}_{\tau\eta})$. Define
$$
\begin{array}{l}
\displaystyle \psi_S: \bigotimes_{b=1}^q SC^*(B_b H)_{\eta} \to
\bigotimes_{a=1}^p SC^*(A_a H)_{\eta}
\\[4mm]
\displaystyle \psi_S(y_1 \otimes \cdots \otimes y_q) =
\!\!\!\!\sum_{u\in \mathcal{M}_0(x_a;y_b;S,\beta)}\!\!\!\!
\epsilon_u \; t^{\tau\eta[u]} \; x_1\otimes \cdots \otimes x_p.\qquad (\textrm{compare } \ref{Subsection Operations on SH(H) Definition})
\end{array}
$$
\emph{Remark (compare \ref{Subsection Twisted symplectic chain complex}). These maps are compatible with the isomorphism of local systems described in \ref{Subsection Novikov bundles of coefficients}. Indeed, changing the representative $\eta$ to $\eta+d\mu$ for $\mu\in
C^1_{deRham}(\overline{M};\R)$ determines the change of basis isomorphism
 $SC^*(H;\underline{\Lambda}_{\tau\eta}) \to SC^*(H;\underline{\Lambda}_{\tau(\eta+d\mu)})$ given by $x \mapsto t^{-\int_{S^1}\! x^*\mu}x$ on generators. If $u\in \mathcal{M}(x_a;y_b;S,\beta)$ contributes $t^{\int_S\! u^*\eta} x_1\otimes \cdots \otimes x_p$ to $\psi_{S}(y_1\otimes \cdots \otimes y_q)$ when twisting by $\eta$, then by Stokes' theorem it will contribute $t^{\int_S\! u^*\eta}t^{-\sum \int_{S^1}\! x_a^*\mu}t^{\sum \int_{S^1}\! y_b^*\mu} x_1\otimes \cdots \otimes x_p$ when twisting by $\eta+d\mu$. So the twisted TQFT structure respects the change of basis.
}

\begin{theorem}\label{Theorem Novikov weights behave well}\strut
\begin{enumerate}
\item The weights $\tau\eta [u] = \int_S u^*\eta$ are locally
constant on $\mathcal{M}(x_a;y_b;S,\beta)$.

\item The weights $\tau\eta [u]$ are additive under the breaking
of Floer solutions: if $u_{\lambda}$ converge to a broken solution
$u\#v$, then $\tau \eta[u_{\lambda}] =\tau \eta[u]+\tau
\eta[v]$.

\item $\tau\eta[\cdot]$ is constant on components of the
compactification of $\mathcal{M}(x_a;y_b;S,\beta)$.
\end{enumerate}
\end{theorem}
\begin{proof}
Suppose $w(\cdot,\cdot,\lambda) = u_{\lambda}(\cdot,\cdot)$ is a
smooth path of solutions in $\mathcal{M}(x_a;y_b;S,\beta)$,
$\lambda\in [0,1]$. Now $d\eta=0$, since $\eta$ is a closed form,
so (1) follows by Stokes:
$$
0 = \int_{S\times [0,1]} w^*d \eta = \int_{S} u_{1}^*\eta -
\int_{S} u_{0}^*\eta = \tau\eta[u_1]-\tau\eta[u_0].
$$

To prove (2), consider a smooth family
$w(\cdot,\cdot,\lambda)=u_{\lambda}\in
\mathcal{M}(x_a;y_b;S,\beta)$ which is parametrized by $\lambda
\in [0,1)$, such that $u_{\lambda} \to u\# v$ as $\lambda \to 1$.

Fix
$\epsilon>0$. Pick a large $c>0$ so that the restrictions
$u',v'$ of $u,v$ to the complement of the neighbourhoods
$(-\infty,-c)\times S^1$, $(c,\infty)\times S^1$ of all the ends
of $u,v$ satisfy
$$|\tau\eta[u]+\tau\eta[v]-\tau\eta[u']-\tau\eta[v']|<\epsilon.$$
This is possible by the calculation of \ref{Subsection
Transgression}, which shows that the integrals evaluated near the
ends are arbitrarily small because $u,v$ converge exponentially
fast.

Denote by $u_{\lambda}'$ the analogous restriction of
$u_{\lambda}$. The domain of $u_{\lambda}'$ is compact so, by
Lemma \ref{Lemma Compactness and Gluing for TQFT}, $u_{\lambda}'$
converges $C^2$-uniformly to $u'$. Moreover there exist
$s_{\lambda}\in \R$ such that
$v_{\lambda}'(s,t)=u_{\lambda}(s+s_{\lambda},t)|_{s\in[-c,c]}$
converges $C^2$-uniformly to $v'$. So for large $c$, and $\lambda$ close to 1,
$$|\tau\eta[u']+\tau\eta[v']-\tau\eta[u_{\lambda}'] -
\tau\eta[v_{\lambda}']|<\epsilon.$$

Since $\epsilon>0$ is arbitrary, to conclude the proof of (2) we
show that for $c\gg 0$ and $\lambda \approx 1$,
$$
|\tau\eta[u_{\lambda}'] +
\tau\eta[v_{\lambda}']-\tau\eta[u_{\lambda}]|<\epsilon.
$$

By Theorem \ref{Theorem Smoothness and exp convergence}, near the ends
$|\partial_s u| < C e^{-\delta|s|}$ and $|\partial_s v| < C
e^{-\delta|s|}$ for some $\delta, C$. For large $c$, these hold at
$|s|=c$, which are the boundaries of the domains for $u',v'$.

Thus, on an end of $u$, say $K=(-\infty,-c]\times S^1$, apply
Stokes' theorem to $K\times [\lambda,1]$ as in (1) to deduce that
$|\int_K u_{\lambda}^*\eta-\int_K u^*\eta|=|\int_{\mu \in
[\lambda,1], s=-c} u_{\mu}^*\eta|$. This quantity is arbitrarily
small for $c\gg 0,\lambda \approx 1$ by the calculation of \ref{Subsection
Transgression}, since $|\partial_s u_{\mu}|$ on $K$ is bounded by
$C e^{-\delta c}$ for $\mu\approx 1$. A similar argument holds
for the other ends of $u$ and $v$.

Finally consider what happens near the breaking. On
the end where $u_{\lambda}$ breaks, define
$w_{\lambda}(s,t)=u_{\lambda}(s,t)|_{s\in [c,
-c+s_{\lambda}]}$. At the boundaries $s=c$, $s=-c+s_{\lambda}$ the
$w_{\lambda}$ converge $C^2$-uniformly to $u(c,t)$, $v(-c,t)$
respectively, so we can assume $\partial_s w_{\lambda}$ is
bounded by a constant multiple of $e^{-\delta c}$ for
$\lambda\approx 1$. Apply Stokes to $\cup_{\mu\in [\lambda,1]}
(c,-c+s_{\mu})\times S^1 \times \{\mu\}$ as in (1) to deduce:
$$\left|\int w_{\lambda}^*\eta - \int_{s\geq c} u^*\eta -
\int_{s\leq -c} v^*\eta\right| = \left|\int_{\mu \in
[\lambda,1],s=c} u_{\mu}^*\eta - \int_{\mu \in
[\lambda,1],s=-c+s_{\mu}} u_{\mu}^*\eta \right|$$
which again is bounded by a constant multiple of $e^{-\delta c}$,
and so is arbitrarily small for large $c$. Thus (2) follows, and
(3) follows from (1) and (2).\end{proof}
\begin{theorem}\label{Theorem Twisted operations
construction} The twisted TQFT map $\psi_S$ is a chain map, so it
induces a map
$\psi_S:\otimes_b SH^*(B_b H)_{\eta} \to \otimes_a SH^*(A_a H)_{\eta}$
independent of the choice of data $(\beta,j,J)$ $($Theorem
\ref{Theorem Homotopy of surfaces}$)$. For cylinders $Z$ the $\psi_Z$ are
twisted continuation maps $($Example \ref{Example Continuation cylinder}$)$. Gluing surfaces corresponds to
compositions $($Theorem \ref{Theorem Gluing for compositions}$)$.
The $\psi_S$ maps are compatible with the twisted
continuation maps $($Theorem \ref{Theorem Operations are compatible
with limit}$)$. In the direct limit, we get operations independent of $H$
$$\psi_S:SH^*(M)_{\eta}^{\otimes q}\to
SH^*(M)_{\eta}^{\otimes p} \qquad (p\geq 1, q\geq 0).
$$
\end{theorem}
\begin{proof}
The claims involve the same moduli spaces used to prove the
untwisted results. Theorem \ref{Theorem
Novikov weights behave well} implies that we count the moduli
spaces with correct weights: broken solutions on the
boundaries $\partial \mathcal{M}$ of a component
$\mathcal{M}$ of a moduli space are counted with equal weights.
\end{proof}

\begin{theorem}\label{Theorem unital ring structure twisted}
$SH^*(M)_{\eta}$ carries a twisted TQFT
structure, in particular it is a unital ring.
The map $c^*:H^*(M)\otimes \Lambda \to SH^*(M)_{\eta}$ of
\ref{Subsection Twisted symplectic cohomology} respects the TQFT,
using Morse operations on $H^*(M)$. In particular the unit is $e=c^*(1)\in SH^0(M)_{\eta}$.
\end{theorem}
\begin{proof}
The proof is analogous to that of Theorems \ref{Theorem unital ring
structure} and \ref{Theorem ring structure on ordinary cohomology}, which is proved via Lemma \ref{Lemma preserve TQFT
structures}. We only need to insert appropriate weights. In the notation of Definition \ref{Definition phi and psi maps},
define the twisted versions of the maps $\phi$, $\psi$ on
generators as follows:
$$\begin{array}{lll}
\displaystyle \phi: SC^*(H^{\delta})_{\eta} \to MC^*(f)\otimes
\Lambda & \quad & \displaystyle \psi: MC^*(f)\otimes \Lambda \to
SC^*(H^{\delta})_{\eta}\\[1mm]
\displaystyle \phi(y)=\!\!\!\!\!\! \sum_{\begin{smallmatrix}v\# u
\in \mathcal{M}_0^{\phi}(x,y) \\ \textrm{image}(v\# u)\subset
M\end{smallmatrix}} \!\!\!\!\!\!\epsilon_{v\# u}\,
t^{\tau\eta[c\circ v]+\tau\eta[u]}\, x & \quad &
\displaystyle \psi(x)=\!\!\!\!\!\!\sum_{u\# v \in
\mathcal{M}_0^{\psi}(y,x)}\!\!\!\!\!\! \epsilon_{u\# v}\,
t^{\tau\eta[u]+\tau\eta[c\circ v]}\, y
\end{array}
$$

The weights for Morse solutions $v$ are redundant:
$t^{\tau\eta[c\circ v]}=1$ since $c\circ v$ only depends on one
variable (see \ref{Subsection Transgression} on why
$c^*\underline{\Lambda}_{\tau\eta}$ is trivial). However,
inserting these trivial weights makes it clear that the proof of
Theorem \ref{Theorem Novikov weights behave well} can be applied
to the above moduli spaces: the weights are locally constant on
the compactifications of $\mathcal{M}^{\phi}(x,y)$, $\mathcal{M}^{\psi}(y,x)$.
\end{proof}
\begin{theorem}\label{Lemma Invariance under Contactomorphs in twisted case}
A $\varphi: \overline{M} \cong \overline{N}$ of
contact type induces
$\varphi_*:SH^*(M)_{\alpha}\cong SH^*(N)_{\varphi_*\alpha}.$
For $\alpha=\tau\eta$, $\varphi_*$ respects the twisted TQFT
structure:
$\varphi_*^{\otimes p} \circ \psi_{S,M} = \psi_{S,N} \circ
\varphi_*^{\otimes q}.$
\end{theorem}
\begin{proof}
In Lemma \ref{Lemma Invariance under Contactomorphs} and Theorem
\ref{Theorem Invariance of TQFT structures} insert weights
$t^{\alpha[v]}$ in the definition of the continuation maps
defining $\varphi_*:SH^*(M)\cong SH^*(N)$, and apply Theorem
\ref{Theorem Novikov weights behave well}.
\end{proof}
%
%
%
%
%
%
\subsection{TQFT structure for other twistings}
\label{Appendix TQFT for other twistings}
\label{Subsection Not all twistings work}
%
%
%
 For general $\alpha \in
H^1(\mathcal{L}M)$, it does not seem possible to define the TQFT
structure on $SH^*(M)_{\alpha}$. Naively one might divide $S$ into
cylinders, the restrictions of $u$ yield $1$-chains in $\mathcal{L}M$, then evaluate
$\alpha$ and add up. This depends heavily on choices and
will not yield an invariant TQFT. We will now show that it is possible to define TQFT operations in two special cases.
\emph{We restrict ourselves, in this discussion, to only the TQFT defined by surfaces $S$ of genus zero.}\footnote{\emph{Technical Remark. For simply connected $M$, in (1) we can also allow surfaces of non-zero genus. This is because for a loop $\gamma$ in $S$ wrapping around a hole of $S$, we can fill the image $u(\gamma)$
in the target space $\overline{M}$ of the map $u: S \to \overline{M}$ by a disc $D$. So in the construction of Theorem \ref{Theorem twisting by widetilde eta}, one can still form a sphere by capping off $u(S)$ at the ends and adding and subtracting such $D$'s.
}}
\begin{enumerate}
 \item  In \ref{Subsection Twisting by $2$-forms from the universal cover} we show how to twist by
$\alpha=\tau\widetilde{\eta}$ for $\widetilde{\eta}\in
H^2(\widetilde{M})$ taken from the universal cover
$\widetilde{M}$. If we restrict to contractible loops, we obtain a TQFT with unit;

\item In \ref{Subsection twisting by 1forms from the base} we twist by
pull-backs $\alpha=ev^*\mu$, for $\mu\in H^1(M)$, where $ev: \mathcal{L}M \to M$, $x \mapsto x(0)$. Then we obtain a TQFT with no unit.
\end{enumerate}

If we restrict to contractible loops, all $\alpha \in H^1(\mathcal{L}_0 M)$ arise from combining (1) and (2) (Remark \ref{Remark get all twistings if simply connected}). For simply connected $M$, all twistings in $H^1(\mathcal{L}M)$ can be achieved using (1).
%
%
%
\subsection{Twisting by closed $2$-forms from the universal cover}
\label{Subsection Twisting by $2$-forms from the universal cover}
%
%
Let $\widetilde{M}$ be the universal cover of $M$. Then the transgression $\tau:
H^2(\widetilde{M}) \to H^1(\mathcal{L} \widetilde{M})$ is an
isomorphism, and
$$
H^1(\mathcal{L} \widetilde{M}) \cong \textrm{Hom}(H_1(\mathcal{L}
\widetilde{M}),\Z) \cong \textrm{Hom}(\pi_1(\mathcal{L}
\widetilde{M}),\Z) \cong \textrm{Hom}(\pi_2(M),\Z),
$$
(via evaluation of $\tau\widetilde{\eta}$ on $1$-chains). For the
component $\mathcal{L}_0 M$ of contractible loops:
$$H^1(\mathcal{L}_0 M)\cong \textrm{Hom}(\pi_1(\mathcal{L}_0 M),\Z) \cong
\textrm{Hom}(\pi_1 M \ltimes \pi_2 M,\Z) \cong \textrm{Hom}(H_1 M
\times \pi_2 M,\Z).$$
Via these, $H^1(\mathcal{L} \widetilde{M}) \subset
H^1(\mathcal{L}_0 M)$, so we can define
$SH^*(M)_{\widetilde{\eta}}=SH^*(M;\underline{\Lambda}_{\tau\widetilde{\eta}}),$
which comes with $c^*:H^*(M)\otimes \Lambda \to
SH^*_0(M)_{\widetilde{\eta}}\subset SH^*(M)_{\widetilde{\eta}}$ (see \ref{Subsection TQFT is compatible with filtrations} for the $SH^*_0$ notation). 

\emph{Remark.
We will use integral singular $1$-cocycles $\alpha$ below, but one could also
use real $1$-cocycles in $\tau H^2(\widetilde{M};\R)\subset
H^1(\mathcal{L}_0 M;\R)$ via $H^1(\mathcal{L} \widetilde{M};\R) \cong
\mathrm{Hom}(\pi_2 M,\R)$.
}

\begin{theorem}\label{Theorem twisting by widetilde eta}
Restrict to $\mathcal{L}_0 M$ (see \ref{Subsection TQFT is
compatible with filtrations}) and consider only the TQFT defined by genus $0$ surfaces. Then $SH^*_0(M)_{\widetilde{\eta}}$ is a
TQFT with unit $\psi_C(1)\in SH^0_0(M)_{\widetilde{\eta}}$ and
$c^*$ respects the TQFT. The analogues of Theorems \ref{Theorem
Viterbo Functoriality}, \ref{Theorem Vanishing criteria},
\ref{Theorem Abb Schw is a TQFT iso} hold.
\end{theorem}
\begin{proof}
Let $\alpha$ denote a singular $1$-cocycle representing the image
of $\tau\widetilde{\eta}$ in $H^1(\mathcal{L}_0 M)$. Pick as base
point for $\mathcal{L}_0 M$ the constant loop $m$ at a point in
$M$.

Let $u:S\to \overline{M}$ be a Floer solution whose ends $z_i$ are
contractible Hamiltonian orbits. Pick maps $\C \to \overline{M}$,
whose images are caps $C_i$ which contract $z_i$ down to $m$. An
appropriate gluing $\sigma=(\cup C_i) \# u(S)$ defines a sphere
$\sigma: S^2 \to \overline{M}$. Now $\alpha$ corresponds to a
homomorphism $\pi_2(M) \to \Z$ under the identification
$H^1(\mathcal{L}_0 M) \cong \textrm{Hom}(\pi_1 M \ltimes \pi_2
M,\Z)$ ($\alpha$ vanishes on $\pi_1(M)$). Thus $\alpha[\sigma]$ is
defined (using that $\mathcal{L}_0 M \simeq \mathcal{L}_0 \overline{M}$ are homotopy equivalent).

Parametrizing $\C$ by polar coordinates, the caps $C_i$ can be
viewed as paths in $\mathcal{L}_0 M$ from $z_i$ to $m$. So $C_i$
is a singular $1$-chain in $\mathcal{L}_0 M$ and $\alpha[C_i]$ is
defined. Let
$$
\widetilde{\eta}[u]=\alpha[\sigma]-\sum \pm \alpha[C_i],
$$
where $\pm$ is the sign of the end $z_i$ of $u$. 

We claim $\widetilde{\eta}[u]$ is independent of the choices of caps. 
Suppose we change a negative cap $C$ to $C'$, so the sphere
$\sigma$ changes to a new sphere $\sigma'$. The gluing $\sigma''=(-C')\# C$ is a sphere in $M$
but also a sum of two $1$-chains in $\mathcal{L}_0 M$, and under
our identifications,
$$
\alpha[\sigma''] = -\alpha[C'] + \alpha[C].
$$
In $\pi_2(M)$, the equality $\sigma''+\sigma = \sigma'$ holds, so
evaluating $\alpha$ we deduce
$$
-\alpha[C'] + \alpha[C] + \alpha[\sigma] = \alpha[\sigma'].
$$
So $\alpha[\sigma]+\alpha[C]$ does not change when we replace $\sigma,C$ with $\sigma',C'$. Similarly $\alpha[\sigma]-\alpha[C]$ does not change when we replace positive caps $C$. So
$\widetilde{\eta}[u]$ only depends on $u$, not on the caps.

The construction of the TQFT for
$SH^*_0(M)_{\widetilde{\eta}}$ is now routine using weights $t^{\widetilde{\eta}[u]}$.

Just as in \ref{Subsection Twisted operations}, we also need to show that when we change the representative cocycle $\alpha$ to $\alpha + df$, with $f\in C^0_{\textrm{sing}}(\mathcal{L}_0 M)$, the TQFT will change in a way compatible with an isomorphism of the $SH^*_0$ groups induced by an isomorphism of the corresponding local systems on $\mathcal{L}_0 M$.
 On a positive cap $C$, viewed as a path from $y$ to $m$ we have by definition: $df[C]=f(m) - f(y)$. So in the above definition, $\widetilde{\eta}[u]$ changes by %
$$\begin{array}{lll}
df[\sigma]-\sum \pm df[C_i] &=& 0-\sum \pm (f(m)-f(z_i))\\
&=& \sum f(y_b) - \sum f(x_a) + (p-q)f(m)                                                                                                                                                                                                                                                                                                                                                                                                                                               \end{array}
$$
where $x_a,y_b$ are respectively the $p$ negative and the $q$ positive asymptotics of $u$. So the TQFT changes compatibly with the isomorphism $\Lambda_x \mapsto \Lambda_x$, $\lambda_x \mapsto t^{-f(x)} t^{f(m)} \lambda_x$ of the local systems, which induces $SH_0^*(M;\underline{\Lambda}_{\alpha}) \cong SH_0^*(M;\underline{\Lambda}_{\alpha+df})$ given by $x \mapsto t^{-f(x)} t^{f(m)} x$ on generators
\end{proof}
\subsection{Twisting by closed $1$-forms from the base.}
\label{Subsection twisting by 1forms from the base}
Let $ev: \mathcal{L}M \to M, x \mapsto x(0)$ be the evaluation at zero, and fix a closed $1$-form $\mu$ on $M$.
Let $u: S\to \overline{M}$ be a Floer solution.

For cylinders
$S$, $ev^*\mu[u] = \int_{\R} ev(u)^*\mu$ evaluates $\mu$ on the $1$-chain in $M$ given by $s\mapsto u(s,0)$. 

For a cap $S$, $ev^*\mu[u]$ cannot be
defined consistently \emph{(if we puncture the cap and parametrize the punctured cap as a cylinder, then the value $ev^*\mu[u]$ would depend on the choice of puncture. In the approach that we will explain below, we would need to embed a suitable graph in the cap, but no choice of graph guarantees that caps glue compatibly with the TQFT structure.)}

For other $S$, let $S'$ be an associated graph (\ref{Subsection Morse-graph
associated to S}). Pick any continuous map $S'\to S$ (not necessarily an embedding) so that the ends of $S'$ (``external vertices") converge to $t  =  0  \in 
S^1$ in the asymptotic circles in $S$. So, to clarify, for $u\in \mathcal{M}(x_a; y_b;S,\beta)$, the external vertices map to $x_a(0),y_b(0)$ via $u$.
 
Associate to each edge $e_i$ a constant $k_i\in \R$. These define a closed $1$-form $\mu_i = k_i \mu$ for each edge $e_i$. Denote by $\mu_a=k_a \mu$, $\mu_b=k_b \mu$ the forms defined in this way on those edges which connect to the external vertices. We now fix these constants $k_a\geq 0,k_b\geq 0$, but we allow the other constants $k_i$ to be arbitrary subject only to the
\emph{zero-sum condition}: at any internal vertex $v$ of the graph we require that the
signed sum of the constants vanishes,
$$
\sum_{\begin{smallmatrix} \textrm{edges $e_i$} \\ \textrm{coming
into $v$}\end{smallmatrix}} k_i - \sum_{\begin{smallmatrix}
\textrm{edges $e_i$} \\ \textrm{leaving $v$}\end{smallmatrix}}
k_i =0.
$$
Integrating $\mu_i$ along the image under $u$ of the edge $e_i$, and adding up, defines $$
\mu[u] = \sum_{e_i} \mu_i[u|_{e_i}].$$ This construction also defines $\mu[u]$ when we deal with Morse
solutions $u: S' \to \overline{M}$ (see \ref{Subsection Morse-graph operations}).  

\noindent {\bf Example.}\! $P'$ (Figure \ref{Figure Graphs}): using $2\mu$ on the
incoming edge, and $\mu$ on each outgoing edge.

\emph{Remark. The zero-sum condition implies the steady-state condition $\textstyle\sum k_a = \sum k_b,$ and conversely the steady-state condition guarantees that there is a choice of $k_i$'s as above.}

Counting
Floer solutions with weight $t^{\mu[u]}$ defines
$$
\psi_{S}: \bigotimes_{b=1}^q SH^*(M)_{ev^*\mu_b} \to
\bigotimes_{a=1}^p SH^*(M)_{ev^*\mu_a} \qquad (p\geq 1, q\geq 0, (p,q)\neq (1,0)).
$$
These $\psi_S$ are well-defined but depend on $S'$ and the choice of map $S'\to S$.

\noindent \textbf{Example.} Take $S=Z$ and $S'=Z'$ as in \ref{Subsection Morse-graph associated to S}. If we map $S'$ inside the line $t=0$ in $Z$ we get $\mu[u]=ev^*\mu[u]$, but if we map $S'$ onto a curve in $Z$ which winds once around the $S^1$ factor then $\mu[u]=ev^*\mu[u] + \mu[\gamma]$ where $\gamma\in H_1(M)$ is the $H_1$-class of the asymptotics of $u$.

So to obtain TQFT operations which glue consistently (assuming that the forms $\mu_i$ agree on
the gluing ends) we need to make choices as follows. We now assume $S$ has genus zero and $p\geq 1, q\geq 0, (p,q)\neq (1,0)$. Such $S$ are generated by $Z,P,Q,S_{20}$ (see \ref{Subsection minimal set of gen surfaces}). Cut the $Z,P,Q,S_{20}$ in half by planes (this produces open model surfaces as in \ref{Subsection open model surfaces S}). We can make this cut so that there are halves $Z_{\mathrm{up}}\subset Z, P_{\mathrm{up}}\subset P,Q_{\mathrm{up}}\subset Q, (S_{20})_{\mathrm{up}}\subset S_{20}$ whose asymptotics consist of the arcs in $S^1$ parametrized by $-1/4\leq t \leq 1/4$. We embed the $Z',P',Q',S_{20}'$ of \ref{Subsection Morse-graph associated to S} into $Z_{\mathrm{up}}, P_{\mathrm{up}},Q_{\mathrm{up}},(S_{20})_{\mathrm{up}}$ (where $S_{20}'$ is the graph with two negative ends meeting at one internal vertex), with the external vertices mapping to $t=0$ in the asymptotic arcs.
 Since under gluing we make the time coordinates agree on the glued ends, any gluing $S$ of copies of $Z,P,Q,S_{20}$ will automatically produce a choice of half-surface $S_{\mathrm{up}}\subset S$ and a graph $S'\subset S_{\mathrm{up}}$.

\begin{lemma} The $\mu[u]$ constructed for $S'\subset S_{\mathrm{up}}$ only depend on the choices of constants $k_a,k_b$ at the ends (and on the connected component of $u\in \mathcal{M}(x_a; y_b;S,\beta)$).
\end{lemma}
\begin{proof}
%
%
Observe that, if we ignored weights, such an $S'$ will be homotopic in $S_{\mathrm{up}}$ to the \emph{standard graph} consisting of $p$ negative ends and $q$ positive ends which meet in one common internal vertex. This relies on the fact that $S$ has genus zero, so $S'$ is contractible.  

Consider first a homotopy of $S'$ inside $S_{\mathrm{up}}$ which does not change the combinatorial type of the graph. Each edge $e_i$ will sweep out via $u$ a $2$-simplex in $\overline{M}$, and these $2$-simplices are joined along the $1$-simplices swept out by the internal vertices (the external vertices are fixed). Evaluating $d\mu_i =0$ on these
$2$-simplices shows that $\mu[u]$ before and after the homotopy has changed by an amount equal to the sum of the $\mu_i$ evaluated on those $1$-simplices. Taking into account the orientations, this amount vanishes by the zero-sum condition. 

Now we show how to change the combinatorial type of the graph. Adding or removing internal vertices which touch only two edges will not affect $\mu[u]$, assuming the zero-sum condition holds. Next, suppose we add to $S'$ an oriented triangle whose edges are assigned the same form $-k\mu$, with at least one vertex being part of the vertices of $S'$. Since the triangle is contractible after mapping it into $S_{\mathrm{up}}$ (since $S$ has genus $0$), and since 
$d\mu=0$, this triangle does not affect $\mu[u]$. Secondly, we can remove two equally oriented edges which carry opposite weights, since these give cancelling contributions to $\mu[u]$.

 These two moves allow us to slide an edge along another: given edges $e'\!:\!v_1\!\to\! v_2$, $e\!:\!v_2\! \to\! v_3$ of weight $k',k$, adding such a triangle with vertices $v_1,v_2,v_3$ allows us to replace $e',e$ by edges: $v_1 \!\to\! v_2$, $v_1 \!\to\! v_3$ of weight $k'-k,k$.

These moves allow us to turn $S'$ into a standard graph. But all standard graphs are homotopic through a homotopy which does not affect the combinatorial type of the graph, so $\mu[u]$ is independent of choices.
\end{proof}

Abbreviate $SH_{\mu}=SH^*(M)_{\mathrm{ev}^*\mu}$.
Arguing as in Theorem \ref{Theorem Operations are compatible with
limit}, we obtain operations $\psi_S: \otimes_b SH_{\mu_b} \to \otimes_a SH_{\mu_a}$ for genus zero surfaces $S$ with $p\geq 1,q\geq 0,(p,q)\neq (1,0)$ which compose correctly. Changing $\mu$ to $\mu+df$ corresponds to an isomorphism of the local systems $x \mapsto t^{-f(x)} x$, just as in \ref{Subsection Twisted operations}.

One could identify $SH_{c\mu} \cong SH_{\mu}$ via the isomorphism $t\mapsto t^{1/c}$ of local systems, for any $c>0$, so that appropriately rescaled operations give a TQFT on $SH^*(M)_{ev^*\mu}$ without unit.

A more natural approach is to take $NSH^*(M)= \bigoplus_{\mu} SH_{\mu}$ summing over all $\mu\in H^1(M)$. One can also twist by $\eta\in H^2(M)$: $NSH^*(M)_{\eta}= \bigoplus_{\mu} SH_{\mu,\eta}$ where $SH_{\mu,\eta}=SH^*(M)_{\mathrm{ev}^*\mu + \tau\eta}$. So $NSH^*(M)=NSH^*(M)_0$. 
For example, the product is $SH_{\mu_1,\eta}\otimes SH_{\mu_2,\eta}\to SH_{\mu_1+\mu_2,\eta}$, the canonical map is $c^*: NH^*(M)=\bigoplus_{\mu} H^*(M)_{\mu} \to NSH^*(M)_{\eta}$ (the Novikov cohomology $H^*(M)_{\mu}$ is the ordinary cohomology with coefficients in the local system defined by $\mu$).
\begin{theorem}
$NSH^*(M)_{\eta}$ is a TQFT (using genus zero $S$, $p\geq 1$) with unit $\psi_C(1)\in SH_{0,\eta}$.
\end{theorem}
\begin{remark}
\label{Remark get all twistings if simply connected}

Restricting to contractible loops, $\alpha \! = \! ev^*\mu+\tau\widetilde{\eta}$
yields all $\alpha \in H^1(\mathcal{L}_0 M)$ since $\pi_1(\mathcal{L}_0
M)  \! \cong \! \pi_1(M)   \ltimes   \pi_2(M)$ so
$H^1(\mathcal{L}_0 M)  \! \cong \!  ev^*H^1(M) \oplus
\mathrm{Hom}(\pi_2(M),\Z)$. So
$\oplus_{\alpha} SH^*_0(M)_{\alpha} =  \oplus_{\widetilde{\eta}}
NSH^*_0(M)_{\widetilde{\eta}}$ is a TQFT with unit (for genus zero, $p\geq 1$), $c^*:NH^*(M) \! \to \!
\oplus_{\alpha} SH^*_0(M)_{\alpha}$.
\end{remark}
%
\subsection{Twisted wrapped theory}
\label{Subsection Wrapped twisted theory}
%
Let $\alpha \in H^1(\Omega(M,L))\cong
H^1(\Omega(\overline{M},\overline{L}))$ (\ref{Subsection Action
functional in wrapped case}). As in \ref{Subsection Novikov
bundles of coefficients}, $\alpha$ defines a local system of
coefficients $\underline{\Omega}_{\alpha}$ on
$\Omega(\overline{M},\overline{L})$. Define $CW^*(L;H)$ over the
Novikov ring $\Lambda$ and insert weights $t^{\alpha[u]}$ in the
definitions of the differential and the continuation maps. This
yields $\Lambda$-modules $HW^*(L;H;\underline{\Omega}_{\alpha})$
and their direct limit defines $HW^*(L)_{\alpha}=HW^*(L;\underline{\Omega}_{\alpha})$.

These come with a canonical map
$$
c^*:H^*(L;c^*\underline{\Omega}_{\alpha}) \to
HW^*(L;\underline{\Omega}_{\alpha}),
$$
where $c:\Omega(\overline{L},\overline{L}) \to
\Omega(\overline{M},\overline{L})$ is the inclusion of paths lying
in $\overline{L}$ (observe that $\Omega(\overline{L},\overline{L})$ has the homotopy type of $L$).

Let $\eta\in H^2(\overline{M},\overline{L};\R)\cong H^2(M,L;\R)
\cong H^2(M/L;\R)$. One can represent $\eta$ by a closed de Rham
two-form which is supported away from $\overline{L}$ (see
\cite[Sec.11.1]{Guillemin-Sternberg}). Evaluating by $\int_C
\sigma^*\eta$ on smooth one-chains $\sigma:C \to \overline{M}$
defines a transgresssion cycle $\tau\eta\in
H^1(\Omega(\overline{M},\overline{L}))$, which can also be
obtained via: (using real coefficients)
$$
\tau: H^2(M/L) \stackrel{ev^*}{\to} H^2(\Omega(M/L)\times S^1)
\stackrel{\textrm{project}}{\longrightarrow} H^1(\Omega(M/L))
\cong \textrm{Hom}(\pi_2(M/L),\R).
$$
This is an isomorphism when $\pi_1(M/L)=1$ (e.g. if $\pi_1(M)=1$).
For
 simply connected $L,M$, $\tau$ is the Hurewicz
isomorphism $H^2(M,L;\R) \cong \textrm{Hom}(\pi_2(M,L),\R)$.

For a wrapped solution $u:S \to \overline{M}$ define
$\textstyle
\tau\eta[u] = \int_S u^*\eta.
$
These weights are locally constant on the moduli spaces of wrapped
solutions since $\tau\eta[u]$ remains constant if we homotope $u$
relative to the ends while keeping $u(\partial S)\subset
\overline{L}$. Inserting weights in \ref{Subsection Wrapped TQFT structure}-\ref{Subsection Wrapped HW is an SH module}, we deduce:

\begin{theorem}\label{Theorem wrapped twisted TQFT and module}
$HW^*(L;\underline{\Omega}_{\tau\eta})$ has a TQFT structure and $HW^*(L;\underline{\Omega}_{\tau\eta})$ is a module over $SH^*(M;\underline{\Lambda}_{\tau\overline{\eta}})$, where
$\overline{\eta} \in H^2(M;\R)$
is the image of $\eta$ via $H^2(M,L;\R)\to H^2(M;\R)$. $\qed$
\end{theorem}
The twisted version of the map $c^*$ from Section \ref{Section Canonical map from ordinary cohomology} is
$c^*:H^*(L)\otimes \Lambda \to HW^*(L;\underline{\Omega}_{\tau\eta})$, and this is a TQFT map by Theorem \ref{Theorem c* map for HWL preserve TQFT} (after inserting Novikov weights as in Theorem \ref{Theorem unital ring structure twisted}).
%
%
%
\section{The action filtration, and the $SH^*_+$, $HW^*_+$ groups}
\label{Section SH+}
%
%
Let $H: \overline{M}\to \R$ be linear at infinity (\ref{Subsection Hamiltonians Linear At Infty}). By \ref{Subsection Liouville domains Definition} we can ensure that the only $1$-orbits of $H$ inside $M$ are critical points of $H$ and, picking $|H|<\delta$ small on $M$, these have action close to zero: ${\mathbb{A}}_H\in (-\delta,\delta)$. We can%
\footnote{\emph{Proof. Recall the action formula $\mathbb{A}_h(R)=-Rh'(R)+h(R)$ of \ref{Subsection Action functional}. If $h$ is convex then $\partial_R \mathbb{A}_h=-Rh''(R)\leq 0$ so the action decreases with $R$. Pick $h$ convex for $R\geq R_0$. For $R\in (1,R_0)$ pick $h'\leq T_0/2$ where $T_0$ is the smallest Reeb period (so there are no $1$-orbits since they have $h'\geq T_0$), with $h(R_0)=2\delta$ and $h'(R_0)=T_0/2$.  Then $1$-orbits on the collar have action less than $\mathbb{A}_h(R_0) = -R_0T_0/2 + 2\delta$. We can pick $R_0\gg 0$. $\qed$}} 
ensure $H=h(R)$ for $R\geq 1$ with $h'$ growing so slowly that the $1$-orbits of $H$ on the collar have arbitrarily negative action, in particular ${\mathbb{A}}_H<-\delta$.

Write $SC^*(A,\infty)_{\alpha}$ for the chain subcomplex of $SC^*(H)_{\alpha}$ generated only by those generators with
action in the interval $(A,\infty)$ (for generic $A \in \R \cup \{\pm
\infty\}$, so not Reeb periods, and where we keep track of the twisting by $\alpha\in H^1(\mathcal{L}M)$ when twistings are present). Notice that this is a subcomplex because the differential counts incoming Floer trajectories, and the action decreases along trajectories.
Define the quotient complex
\\[0.5mm]
\begin{tabular*}{\textwidth}{l@{\extracolsep{\fill}}cr@{\extracolsep{0pt}}} 
\strut & $SC^*(A,B)_{\alpha} = SC^*(A,\infty)_{\alpha}/SC^*(B,\infty)_{\alpha}.$ & \strut 
\end{tabular*}
\\[0.5mm]
Increasing $A$ gives
a subcomplex:
$SC^*(A',B)_{\alpha}  \hookrightarrow
 SC^*(A,B)_{\alpha}$ for $A  <  A'  <  B.$\\
Decreasing $B$ gives quotient maps: 
$SC^*(A,B)_{\alpha}  \to   SC^*(A,B')_{\alpha}$ for $A<B'<B.$

\noindent\textbf{Example.} \emph{$SC^*(-\delta,\infty) \equiv SC^*(-\delta,\delta)$ is a subcomplex of $SC^*(H)$ and it computes precisely the cohomology $SH^*(H^{\delta})$ from Section \ref{Section Canonical map from ordinary cohomology}, which one can identify with $H^*(M)$.}\\
\noindent \textbf{Example.} \emph{Define $SC^*_+(H)_{\alpha}=SC^*(-\infty,-\delta)_{\alpha}$. Define $SH^*_+(M)_{\alpha}$ to be the direct limit of the cohomology groups $SH^*_+(H)_{\alpha}$ as we increase the slope of $H$ at infinity. In the untwisted setup, the $SH^*_+(M)$ groups are well-known \cite{Viterbo1,Bourgeois-Oancea}.}

These two examples fit naturally into a short exact sequence
\\[0.5mm]
\begin{tabular*}{\textwidth}{l@{\extracolsep{\fill}}cr@{\extracolsep{0pt}}} 
\strut & $0 \to SC^*(-\delta,\delta)_{\alpha} \stackrel{c^*}{\to}
SC^*(-\infty,\delta)_{\alpha} \to SC^*(-\infty,-\delta)_{\alpha} \to 0.$ & \strut 
\end{tabular*}
\\[0.5mm]
(More generally, for  $s < m < \ell$, one gets 
$
0 \to SC^*(m,\ell)_{\alpha} \hookrightarrow SC^*(s,\ell)_{\alpha} \to
SC^*(s,m)_{\alpha} \to 0.
$)
Taking the associated long exact sequence in cohomology, and taking direct limits, yields:
\begin{lemma}\label{Lemma LES for SH+}
There is a long exact sequence induced by action-restriction maps
\\[0.5mm]
\begin{tabular*}{\textwidth}{l@{\extracolsep{\fill}}cr@{\extracolsep{0pt}}} 
\strut & 
$\cdots \to H^*(M)_{\alpha} \stackrel{c^*}{\to} SH^*(M)_{\alpha}
\to SH^*_+(M)_{\alpha} \to H^{*+1}(M)_{\alpha} \to\cdots \qquad \qed$
 & \strut 
\end{tabular*}
\end{lemma}
\noindent \textbf{Example.} $\overline{M}=T^*N$: we assume $\K$ has characteristic $2$ (we explain why in Section \ref{Section Application Loop space homology and string topology}). By Abbondandolo-Schwarz \cite{Abbondandolo-Schwarz} (and the twisted analogue \cite[Thm 3]{Ritter}) the isomorphism $SH^*(T^*N)_{\alpha} \cong
H_{n-*}(\mathcal{L}N)_{\alpha}$ respects certain action
filtrations (see Section \ref{Section Application Loop space homology and string topology}). In this case:
\begin{corollary}\label{Corollary LES for SH+ of T*M}
For $\overline{M}=T^*N$ the long exact sequence is that of the pair
$(\mathcal{L}N,N)$, viewing $c:N \hookrightarrow \mathcal{L}N$ as the subspace of constant loops,
$$
\xymatrix@R=10pt@C=16pt{ \cdots \ar@{->}[r] & H^*(T^*N)_{\alpha}
\ar@{->}[d]^{\wr} \ar@{->}[r]^{c^*} & SH^*(T^*N)_{\alpha}
\ar@{->}[d]^{\wr} \ar@{->}[r] & SH_+^*(T^*N)_{\alpha}
\ar@{->}[d]^{\wr} \ar@{->}[r] & \cdots\qquad
\\
 \cdots \ar@{->}[r] & H_{n-*}(N)_{\alpha} \ar@{->}[r]^{c_{n-*}} &
 H_{n-*}(\mathcal{L}N)_{\alpha}
 \ar@{->}[r] & H_{n-*}(\mathcal{L}N,N)_{\alpha}
\ar@{->}[r] & \cdots \qed}
$$
\end{corollary}
 Now consider the wrapped analogue of the construction. The action contains the additional term $f(x(1))-f(x(0))$ (see \ref{Subsection Action functional in wrapped case}). This will not affect the ${\mathbb{A}}_H<-\delta$ estimate on the collar, since we can make the $-Rh'(R)+h(R)$ contribution arbitrarily small, whereas $f$ is locally constant on $\overline{L}\setminus L$ and hence bounded. So we can define the subcomplex $CW^*(A,\infty)_{\alpha}\subset CW^*(L;H)_{\alpha}$ for $\alpha \in H^1(\Omega(M,L))$; the quotient complex $CW^*(A,B)_{\alpha}$; and $HW^*_+(L) = HW^*(-\infty,-\delta)_{\alpha}$.
\begin{lemma}\label{Lemma HW+}
There is a long exact sequence for wrapped Floer cohomology,
\\[0.5mm]
\begin{tabular*}{\textwidth}{l@{\extracolsep{\fill}}cr@{\extracolsep{0pt}}} 
\strut & 
$
\cdots \to H^*(L)_{\alpha} \stackrel{c^*}{\to} HW^*(L)_{\alpha}
\to HW^*_+(L)_{\alpha} \to H^{*+1}(L)_{\alpha} \to\cdots \quad \qed
$
 & \strut 
\end{tabular*}
\end{lemma}
%
%

%
\section{Viterbo functoriality}
\label{Section Viterbo Functoriality}
%
%
\subsection{Liouville subdomains}
\label{Subsection Liouville subdomains}
Observe Figure \ref{Figure Viterbo restriction picture} (ignore $\overline{L}$). A \emph{Liouville subdomain} $i:(W,\theta_W) \hookrightarrow
(\overline{M},\theta)$ is the embedding of a Liouville domain
$(W,\theta_W)$ of the same dimension as $M$ such that 
$i^*\theta =
e^\rho\theta_W+\textrm{exact form}$, for some $\rho\in \R$.
\begin{example}\label{Example Weinstein neighbourhood}
Suppose $L\subset T^*N$ is a closed exact Lagrangian submanifold
$(\theta|_{L}=\textrm{exact}$, and $\mathrm{dim}\, L=\mathrm{dim}\, N)$. This induces a Weinstein embedding
$i:DT^*L \hookrightarrow T^*N$ of a small disc cotangent bundle
of $L$ onto a neighbourhood of $L$, and $i$ is a Liouville
subdomain.
\end{example}
\begin{lemma}\label{Lemma can assume Liouville subdomains are
included} For the purposes of symplectic cohomology, one can
always assume a Liouville embedding is an inclusion $W\subset \overline{M}$,
using the same $\theta, Z$ for $W$ as for $\overline{M}$.
\end{lemma}
\begin{proof}
We can always assume $\rho=1$ by redefining $\theta_W$ to be
$e^{\rho}\theta_W$. This does not affect the symplectic cohomology
theory for $\overline{W}$ (it rescales all
Hamiltonians by $e^{\rho}$).

Identify $W$ and $i(W)$. We are given $\theta_W = \theta - df$,
with $f:W\to \R$. Extend $f$ smoothly to $\overline{M}$, so that
it vanishes outside of a compact neighbourhood of $W$. Then extend
$\theta_W$, $Z_W$ to all of $\overline{M}$ by $\theta_W = \theta -
df$ and $\omega(Z_W,\cdot)=\theta_W$.
Since $d\theta_W = d\theta$ everywhere and $\theta_W = \theta$ for large $R$ on
the collar of $\overline{M}$, the Hamiltonian vector fields agree
and we can use the same almost complex structure. So the Floer
theory of $\overline{M}$ is not affected if we replace $\theta$ by $\theta_W$.

\emph{Remark. If $W$ is not entirely contained in $M$, then the change in $\theta$ does affect the definition of the coordinate $R$ and thus of the class of linear Hamiltonians. But in that case, the identity $(\overline{M},\theta) \to (\overline{M},\theta_W)$ is a symplectomorphism of contact type so $SH^*$ is invariant by \ref{Section Invariance under contactomorphisms}.
}
\end{proof}
\begin{remark}
Liouville embeddings $i\co\!\! (W,\theta) \hookrightarrow
(\overline{M},\theta)$ can be extended to
$i\co\! \! (\overline{W},\theta) \hookrightarrow
(\overline{M},\theta)$: define $i$ on the collar to be the
Liouville flow of $Z_W$ in $\overline{M}$ $($so $[1,\infty) \times
\partial W \ni (R=e^r,y) \mapsto \varphi^r_{Z_W}(y) \in \overline{M})$. Since
$Z_W$ is outward pointing along $\partial W$, this flow will not
reintersect $W$.
\end{remark}

By Lemma \ref{Lemma can assume Liouville subdomains are included},
we always assume $W\subset \overline{M}$. Let $C_W= [1,1+\varepsilon) \times
\partial W$ be a collar of $\partial W\subset \overline{M}$.
Choose $J$ on $\overline{M}$ of contact type near $R=e^r =
1+\varepsilon$ (using the $R$ coordinate of the $C_W$ collar). Choose $H:\overline{M}\to \R$ with $H=mR$ near
$R=1+\varepsilon$ and $H\geq 0$ on $\overline{M}\setminus (W\cup
C_W)$.
\begin{corollary}\label{Corollary No escape}
Any Floer solution with asymptotics in $W \cup C_W$ cannot escape
$W$.
\end{corollary}
\begin{proof}
Let $V=\overline{M}\setminus (W \cup [1,R) \times \partial W)$,
with $R$ close to $1+\varepsilon$ chosen generically so that
$\partial V = \{ R \} \times \partial W$ intersects the Floer solution $v$
transversely. Apply Lemma \ref{Lemma no escape} to the restriction
$u$ of $v$ to the preimage of each component of $\mathrm{im}\,v\cap V$. Hence
$\mathrm{im}\, v\subset W\cup C_W$.
\end{proof}
%
\subsection{Viterbo Functoriality}
\label{Subsection Viterbo Functoriality}
%
%
For Liouville subdomains $W\subset \overline{M}$, Viterbo \cite{Viterbo1}
constructed a restriction map $SH^*(M)\to SH^*(W)$ and McLean
\cite[Lemma 10.2]{McLean} proved that it is a ring homomorphism. We now
prove a stronger statement:
\begin{theorem}\label{Theorem Viterbo Functoriality}
Let $ i:(W,\theta_W) \hookrightarrow (\overline{M},\theta) $ be a Liouville
subdomain. Then there exists a $\mathrm{restriction\; map}$,
$SH^*(i)_{\eta}:SH^*(M)_{\eta} \to SH^*(W)_{i^*\eta},$
which is a TQFT map which fits into a commutative diagram which respects TQFT
structures:
$$
\xymatrix@R=12pt{SH^*(W) \ar@{<-}[r]^-{SH^*(i)} \ar@{<-}[d]_-{c^*}
& SH^*(M) \ar@{<-}[d]^-{c^*} & & SH^*(W)_{i^*\eta}
\ar@{<-}[r]^-{SH^*(i)_{\eta}} \ar@{<-}[d]^-{c^*} &
SH^*(M)_{\eta} \ar@{<-}[d]_-{c^*} \\
H^*(W) \ar@{<-}[r]^-{i^*} & H^*(M) & & H^*(W)\otimes \Lambda
\ar@{<-}[r]^-{i^*} & H^*(M)\otimes \Lambda}
$$
In particular, all maps are unital ring homomorphisms.
\end{theorem}
\begin{proof}
\begin{figure}[ht]
\input{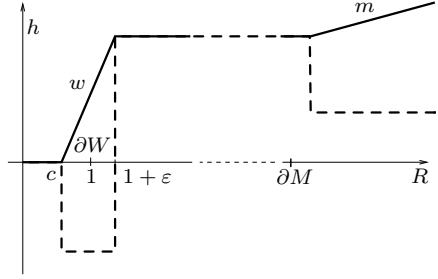}
\caption{A step-shaped Hamiltonian with slopes $w, m$.
The dashed line is the action function $\mathbb{A}(R)=-Rh'(R)+h(R)$. See also Figure \ref{Figure Viterbo restriction picture}}\label{Figure
Viterbo functoriality}
\end{figure}

We may assume $W\subset \overline{M}$ and $\theta_W=\theta|_W$ by
Lemma \ref{Lemma can assume Liouville subdomains are included}.

Pick a (smoothing of the) step-shaped Hamiltonian $H\geq 0$ in
Figure \ref{Figure Viterbo functoriality}. More precisely, the
Liouville flow parametrizes a tubular neighbourhood
$(0,1+\varepsilon) \times \partial W$ of $\partial W \subset
\overline{M}$ (compare \ref{Subsection Liouville domains Definition}). This determines an $R$-coordinate on that neighbourhood with $\partial W=\{ R=1 \}$ (which in general is not related to the $R$-coordinate that we use on the collar of $\overline{M}$). Using these $R$-coordinates on the two
collars, define $H$ to be linear of slope $m$ for $R\gg 0$ on the collar of
$\overline{M}$ and linear of slope $w\gg m$ on
$(c,1+\varepsilon) \times \partial W$. Make $H$ constant elsewhere (rounding off the corners in the graph). Here $w,m$ are
generic (so not Reeb periods of $\partial W$, $\partial M$). In the regions where the graph is flat we make a time-independent positive Morse perturbation of $H$ so the $1$-orbits there are critical points of $H$ and their action is the value of $H\geq 0$.

Even if $W$ is not contained in $M$, we can fix a large enough $R_0$ so that the region $R\geq R_0$ on the collar of $\overline{M}$ where we make $H$ have slope $m$ is away from $W$.

We now show how to separate the action values of the $1$-orbits lying in $W$ and those outside $W$ (we mimic the discussion of \cite[Sec.4.3]{Ritter} which is based on Viterbo's original argument \cite{Viterbo1}). By construction, the only non-constant $1$-orbits of $H$ arise near the regions where the graph of Figure \ref{Figure Viterbo functoriality} has corners. By \ref{Subsection Liouville domains Definition}, the action of these orbits is $\mathbb{A}(R)=-Rh'(R)+h(R)$ (in the appropriate $R$-coordinate), which is the value where the tangent line to the graph of $H$ intersects the vertical axis.

So $w\gg m$ ensures that the smallest values of $\mathbb{A}(R)$ outside $W$ arise near $R=1+\varepsilon$, since the $1$-orbits outside the corner $R= 1+\varepsilon$ will have large positive action if we make $w$ large.

Fix $0<\delta_w<1$ such that no Reeb periods are in $[w-\delta_w,w+\delta_w]$ (the Reeb periods are a discrete subset of $[0,\infty)$ by \ref{Subsection Liouville domains Definition}). By \ref{Subsection Liouville domains Definition}, the slopes of $h$ at $1$-orbits are Reeb periods, so 
the $1$-orbits near $R=1+\varepsilon$ have $\mathbb{A}(R)\geq -(1+\varepsilon)(w-\delta_w)+h(1+\varepsilon)\geq (1+\varepsilon)\delta_w-wc$, using $h(1+\varepsilon) = w(1+\varepsilon-c)$. So, for $c\ll 1/w$, \textbf{the $\mathbf{1}$-orbits outside $\mathbf{W}$ have action in $\mathbf{(\delta_w,\infty)}$}.

The $1$-orbits in $W$ near $R=c$ have action in $[-wc,0]$ and the critical points of $H$ in $W$ have arbitrarily small positive action. So given $w$, for $c\ll 1/w$  \textbf{the $\mathbf{1}$-orbits in $\mathbf{W}$ have action in $\mathbf{(-\delta_w/G_c,\delta_w/G_c)}$} for a large number $G_c\geq 1$, such that $G_c \to \infty$ as we let $c\to 0$.

Let $H'$ be another step-shaped Hamiltonian for $w'\geq w,m'\geq m$ with $w'\gg m'$ so that
the $1$-orbits outside the corner $R= 1+\varepsilon$ have large action (action $\geq 1$ suffices). We can modify $H$ by decreasing $c$ so that $\delta_w/G_c < \delta_{w'}$. We then modify $H'$ by decreasing $c'$ so that: 

(1) $c'<c$, so $H'>H$, so there is a monotone homotopy $H_s$ from $H'$ to $H$ with $\partial_s H_s\leq 0$; 

(2) $\delta_{w'}/G_{c'}<\delta_w/G_c<\delta_{w'}$.

\noindent \emph{Sub-claim: there is a commutative diagram}
\\[0.5mm]
\begin{tabular*}{\textwidth}{l@{\extracolsep{\fill}}cr@{\extracolsep{0pt}}} 
\strut & 
$
\xymatrix@R=10pt{ SC^*(W,i^*H')_{i^*\eta}
\ar@{<-}[r]^-{\;i^{-1}}_-{\cong} \ar@{<-}[d]_-{\mathrm{cont.}} & SC^*(M,H';{\mathbb{A}}_{H'}<\delta_{w}/G_c)_{\eta}
\ar@{<-}[rr]^-{\mathrm{action-restrict}} \ar@{<-}[d]^-{\mathrm{cont.}} &&
SC^*(M,H')_{\eta}  \ar@{<-}[d]^-{\mathrm{cont.}} \\
SC^*(W,i^*H)_{i^*\eta} \ar@{<-}[r]^-{\;i^{-1}}_-{\cong} &
SC^*(M,H;{\mathbb{A}}_H<\delta_w/G_c)_{\eta} \ar@{<-}[rr]^-{\mathrm{action-restrict}} && SC^*(M,H)_{\eta}}
$
 & \strut 
\end{tabular*}
\\[0.5mm]
\emph{where: the vertical maps are the (twisted) continuation maps defined by $H_s$; and the horizontal maps in the second square are action-restriction quotient maps (see Section \ref{Section SH+}).}

We now prove the \emph{Sub-claim}. The second square commutes because the action decreases along Floer continuation solutions for $H_s$ (using $\partial_s H_s\leq 0$ and 
the energy estimate in \ref{Subsection Floer continuation solutions}). In the first square, by the action constraint the only $1$-orbits being considered lie in $W$. By Corollary \ref{Corollary No escape}, Floer trajectories with ends in $W$ do not escape $W$. So the $i^{-1}$ maps are well-defined
identifications of chain complexes. The first square commutes because, by Corollary \ref{Corollary No escape} and Remark \ref{Remark no escape Lemma comments}(\ref{Item no escape for z dependent H}), also Floer continuation solutions with ends in $W$ do not escape $W$.

We now prove that the above diagram respects the TQFT structure on cohomology, in the sense of Theorem \ref{Theorem Operations are compatible with
limit} (for $A_a=A_a'$, $B_b=B_b'$). For the first square, this follows because the Floer solutions in $\overline{M}$ with asymptotics in $W$ and the Floer solutions in $\overline{W}$ with asymptotics in $W$ must both lie in $W$ by Corollary \ref{Corollary No escape}, and so they coincide via $i^{-1}$. For the second square, we first need to consider the behaviour of the action-restriction map under TQFT operations.

Consider $\psi_S: \otimes_b SC^*(B_b H)\to \otimes_a SC^*(A_a H)$ on $\overline{M}$. By Theorem \ref{Theorem Operations are compatible with limit}, on cohomology $\psi_S$ does not change if we increase $w$ and decrease $c$: it only depends on the slope $m$ at infinity. 
By the above estimates, for $w\gg m, c\ll 1/w$ we can ensure that: $G_c> p+q-1$; the $1$-orbits of $A_a H,B_b H$ in $W$ have action in $(-\delta_w/G_c,\delta_w/G_c)$; the $1$-orbits of $A_a H,B_b H$ outside $W$ have action $> \delta_w>\delta_w/G_c$. Therefore, in particular, all $1$-orbits of $B_bH$ have $\mathbb{A}_{B_b H}>-\delta_w/G_c$.

For action-restrictions to commute with $\psi_S$, we need to ensure that no Floer solution $u:S \to \overline{M}$ which contributes a non-zero multiple of $\otimes_a x_a$ to $\psi_S(\otimes_b y_b)$ has some $\mathbb{A}_{B_kH}(y_k)> \delta_w$ but all $\mathbb{A}_{A_a H}(x_a)<\delta_w/G_c$ (this would be a problem since the action-restriction of $\otimes_b y_b$ is zero, but the action-restriction of $\otimes_a x_a$ is non-zero). Suppose by contradiction that this was the case. Then by the energy estimate in \ref{Subsection Energy Appendix Wrapped} and using $G_c>p+q-1$ and $\mathbb{A}_{B_b H}(y_b)>-\delta_w/G_c$,
$$
E(u) = \sum_a \mathbb{A}_{A_aH}(x_a) - \sum_b \mathbb{A}_{B_bH}(y_b) <  p\frac{\delta_w}{G_c} + (q-1)\frac{\delta_w}{G_c}  - \delta_w  < 0
$$
contradicting $E(u)\geq 0$. So the only contributions to the TQFT operation $\psi_S$ which survive under action-restriction involve Floer solutions $u$ all of whose asymptotics lie in $W$. A similar argument holds for $H'$ in place of $H$. This, combined with Theorem \ref{Theorem Operations are compatible with
limit}, implies that the second square respects the TQFT operation $\psi_S$ on cohomology. 
This proves the \emph{Sub-claim}. 

Call $i_{H}:SH^*(W,i^*H)_{i^*\eta} \leftarrow SH^*(M,H)_{\eta}$ the composite on cohomology of the horizontal maps in the \emph{Sub-claim}. By the \emph{Sub-claim}, the direct limit of the $i_{H'}$ as we let $w'\gg m' \to \infty$ is defined. This limit defines the map $SH^*(i)_{\eta}$ of the Theorem. By construction, $SH^*(i)_{\eta}$ fits into a commutative diagram with $i_H$. To obtain the diagram in the Theorem, define $H$ by taking $w\ll 1$ and then $m\ll w$, so that no Reeb periods for $\partial W,\partial M$ lie in $[0,w],[0,m]$. So the only $1$-orbits of $H$ are critical points and, by Section \ref{Subsection The maps c from ordinary cohomology} and Theorem \ref{Theorem unital ring structure twisted}, $i_H$ can be identified (in a TQFT-preserving way) with the restriction on Morse cohomology $MH^*(W,i^*H)\otimes \Lambda \leftarrow MH^*(M,H)\otimes \Lambda$, which indeed is the pull-back on ordinary cohomology.
%
%
%
%
\end{proof}
\begin{theorem}\label{Theorem Properties of restriction map}
\strut\begin{enumerate}

\item The restriction map $SH^*(i)_{\eta}: SH^*(M)_{\eta}\to
SH^*(W)_{i^*\eta}$ is invariant under isotopies of $i:
W\hookrightarrow \overline{M}$ among embeddings satisfying
$i^*\theta_M-e^{\textrm{constant}}\theta_W=\textrm{exact}$.

\item Functoriality: given nested Liouville subdomains $W'
\hookrightarrow W \hookrightarrow \overline{M}$, then
$SH^*(M)_{\eta}\! \to SH^*(W)_{i^*\eta}\! \to SH^*(W')_{i'^*\eta}$
equals $SH^*(M)_{\eta}\!\to SH^*(W')_{i'^*\eta}$.

\item $SH^*(i)_{\eta}=\textrm{id}$ for the following maps: $i=\textrm{id}:M \subset \overline{M}$; the Liouville flow $i=\varphi_Z^{\tau}: M \hookrightarrow \overline{M}$ for time $\tau\in \R$; and the inclusions $i:M\cup [1,R_0]
\times
\partial M\hookrightarrow \overline{M}$.
\end{enumerate}
\end{theorem}
\begin{proof} (1) Consider a family $i_{\lambda}$,
$0\leq \lambda \leq 1$, of such embeddings. By the proof of Lemma \ref{Lemma
can assume Liouville subdomains are included} (and the \emph{Remark} contained therein), this reduces to considering a fixed $W\subset \mathrm{int}(M)$ with Liouville form
$\theta_{\lambda}=\theta_0 + df_{\lambda}$ on $\overline{M}$, such that $f_{\lambda}$ is supported in a neighbourhood
of $W$ in $\mathrm{int}(M)$ and $\theta_{\lambda}|_{W_{\varepsilon}} = \theta_W$ where $W_{\varepsilon}=W\cup (1,1+\varepsilon)\times \partial W$ (we extend $\theta_W$ to the collar just as for $\overline{W}$). Notice that the $R$-coordinate on the collar of $\overline{M}$ will be independent of $\lambda$.

Choose the step-shaped Hamiltonian $H$ with slopes $w,m$ so that its constant step includes the bounded region outside $W_{\varepsilon}$ where $\theta_{\lambda} \neq \theta_0$ (in the two regions where we make $H$ linear the two $R$-coordinates do not depend on $\lambda$ since $\theta_{\lambda}=\theta_0$ there). Now consider the computation of the actions in the proof of Theorem \ref{Theorem Viterbo Functoriality}: near the regions where $H$ has slopes $w,m$ the form $\theta_{\lambda}$ equals $\theta_W,\theta_0$ respectively, and the action estimates in the region where $H$ is constant do not involve $\theta_{\lambda}$ (the generators there, after a small perturbation, are critical points of $H$, and these have actions equal to the value of $H$). As explained in the proof of Theorem \ref{Theorem Viterbo Functoriality}, the Floer solutions joining orbits in $W$ stay entirely in $W$, so the Floer theory there only depends on the form $\theta_W$. Finally, the restriction map for each $\lambda$ can be constructed by using action-restrictions for such Hamiltonians $H$, so we have shown that the restriction maps do not depend on whether we use $\theta_0$ or $\theta_1$ on $\overline{M}$.

(2) Using an appropriate step-shaped Hamiltonian with two steps
instead of one, one can separate the action values of the
generators in $W'$, $W\setminus W'$ and $\overline{M} \setminus
W$. So claim (2) follows because the composite of action-restriction maps is an action-restriction map.

(3) For $i\!=\!\textrm{id}$, the step-shaped Hamiltonian induces the continuation $SH^*(H^m)\!
\to \! SH^*(H^w)$ in the notation of \ref{Subsection Hamiltonians Linear
At Infty}. So the direct limit as $w\gg m\to \infty$ is the
identity. For $i=\varphi_Z^{\tau}:M  \hookrightarrow \overline{M}$, $i^*\theta_M=e^{\tau}\theta_M$, and since $i$ is isotopic to the identity, $SH^*(i)_{\eta}=\mathrm{id}$ by (1). The inclusions $i$ in (3) are isotopic to the identity via $(\varphi_Z^{\tau})_{\tau\in [0,\log R_0]}$, so $SH^*(i)_{\eta}=\mathrm{id}$ by (1).
\end{proof}
%
%
\subsection{Wrapped solutions with asymptotics in $W$ do not escape $W$}
\label{Subsection No escape for wrapped solutions}
%
Observe Figure \ref{Figure Viterbo restriction picture}. Let $W \hookrightarrow M$ be a Liouville subdomain, and $L$ an
exact Lagrangian with transverse Legendrian intersections with
$\partial W, \partial M$ (by Lemma \ref{Lemma can assume Liouville
subdomains are included}, we can assume $W\subset M$,
$\theta_W=\theta|_W$). So explicitly:
$\theta|_L=df$ for $f:L \to \R$, and $df=0$ on $L\cap(\partial W \cup \partial M)$.

\begin{figure}[ht]
\input{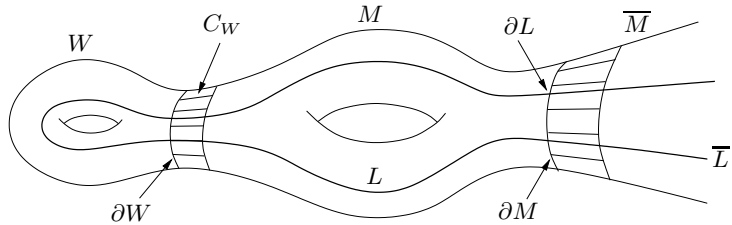}
\caption{A Liouville subdomain $W\subset M$, and a Lagrangian $\overline{L}\subset \overline{M}$. The indicated collars are where conical ends can be attached to form $\overline{W},\overline{M}$.} \label{Figure Viterbo restriction picture}
\end{figure}

In preparation for the construction of the wrapped Viterbo restriction map in \ref{Subsection Wrapped Viterbo restriction}, we make an additional assumption (the necessity of which is explained in \cite[Ex.4.2]{Abouzaid-Seidel}):\\[1mm]
\indent
\boxed{\emph{One can extend $f:L\to \R$ to $f:M\to
\R$ so that $f$ is locally constant on
$\partial W \cup \partial M$.}}\\[1mm]
The condition we really want is that $f|_L=0$ near $\partial W \cup \partial M$ in order to apply Lemma \ref{Lemma no escape with Lag bdry condns}. The somewhat weaker assumption above actually guarantees this condition can be achieved after deforming $L$ by a Hamiltonian isotopy of $M$ relative to $\partial W \cup
\partial M$ (\cite[Lemma 4.1]{Abouzaid-Seidel}). In particular, $L$ then has the form $(\textrm{interval})\times (L\cap
\partial W)$ near $\partial W$ and similarly near $\partial M$, and we can extend $f$ to $\overline{M}$ by defining it to be locally constant on the collar of $\overline{M}$. 

By the no escape Lemma \ref{Lemma no escape with Lag bdry condns}, Corollary \ref{Corollary No escape} generalizes to:

\begin{corollary} \label{Corollary No escape with Lagr bdry condns}
Any wrapped Floer solution with asymptotics in $W \cup C_W$ cannot escape
$W$.
\end{corollary}
%
%
\subsection{Wrapped Viterbo restriction}
\label{Subsection Wrapped Viterbo restriction}
%
%
Continuing with the notation and assumptions introduced above Corollary \ref{Corollary No escape with Lagr bdry condns}, we now prove the wrapped analogue of Theorem \ref{Theorem Viterbo Functoriality}. The construction of the Viterbo restriction for wrapped cohomology is due to Abouzaid-Seidel \cite{Abouzaid-Seidel}: they prove that it respects the $A_{\infty}$-operations on a chain-level telescope model for $HW^*(L)$.

\begin{theorem}\label{Theorem wrapped Viterbo}\label{Theorem wrapped twisted Viterbo}
There is a restriction map
$HW^*(L\subset M) \to HW^*(L\cap W \subset W),$
commuting with $c^*$, and compatible with the TQFT and the
$SH^*$-module structure. 

More generally there is a twisted restriction
$HW^*(L;\underline{\Omega}_{\alpha}) \to HW^*(L\cap W;\underline{\Omega}_{\alpha|W})$, which for $\alpha=\tau\eta$ respects the structures of Theorem \ref{Theorem wrapped twisted TQFT and module}.

\end{theorem}
\begin{proof}
Observe that if we can separate the action values as in the proof of
Theorem \ref{Theorem Viterbo Functoriality}, then the proof is identical after replacing the groups $SC^*(L;H)$, $SC^*(W;H)$ by $CW^*(L;H)$, $CW^*(L\cap W;H)$, and using the no escape Corollary \ref{Corollary No escape with Lagr bdry condns} instead of Corollary \ref{Corollary No escape}.

We now show how to separate the action values. In the regions where $f=0$, the action of Hamiltonian chords is $-Rh'(R)+h(R)$ 
(see \ref{Subsection Action functional in wrapped case}). Thus, in those regions, the same estimates that we made in the proof of Theorem \ref{Theorem Viterbo Functoriality} for the step-shaped Hamiltonian $H$ of slopes $w,m$ will also hold in the wrapped setup. The non-constant Hamiltonian chords lie in the regions where we smoothen the corners of the graph of $H$ drawn in Figure \ref{Figure Viterbo functoriality}: near $\partial M,\partial W$ we indeed have $f=0$, but the region near $R=c$ inside $W$ is problematic since it is far away from $\partial W$ as we let $c\to 0$, so $f=0$ may fail there.

Let $\varphi=\varphi_Z^{\log c}$ be the Liouville flow for time $\log
c<0$. Push $L\cap W$ into $\varphi(L\cap W)$. Since $L$ has the
form $(\textrm{interval})\times (L\cap
\partial W)$ near $\partial W$, the flow rescales the interval by
$c$, so we can smoothly join $\varphi(L\cap W)$ to
$\overline{L}\setminus W$ by some $(\textrm{interval})\times
(L\cap
\partial W)$. Call $P$ the resulting Lagrangian. The part of $P$ contained in the region $c\leq R \leq 1+\varepsilon$ of $W\cup C_W$ is conical: $[c,1+\varepsilon]\times (L \cap \partial W)$. So $\theta|_P=0$ on that part, and so globally $\theta|_P=df_P$ for a function $f_P:\overline{M}\to \R$ which is zero also in the problematic region near $R=c$. Thus for $(P, H, W, M,\theta)$ the action values can be separated. So we have a restriction
$$
HW^*(P;H) \stackrel{\textrm{restrict}}{\longrightarrow}
HW^*(P;H,{\mathbb{A}}_H<\delta_w/G_c) \cong HW^*(W\cap P;H|_W).
$$
We claim that this naturally induces a restriction $HW^*(L;H^m)\to
HW^*(W\cap L;H^w)$ for some $H^m,H^w$ with slopes $m,w$ at
infinity on $\overline{M},\overline{W}$. This claim implies the
Theorem after taking direct limits $w\gg m\to \infty$.

Since the above construction gives a family of exact Lagrangians $P_{\lambda}$ interpolating $P_0=\overline{L}$ and $P_1=P$, there is a Hamiltonian isotopy $\psi$ of $\overline{M}$ supported in $W$ which maps $\overline{L}$ onto
$P$. This induces a natural isomorphism $\psi^*:HW^*(P;H)\cong HW^*(L;H^m)$
by pulling back the Floer data, which determines the Hamiltonian $H^m=\psi^*H$ with slope $m$ at infinity.

Secondly, the above Liouville map $\varphi$ also induces a natural pull-back isomorphism
$$
\varphi^*:HW^*(W\cap P;H|_W;\theta) \cong HW^*(W\cap
L;\varphi^*(H|_W);c\,\theta),
$$
where $\varphi^*(H|_W)=H|_W\circ\varphi$ has slope $wc$ at
infinity on $\overline{W}$. Rescaling $c\,\theta$ and
$\varphi^*(H|_W)$ by $1/c$ preserves the Hamiltonian vector field
and thus the whole Floer theory:
$$
HW^*(W\cap L;\varphi^*(H|_W);c\,\theta) \equiv HW^*(W\cap
L;H^w;\theta),
$$
where $H^w=\varphi^*(H|_W)/c$ has slope $w$ at infinity, as
required.

Both these isomorphisms are compatible with the algebraic structures (the $c^*$-map, the TQFT, and the $SH^*$-module structure) since analogous isomorphisms can be defined for the algebraic structures again simply by pulling-back the relevant Floer data by global maps (in the first case, a Hamiltonian isotopy $\psi$, and in the second case, the Liouville map $\varphi$). We briefly illustrate how this works in the case of the $c^*$-maps, but we omit a discussion of the other algebraic structures for sake of brevity.

Recall $c^*:H^*(L) \to HW^*(L)$ arises at the Floer level as a direct limit of continuation maps $c^*:HW^*(L;H^{\delta}) \to HW^*(L;H)$ which increase the slope of $H$ at infinity (see Section \ref{Section Canonical map from ordinary cohomology}). Let $H$ be the step-shaped Hamiltonian from above (with slopes $w\gg m \gg 1$) and let $H^{\delta}$ be a step-shaped Hamiltonian with slopes $\delta_1,\delta_2$ where $1\gg \delta_1 \gg \delta_2>0$. These conditions are chosen so as to ensure that restriction maps can be constructed for $H$ and for $H^{\delta}$, and that the slopes $\delta_1,\delta_2$ are appropriate for defining the $c^*$ maps for $L\cap W$ and $L$. Let $H_s$ be a monotone homotopy interpolating $H$ with $H^{\delta}$ with $\partial_s H_s\leq 0$. By our previous discussion, we obtain the following commutative diagram:
$$
\xymatrix@R=12pt@C=16pt{HW^*(W\cap L;H^w) \ar@{<-}^-{\varphi^*}_{\cong}[r] \ar@{<-}_-{c^*}[d] &
HW^*(W\cap P;H|_W) \ar@{<-}_-{\mathrm{cont.}}[d]  \ar@{<-}^-{\textrm{restr.}}[r] &
HW^*(P;H) \ar@{<-}_-{\mathrm{cont.}}[d]  \ar@{<-}^-{(\psi^*)^{-1}}_{\cong}[r]  &
HW^*(L;H^m) \ar@{<-}^-{c^*}[d]\\
HW^*(W\cap L;H^{\delta_1}) \ar@{<-}^-{\varphi^*}_{\cong}[r] &
HW^*(W\cap P;H^{\delta}|_W) \ar@{<-}^-{\textrm{restr.}}[r] &
HW^*(P;H^{\delta}) \ar@{<-}^-{(\psi^*)^{-1}}_{\cong}[r] &
HW^*(L;H^{\delta_2})}
$$
where: the first square arises from pulling back via $\varphi$ the data $H^{\delta}|_W, H|_W, H_s|_W,J|_W$; the second square arises from action-restrictions since for the Lagrangian $P$ we have action estimates analogous to those in the proof of Theorem 
\ref{Theorem Viterbo Functoriality}; and the last square arises from pulling back via $\psi$ the data $H^{\delta},H,H_s,J$.
\end{proof}
%
%
%
%
%
%
%
\section{Vanishing Criteria}
\label{Section Vanishing criteria}
\label{Subsection Vanishing of symplectic cohomology}

%
%
This Section relies on Theorem \ref{Theorem unital ring structure twisted} and Section \ref{Subsection Definition of the unit}. We recall that by construction the map $c^*$ factors via $c_H^*:H^*(M) \to SH^*(H)_{\eta}$ as
$c^*:H^*(M) \to SH^*(H)_{\eta} \to SH^*(M)_{\eta}$
for any Hamiltonian $H$ linear at infinity. Denote $1\mapsto e_H \mapsto e_M$ the image of the unit via $c^*$. 

The results in this Section hold also in the untwisted setup (simply ignore $\eta$).
\begin{theorem}\label{Theorem Vanishing criteria} \strut
\begin{enumerate}
 \item $SH^*(M)_{\eta}=0$ if and only if the unit $e_M=0\in SH^0(M)_{\eta}.$
 \item For a Liouville subdomain
$i:(W,\theta_W)\! \hookrightarrow \! (\overline{M},\theta)$,
if $SH^*(M)_{\eta} \! = \! 0$ then $SH^*(W)_{i^*\eta} \! = \! 0.$
\end{enumerate}
\end{theorem}
\begin{proof}
1) If $e_M=0$ then for all $y\in SH^*(M)_{\eta}$,
$y = e_M\cdot y = 0 \cdot y = 0, $
so $SH^*(M)_{\eta}=0$.\\
2) Suppose $SH^*(M)_{\eta}=0$, so $e_M=0$. The restriction map $\varphi:
SH^*(M)_{\eta} \to SH^*(W)_{i^*\eta}$ respects the TQFT by Theorem \ref{Theorem Viterbo Functoriality}. So it sends unit to unit. Thus $e_W=\varphi(e_M)=\varphi(0)=0$. But by (1) applied to $W$, $e_W=0$ forces $SH^*(W)_{i^*\eta} = 0$.
\end{proof}

\begin{theorem}\label{Theorem vanishing at finite stage}
The following conditions are equivalent
\begin{enumerate}
 \item $SH^*(M)_{\eta}=0$;
 \item $e_H=0$ for some Hamiltonian $H=H^{\ell}$ linear of slope $\ell>0$ at infinity;
 \item $c^*_H:H^*(M) \to SH^*(H)_{\eta}$ vanishes for some $H$;
 \item $c^*:H^*(M)\to SH^*(M)_{\eta}$ vanishes;
 \item for some $\ell\!>\!0$, the continuation $SH^*(H^m)_{\eta} \!\to\! SH^*(H^{\ell + m})_{\eta}$ is zero for all slopes $m$.
\end{enumerate}
\end{theorem}
\begin{proof}
By definition, $e_M=\varinjlim e_H=0\Leftrightarrow$ some $e_H=0$. So (1)$\Leftrightarrow$(2) by
Theorem \ref{Theorem Vanishing criteria}. 

If $e_H=0$ then  $0=\psi_P(e_H, y) = c^*_H y \in SH^*(H)_{\eta}$ for all $y\in H^*(M)$ using the product $\psi_P: SH^*(H)_{\eta} \otimes H^*(M) \to SH^*(H)_{\eta}$. So (2)$\Rightarrow$(3). For (3)$\Rightarrow$(2) recall that $e_H=c^*_H(1)$.

The factorization $c^*\!:\!H^*(M)\!\to\! SH^*(H^m)_{\eta}\! \to\! SH^*(H^{m'})_{\eta}\! \to\! SH^*(M)_{\eta}$, for slopes $0<m\leq m'$, shows that the rank of the image $c^*_H(H^*(M))\subset SH^*(H)_{\eta}$ decreases as the slope of $H$ increases. So the rank must stabilize for large enough slopes. So (3)$\Leftrightarrow$(4).

Finally (5)$\Rightarrow$(3) by taking $m$ to be very small, and (2)$\Rightarrow$(5) by Theorem \ref{Theorem characterization of continuation maps}.
 \end{proof}
\begin{theorem}\label{Theorem Vanishing criteria for SH_*}
The following are equivalent
\begin{enumerate}
 \item $SH_*(M)_{\eta}=0$;
 \item the counit $\psi^C:SH_*(M)_{\eta} \to \K$ vanishes;
 \item $c_*^H:SH_*(H)_{\eta}\to H_*(M)$ vanishes for some $H$;
 \item $c_*:SH_*(M)_{\eta}\to H_*(M)$ vanishes.
\end{enumerate}
\end{theorem}
\begin{proof} 
 This can be proved either directly as for Theorem \ref{Theorem vanishing at finite stage}, or by dualizing the statement of Theorem \ref{Theorem vanishing at finite stage} using the dualization results of Section \ref{Subsection Symplectic homology is the dual of symplectic cohomology}.
\end{proof}

\begin{theorem}\label{Theorem Vanishing of homology iff
cohomology}  $SH^*(M)_{\eta}=0$ if and only if $SH_*(M)_{-\eta}=0$.
\end{theorem}
\begin{proof}
This follows by Theorems \ref{Theorem vanishing at finite stage} and \ref{Theorem Vanishing criteria for SH_*} since $c^*,c_*$ are dual to each other (\ref{Subsection Symplectic homology is the dual of symplectic cohomology}).
\end{proof}
\begin{theorem}\label{Theorem wrapped vanishing criterion}
The following are equivalent
\begin{enumerate}
 \item $HW^*(L)_{\eta}=0$;
 \item the unit $\mathcal{W}_C(1)=c^*(1)\in HW^0(L)_{\eta}$ vanishes;
 \item $c^*_{H^{\ell}}(1) = 0 \in HW^0(L;H^{\ell})_{\eta}$ for some slope $\ell>0$;
 \item $c^*_H :H^*(L) \to HW^*(L;H)_{\eta}$ vanishes for some $H$;
 \item $c^*:H^*(L) \to HW^*(L)_{\eta}$ vanishes.
 \item for some $\ell\!>\!0$, the continuation $HW^*(L;H^m)_{\eta} \!\to\! HW^*(L;H^{\ell + m})_{\eta}$ is zero for all $m$.
\end{enumerate}
\end{theorem}
\begin{proof}
 This follows by mimicking the proof of Theorem \ref{Theorem vanishing at finite stage} in the wrapped setup.
\end{proof}
\begin{theorem}\label{Theorem SH=0 implies HW=0}\strut
\begin{enumerate}
 \item If $SH^*(M) =  0$ then $HW^*(L)  =  0$;
 \item If $SH^*(M)_{\overline{\eta}} = 0$ then $HW^*(L)_{\eta}  =  0$ (in the notation of Theorem \ref{Theorem
wrapped twisted TQFT and module}).
\end{enumerate}
\end{theorem}
\begin{proof}
This follows from the fact that $HW^*$ is a module over $SH^*$ (Theorems \ref{Theorem wrapped module over SH} and \ref{Theorem wrapped twisted TQFT and module}): if $e=0$ in $SH^*$, then $y=e\cdot y = 0\cdot y =0$ for all $y\in HW^*$.
\end{proof}
%
%
%
%
%
\section{Arnol'd chord conjecture}
\label{Section Application Arnol'd Chord conjecture}
%
%
The chord conjecture states that a contact manifold containing a Legendrian
submanifold has a Reeb chord with ends on the Legendrian. This was originally stated for Legendrian knots in the standard $3$-sphere by Arnol'd \cite{Arnold}, and we recommend the work of Cieliebak \cite{Cieliebak} for a modern Floer-theoretic     approach to the general statement. Our setup \ref{Subsection Lagrangians inside Liouville domains}
involves the Legendrian $\partial L = L\cap \partial M$ inside the
contact manifold $\partial M \subset \overline M$.
\begin{theorem}\label{Theorem Vanishing of HW implies
Arnol'd Chord Conj} If $HW^*(L)=0$ then the Arnol'd chord
conjecture holds for $\partial L \subset
\partial M$ $($existence of a Reeb chord$)$.
For a generic contact form $\alpha$ there are at least
$\textrm{rank}\; H^*(L)$ chords. The same holds if
$HW^*(L)_{\eta}=0$ for $\eta\in H^2(M,L;\R)$.
\end{theorem}
\begin{proof}
Suppose there are no Reeb chords. Then we do not need to perturb
$\alpha$ to avoid transversality issues caused by degenerate Reeb
chords. By Lemma \ref{Lemma Reeb chords that can arise} there are no Hamiltonian chords on the collar of $\overline{M}$, and by the maximum principle in \ref{Subsection Maximum principle for the Lagrangian setting} all wrapped trajectories lie in $M$ where we can ensure $H|_M=H^{\delta}|_M$ for any Hamiltonian $H$ linear at infinity. So, by Section \ref{Section Canonical map from ordinary cohomology},
$c^*:H^*(L) \cong HW^*(L;H^{\delta}) \cong HW^*(L)$, so $HW^*(L)\neq 0$.

Now suppose $HW^*(L)=0$, $\alpha$ generic. By Lemma \ref{Lemma
HW+} if $HW^*(L)=0$
then $HW^*_+(L)\cong H^{*+1}(L)$, so there are at least
$\textrm{rank}\; H^*(L)$ distinct Hamiltonian chords on the collar
(see the \emph{Technical Remarks} in \ref{Subsection Wrapped TQFT structure}), and the same holds for
Reeb chords by Lemma \ref{Lemma Reeb chords that can arise}. The
twisted case is analogous, since $HW^*(L)_{\eta}=0$ implies
$HW^*_+(L)_{\eta}\cong H^{*+1}(L)\otimes \Lambda$.
\end{proof}

\emph{Remark. The proof is similar to Viterbo's applications \cite{Viterbo1} to
the \emph{Weinstein conjecture} $($existence of a closed Reeb orbit$)$:
if there are no closed Reeb orbits then $c^*\!:\!H^*(M)\!\to\!
SH^*(M)$ is an isomorphism. So the Weinstein conjecture holds if
$SH^*(M)=0$ or $SH^*(M)_{\eta}=0$.}\\
\indent \emph{Remark. In general $\textrm{rank}\; H^*(L) \geq \frac{1}{2}\textrm{rank}\; H^*(\partial L)$ (use the long exact sequence for the pair $(L,\partial L)$, and $\textrm{rank}\; H^*(L)=\textrm{rank}\;H^*(L,\partial L)$ by Poincar\'{e} duality and universal coefficients).}
\begin{example} Let $M$ be a subcritical Stein
manifold. Then Theorem \ref{Theorem Vanishing of HW implies
Arnol'd Chord Conj} applies since $SH^*(M)=0$ by Cieliebak
\cite{Cieliebak} so $HW^*(L)=0$ by Theorem \ref{Theorem SH=0
implies HW=0}.

Recent work of Bourgeois-Ekholm-Eliashberg \cite[Remark 6.2]{BEE} can tackle the chord conjecture for Legendrian spheres in the boundary of Stein domains with vanishing $SH^*$.
\end{example}

Let $N$ be a simply connected closed manifold, and $L\subset
DT^*N$ an exact Lagrangian with transverse Legendrian intersection
$\partial L = L \cap ST^*N$.

\begin{theorem}\label{Theorem Arnold Conj application 1}
If $H^2(T^*N)\to H^2(L)$ is not injective, then the chord conjecture holds and generically there are at least
$\textrm{rank}\; H^*(L)$ Reeb chords.
\end{theorem}
\begin{proof}
The non-injectivity implies the existence of an $\eta \in
H^2(T^*N,L)$ with $\overline{\eta}\neq 0 \in H^2(T^*N)$. Combining
Theorem \ref{Theorem SH=0 implies HW=0} and the vanishing
$SH^*(T^*N;\Z/2)_{\overline{\eta}}=0$ (see \ref{Subsection Vanishing of the
Novikov homology of the free loopspace}) we deduce that
$HW^*(L;\Z/2)_{\eta}=0$. Thus the claim follows by Theorem \ref{Theorem
Vanishing of HW implies Arnol'd Chord Conj}.
\end{proof}
Theorem \ref{Theorem Arnold Conj application 1} applies for
example for $L$ a conormal bundle to a submanifold $K\subset N$
such that $H^2(N)\to H^2(K)$ has kernel (Example \ref{Example
conormal bdle}).

Theorem \ref{Theorem Arnold Conj application 1} also holds after attaching subcritical handles to $DT^*N$, since
Cieliebak \cite{Cieliebak} proved that $SH^*$ does not change for action reasons (so this also holds for twisted $SH^*$).
\begin{theorem}
For any ALE space (see \ref{Introduction on Arnold conjecture}) the chord
conjecture holds for any $\partial L$ and generically there are at least
$\textrm{rank}\; H^*(L)$ Reeb chords.
\end{theorem}
\begin{proof}
By \cite{Ritter2}, $SH^*(M)_{\overline{\eta}}=0$ for generic
$\overline{\eta} \in H^2(M)$. Now $L$ is a $2$-manifold with
boundary so $H^2(L)=0$, so $\overline{\eta}$ lifts to an $\eta \in
H^2(M,L)$. Now see the proof of Theorem \ref{Theorem Arnold Conj
application 1}.
\end{proof}
%
%
%
\section{Exact contact hypersurfaces}
\label{Section Applicatoin 2 Exact contact hypersurf}
%
\subsection{Exact contact hypersurfaces}
\label{Subsection Exact embeddings}
%
The definition and study of exact contact hypersurfaces first appeared in Cieliebak-Frauenfelder \cite{Cieliebak-Frauenfelder}. We briefly recall the definition.

\begin{definition}
An exact contact hypersurface is an embedding
$i:(\Sigma^{2n-1},\xi) \hookrightarrow
(\overline{M}^{2n},d\theta)$ of a closed contact manifold
$(\Sigma,\xi)$ such that for some $1$-form $\alpha \in
\Omega^1(\Sigma)$,
$\xi = \ker \alpha$ and $\alpha-i^*\theta=\textrm{exact}.$
Moreover, we will always assume that $i(\Sigma)$ separates $\overline{M}$
into a compact connected submanifold $W\subset \overline{M}$ and
an unbounded component $\overline{M}\backslash W$.
\end{definition}

\begin{remark}
The separating assumption is automatic if $H_{2n-1}(M)=0$. Indeed, if
$\overline{M}\setminus \Sigma$ is connected, then there is a loop $\gamma$ in
$\overline{M}$ cutting $\Sigma$ once transversely, so the intersection number
$[\gamma]\cdot [\Sigma]$ is non-zero, so $[\Sigma] \in
H_{2n-1}(\overline{M})\cong H_{2n-1}(M)$ is non-trivial, contradiction.

When $H^1(\Sigma)=0$, the condition $\alpha-i^*\theta=\textrm{exact}$
 is equivalent to
$d\alpha=i^*\omega$.
\end{remark}

\begin{example}
Sphere bundles $ST^*N \!\hookrightarrow \!T^*N$ and the boundaries
$ST^*L \!\hookrightarrow\! M$ of the Liouville domains in Example \ref{Example
Weinstein neighbourhood} $($assume $n\geq 2$ for connectedness$)$. The
separating condition for $\Sigma \hookrightarrow T^*N$ is
automatic if $\textrm{dim}(N)=n\geq 2$, since $H_{2n-1}(T^*N)
\cong H_{2n-1}(N) = 0$.
\end{example}

\begin{lemma}\label{Lemma Can assume hypersurface is included}
For the purposes of symplectic cohomology, we can always assume
that an exact contact hypersurface $\Sigma$ bounds a Liouville
subdomain $W$ of $M$.
\end{lemma}
\begin{proof}
Assume $i(\Sigma) \subset M$ by redefining $M$ to be $M\cup
[1,R_0]\times \partial M$ (Theorem \ref{Theorem Properties of restriction map}). Identify $i(\Sigma)=\Sigma$. Use a bump
function to extend $\alpha$ to a neighbourhood of $\Sigma\subset
\overline{M}$. Near $\Sigma$, $\alpha-\theta=df$ for a function
$f$ supported near $\Sigma$. Then replace $\theta$ by $\theta+df$:
$SH^*(M)$ has not changed (compare Lemma \ref{Lemma can assume
Liouville subdomains are included}). Thus we can assume $W\subset
M$ with $\alpha=\theta|_{\Sigma}$ a contact form. Define $Z$ by
$\omega(Z,\cdot)=\theta$, so $\omega(Z,\cdot)|_{\Sigma}=\alpha$.
Thus, $\omega^n(Z,\ldots)|_{\Sigma} = \alpha \wedge
(d\alpha)^{n-1} \neq 0$ pointwise (contact condition) and since
the flow of $Z$ expands volumes, $Z$ must be pointing strictly outwards along
$\Sigma$. So $W\subset M$ is Liouville.
\end{proof}
%
\subsection{Stretching-of-the-neck argument}
\label{Subsection Stretching of the neck}
%
%
(Bourgeois-Oancea
\cite[Sec.5.2]{Bourgeois-Oancea}) Consider an isolated Floer
trajectory in $\overline{M}$ joining orbits lying in the collar, assuming $J$ is of contact type on the collar.
Consider what happens to this Floer trajectory as you stretch a
neighbourhood of $\partial M$ (more precisely: you insert a collar
$[R_0,1] \times
\partial M$ in between $M$ and $[1,\infty)\times \partial
M$, and you rescale $\omega,H$ on $M$ by $1/R_0$ so that the data
glues correctly). By a compactness argument, this $1$-family of
Floer trajectories will converge in the limit to a Floer cylinder
with punctures 
$$(\R\times S^1)\setminus \{ \textrm{punctures} \} \to \partial M \times
\R$$
and the map rescaled near the punctures converges to Reeb orbits,
each of which is ``capped off" by an isolated holomorphic plane
$\C\to \overline{M}$ (converging to the Reeb orbit at infinity).
The proof that the only holomorphic curves capping off the Reeb
orbits in $\overline{M}$ are planes (and not
holomorphic cylinders, for example) is a consequence of a subtle action
argument \cite[Sec.5.2, Proof of Prop.5, Step 1]{Bourgeois-Oancea}
which shows that the limit curve must have a \emph{connected}
component containing both Hamiltonian orbits (the fixed ends of
the Floer trajectory). A dimension count and
the fact that cylinders have genus $0$ then proves that the above are the
only limits.

\begin{remark}[Transversality]\label{Remark transversality for
neck stretching} In general, this argument would require knowing
transversality for symplectic field theory \cite[Remark 9,
sec.3.1]{Bourgeois-Oancea}. In our applications this is not
necessary, as explained to us by Oancea: we will always assume
that the virtual dimension $(2n-3)-|x|$ of the moduli space of
holomorphic planes $\C \to \overline{M}$ converging to a Reeb
orbit $x$ at infinity is $\geq 1$, therefore the above punctured
Floer cylinders do not have punctures. Indeed, if they had
punctures, then since the virtual dimension of the planes is $\geq
1$, the virtual dimension of the component containing the two
Hamiltonian orbits (modulo the $\R$-translation action) would be
negative. But for that component, transversality is guaranteed by
a time-dependent perturbation of the almost complex structure just
like for non-punctured Floer cylinders. Thus that component does
not exist, which is a contradiction.
\end{remark}
%
\subsection{Independence of the filling}
\label{Subsection Independence of the filling}
%
%
\begin{theorem}[Bourgeois-Oancea \cite{Bourgeois-Oancea},
Cieliebak-Frauenfelder-Oancea {\cite[Theorem 1.14]{Cieliebak-Frauenfelder-Oancea}}]\label{Theorem no holomorphic
planes}
\indent Let $W$ be a Liouville domain with
$c_1(W)|_{\pi_2(W)}=0$ such that all closed Reeb orbits $x$ in
$\partial W$ which are contractible in $W$ have $|x|\!<\!2n-3$.
Then for small $C>0$, all isolated Floer trajectories in
$\overline{W}$ connecting orbits in the collar lie in $R\geq C$.

For a disc cotangent bundle $W\!=\!DT^*L$ these conditions are
satisfied if $\dim L\!\geq\! 4$.
\end{theorem}
\begin{proof}
If no such $C$ existed, the Section \ref{Subsection Stretching of the neck}
implies the existence of an isolated punctured Floer cylinder
capped off by at least one holomorphic plane in $\overline{W}$.
But $(2n-3)-|x|\geq 1$, so Remark \ref{Remark transversality for
neck stretching} implies that the cylinder cannot have punctures,
which is a contradiction.

For $W^{2n}\!=\! DT^*L$, $c_1(W)\!=\!0$. Our grading convention is
$SH^*(T^*L;\Z/2)\cong H_{n-*}(\mathcal{L}L;\Z/2)$, so
$|x|=n-\textrm{index}(\gamma_x)$ where $\gamma_x$ is the
closed geodesic (Reeb orbit) corresponding to $x$. So
$(2n-3)-|x|= \textrm{index}(\gamma_x) + \dim L - 3$. The claim
follows since $\textrm{index}(\gamma_x)\geq 0$.
\end{proof}
\begin{definition}[Boundary symplectic cohomology]
Let $(\Sigma,\alpha)$ be a contact manifold admitting a filling\footnote{``Filling'' here means that there is an exact contact hypersurface $\Sigma \hookrightarrow W$ mapping onto $\partial W$.} by a Liouville domain $W$ with
$c_1(W)|_{\pi_2(W)}=0$ such that all closed Reeb orbits $x$ in $\Sigma$ which are contractible in $W$ have $|x|<2n-3$. For $\eta\in H^2(W)$, define
$$
BSH^*(\Sigma)_{\eta|_{\Sigma}} = SH^*_+(W)_{\eta} \quad (\textrm{constructed in Section } \ref{Section SH+}).
$$
\end{definition}

\begin{corollary}\label{Corollary Independence of filling}
\label{Corollary exact contact hyp in cot bdles}\strut \begin{enumerate}
\item $BSH^*(\Sigma)_{\eta|_{\Sigma}}$ is independent of the choice of filling $W$ in the above Definition. In
particular, it only depends on the restriction of $\eta$ to $H^2(\partial
W)$.

\item Let $i:\Sigma \hookrightarrow \overline{M}$ be an exact contact
hypersurface with $c_1(M)=0$ such that all closed Reeb orbits $x$ in
$\Sigma$ have $|x|<2n-3$ in $\overline{M}$. Then
$BSH^*(\Sigma)_{\eta|_{\Sigma}}\cong SH^*_+(W)_{\eta}$ for the
Liouville subdomain $W$ bounding $i(\Sigma)$ in $\overline{M}$
$($Lemma \ref{Lemma Can assume hypersurface is included}$)$.
\end{enumerate}
\end{corollary}
\begin{proof}
1) We can choose the same Hamiltonian and almost complex structure on
$R\!\geq\! C$ for two such fillings, so the Floer equations are
identical. Twistings are determined by $\eta$ restricted to
$R\!\geq\! C$, but this region is homotopy equivalent to $\partial
W$.

2) $c_1(M)\!=\!0$, so $c_1(W)\!=\!0$ by naturality, so gradings for $W,M$ can both
be calculated with respect to the trivialization of the canonical
bundle of $\overline{M}$ (\ref{Subsection Maslov index and
Conley-Zehnder index}). Now apply 1).
\end{proof}
%
%
\subsection{Obstructions to
exact $ST^*L \hookrightarrow T^*N$, simply connected
case}\!\!\label{Subsection Symplectic cohomology obstructions
simply connected case}
\strut\\
\emph{Convention: In sections \ref{Subsection Symplectic cohomology obstructions
simply connected case}-\ref{Subsection SH
obstructions in general}, ordinary (co)homology is computed over $\Z$ coefficients.}

\begin{theorem}\label{Theorem exact contact hypersurf result simply conn case}
Let $L,N$ be closed simply connected $n$-manifolds, $n\geq 4$. For
any exact contact hypersurface $ST^*L \hookrightarrow T^*N$, the
following hold
\begin{enumerate}
\item $H^2(N) \to H^2(L)$ is injective;

\item $\pi_2(L) \to \pi_2(N)$ has finite cokernel;

\item if $H^2(N)\neq 0$ then\footnote{\label{Footnote technical}\emph{Technical Remark:} $SH^*(T^*N)_{\eta} \cong H_{n-*}(\mathcal{L}N)_{\eta}$ is known to hold only for $\mathrm{char}\,\K=2$ (see Section \ref{Section Application Loop space homology and string topology}). This does not affect (1) and (2): those arise from the integral cohomology classes which define the parallel transport maps of the local systems, so it does not matter if we take $\K=\Z/2$ in the definition of $\Lambda$. However, (3) heavily involves $\Lambda$ in the last lines of the proof. So, strictly speaking, our argument only proves $H_*(L;\K) \cong H_*(W;\K)$ for $\mathrm{char}\,\K=2$. However, assuming the predictions explained at the start of Section \ref{Section Application Loop space homology and string topology}, the argument holds for any $\K$: the twisted $SH^*(W)$ which vanishes in the proof is in fact the one twisted not just by $\tau(\eta|_W)$ but also by $\tau(w_2(TN)|_W)$. But under the isomorphism $SH^*_+(W)_{\textrm{twisted}} \cong H^{*+1}(W)\otimes \Lambda$ the twisting becomes irrelevant since the system $c^*\underline{\Lambda}_{\tau(w_2(TN)|_W)}$ is trivial on $W$ (see \ref{Subsection Transgression}).
} 
$H_*(L) \cong H_*(W)$ (where $W$ is the filling of $ST^*L\subset T^*N$).
\end{enumerate}
\end{theorem}
\begin{remark}\label{Remark what the maps are L W N}
The maps in $(1)$, $(2)$, $(3)$ are defined by the composites
$\strut\\ \strut\qquad\pi_2(L)
\cong \pi_2(ST^*L) \to \pi_2(W) \to \pi_2(T^*N) \cong
\pi_2(N),\\
\strut\qquad H^2(N)\cong H^2(T^*N) \to H^2(W) \to H^2(ST^*L) \cong H^2(L).\\
$
Using $\textrm{dim}(L)\geq 4$, we identify $\pi_1(L),\pi_2(L),
H_1(L), H^2(L)$ with the analogous groups for $\partial W =ST^*L$
via the long exact sequence in homotopy for $S^{n-1}\to ST^*L\to
L$ and the Gysin sequence $H^*(ST^*L)\to H^{*-(n-1)}(L) \to
H^{*+1}(L) \to H^{*+1}(ST^*L)$.
\end{remark}
\begin{proof}[Proof of Theorem \ref{Theorem exact contact hypersurf result simply conn case}] Suppose by contradiction that $\eta\neq 0\in H^2(N)$ vanishes in
$H^2(L)\cong H^2(\partial W)$. By Corollary \ref{Corollary Full
loopspace homology vanishes}, $SH^*(T^*N)_{\eta}=0$. By Theorem
\ref{Theorem Vanishing criteria}, $SH^*(W)_{\eta|W}=0$. By Lemma
\ref{Lemma LES for SH+}, $SH^*_+(W)_{\eta|W}\cong
H^{*+1}(W)\otimes \Lambda$ has finite rank. By Corollary
\ref{Corollary exact contact hyp in cot bdles},
$SH^*_+(W)_{\eta|W}=BSH^*(ST^*L)_{\eta|\partial W}$. Using $\eta|_{\partial W}=0 \in H^2(\partial
W)$ and Corollary \ref{Corollary LES for SH+ of T*M},
$$BSH^*(ST^*L)_{\eta|\partial W} \cong BSH^*(ST^*L)\otimes \Lambda \cong H_{n-*}(\mathcal{L}L,L)\otimes \Lambda.$$
This has infinite rank as $H_*(\mathcal{L} L)$ does
\cite{Sullivan}, using $L$ is closed with $\pi_1(L)=1$. So
$SH^*_+(W)_{\eta|W}$ has both finite and infinite rank.
Contradiction, proving (1).

We have in fact just proved more, since the isomorphism class of the local system used to twist $SH^*$ actually depends on the transgressed classes $\tau\eta$ and $\tau (\eta|_{\partial W})$ (recall \ref{Subsection Transgression}). So we actually proved that if $\tau\eta \neq 0$ in
$H^1(\mathcal{L}N)$, then $\tau(\eta|_{\partial W}) \neq 0$ in
$H^1(\mathcal{L}\partial W)  \cong   H^1(\mathcal{L} L)$.

In general, for simply connected manifolds $X$, the transgression $\tau: H^2(X) \to H^1(\mathcal{L}X)$ is an isomorphism and one can identify the groups $H^1(\mathcal{L}X) \cong \textrm{Hom}(H_1(\mathcal{L}X),\Z) \cong \textrm{Hom}(\pi_2(X),\Z)$. So, using this general fact for $X=N$ and $X=L$, the above actually proved that
$$
\textrm{Hom}(\pi_2(N),\Z) \to \textrm{Hom}(\pi_2(L),\Z),\;\;
\tau\eta \mapsto \tau(\eta|_{\partial W})
$$
must be injective. Dualizing this statement yields (2).

If $H^2(N)\neq 0$ then pick an $\eta\neq 0\in H^2(N)$. By Corollary \ref{Corollary Full
loopspace homology vanishes},
$SH^*(T^*N)_{\eta}=0$ so $SH^*(W)_{\eta|W}=0$. By $(1)$ it
restricts to $\eta|_{\partial W}\neq 0 \in H^2(ST^*L)\cong
H^2(T^*L)$. By Corollary \ref{Corollary Full loopspace homology
vanishes}, $SH^*(T^*L)_{\eta|\partial W}=0$. Using Lemma
\ref{Lemma LES for SH+} and Corollary \ref{Corollary exact
contact hyp in cot bdles},
$$
H^{*+1}(W) \otimes \Lambda \cong
SH^*_{+}(W)_{\eta|W} \cong
BSH^*(ST^*L)_{\eta|_{\partial W}} \cong
SH^*_{+}(DT^*L)_{\eta|\partial W} \cong
H^{*+1}(L) \otimes \Lambda.
$$
Passing to homology by universal coefficients, proves (3) over $\K$ coefficients. In fact it also holds over $\Z$ coefficients by \ref{Subsection Choice of coefficients} (defining $\Lambda$ over $\Z$ instead of over
$\K$).
\end{proof}
%
%
%
\subsection{Obstructions to
exact $ST^*L \hookrightarrow T^*N$, general case}\label{Subsection
Symplectic cohomology obstructions}

The key to Theorem \ref{Theorem exact contact hypersurf result
simply conn case} was that $H_*(\mathcal{L}L)$ had infinite rank.
This holds for closed $L$ with $\pi_1(L)=1$, since Sullivan
\cite{Sullivan} showed $H_*(\mathcal{L} L;\Q)$ is infinite
dimensional via rational homotopy theory. Sullivan's proof also
works for nilpotent spaces. We do not know when the result holds
for finite $\pi_1(L)$ (we remark that $L$ need not be nilpotent: $\R P^2$ has
$\pi_1=\Z/2$ but does not act nilpotently on $\pi_2=\Z$). When
$\pi_1(L)$ has infinitely many conjugacy classes (for example
infinite abelian $\pi_1(L)$) the result holds because the
connected components of $\mathcal{L}L$ are indexed by these
conjugacy classes, so $\textrm{rank}\, H_0(\mathcal{L}L)=\infty$.

\noindent
\emph{Convention: In sections \ref{Subsection Symplectic cohomology obstructions
simply connected case}-\ref{Subsection SH
obstructions in general}, ordinary (co)homology is computed over $\Z$ coefficients.}\\
\emph{Notation: Denote $\mathcal{L}_0 L\subset \mathcal{L}L$ is the subspace of contractible loops in $L$.}

\begin{theorem}\label{Theorem exact contact hypersurf result}
Let $L,N$ be closed $n$-manifolds, $n\geq 4$, with $N$ of finite
type $($\ref{Subsection Vanishing of the Novikov homology of the
free loopspace}$)$. Let $ST^*L \hookrightarrow T^*N$ be an exact
contact hypersurface. 

\noindent If $\mathrm{rank}\, H_*(\mathcal{L}_0
L)=\infty$ then:
\begin{enumerate}
\item $\pi_2(L) \to \pi_2(N)$ has finite cokernel.

\item if $N$ is simply connected, $H^2(N) \to H^2(L)$ is
injective;

\item $H^2(\widetilde{N}) \to H^2(\widetilde{L})$ and
$H^2(N)\setminus\ker \pi^* \to H^2(L)$ are injective;\\
$($Here $\pi\!:\!\widetilde{N}\!\to\! N$ is the universal
cover, inducing $\pi^* \! : \! H^2(N) \! \to
\! H^2(\widetilde{N}))$

\item if $L$ has finite type and $\textrm{rank}\,(\pi_2 N)\neq 0$,
then\footnote{See the technical remark in the footnote \ref{Footnote technical} to Theorem \ref{Theorem exact contact hypersurf result simply conn case}.} $H_*(L) \cong H_*(W)$ where $W$ is the filling of $ST^*L\subset T^*N$.
\end{enumerate}
If $\mathrm{rank}\, H_*(\mathcal{L} L)=\infty$ then: $\mathrm{(1), (2), (3)}$ hold after replacing
$\pi_2(L)$, $H^2(\widetilde{L})$
 by $\pi_2(W)$, $H^2(\widetilde{W})$ respectively. Also if any one of $\mathrm{(1)-(4)}$ failed, then
 Corollary \ref{Corollary exact contact hypersurf stuff} $\mathrm{(1)-(4)}$
 all hold and $\ker \pi^*=H^2(N)$
 $\mathrm{(}$so the image of the Hurewicz homomorphism
  $\pi_2(M)\to H_2(M)$ is torsion$\mathrm{)}$.
\end{theorem}
\begin{remark}\label{Remark about STL theorem}
The distinction $(2),(3)$ is because we need $\tau\eta  \! \neq \!
0  \! \in  \! H^1(\mathcal{L}_0 N)$ in Corollary \ref{Corollary
Full loopspace homology vanishes}: $\tau  \! : \! H^2(N) \!
\to \!  H^1(\mathcal{L}_0 N)$ is an isomorphism if $\pi_1(N) \! =
\! 1$, but otherwise has kernel $\ker \pi^*$. We identify $H^2(\widetilde{N})\cong
\mathrm{Hom}(\pi_2(N),\Z)$ with $\mathrm{im}\,(\tau) \subset
H^1(\mathcal{L}_0 N)$ (\ref{Subsection Twisting by $2$-forms
from the universal cover}). The $2^{\textrm{nd}}$ part of $(3)$ and
$(2)$ can be proved as in Theorem \ref{Theorem exact contact
hypersurf result simply conn case} using $\textrm{rank}\,
H_*(\mathcal{L}L)=\infty$. Note: $(1)$ implies $(2)$,$(3)$.
The $1^{\textrm{st}}$ part of $(3)$ is dual to $(1)$, the
$2^{\textrm{nd}}$ part follows by the diagram
$$
\xymatrix@R=10pt@C=12pt{ H^2(N) \ar@{->}[r]^-{\tau} \ar@{->}[d] &
\mathrm{Hom}(\pi_2(N),\Z)\subset H^1(\mathcal{L}_0 N) \ar@{->}[d] \\
H^2(L) \ar@{->}[r]^-{\tau} & \mathrm{Hom}(\pi_2(L),\Z)\subset
H^1(\mathcal{L}_0 L) }
$$
\end{remark}
\begin{proof}
To prove (1) we twist by $2$-forms $\widetilde{\eta}\neq 0 \in
H^2(\widetilde{N})$ on the universal cover of $N$ as explained in
\ref{Subsection Twisting by $2$-forms from the universal cover}.
By transgression, these give rise to $1$-forms on the space
$\mathcal{L}_0 N$ of contractible loops in $N$. By Theorem
\ref{Theorem twisting by widetilde eta} there is a TQFT structure
on the subgroup $SH^*_0(M)_{\widetilde{\eta}}$ of
$SH^*(M)_{\widetilde{\eta}}$ generated by the contractible orbits.

Pass to universal covers and use transgressions (\ref{Subsection
Twisting by $2$-forms from the universal cover}) to identify the
two rows:
$$
\xymatrix@R=6pt@C=12pt{%
H^2(\widetilde{N}) \ar@{->}[r] \ar@{=}[d] & H^2(\widetilde{W})
\ar@{->}[r] \ar@{=}[d] & H^2(\widetilde{L})\ar@{=}[d]\\
\textrm{Hom}(\pi_2 N, \Z) \ar@{->}[r] & \textrm{Hom}(\pi_2 W, \Z)
\ar@{->}[r] & \textrm{Hom}(\pi_2 L, \Z) }
$$
\textbf{Case 1:} $\textrm{rank}\, H_*(\mathcal{L}_0 L)=\infty$.
Suppose by contradiction that $\widetilde{\eta}$ vanishes under
this composition: $\widetilde{\eta}|_{\widetilde{L}}=0 \in
H^2(\widetilde{L})$. By \cite{Ritter},
$SH^*_0(T^*N)_{\widetilde{\eta}} \cong H_{n-*}(\mathcal{L}_0
N)_{\widetilde{\eta}}=0$ since $N$ has finite type. As before,
$SH^*_{0}(W)_{\widetilde{\eta}}=0$, so
$SH^*_{+,0}(W)_{\widetilde{\eta}}= H^{*+1}(W)\otimes \Lambda$ has
finite rank. But, by independence of the filling, this group is
isomorphic to
$$SH^*_{+,0}(DT^*L)_{\widetilde{\eta}|{\widetilde{L}}}=SH^*_{+,0}(DT^*L)_0 \cong H_{n-*}(\mathcal{L}_0 L,L)\otimes \Lambda,$$
which has infinite rank, contradiction. So the composition is
injective, so dualizing gives (1). By Remark \ref{Remark about STL
theorem} it remains to prove (4). If $\textrm{rank}\, \pi_2(N)\neq
0$, pick $\tau\widetilde{\eta}\neq 0$. By $(2)$, also
$\tau\eta|_{\widetilde{L}} \neq 0$, so since $L$ has finite type,
$SH^*_0(T^*L)_{\widetilde{\eta}|{\widetilde{L}}}=0$. Thus
$$
H^{*+1}(W) \otimes \Lambda \cong
SH^*_{+,0}(W)_{\widetilde{\eta}} \cong
BSH_0^*(ST^*L)_{\widetilde{\eta}} \cong
SH^*_{+,0}(DT^*L)_{\widetilde{\eta}} \cong
H^{*+1}(L) \otimes \Lambda.
$$
Passing to homology by universal coefficients and using
\ref{Subsection Choice of coefficients}, proves (4).

\textbf{Case 2:} $\textrm{rank}\, H_*(\mathcal{L} L)=\infty$.
Assume by contradiction that $\widetilde{\eta}|_{\widetilde{W}}=0
\in H^2(\widetilde{W})$. As before,
$SH^*_{0}(W)_{\widetilde{\eta}}=0$. Since
$\widetilde{\eta}|_{\widetilde{W}}=0$, this implies
$SH^*_{0}(W)=0$. By Theorem \ref{Theorem Vanishing criteria},
$SH^*(W)=0$ since the unit lies in $SH^*_0(W)=0$. By Lemma
\ref{Lemma LES for SH+}, $SH^*_+(W)\cong H^{*+1}(W)$ has finite
rank. But by Corollary \ref{Corollary Independence of filling} and
Corollary \ref{Corollary LES for SH+ of T*M},
$$SH^*_+(W) \cong BSH^*(ST^*L) \cong H_{n-*}(\mathcal{L}L,L),$$
which has infinite rank since $H_*(\mathcal{L}L)$ does.
Contradiction. So $H^2(\widetilde{N}) \to H^2(\widetilde{W})$ is
injective and dually $\pi_2(W)\to \pi_2(N)$ has finite cokernel.
\end{proof}
%
%
%
%
%
\subsection{Pathological $L$}\label{Subsection Pathological L}
%
Write $[S^1,L]$ for the set of free homotopy classes of maps $S^1
\to L$, equivalently: it is the set of conjugacy classes of
$\pi_1(L)$.

Call $L$ \emph{pathological} if $[S^1,L]\neq 1$ is a finite set
and $H_*(\mathcal{L}L)$ has finite rank. This is the only case
when Theorem \ref{Theorem exact contact hypersurf result} may not
apply.

\begin{corollary}\label{Corollary exact contact hypersurf stuff}
Let $L,N$ be closed $n$-manifolds, $n\geq 4$, with $N$ of finite
type and $L$ pathological. Let $ST^*L\hookrightarrow T^*N$ be an
exact contact hypersurface. Then either the results of Theorem
\ref{Theorem exact contact hypersurf result} all hold, or else the
following must hold:
\begin{enumerate}
\item $H_1(L)\to H_1(W)$ and $[S^1,L] \to [S^1,W]$ both vanish;

\item $\textrm{rank } H^{n+1}(W)=\# [S^1,L]-1 \geq 1$;

\item if $L$ is of finite type then $H^2(\widetilde{N}) \to
H^2(\widetilde{L})$, $H^2(N)\setminus\ker\pi^*\to H^2(L)$ vanish;
\item if $L$ is of finite type then the image of $\pi_2(L) \to
\pi_2(N)$ is torsion.
\end{enumerate}
In particular, if $W$ is Stein then the results of Theorem \ref{Theorem exact contact hypersurf
result} hold.
\end{corollary}

\begin{proof}
For the last claim: $H^{n+1}(W) \neq 0$ by (2), but
$H^{n+1}(W)\!=\!0$ for Stein $W^{2n}$.

The connected components $\mathcal{L}_c L$ of $\mathcal{L}L$ are indexed by
$c\in [S^1,L]$. Filter $SH^*(DT^*L)$ by $[S^1,L]$ as in \ref{Subsection
TQFT is compatible with filtrations}, and consider $i:[S^1,L]\cong
[S^1,ST^*L] \to [S^1,W]$.

Suppose Theorem \ref{Theorem exact contact hypersurf result}
fails. By Remark \ref{Remark about STL theorem} some
$\widetilde{\eta}\!\neq 0\!\in H^2(\widetilde{N})$ has
$\widetilde{\eta}|_{\widetilde{L}}\!=\! 0\!\in\!
H^2(\widetilde{L})$. By the proof of Theorem \ref{Theorem exact
contact hypersurf result} $SH^*(W)=0$ and, using Lemma
\ref{Lemma LES for SH+},
$$
\begin{array}{ll}
H_0(\mathcal{L}L,L) \cong SH^n_+(DT^*L) \cong
SH^n_+(W)\cong\displaystyle\!\!\!\!\bigoplus_{w\in [S^1,W]}
\!\!\!\!SH^n_{+,w}(W) \cong H^{n+1}(W),
\\[3mm]
H_0(\mathcal{L}L,L) \cong \!\! \displaystyle
\!\!\!\!\bigoplus_{0\neq c \in [S^1,L]}\!\!\!\!
H_0(\mathcal{L}_cL) \displaystyle \cong \!\! \!\!\bigoplus_{0\neq
c \in [S^1,L]}\!\!\!\! SH^n_{+,c}(DT^*L) \hookrightarrow \!\!\!\!
\bigoplus_{0\neq c \in [S^1,L]}\!\!\!\! SH^n_{+,i(c)}(W).
\end{array}
$$

Now $H_0(\mathcal{L}_c L) = \Z$, so the two lines imply
$\textrm{rank } H^{n+1}(W)=\# [S^1,L] -1$. Also,
$SH^n_{+,w}(W)\cong SH^n_w(W)=0$ for $0\neq w \in [S^1,W]$. So,
since the second line respects summands, $i$ must vanish (and
similarly if we filter by $H_1$, \ref{Subsection TQFT is
compatible with filtrations}), proving (1),(2).

Suppose $\widetilde{\eta} \!\in\! H^2(\widetilde{N})$ has
$\widetilde{\eta}|_{\widetilde{L}} \!\neq 0 \!\in\!
H^2(\widetilde{L})$. Then $SH^*_0(T^*N)_{\widetilde{\eta}}\!=\!0$,
$SH^*_0(W)_{\widetilde{\eta}}\!=\!0$,
$$
H^{*+1}(W) \otimes \Lambda \cong
SH^*_{+,0}(W)_{\widetilde{\eta}} \cong
BSH^*_0(ST^*L)_{\widetilde{\eta}} \cong
SH^*_{+,0}(DT^*L)_{\widetilde{\eta}} \cong
H_{n-*}(\mathcal{L}_0 L,L)_{\widetilde{\eta}}.
$$
But for finite type $L$, $H_*(\mathcal{L}_0
L)_{\widetilde{\eta}}=0$ since $\widetilde{\eta}\neq 0 \in
H^2(\widetilde{L})$. Thus, $H_{n-*}(\mathcal{L}_0
L,L)_{\widetilde{\eta}} \cong H_{n-*-1}(L)$ by Corollary
\ref{Corollary LES for SH+ of T*M}. So $H^{n+1}(W)\cong
H_{-1}(L)=0$ contradicts (2). Thus $H^2(\widetilde{N}) \to
H^2(\widetilde{L})$ vanishes, proving the $1^{\textrm{st}}$ part
of $(3)$. Taking transgressions and dualizing implies $(4)$, which
implies the $2^{\textrm{nd}}$ part of $(3)$ (Remark \ref{Remark
about STL theorem}).
\end{proof}
%
%
%
\subsection{Obstructions to
exact contact embeddings}\label{Subsection SH
obstructions in general}
%
The only input about $\overline{M}=T^*N$ that we used in the proofs of Theorem \ref{Theorem exact contact hypersurf result} and Corollary \ref{Corollary exact contact hypersurf stuff} was the vanishing theorem $SH^*(M)_{\eta}=0$ from Corollary \ref{Corollary Full loopspace homology vanishes}. Thus, the same proofs show more generally the following:
\begin{theorem}\label{Theorem Obstructions to exact hyp} Let $M^{2n}$ be a Liouville domain with $c_1(M)=0$, $n\geq 4$.
Let $L^n$ be a non-pathological closed manifold admitting an exact
contact embedding $i:ST^*L \hookrightarrow \overline{M}$. Then
\begin{enumerate}
\item If $SH^*(M)_{\eta}=0$ then $i^*\eta\neq 0 \in H^2(L)$ and
$\tau(i^*\eta)\neq 0 \in H^1(\mathcal{L}_0 L)$;

\item If $SH^*(M)_{\widetilde{\eta}}=0$ then $i^*\widetilde{\eta}\neq 0
\in H^2(\widetilde{L})$;

\item If $L$ is of finite type and $SH^*(M)_{\widetilde{\eta}}=0$,
then\footnote{See the technical remark in the footnote \ref{Footnote technical} to Theorem \ref{Theorem exact contact hypersurf result simply conn case}.} $H_*(L) \cong H_*(W)$ where $W$ is the filling of $ST^*L\subset \overline{M}$.
\end{enumerate}
For pathological $L$: if $(1)\!$ or $\!(2)$ fails, then Corollary
\ref{Corollary exact contact hypersurf stuff}(1)(2) hold and $W$ is
not Stein. \\ For finite type pathological $L$:
$SH^*(M)_{\eta}\!=\!0$ implies $\tau(i^*\eta)\!=\! 0$, $SH^*(M)_{\widetilde{\eta}}\!=\!0$ implies $i^*\widetilde{\eta}\!=\! 0$.
\end{theorem}
\begin{corollary}\label{Corollary pathological STL in general M}
In the setup of Theorem \ref{Theorem Obstructions to exact hyp},
if $SH^*(M)\!=\!0$ or $SH^*(W)\!=\!0$, then:
\begin{itemize}
\item $L$ must be pathological and $W$ is not Stein;

\item $H_1(L)\to H_1(W)$ and $[S^1,L] \to [S^1,W]$ both vanish;

\item $\textrm{rank } H^{n+1}(W)=\# [S^1,L]-1 \geq 1$.
\end{itemize}
\end{corollary}
\noindent \textbf{Example.} Corollary \ref{Corollary pathological
STL in general M} applies to subcritical Stein manifolds $M$ with
$c_1(M)=0$, since $SH^*(M)=0$ by Cieliebak \cite{Cieliebak}. This
mildly strengthens \cite[Cor. 1.18]{Cieliebak-Frauenfelder-Oancea}
%
%
%
%
\section{Displaceability of $\partial M$ and Rabinowitz Floer theory} \label{Section Displacement implies vanishing}
\subsection{Rabinowitz Floer cohomology}
The Lagrange multiplier analogue of $SH^*$ is \emph{Rabinowitz
Floer cohomology}, due to Cieliebak-Frauenfelder
\cite{Cieliebak-Frauenfelder}. It involves changing the action
${\mathbb{A}}_H$ by rescaling $H$ by a \emph{Lagrange multiplier} variable
$\lambda\in \R$:
$$\textstyle
\mathbb{L}_H: \mathcal{L}\overline{M} \times \R \to \R, \quad
\mathbb{L}_H(x,\lambda) = - \int x^*\theta + \lambda \int_0^1 H(x(t)) \,
dt.
$$

This tweak has a considerable effect: it forces \emph{all}
critical points of the new action functional $\mathbb{L}_H$ to lie in $H^{-1}(0)$.
So take $H:\overline{M}\to \R$ such that $H^{-1}(0)=\partial M$, with $H=h(R)$ near $\partial M=\{R=1\}$ with $h'(1)=1$. Then the critical points of $\mathbb{L}_H$ are the solutions $x:S^1 \to \partial M$ of
$\dot{x}=\lambda \mathcal{R}(x)$ with action
$\mathbb{L}_H=-\lambda$. So, after reparametrizing $y(t)=x(t/ |\lambda|)$,
these critical points are precisely: the Reeb orbits of period
$\lambda>0$, the constant loops in $\partial M$ ($\lambda=0$), and
the negative Reeb orbits of period $-\lambda>0$. The chain complex
generated by the critical points of $\mathbb{L}_H$ has a differential
defined by counting negative gradient flowlines of $\mathbb{L}_H$. The
resulting cohomology is called \emph{Rabinowitz Floer cohomology}
$RFH^*(M)$ (see \cite{Cieliebak-Frauenfelder} for details).
\begin{theorem}
[Cieliebak-Frauenfelder-Oancea,
\cite{Cieliebak-Frauenfelder-Oancea}]\label{Theorem CFO RFH lemma}
There is a long exact sequence
$$
\cdots \to SH_{2n-*}(M)
\to SH^{*}(M) \to RFH^{*}(M) \to SH_{2n-*-1}(M) \to \cdots
$$
where the map $SH_{2n-*}(M) \to SH^{*}(M)$ is the composite
$$
\xymatrix@C=15pt{ SH_{2n-*}(M) \ar@{->}[r]^{c_*} & H_{2n-*}(M)
\ar@{->}[rr]^-{\textrm{Poincar\'{e}}} && H^{*}(M,\partial M)
\ar@{->}[rr]^-{\textrm{inclusion}^*} && H^{*}(M)
\ar@{->}[r]^-{c^*} & SH^{*}(M).}
$$
\end{theorem}
\begin{remark}
The Lemma is written in our conventions. Comparison: conventions
in \cite{Cieliebak-Frauenfelder-Oancea} dictate $SH_*(T^*N;\Z/2) \cong
H_*(\mathcal{L}N;\Z/2)$ instead of our $SH^{n-*}(T^*N;\Z/2) \cong
H_*(\mathcal{L}N;\Z/2)$.
\end{remark}
%
%
\subsection{Vanishing theorem}
%
\begin{theorem}\label{Theorem RFH=0 iff SH=0}
$RFH^*(M)=0$ \,iff\, $SH^*(M)=0$ \,iff\, $SH_*(M)=0$.

\noindent If one among $SH_0,SH^0,RFH^0$ is $0$ $($even just on the component $\mathcal{L}_0 \overline{M}$ of contractible loops, \ref{Subsection
TQFT is compatible with filtrations}$)$, then all
$SH_*,SH^*,RFH^*$ are $0$. Here we used
$\Z$-gradings if $c_1\!=\!0$ $($or $c_1|_{\pi_2 M}\!=\!0$ if we restrict to
$\mathcal{L}_0 \overline{M})$ otherwise use $\Z/2\Z$-gradings (see \ref{Subsection Maslov index and
Conley-Zehnder index}).
\end{theorem}
\begin{proof}
By Theorem \ref{Theorem Vanishing of homology iff cohomology},
$SH^*(M) \! \!  = \! \!  0$ iff $SH_*(M) \! \!  = \! \!  0$ iff
one of them vanishes in degree $0$ (where the unit/counit lie). By
Theorem \ref{Theorem CFO RFH lemma}, if $SH^*(M) \! = \! SH_*(M) \! =
\! 0$ then $RFH^*(M) \! = \! 0$. Conversely, suppose $RFH^0(M)=0$
(even restricted to contractible loops, since the long exact
sequence of Theorem \ref{Theorem CFO RFH lemma} respects the filtrations
of \ref{Subsection TQFT is compatible with filtrations}). By Theorems
\ref{Theorem Vanishing criteria} and \ref{Theorem Vanishing of
homology iff cohomology} it suffices to show that the unit $e\in
SH^0(M)$ vanishes.

\textbf{Case 1.} Suppose we have a $\Z$-grading. Theorem \ref{Theorem
CFO RFH lemma} yields the exact sequence $SH_{2n}(M)\to SH^0(M)
\to RFH^0(M)$, whose first map is the composite
$$
SH_{2n}(M) \to H_{2n}(M) \to H^{0}(M,\partial M) \to H^{0}(M) \to
SH^{0}(M).
$$
Now $M$ is a $2n$-manifold with non-empty boundary so
$H_{2n}(M)=0$. Thus the above composite is zero, so in the above
exact sequence, $SH^0(M)\hookrightarrow RFH^0(M)$ is injective.
Thus $RFH^0(M)=0$ implies $SH^0(M)=0$ so $e=0$.

\textbf{Case 2.} Suppose we have a $\Z/2\Z$-grading. By Theorem
\ref{Theorem CFO RFH lemma}, $e$ is in the image of
$SH_{\textrm{even}}(M) \! \to \! SH^\textrm{even}(M)$ since it
maps to zero via $SH^{\textrm{even}}(M) \! \to \!
RFH^{\textrm{even}}(M) \! = \! 0$. The $H^0(M)$ summands of the
image of $SH_{\textrm{even}}(M) \to H_{\textrm{even}}(M)\to
H^{\textrm{even}}(M,\partial M) \! \to \! H^{\textrm{even}}(M)$
must vanish since they factor through $H_{2n}(M)=0$. So since $e$
is in the image of $SH_{\textrm{even}}(M)\!\to\!
SH^\textrm{even}(M)$, $e=c^*(z)$ for some $z\in H^*(M)$ with no $H^0(M)$-summand. But $c^*$ is a TQFT map by Theorem \ref{Theorem ring structure on ordinary cohomology} and $z^{n+1}=0$ for degree reasons,
so $0=c^*(z^{n+1})=c^*(z)^{n+1}=e^{n+1}=e$.
\end{proof}
%
\subsection{Displaceability of $\partial M$ implies vanishing of $SH^*(M)$}
%
Call an exact symplectic manifold $(X^{2n},d\theta)$ \emph{convex} if: it is connected without boundary; it admits an exhaustion $X=\cup_k X_k$ by compact sets $X_k \subset X_{k+1}$ for which the Liouville vector field $Z$ defined by $i_Z d\theta = \theta$ points strictly outwards along $\partial X_k$; $Z$ is complete, meaning its flow is defined for all time; and $Z\neq 0$ outside of some $X_k$. For example, $X=\overline{M}$ for a Liouville domain $M$.
\begin{theorem}\label{Theorem displaceable implies SH is zero}
If $M$ admits an
exact embedding into an exact convex symplectic manifold $X$, such that $\partial M$ is displaceable by a compactly supported Hamiltonian flow inside $X$, then: $SH^*(M)\!=\!0$,
there are no closed exact Lagrangians in $M$, and the Arnol'd chord conjecture holds for any Legendrian arising as in \ref{Subsection Lagrangians inside Liouville domains}.
\end{theorem}
\begin{proof}
We work over $\K=\Z/2$ throughout the proof.
If $\partial M\! \subset\! X$ is displaceable then $RFH^*(\partial M\!\subset\! X)\!=\!0$ by
Cieliebak-Frauenfelder \cite{Cieliebak-Frauenfelder}. But $RFH^*$ only depends on the filling $M$ of $\partial M$, so $RFH^*(M)=0$. By Theorem
\ref{Theorem RFH=0 iff SH=0}, $SH^*(M)=0$.

If $L\subset M$ was a closed exact Lagrangian then for $W\cong DT^*L$ as in Example
\ref{Example Weinstein neighbourhood} the restriction
$0=SH^*(M)\to SH^*(W)\cong H_{n-*}(\mathcal{L}L)$ implies
$H_*(\mathcal{L}L)=0$ by Theorem \ref{Theorem Vanishing criteria}, which is
absurd since $H_0(\mathcal{L}L)\neq 0$. For the last claim use Theorems \ref{Theorem SH=0 implies HW=0} and \ref{Theorem Vanishing of HW implies
Arnol'd Chord Conj}.
\end{proof}
\noindent \textbf{Example.} If $ST^*L
\hookrightarrow X$ is an exact contact hypersurface such that $ST^*L$ is displaceable, with $c_1(X)=0$ and
$n\geq 4$, then
$SH^*(W)=0$ so Corollary \ref{Corollary pathological STL in
general M} holds. So for non-pathological $L$, $ST^*L$ is never
displaceable (this mildly strengthens Theorem 1.17 of
\cite{Cieliebak-Frauenfelder-Oancea}).
%
%
%
%
\section{String topology}
\label{Section Application Loop space homology and string topology}
%
%
%
\noindent For closed
$n$-manifolds $N$, Viterbo \cite{Viterbo1} proved that
$
SH^*(T^*N;\Z/2) \cong H_{n-*}(\mathcal{L}N;\Z/2),
$
and there are now also alternative proofs by Abbondandolo-Schwarz
\cite{Abbondandolo-Schwarz} and Salamon-Weber
\cite{Salamon-Weber}. We use approach
\cite{Abbondandolo-Schwarz} since we used it in \cite{Ritter} to
prove
$
SH^*(T^*N;\Z/2)_{\eta} \cong H_{n-*}(\mathcal{L}N;\Z/2)_{\eta}.
$
Our goal is to show that these isomorphisms respect the TQFT. 

It was initially believed that the isomorphism $SH^*(T^*N) \cong H_{n-*}(\mathcal{L}N)$ should hold as written for any choice of coefficients when $N^n$ is an orientable manifold. However, recent observations by Seidel \cite{Seidel} show that in fact this is not quite correct due to an ambiguity in the choice of orientation signs for the relevant moduli spaces. The prediction, according to \cite{Seidel}, is that the isomorphism holds as written over any coefficients if the orientable manifold $N$ is spin (meaning $w_2(TN)=0$). For a general closed manifold $N$ the prediction \cite{Seidel} is:
$$
SH^*(T^*N)_{\eta} \cong H_{n-*}(\mathcal{L}N;\underline{\Z}_{ev^*(w_1(TN))} \otimes_{\Z} \underline{\Z}_{\tau_2 (w_2(TN))} \otimes_{\Z} \underline{\Lambda}_{\tau \eta}),
$$
where we now explain the local system of coefficients used to compute the homology of $\mathcal{L}N$:
\begin{enumerate}
 \item  $\tau_2: H^2(N;\Z/2) \to H^1(\mathcal{L}N;\Z/2)$ is the transgression map over $\Z/2$ (see \ref{Subsection Transgression});
\item the map $ev:\mathcal{LN} \to N$ is the evaluation $x \mapsto x(0)$;
 \item  $\underline{\Z}_{\alpha}$ for $\alpha\in H^1(\mathcal{L}N;\Z/2)$ is a local system defined analogously to the local system $\underline{\Lambda}_{\alpha}$ of \ref{Subsection Novikov bundles of coefficients}: the fibre $\Z_x$ over $x\in \mathcal{L}N$ is $\Z$, and the parallel transport map over a path $u$ in $\mathcal{L} N$ connecting $x$ to $y$ is the multiplication isomorphism
$\pm 1 \ni \alpha[u] \co  \Z_{y}
\to \Z_{x}.$

\item $w_1(TN)\in H^1(N;\Z/2)$ and $w_2(TN)\in H^2(N;\Z/2)$ are the Stiefel-Whitney classes of $N$, so in particular $\underline{\Z}_{ev^*(w_1(TN))}$ is the pull-back via $ev$ of the orientation sheaf of $N$.
\end{enumerate}

To avoid these complications, \textbf{we assume throughout this Section that the field $\K$ has characteristic $2$}, so that orientation signs are irrelevant and we only assume that $N$ is a closed manifold. 
We write $H_*(\mathcal{L}N)_{\eta}$ to keep track of twistings by
$\eta\in H^2(N)$ when present (in which case it means $H_*(\mathcal{L}N;\underline{\Lambda}_{\tau\eta})$ using coefficients in the local system defined in \ref{Subsection Novikov bundles of coefficients}).
%
%
\subsection{The Abbondandolo-Schwarz construction \cite{Abbondandolo-Schwarz}}
\label{Subsection Abb-Sch construction}
We first identify $H_*(\mathcal{L}N)_{\eta}$ with the Morse
homology $MH_*(\mathcal{L}N)_{\eta}$ of the Lagrangian action
functional
$$\textstyle \mathcal{E}(\gamma) = \int_0^1 L(\gamma,\dot{\gamma})\, dt,$$
for a suitable function $L: TN \to \R$ which is quadratic in the
fibres, for example the Lagrangian functions
$L(q,v)=\frac{1}{2}|v|^2 - U(q)$ from classical mechanics.
One actually replaces $\mathcal{L}N=C^{\infty}(S^1,N)$ with the homotopy
equivalent space $W^{1,2}(S^1,N)$, and one proves that Morse
homology for $\mathcal{E}$ on this space is well-defined for most
$L$ (the quadratic growth of $L$ ensures that $\mathcal{E}$ is bounded below and that it has finite-dimensional unstable manifolds).

The choice of $L$ determines the choice of a quadratic Hamiltonian
on $T^*N$,
$$H(q,p)=\max_{v\in T_q N}(p\cdot v - L(q,v)).$$
For example: $L=\tfrac{1}{2}|v|^2$ gives $H=\tfrac{1}{2}|p|^2$.
The convexity of $L$ implies that the maximum is achieved at
precisely one point $p$, which equals the vertical differential of $L$ at $(q,v)$. This defines a fiber-preserving
diffeomorphism, called \emph{Legendre transform}:
$$\mathfrak{L}:TN \to T^*N, \quad (q,v) \mapsto (q,dL(q,v)|_{T^{\textrm{vert}}_{(q,v)}(TN)\,\equiv\, T_q N}).$$

This relates the dynamics of $H$ on $T^*N$ with that of
$\mathcal{E}$ on $\mathcal{L}N$: $\gamma \in
\textrm{Crit}(\mathcal{E})$ is a critical point of $\mathcal{E}$
in $\mathcal{L}N$ if and only if
$x=\mathfrak{L}(\gamma,\dot{\gamma})\in \textrm{Crit}({\mathbb{A}}_H)$ is a
$1$-periodic orbit of $X_H$ in $T^*N$. One can even relate the
indices and the action values of $\mathcal{E}$ and ${\mathbb{A}}_H$.

The isomorphism $H_*(\mathcal{L}N)_{\eta} \to
SH^{n-*}(T^*N)_{\eta}$ is induced by the chain isomorphism
$$\varphi: MC_*(\mathcal{E})_{\eta} \to SC^{n-*}(H)_{\eta}, \quad \varphi (\gamma) = \!\!\!\!\sum_{v\# u \in
\mathcal{M}^+_0(\gamma,x)}\!\!\!\! \epsilon_{v\# u}\,
t^{-\tau\eta[v]+\tau\eta[\pi u]}\, x$$
where $\epsilon_{v\# u}\!\in\! \{\pm 1\}$ are orientation signs and where
$\mathcal{M}^+(\gamma,x)$ consists of pairs $v\# u$, with
$
v\!:\!(-\infty,0] \!\to\! \mathcal{L}N
$
 a $-\nabla \mathcal{E}$ flowline converging to
$\gamma\!\in\! \textrm{Crit}(\mathcal{E})$ as $s\to -\infty$, and $u\!:\! (-\infty,0]\!\times\! S^1
\!\to\! T^*N$ solves Floer's
equation and converges to $x\in \textrm{Crit}({\mathbb{A}}_H)$ as $s\to -\infty$, with
Lagrangian boundary condition $\pi \circ u(0,t) = v(0,t)$ where $\pi\!:\! T^*N \!\to\! N$ is the projection.

As usual, one actually makes a $C^2$-small time-dependent perturbation
of $L$ (and hence also of $H$), but we keep the notation simple. Section \ref{Appendix Using non-linear Hamiltonians} explains why we can use a quadratic Hamiltonian instead of taking a direct
limit over linear Hamiltonians, and in \ref{Subsection levicivita J vs contact type J} we explain why the use of non-contact type $J$ in Abbondandolo-Schwarz is also not problematic.

Finally, $MH_*(\mathcal{L}N;\mathcal{E})_{\eta} \cong H_*(\mathcal{L}N)_{\eta}$ is essentially obtained by mapping a critical point $\gamma$ of $\mathcal{E}$ to the pseudo-cycle given by the unstable manifold $W^u(\gamma;\mathcal{E})$.

\begin{remark}\label{Remark AS iso respects filtration}
In our conventions \cite[Sec.3]{Ritter}, any $w: [0,1] \to T^*N$ has $\mathcal{E}(\pi w) \geq -\mathbb{A}_H(w)$. Since $\mathcal{E}$ decreases along $v$ and $-\mathbb{A}_H$ decreases along $s\mapsto u(-s,\cdot)$, the isomorphism $\varphi$ respects action-filtrations: $\varphi_{<c}:MH_*(\mathcal{E}<c) \cong SH^{n-*}(\mathbb{A}_H>-c)$ for all $c\in \R$ (see Section \ref{Section SH+}).

For small $c>0$, this filtered isomorphism becomes $MH_*(N;L) \!\cong\! MH^{n-*}(T^*N;H)$, which is the Poincar\'e duality $H_*(N) \!\cong\! H^{n-*}(T^*N) \!\cong\! H^{n-*}(N)$. 
So via $\varphi$, the map $c^*:H^*(T^*N) \to
SH^*(T^*N)_{\eta}$ becomes the inclusion of constant loops
$
c_*:H_*(N) \to H_*(\mathcal{L}N)_{\eta}.
$
 In particular, $\varphi^{-1}(e) = c_*[N]$ since $[N]$ is the Poincar\'e dual of $1$ and $e=c^*(1)$ (Theorem \ref{Theorem unit is image of 1}).
\end{remark}
%
\subsection{TQFT structure on $H_*(\mathcal{L}N)$}
\label{Subsection Morse operations for free loop space}
%
%
Given a graph $S'$ as in
\ref{Subsection Morse-graph operations}, use the Morse function
$\mathcal{E}_i(\gamma)=\int_0^1 L_i(\gamma,\dot{\gamma})\, dt$ for edge $e_i$ for 
a generic $L=L_i$ as above. This yields
$$
\psi_{S'}: MH_*(\mathcal{L}N)_{\eta}^{\otimes p} \to
MH_*(\mathcal{L}N)_{\eta}^{\otimes q}  \qquad (p\geq 1, q\geq 0).
$$
Notice we use homological conventions (see \ref{Subsection TQFT on SH summary}), so the operation goes ``from left to right'' in Figure
\ref{Figure Graphs}. For example $\psi_{Q'}: MH_*(\mathcal{L}N)_{\eta}^{\otimes 2}
\to MH_*(\mathcal{L}N)_{\eta}$ is a product. 

\emph{Remark: $p\geq 1$ is needed since only the unstable manifolds of the $\mathcal{E}_i$ are finite-dimensional, the stable ones are infinite-dimensional. 
%
%
In particular we cannot construct a unit via Morse theory since $\mathcal{E}$ does not have a maximum (but one can construct a counit, compare \ref{Subsection Morse-graph operations}). Nevertheless there is a unit $c_*[N]$ (see Remark \ref{Remark AS iso respects filtration} and the Remarks in \ref{Subsection Chas-Sullivan loop product} and \ref{Subsection varphi is a TQFT isomorphism}).
}

Identifying $MH_*(\mathcal{L}N)_{\eta}\!\cong\! H_*(\mathcal{L}N)_{\eta}$,
these Morse operations define a TQFT on
$H_*(\mathcal{L}N)_{\eta}$.

A string topology focused description of
these can be found in Cohen-Schwarz
\cite{Cohen-Schwarz}.
%
\subsection{The Chas-Sullivan loop product}
\label{Subsection Chas-Sullivan loop product}
%

Let $ev:\mathcal{L}N \to N$ be evaluation at $0$. Let
$\sigma:\Delta^a \to \mathcal{L}N$ and $\tau: \Delta^b \to
\mathcal{L}N$ be two singular chains, thought of as
$a,b$-dimensional families of loops. When two loops from those two
families happen to have the same base-point, we can form a
``figure 8" loop. Let $E=\{(s,t) \in \Delta^a \times \Delta^b:
ev(\sigma(s))=ev(\tau(t)) \}$ parametrize those match-ups and let
$j_E:E \to \mathcal{L}N$ output those figure-8 loops. Now $E=(ev\times
ev)^{-1}(\textrm{diagonal})$ via $ev\times ev: \Delta^a \times
\Delta^b \to N\times N$. So $E$ is a manifold when the two
families of basepoints, $ev(\sigma)$ and $ev(\tau)$, intersect
transversely in $N$, and then $\textrm{dim}(E)=a+b-n$. 
The \emph{loop product} $\sigma\cdot\tau$ of \cite{Chas-Sullivan} is then defined to be the cycle $(j_E)_*[E]
\in H_{a+b-n}(\mathcal{L}N)$  (see
\cite{Abbondandolo-Schwarz2} for details).
\begin{theorem}[Abbondandolo-Schwarz \cite{Abbondandolo-Schwarz2}]\label{Theorem
Abbondandolo-Schwarz product theorem} The map
$\varphi\co MH_*(\mathcal{L}N)\to SH^{n-*}(T^*N)$ preserves the
product structure, so $\varphi \circ \psi_{Q'} = \psi_P \circ
\varphi^{\otimes 2}$. Moreover the product $\psi_{Q'}$ on $MH_*(\mathcal{L}N)$
corresponds to the loop product on $H_*(\mathcal{L}N)$.
\end{theorem}
\begin{figure}[ht]
\includegraphics[scale=0.65]{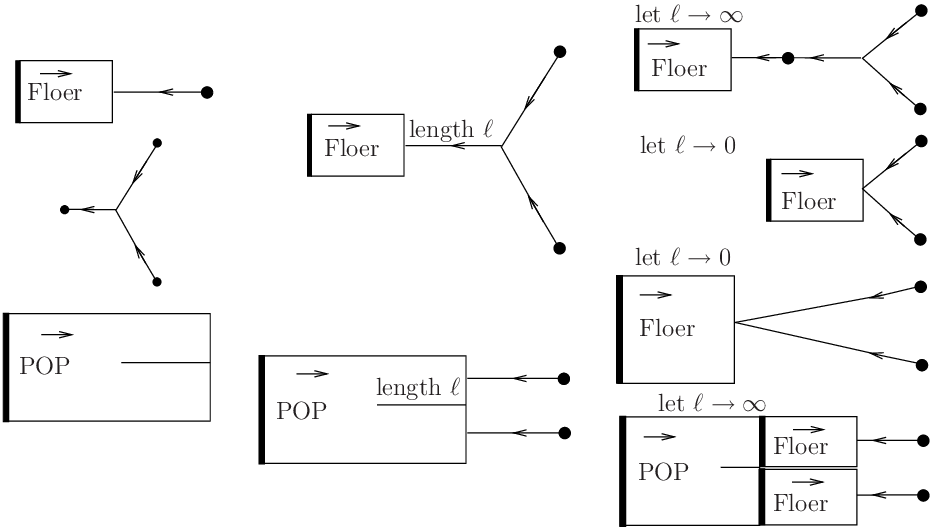}
\caption{Dark circles are critical points of $\mathcal{E}$, dark
rectangles are Hamiltonian orbits. POP are pairs of pants, Floer
are Floer trajectories. The configurations in the left column are counted by
$\varphi$, $\psi_{Q'}$, $\psi_P$.}\label{Figure AbbSchw loop
product}
\end{figure}
\noindent\emph{Sketch Proof.} Observe Figure \ref{Figure AbbSchw
loop product}. Consider the $1$-dimensional family of solutions
corresponding to the configurations in the middle column ($\ell$
is a free parameter). On the right are the limiting configurations
as $\ell \to 0$ or $\infty$. The middle two pictures in the last
column are the same. So the maps counting the top and bottom
configurations in the last column are chain homotopic. These are respectively
$\varphi\circ \psi_{Q'}$ and $\psi_P \circ \varphi^{\otimes 2}$.
\qed
\\[1mm]
\emph{Remark. Theorem \ref{Theorem Abbondandolo-Schwarz product theorem}
and Remark \ref{Remark AS iso respects filtration} prove that the unit for the loop product is $c_*[N]$.
}

\subsection{$\varphi$ is a TQFT isomorphism}
\label{Subsection varphi is a TQFT isomorphism}
%
To $S$ we
associated a graph $S'$ (\ref{Subsection Morse-graph associated to S}). Let $\overline{S'}$ be the graph obtained
from $S'$ by switching the orientations of the edges ($p,q$ get
interchanged). So for $S$ with $p+q$ punctures, $\psi_{\overline{S'}}: MH_*(\mathcal{L}N)^{\otimes q} \to MH_*(\mathcal{L}N)^{\otimes p}$.

\textbf{Example.} for $P,Q$ of \ref{Subsection minimal
set of gen surfaces} we get the $P',Q'$ of \ref{Subsection
Morse-graph associated to S} and $\overline{P'}=Q'$, $\overline{Q'}=P'$.
\begin{theorem}\label{Theorem Abb Schw is a TQFT iso}
The isomorphism $\varphi: H_*(\mathcal{L}N)_{\eta} \cong
MH_*(\mathcal{L}N)_{\eta} \to SH^{n-*}(T^*N)_{\eta}$ preserves the
$($possibly twisted$)$ TQFT structures:
$\varphi^{\otimes p} \circ \psi_{\overline{S'}} =
\psi_S \circ \varphi^{\otimes q}$ where $p\geq 1,q\geq 1.$
\end{theorem}
\begin{proof}
For general $S$ with $p\geq 1,q\geq 1$, decompose $S$ into copies of $Z,P,Q$ as in
Theorem \ref{Theorem decomposing S} (this can be done without using $C$ by the Remark in \ref{Section POP decomposition}). So $\psi_S$ is a
composition of operations $\psi_Z,\psi_P,\psi_Q$. Thus
$\psi_{\overline{S'}}$ is the corresponding composition of
operations obtained by replacing $Z,P,Q$ by
$\overline{Z'}, \overline{P'}, \overline{Q'}$. So
the general relation follows if we can prove it for $Z,P,Q$.

For $S=Z$, $\psi_Z$ and $\psi_{\overline{Z'}}$ are the identity, so
there is nothing to prove. 
%

For $S=P$, the relation $\varphi \circ \psi_{Q'} = \psi_P \circ
\varphi^{\otimes 2}$ is Theorem \ref{Theorem Abbondandolo-Schwarz
product theorem}.

For $S=Q$, we need to check $\varphi^{\otimes 2} \circ \psi_{P'} = \psi_Q \circ \varphi$. One can either prove this by an argument similar to Figure \ref{Figure AbbSchw loop product} using the techniques of Abbondandolo-Schwarz \cite{Abbondandolo-Schwarz2}, or one can use the second factorization in the proof of Theorem \ref{Theorem description of coproduct} as follows. The coproduct $Q$ factors as a gluing of $(P\sqcup Z)\#(Z\sqcup (Q\# C))$ where one can arrange the weights so that $Q\#C$ uses $H^{\delta}$ on each output (recall $H^{\delta}$ from Section \ref{Section Canonical map from ordinary cohomology}). Since Theorem \ref{Theorem Abb Schw is a TQFT iso} holds for $P,Z$, we reduce to checking that $\psi_{Q}(1) = \psi_Q(\varphi([N])) = \varphi^{\otimes 2}(\psi_{P'}[N])$ (recall $\varphi(c_*[N])=c^*(1)$ by Remark \ref{Remark AS iso respects filtration}, and that using $H^{\delta}$ instead of $H$ is equivalent to restricting to action values close to zero by Section \ref{Section SH+}).
By Theorem \ref{Theorem ring structure on ordinary cohomology}, we can identify $SH^*(H^{\delta})$ with the Morse cohomology, and $\psi_{Q}(1)=\psi_{Q'}(1)=\pm \chi(N) \mathrm{vol}_N^{\otimes 2}$ by Example \ref{Example coproduct}. 
One then checks that by restricting to $\mathcal{E}$-actions close to zero (see Remark \ref{Remark AS iso respects filtration}), the TQFT solutions for $\mathcal{L}N$ become time-independent (this would involve mimicking Section \ref{Subsection Floer trajectories converging to broken Morse trajectories}, which we will not carry out). 
We now compute $\psi_{P'}[N]$. Since $[N]$ is represented by the maximum of $L|_N$, the first edge of $P'$ will sweep out $[N]$. We need to perturb $\mathcal{E}=\int L\, dt$ on one of the two outgoing edges of $P'$ to achieve transversality. We do this by using $L'=L\circ \varphi_K^1:TN \to \R$ instead of $L$ on that edge, where $\varphi_K^1: TN \to TN$ is a time-1 Hamiltonian flow so that $\varphi_K^1(N), N$ are transverse in $TN$. Thus the only non-zero contribution to 
$\psi_{P'}[N]$ is the count of $w\in N\cap \varphi_K^1(N)$ which admit a configuration consisting of a semi-infinite $-\nabla L|_N$ flowline from $w$ to the minimum of $L|_N$ and a semi-infinite $-\nabla L'|_N$ flowline from $w$ to the minimum of $L'|_N$. So $\psi_{P'}([N])= \pm \chi(N) [\mathrm{pt}]^{\otimes 2} \in H_0(N)^{\otimes 2}$. So $\psi_{Q}(1) = \varphi^{\otimes 2}(\psi_{P'}[N])$ since $\varphi([\mathrm{pt}]) = \mathrm{vol}_N$ by Remark \ref{Remark AS iso respects filtration}.

For the twisted case, $\varphi$ was constructed in \cite{Ritter}.
For the argument above to hold in the twisted case, it suffices
that the weights are locally constant on the TQFT moduli spaces
and the $\mathcal{M}^+_0(\gamma,x)$ moduli spaces. The former is
 Theorem \ref{Theorem Novikov weights behave well}, the latter is proved similarly and was done in detail in \cite{Ritter}.
\end{proof}
\emph{Remark. Action-restricting the above result as in Remark \ref{Remark AS iso respects filtration}, one obtains that $c_*: H_*(N) \to H_*(\mathcal{L}N)_{\eta}$
is a TQFT map, with unit $[N]\mapsto c_*[N]$, using the (twisted) loop product 
on $H_*(\mathcal{L}N)_{\eta}$, and the intersection product on $H_*(N)$ (the Poincar\'e dual of the cup product).
}
\begin{corollary}
For a closed exact
Lagrangian submanifold $L\subset T^*N$, Theorem \ref{Theorem
Viterbo Functoriality} applied to Example \ref{Example Weinstein neighbourhood} yields via Theorem \ref{Theorem Abb Schw is a TQFT iso} the commutative TQFT diagram of
transfer maps
$$
\xymatrix@R=12pt@C=12pt{ H_*(\mathcal{L}L)_{\eta|L} \ar@{<-}[r]^-{}
\ar@{<-}[d]^-{c_*} &
H_*(\mathcal{L}N)_{\eta} \ar@{<-}[d]^-{c_*} \\
H_*(L)\otimes \Lambda \ar@{<-}[r]^-{} & H_*(N)\otimes \Lambda }
$$
Remark. The untwisted diagram (without TQFT) is due to Viterbo \cite{Viterbo3}. $\qed$
\end{corollary}
%
%
\subsection{The based loop space}
\label{Subsection Based loop space}
%
The wrapped analogue of $SH^*(T^*N)\cong H_{n-*}(\mathcal{L}N)$ for a fibre $T^*_{q_0}N\subset T^*N$ is
Abbondandolo-Schwarz's ring isomorphism
\cite{Abbondandolo-Schwarz2}:
$$
HW^*(T_{q_0}^*N) \cong H_{n-*}(\Omega N),
$$
using the Pontryagin product on $H_*(\Omega N)$ induced by
concatenation of based loops. As in Theorem \ref{Theorem Abb Schw
is a TQFT iso}, this respects the TQFT and the twisted
analogue holds.
%
\subsection{Vanishing of the Novikov homology of the free
loop space} \label{Subsection Vanishing of the Novikov homology of
the free loopspace}
%
\strut\\
\emph{Convention. Recall, as mentioned at the start of Section \ref{Section Application Loop space homology and string topology}, that we assume $\mathrm{char}(\mathbb{K})=2$.}

\noindent \textbf{Definition.} A space $N$ has \emph{finite type} if $\pi_m(N)$ is finitely
generated for each $m\geq 2$.\\
\strut\textbf{Examples.} 1) simply connected closed manifolds, and \\
\strut\quad\qquad\qquad\; 2) closed manifolds with trivial $\pi_1$ action
on higher homotopy groups.\\
For closed $N$ of
finite type, the
Novikov homology
$H_*(\mathcal{L}_0N;\underline{\Lambda}_{\tau\eta})\!=\!0$ for $\tau\eta\!\neq\! 0 \!\in\! H^1(\mathcal{L}_0 N)$ by \cite{Ritter} where $\mathcal{L}_0 N$ is the component of contractible loops. We can now
extend it to $\mathcal{L}N$:
\begin{corollary}\label{Corollary Full loopspace homology vanishes}
For closed manifolds $N$ of finite type, and any $\tau\eta\neq
0\in H^1(\mathcal{L}_0 N)$,
$$SH^*(T^*N)_{\eta} \cong H_{n-*}(\mathcal{L}N;\underline{\Lambda}_{\tau\eta})
=0.$$
\end{corollary}
\begin{proof}
By Theorem \ref{Theorem Abb Schw is a TQFT iso},
$SH^*(T^*N)_{\eta}\cong H_{n-*}(\mathcal{L}N)_{\eta}$ is a ring
with unit $c_*[N]\in H_n(\mathcal{L}_0N)_{\eta}$. But
$H_*(\mathcal{L}_0N)_{\eta}=0$ by \cite{Ritter}. So the claim
follows by Theorem \ref{Theorem Vanishing criteria}.
\end{proof}
\noindent$\textbf{Remark.}$ $\tau\eta\!\neq\! 0\in H^1(\mathcal{L}_0 N)$ iff
$\pi^*\eta\!\neq\! 0 \in H^2(\widetilde{N})$ for the universal cover
$\pi:\widetilde{N}\!\to\! N$.
%
%
%
\section{The $c^*$ maps preserve the TQFT structure}
\label{Section c* maps preserve TQFT}
%
%
%
%
%
%
\subsection{PSS map}
\label{Subsection The c* maps are unital ring maps}
%
%
Theorem \ref{Theorem ring structure on ordinary cohomology} is the analogue of the Piunikhin-Salamon-Schwarz \cite{PSS} ring isomorphism $FH^*(M,H) \to QH^*(M)$ which holds for weakly monotone \emph{closed} symplectic manifolds $(M,\omega)$. It turns the pair of
pants product on Floer cohomology into the quantum cup product on quantum cohomology. We will briefly survey their proof.
\begin{figure}[ht]
\input{pss2.pstex_t}
\label{Figure PSS}
\end{figure}

Pick a
generic Morse function $f$ and a Hamiltonian $H$ on $M$. Pick a
homotopy $H_s$ interpolating $0$ and $H$. The PSS map $\phi:
FC^*(H) \to MC^*(f)$ on generators is $\phi(y)=\sum N_{x,y} x$ where $N_{x,y}$ is the oriented count of isolated \emph{spiked
discs} converging to $x,y$ (we omit a detail here: one actually has to work over a Novikov ring since $\omega$ is not exact \cite[Sec.4]{PSS}). Spiked discs are maps $u:\C \to M$
such that $u(e^{2\pi(s+it)})$ satisfies Floer's continuation
equation for $H_s$, converging at $s=\infty$ to a 1-orbit
$y$ for $H$ and at $s=-\infty$ to a point $u(0)\in W^u(x,f)$ in the unstable
manifold of $f$. The ``spike" is the $-\nabla f$ flow line connecting $x$ to $u(0)$. An inverse $\psi:
MC^*(f) \to FC^*(H) $ (up to chain homotopy) is defined by
counting isolated spiked discs flowing in the reverse direction:
$u(e^{-2\pi(s+it)})$ is an Floer continuation solution
for $H_{-s}$ converging at $s=-\infty$ to a $1$-orbit $y$ for $H$ and
at $s=\infty$ to a point $u(0)\in W^s(x,f)$.

\begin{figure}[ht]
\input{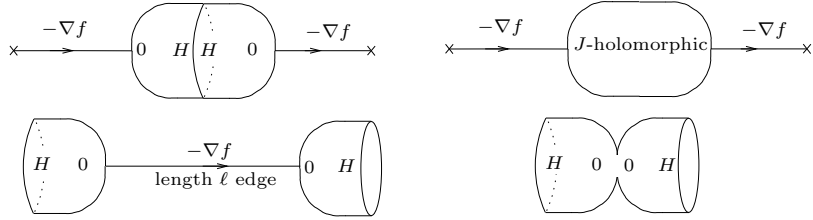}
\caption{Counting the solutions defining
$\phi\circ \psi$ and $\psi\circ \phi$.}\label{Figure PSS compositions}
\end{figure}

To show that $\phi, \psi$ are inverses up to chain homotopy,
consider $\phi\circ \psi$ and $\psi\circ \phi$.
After gluing solutions, these count configurations on the left in
Figure \ref{Figure PSS compositions}. Up to chain homotopy, the
maps don't change if we homotope respectively $H$ and $\ell$ to
zero. So we count the configurations on the right. In the
first case, for dimensional reasons the $J$-holomorphic sphere in
the figure is constant for generic $J$ (this uses the weak monotonicity of $M$). So we count isolated Morse
continuation solutions for a constant homotopy $f$, which are
constants (otherwise there is a $1$-family of
reparametrized solutions via $s\mapsto s+\textrm{constant}$). So
$\phi \circ \psi\! \simeq\! \textrm{id}$ are chain homotopic.  In
the second case, we used a gluing theorem for $J$-holomorphic
curves to obtain a cylinder (for such gluing arguments we recommend the appendix in \cite{McDuff-Salamon}). Up to chain homotopy, we can homotope
$H_s$ to an $s$-independent $H$. So we count isolated
Floer continuation solutions for a constant homotopy $H$, which
again must be constants. So $\psi \circ \phi \!\simeq\! \textrm{id}$.

\begin{figure}[ht]
\input{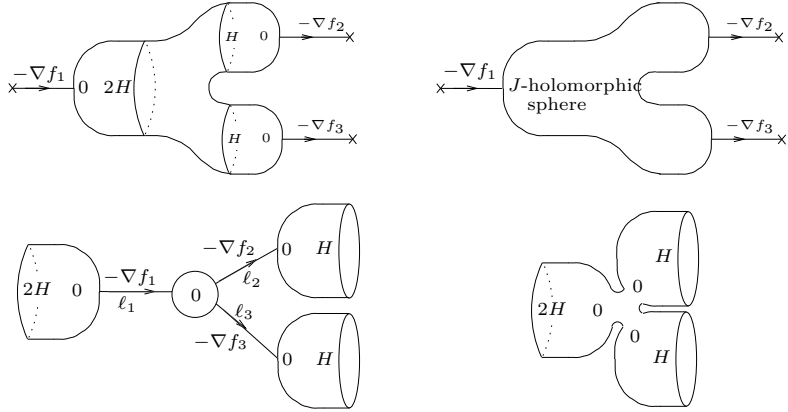}
\caption{Capping off Floer and Morse product solutions.}\label{Figure PSS ring structure}
\end{figure}

To prove that $\phi,\psi$ respect the ring structure observe Figure
\ref{Figure PSS ring structure}. We glue spiked discs to cap off,
respectively, Floer and quantum-Morse product solutions (the
quantum product replaces the vertex of the Morse-product
graph $P'$ of \ref{Subsection Morse-graph operations} with a $J$-holomorphic sphere). Up to chain homotopy, we can
homotope respectively $H$ and $\ell_1,\ell_2,\ell_3$ to zero,
reducing to the configurations on the right. In the second case, we used the
gluing theorem for $J$-holomorphic curves to get a pair of pants
Floer solution. So up to chain homotopy, we obtain
the products on $QH^*$ and $FH^*$ respectively. A
similar argument works for all (quantum) TQFT operations.
%
\subsection{Proof for Liouville domains}
\label{Subsection The c* maps are unital ring maps part 2}
%
For Liouville domains $M$, on the one hand the above construction simplifies because the $J$-holomorphic spheres must be constant by Stokes' theorem (so only the non-quantum Morse operations defined in \ref{Subsection Morse-graph operations} contribute), on the other hand the construction complicates because of the non-compactness of $\overline{M}$. Indeed, the compactness of the above moduli spaces must fail, for otherwise $\phi,\psi$ would always induce an isomorphism $SH^*(M) \cong H^*(M)$ which is false in general. Compactness in fact fails for $\phi$ since it uses a non-monotone homotopy $H_s$. The other map, 
$\psi:MC^*(f)\to SC^*(H)$, is actually well-defined by Lemma \ref{Lemma Maximum principle for Floer solns} since we can pick $H_{-s}$ to be monotone, and we ensure $-\nabla f$ is pointing inwards on the collar of $\overline{M}$ so that $MH^*(f)$ computes $H^*(M)$.
\begin{lemma}\label{Lemma Hamiltonians H0}
For $H=H^{\delta}$ as in Section \ref{Section Canonical map from ordinary cohomology},
$\psi: MH^*(f)\to SH^*(H^{\delta})$ is equal to $c^*_{H^{\delta}}:H^*(M) \cong
SH^*(H^{\delta})$ (see Theorem \ref{Theorem unit is image of 1}).
\end{lemma}
\begin{proof}
$H^{\delta}$ is $C^2$-small and Morse so the Floer solutions counted by $\psi$
are time-independent. Now $\psi$ changes by a chain homotopy if we
homotope $f,H_s$ to $H,H$. So we end up counting isolated Morse
continuation solutions for the constant homotopy $H$, which must
be constant solutions. So $\psi \simeq c^*_{H^{\delta}}$ are chain homotopic.
\end{proof}
\noindent \textbf{Goal:} To prove Theorem \ref{Theorem ring
structure on ordinary cohomology} by the methods of
\ref{Subsection The c* maps are unital ring maps} it remains to
construct a well-defined map
$\phi: SH^*(H^{\delta}) \to H^*(M)$
inverse to $\psi$. We will do this in \ref{Subsection Non-monotone
homotopies of small Hamiltonians}.
%
%
\subsection{Floer trajectories converging to broken Morse trajectories}
\label{Subsection Floer trajectories converging to broken Morse
trajectories}
%
%
The analytical machinery that Section \ref{Subsection Non-monotone homotopies of small Hamiltonians} is based upon is due to Salamon-Zehnder \cite{Salamon-Zehnder}. For the reader's convenience, we briefly review how they applied this machinery to prove $FH^*(H) \!\cong\! H^*(M)$ for
$C^2$-small $H$ on closed symplectic manifolds
$(M,\omega)$, assuming $M$ is symplectically aspherical: $\omega|_{\pi_2(M)}\!=\!0$ (this ensures in particular that $J$-holomorphic spheres are constant).

Let $S^1_T=\R/T\Z$. We will be considering $T$-periodic Floer
theory, as opposed to the $1$-periodic theory, and we will study
the limit $T\to 0$. Consider solutions $u: \R \times S^1_T \to \R^{2n}$ to Floer's
local equation $F_T (u)=0$, where $F_T: W^{1,2}(\R \times
S^1_T;\R^{2n}) \to L^2(\R \times S^1_T;\R^{2n})$ is the local Floer
operator defined by
$$\textstyle F_T(u) =\partial_s u + J\partial_t u + (S+A)u,$$
where $J$ is standard, $S=S(s)$ and $A=A(s)$ are continuous
matrix valued functions on $\R$ which are
respectively symmetric and antisymmetric, and which converge to
$S_{\pm}$ and $A_{\pm}=0$ as $s\to \pm\infty$ (the time-independence of $S$ and $A$ is the local equivalent of the assumption that the Hamiltonian $H$ is time-independent).
\\[1mm]
\noindent \textbf{Claim 1.}
\emph{The solutions $u$ of $F_T(u)=0$ are time-independent
whenever the $L^2$-norms satisfy $\| S \| + \| A \| \leq c <
1/T$ for some constant $c$.}

\noindent \emph{Sketch Proof.}\;\cite[Prop.4.2]{Salamon-Zehnder} The averaged
solution $\overline{u}(s)=(\int_0^T u(s,t)\, dt) /T$ is a solution
$F_0(\overline{u})=0$ of the local Morse operator $F_0: W^{1,2}(\R,\R^{2n})\to L^2(\R,\R^{2n})$,
$$\textstyle \textstyle F_0(\overline{u}) = \partial_s \overline{u} + (S+A)\overline{u},$$
and then one checks that $u(s,t)\equiv\overline{u}(s)$ when $S,A$
are small in the above sense. We remark that this result can also be proved using Fourier series. \qed
\\[1mm]\indent
Assume the hypothesis of Claim 1, and assume in addition that $S_{\pm}$ are non-singular (which is the
local equivalent of the non-degeneracy condition of orbits in
Floer theory).
\\[1mm]
\noindent \textbf{Claim 2.}
\emph{$F_T$ is onto iff $F_0$ is onto.}

\noindent \emph{Sketch proof.}\;\cite[Cor.4.3]{Salamon-Zehnder} This follows because they have the
same Fredholm index, and by Claim 1 they have isomorphic
kernels.\qed
\\[1mm]\indent
The local result can now be applied globally. Pick an
$\omega$-compatible almost complex structure on the closed
symplectic manifold $(M,\omega)$ such that $\omega(\cdot,J\cdot)$
is a Morse-Smale metric for a given Morse function $H:M\to \R$.
\\[1mm]
\noindent \textbf{Claim 3.}
\emph{For small $T>0$, every isolated finite energy Floer
trajectory $u:\R \times S^1_T \to M$ is time-independent and
regular $($the linearization of the Floer operator at $u$ is
onto$)$.}

\noindent \emph{Sketch proof.}\;\cite[Thm.7.3]{Salamon-Zehnder} When $u$ is a time-independent
Floer trajectory, i.e. a $-\nabla H$ flow, we reduce to the above
local case by picking an orthonormal frame: $F_0$ is onto by the
Morse-Smale condition, so $F_T$ is onto, so $u$ is regular.

To prove that $u$ must be time-independent we need a Gromov
compactness trick. By contradiction suppose there are isolated finite
energy Floer solutions $u_n:\R \times S^1_{T_n} \to M$ with
periods $T_n \to 0$, which are not time-independent, connecting
two given critical points $x,y$ of $H$ (we assume $T_n$ is so
small that there are no $T_n$-periodic Hamiltonian orbits, this
uses the fact that $H$ is Morse). The energy
\\[0.5mm]
\begin{tabular*}{\textwidth}{l@{\extracolsep{\fill}}cr@{\extracolsep{0pt}}} 
\strut & 
$
\displaystyle E(u_n)=\frac{1}{T_n}\int_{-\infty}^{\infty}\int_0^{T_n}
|\partial_s u_n|^2 \, dt\, ds = H(x)-H(y)
$
 & \strut 
\end{tabular*}
\\[0.5mm]
is bounded, therefore by Gromov compactness, a subsequence of the
$u_n$ must converge to some broken Morse trajectory $v_1\#\cdots
\#v_m$ for $H$, where $m$ is bounded by the difference of the
Morse indices of $x,y$. Notice that this uses the assumption $\omega|_{\pi_2(M)}=0$ to exclude sphere bubbling (sphere bubbles can otherwise naturally appear and in general one should expect the Gromov limits to consist of Morse trajectories intermediated by trees of $J$-holomorphic spheres -- an example of this phenomenon for the non-exact total space of $\mathcal{O}_{\mathbb{P}^1}(-1)$ is explained in \cite{Ritter4}).
Since the $u_n$ are isolated (up to $\R$-reparametrization), the index of the Fredholm problem is $1$, so also the broken Morse trajectory must have total index $1$. Therefore $m=1$ and $v=v_1$ is an isolated
Morse trajectory, i.e. a time-independent isolated Floer
trajectory.

We claim that for large $n$, $u_n(s+s_n',t)\equiv v(s)$ coincide
for some $s_n'\in \R$. This would contradict the assumption that
$u_n$ is not time-independent. So suppose by contradiction that
$u_n$ and $v$ never coincide in this way. 
If there are integers $C_n$ with $C_n T_n = T$, with $T$ so small
that $F_T$ is onto at $v$ by the first part of the proof, then
$F_T(u_n)=0$ and $u_n \to v$ contradicts the fact that $v$ is an
isolated solution of $F_T$. In general, pick integers $C_n$ with
$C_n T_n \to T$ and apply the same argument uniformly in an
interval around $T$.\qed
\\[1mm]
\noindent \textbf{Claim 4.} \emph{Given a Morse function $H:M \to
\R$ and a small enough $T>0$, all isolated Floer trajectories for $T\cdot
H$ are time-independent and regular.}

\noindent \emph{Proof.} The $T$-periodic and $1$-periodic Floer
theories are isomorphic, by sending a $T$-periodic Floer
trajectory $u(s,t)$ for $H$ to the $1$-periodic Floer trajectory
$u(sT,tT)$ for $T\cdot H$. By Claim 3, for $T\ll 1$, the
$T$-periodic Floer trajectories for $H$ are time-independent and
regular, so the same is true for the $1$-periodic ones for $T\cdot
H$.\qed
\\[1mm]\indent
We can apply the same arguments to Floer continuation solutions
$u:\R \times S^1_T \to M$ and Morse continuation solutions for
homotopies $H_s: M\to \R$ connecting two Morse functions and
picking $J_s$ so that $g_s=\omega(\cdot,J_s\cdot)$ is Morse-Smale
for $H_s$. Thus,
\\[1mm]
\noindent \textbf{Claim 5.} \emph{For small $T>0$, all
isolated Floer continuation solutions for $T\cdot H_s$ are time-independent
and regular.}\\[1mm]
\indent
We now show that the arguments of Claims 3-5 also hold for non-isolated Floer solutions. This result is stated in Salamon-Zehnder \cite[page 1343]{Salamon-Zehnder}: they mention that this generalization can be proved by a parameter-valued gluing argument. Strictly speaking, we only need this generalization in the proof of Lemma \ref{Lemma psi is a chain map}, and we mention in Remark \ref{Remark phi solutions} a work-around.
\begin{theorem}\label{Theorem parameter valued gluing}
  For a time-independent $H$ which is Morse-Smale for the metric induced by an almost complex structure $J$, and for small enough $T>0$, the finite energy Floer trajectories for $T\cdot H$ are time-independent and
regular. Similarly for Floer continuation solutions of $T\cdot H_s$.
\end{theorem}
\begin{proof}
We take a closer look at the proof of Claim 3.
Let $\mathcal{M}_T$ be the moduli space of Floer trajectories $u:\R \times S^1_T \to M$ joining two critical points $x,y$ of $H$ (for small $T$ the $1$-orbits of $T\cdot H$ are $\textrm{Crit}(H)$). The proof is by induction on the (finitely many possible values of the) Morse index difference between $x,y$. To keep the notation under control, we assume the index difference between $x,y$ is $2$, so the inductive hypothesis is the result for isolated Floer trajectories. Our proof of the inductive step from 1 to 2 easily generalizes to the general case.

Let $\mathcal{M}_0$ be the moduli space of Morse trajectories from $x$ to $y$. By the Morse-Smale assumption, this moduli space is a smooth, typically non-compact, manifold of the expected dimension $1$ (the Morse index difference minus the $\R$-reparametrization freedom), and it has a compactification $\overline{\mathcal{M}}_0$  consisting of all possibly-broken Morse trajectories from $x$ to $y$. So in particular, $\overline{\mathcal{M}}_0$ contains all Gromov limits that arise from sequences $u_{T_n}\in \mathcal{M}_{T_n}$ as $T_n \to 0$ (as emphasized in the proof of Claim 3, this uses that $J$-holomorphic spheres in $M$ are constant).

Let us abbreviate by $F_0,F_T$ the linearizations of the relevant Fredholm maps which define the moduli spaces $\mathcal{M}_0,\mathcal{M}_T$. Since Fredholm indices are locally constant and $F_0$ has index $2$, also 
$F_T$ must have index $2$ for small enough $T$ (the maps in $\mathcal{M}_T$ must get close to the maps in $\mathcal{M}_0$, otherwise thanks to the energy estimate $E=H(x)-H(y)$ a Gromov compactness argument as $T\to 0$ would produce a broken Morse trajectory which is not in $\overline{\mathcal{M}}_0$). So $\mathcal{M}_T$ has virtual dimension $1$. For small enough $T$, by Claim 2, the operators $F_0,F_T$ have isomorphic kernels. But since their Fredholm indices equal, also their cokernels have the same rank. Regularity of $F_0$ means that this cokernel vanishes, so also $F_T$ is regular. Therefore $\mathcal{M}_T$ is a smooth $1$-dimensional manifold, and it has a compactification $\overline{\mathcal{M}}_T$ by broken Floer trajectories. The boundary $\partial \overline{\mathcal{M}}_T$ consists of broken Floer trajectories, which by the inductive hypothesis must be time-independent for all small enough $T$, so $\partial \overline{\mathcal{M}}_T \subset \partial \overline{\mathcal{M}}_0$.

Observe that $\overline{\mathcal{M}}_0\subset \overline{\mathcal{M}}_T$ viewing Morse trajectories for $H$ as time-independent Floer trajectories. But both $\overline{\mathcal{M}}_0,\overline{\mathcal{M}}_T$ are compact $1$-manifolds so the inclusion $\overline{\mathcal{M}}_0\subset \overline{\mathcal{M}}_T$ identifies those connected components which contain some possibly-broken time-independent trajectory. So for these particular connected components of $\overline{\mathcal{M}}_T$ the Floer trajectories are time-independent.


Now run the final argument of Claim 3: we can assume there are integers $C_n$ with $C_nT_n = T$ with $T$ small so that $F_T$ is onto (otherwise apply the argument in a small interval around $T$ and use $C_n T_n \to T$). Suppose by contradiction that for $T_n \to 0$, there are $u_n \in \mathcal{M}_{T_n}$ lying in a connected component of $\overline{\mathcal{M}}_{T_n}$ which does not contain any possibly-broken time-independent trajectory. As $n\to \infty$, $u_n$ converges to a possibly-broken time-independent trajectory $u\in \overline{\mathcal{M}}_{0}$. Since $F_T(u_n)=F_{C_n T_n}(u_n)=0$, both $u_n,u\in \overline{\mathcal{M}}_T$. But $\overline{\mathcal{M}}_T$ is a compact $1$-manifold, so $u_n$ converging to $u$ implies that for large $n$ all the $u_n$ lie in the same connected component of $\overline{\mathcal{M}}_T$ as $u$ -- but this contradicts the choice of $u_n$ since $u$ is time-independent.
\end{proof}
%
\subsection{Non-monotone homotopies of small Hamiltonians}
\label{Subsection Non-monotone homotopies of small Hamiltonians}
We now deal with the ``Goal" of \ref{Subsection The c* maps
are unital ring maps part 2}. 
From now on, $M$ is a Liouville domain, $H=H^{\delta}$ is as in Section \ref{Subsection The maps c from ordinary
cohomology} so the $1$-orbits of $H$ are the critical points of $H$.
Pick a Morse function $f:\overline{M}\to \R$ and a (non-monotone) homotopy $H_s$ from $0$ to $H$
 of the form $f=f(R)$, $H_s=h_s(R)$ on the collar, with
$f'(R)>0$ and $h_s'(R)>0$ (unless $H_s\equiv 0$). This ensures
that $-\nabla f$ and $-\nabla H_s$ are inward-pointing on the
collar, so in particular the Morse cohomologies of $f$ and $H$ are well-defined and are isomorphic to $H^*(M)$.
\begin{definition}
A \emph{$\phi$-solution} $v\# u$ from $x\in
\textrm{Crit}(f)$ to $y(t)\equiv y\in \textrm{Crit}(H)$ consists of:
\begin{enumerate}
 \item $v:(-\infty,0] \to \overline{M}$, which is a $-\nabla f$ flowline from $x$ to $u(0)$;
 \item $u:\C \to \overline{M}$, such that $u(e^{2\pi(s+it)})$ is a Floer continuation solution for $H_s$
converging to $y(t)$ as $s\to \infty$.
\end{enumerate}
Write
$\mathcal{M}^{\phi}(x,y)$ for the moduli space of
$\phi$-solutions. We call it a \emph{$T$-periodic $\phi$-solution}
if we instead use $u(e^{2\pi(s+it)/T})$ and parametrize $t$ by $\R/T\Z$
instead of $\R/\Z$.
\end{definition}

We choose the relevant almost complex structures $J$, $J_s$ so that they induce Morse-Smale metrics for
$H$, $H_s$, and we choose $f$ generically so that the time-independent $\phi$-solutions are regular solutions of the defining Fredholm problem.

\begin{lemma}\label{Lemma phi solutions}
Fix a compact $M'\subset \overline{M}$ containing $M$. For small
enough $T>0$ all $\phi$-solutions for $T\cdot H_s$ lying in $M'$
are time-independent, regular and lie in $M$.
\end{lemma}

\begin{proof}
We dicuss the case of isolated solutions - the general case is proved by mimicking the proof of Theorem \ref{Theorem parameter valued gluing}.  

Apply the Gromov compactness trick of \ref{Subsection Floer
trajectories converging to broken Morse trajectories} Claim 3: suppose by
contradiction there are $T_n$-periodic isolated $\phi$-solutions $v_n \#
u_n \in \mathcal{M}^{\phi}(x,y)$ lying in $M'$, with $T_n\to 0$,
such that $u_n$ is not time-independent. The $v_n \# u_n$ have
bounded energy
$$
\begin{array}{lll}
E(v_n)+E(u_n)& = & (f(x)-f(u(0)))+(H(u(0))-H(y))\\
& \leq & f(x)-H(y)+\max(f|_M) + \max (H|_M)
\end{array}
$$
because $u(0)$ must lie in $M$ (since $-\nabla f$ is
inward-pointing on the collar). So by Gromov compactness we can
extract a subsequence converging to a time-independent broken $\phi$-solution. But since $v_n \#
u_n$ are isolated, the index of the Fredholm problem is $0$, so the broken $\phi$-solution also has total index $0$, so in fact it is not broken. So this time-independent $\phi$-solution $v\# u$ is isolated, and it is regular by the comments above the statement of the Lemma.

We claim that for large $n$, $u_n(s,t)\equiv u(s)$ coincide, contradicting the time-dependence of $u_n$. If $u_n$ and $u$ never coincide in this way, then a time-rescaling argument as in the proof of Claim 3 of \ref{Subsection Floer trajectories converging
to broken Morse trajectories} would contradict that $v\# u$ is an isolated $\phi$-solution. Thus
for small enough $T>0$, all $T$-periodic $\phi$-solutions lying in
$M'$ are time-independent and regular.

For time-independent $\phi$-solutions, since the ends lie in $M$,
the flowline cannot exit $M$ otherwise $-\nabla f$ or $-\nabla
H_s$ would point outward at some point on the collar. The result
follows by a rescaling trick like in Claim 4 of \ref{Subsection
Floer trajectories converging to broken Morse trajectories}.
\end{proof}

\begin{lemma}\label{Lemma no escape trick for phi solutions}
Let $\mathcal{M}$ be a connected component of
$\mathcal{M}^{\phi}(x,y)$. Then $\mathcal{M}$ consists either
entirely of solutions lying in $M$ or entirely of solutions which
escape $M'$.
\end{lemma}
\begin{proof}
By Lemma \ref{Lemma phi solutions}, $\phi$-solutions lying in $M'$
must lie in $M$, so they cannot be $C^0$-close to a
$\phi$-solution which escapes $M'$.
\end{proof}

Similarly define $\psi$-solutions,
$\mathcal{M}^{\psi}(y,x)$ for $f$, $H_{-s}$ (see \ref{Subsection
The c* maps are unital ring maps}), with
$H_{-s}$ monotone. Recall that for these the maximum principle holds, so:

\begin{lemma}\label{Lemma psi solutions}
All $\psi$-solutions are contained in $M$. Moreover, for small
$T>0$, all $\psi$-solutions for $T\cdot H_{-s}$ are
time-independent, regular, and contained in $M$. $\qed$
\end{lemma}

\begin{definition}\label{Definition phi and psi maps}
Replace
$H_s$ by $T\cdot H_s$ so that the previous two lemmas hold. Define
$ \phi: SC^*(H) \to MC^*(f)$ and $\psi: MC^*(f) \to SC^*(H)
$
by counting only the isolated $\phi,\psi$-solutions \textbf{which
lie in $\mathbf{M}$} (so we ignore the $\phi$-solutions which exit $M'$).
\end{definition}

\begin{lemma}\label{Lemma psi is a chain map}
$\phi$ and $\psi$ are chain maps.
\end{lemma}
\begin{proof}
Consider a $1$-dimensional connected component
$\mathcal{M}\subset\mathcal{M}^{\psi}_1(y,x)$. The boundary
consists of broken $\psi$-solutions, where a Floer or a Morse
trajectory breaks off at one of the two ends. The count of these
broken solutions proves that $d_{\textrm{Floer}}\circ \psi - \psi
\circ d_{\textrm{Morse}}=0$. The same proof works for $\phi$ by
Lemma \ref{Lemma no escape trick for phi solutions}, since we can
ignore those $\mathcal{M}\subset \mathcal{M}^{\phi}_1(x,y)$ which
contain a $\phi$-solution which escapes $M'$.\end{proof}

\begin{remark}\label{Remark phi solutions} 
One can prove Lemma \ref{Lemma psi is a chain map} while only using the result of Lemma \ref{Lemma phi solutions} for isolated solutions (so we avoid appealing to Theorem \ref{Theorem parameter valued gluing}). Indeed, suppose by contradiction that there exist $T_n$-periodic $\phi$-solutions $v_n\# u_n$ of index $1$ lying in $M'$ but which exit $M$, where $T_n \to 0$. After passing to a subsequence, $v_n \# u_n$ converges as $T_n \to \infty$ either to a genuine time-independent $\phi$-solution of index $1$ or to a broken time-independent $\phi$-solution. But both those limits are contained in the interior of $M$, so also $v_n \# u_n$ lies in $M$, contradiction. 
\end{remark}

\begin{theorem}\label{Theorem psi phi are inverses}
$\psi,\phi$ are inverse to each other up to chain homotopy. In
particular, $\phi: SH^*(H^{\delta}) \to H^*(M)$ is an inverse for
$\psi=c^*: H^*(M) \to SH^*(H^{\delta})$.
\end{theorem}
\begin{proof}
Run the argument of \ref{Subsection The c* maps are unital ring
maps}, using that all $J$-holomorphic spheres are constant since the symplectic form is exact. The proof of Figure \ref{Figure PSS compositions} involves
the connected $1$-dimensional components $\mathcal{M}$ of the
moduli space $\bigcup_{c\in [0,1]}
\mathcal{M}^{\phi\circ \psi}(x,x';f,c\cdot H_s)$ of glued
solutions for $\phi \circ \psi$ for $c\cdot H_s$. The count of all
$\partial \mathcal{M}$ proves the existence of the chain homotopy
mentioned in Figure \ref{Figure PSS compositions} like in the
proof of Theorem \ref{Theorem Homotopy of surfaces}. We need to
justify why we can ignore those $\mathcal{M}$ which contain a
solution which escapes $M'$. This is proved like Lemma \ref{Lemma
no escape trick for phi solutions}, since the glued solutions for
$c\cdot H_s$ either lie in $M$ or escape $M'$, by combining Lemmas
\ref{Lemma phi solutions} and \ref{Lemma psi solutions}. \\ \indent
\emph{Technical Remark. More precisely, this last argument is as follows. The glued solutions depend on a gluing parameter $\lambda>0$ and as $\lambda\to \infty$ they converge to broken solutions, which by assumption lie in $M$. So for $\lambda\gg 0$, we claim that the glued solutions in $\mathcal{M}$ must also lie in $M$. Suppose not. Then there is a sequence of glued solutions $w_n^{\phi}\#_{\lambda_n} w_n^{\psi}\in \mathcal{M}$ lying in $M'$ which exit $M$ with $\lambda_n \to \infty$ as $n\to \infty$. By passing to a subsequence, we can assume that they arise from gluings for $c_n\cdot H_s$ with $c_n\to c_{\infty}\in [0,1]$. In the limit, this converges to a broken solution in $M'$ consisting of a $\phi$-solution and a $\psi$-solution (defined using $c_{\infty} H$) such that at least one of them does not lie in the interior of $M$. So this contradicts Lemma
\ref{Lemma phi solutions} or \ref{Lemma psi solutions}.
}\\ \indent
So $\phi\circ \psi$ is chain homotopic to the map we get at $c=0$,
which is the identity since it counts isolated Morse continuation
solutions for a constant homotopy (\ref{Subsection The c* maps are
unital ring maps}). Similarly $\psi\circ \phi \simeq \textrm{id}$.
\end{proof}
\subsection{Proof of Theorem \ref{Theorem ring structure on ordinary
cohomology}}\label{Subsection S,S' pairs}
Let $H=H^{\delta},H_s,f$ be as above. Let $S$ be a model surface
with weights $A_a,B_b \leq 1$ (see \ref{Subsection Model Riemann
Surface}). Denote by
$$\phi_a: SH^*(A_a H^{\delta}) \to MH^*(f_a) \;\textrm{ and }\;\psi_b:
MH^*(f_b) \to SH^*(B_b H^{\delta})
$$
the $\phi$-,$\psi$-maps obtained for $A_a \cdot H_s$, $B_b \cdot
H_{-s}$ and Morse functions $f_a,f_b$ obtained by generically
perturbing $f$. Let $S'$ be the graph associated to $S$
(\ref{Subsection Morse-graph associated to S}, \ref{Subsection
Morse-graph operations}). Then
$$\psi_S:\otimes_b SH^*(B_b H^{\delta})\to \otimes_a SH^*(A_a H^{\delta})
\;\textrm{ and }
\;\psi_{S'}:\otimes_b MH^*(f_b) \to \otimes_a MH^*(f_a).$$

\begin{lemma}\label{Lemma preserve TQFT structures}
$\otimes_a \phi_a \circ \psi_S \circ \otimes_b \psi_b$ is chain
homotopic to $\psi_{S'}$ (compare Figure \ref{Figure PSS ring structure}).
\end{lemma}
\begin{proof}
The proof follows like Theorem \ref{Theorem psi phi are inverses} using the
fact that glued solutions of $\otimes_a \phi_a \circ \psi_S \circ
\otimes_b \psi_b$ either lie in $M$ or escape $M'$, combining
Lemmas \ref{Lemma Maximum principle for Floer solns}, \ref{Lemma
phi solutions}, \ref{Lemma psi solutions}.
\end{proof}
Similarly, $\otimes_a \psi_a \circ \psi_{S'} \circ \otimes_b
\phi_b$ is chain homotopic to $\psi_S$, in the obvious notation.
Theorem \ref{Theorem ring structure on ordinary cohomology} then
follows upon taking direct limits and using Theorems \ref{Theorem
Operations are compatible with limit} and \ref{Theorem psi phi are
inverses}.
%
\subsection{The maps $c^*:H^*(L) \to HW^*(L)$ preserve the TQFT (Theorem \ref{Theorem c* map for HWL preserve TQFT})}\label{Subsection c* respects TQFT Lagrangian case}
Recall from Section \ref{Section Canonical map from ordinary cohomology} that in the wrapped case there is a map $c^*:H^*(L) \cong HW^*(L;H^{\delta}) \to HW^*(L)$. To prove that this map is a TQFT map one needs to show that continuation maps 
$HW^*(L;H_+) \to HW^*(L;H_-)$ are compatible with the TQFT, and that the isomorphism $H^*(L)\otimes \Lambda \to HW^*(L;H^{\delta})$ of Section \ref{Section Canonical map from ordinary cohomology} is compatible with the TQFT. The proof of the former is analogous to the proof for the $SH^*$ groups (Theorem \ref{Theorem Operations are compatible with
limit}), but the proof of the latter requires a PSS-description of the isomorphism $H^*(L)\otimes \Lambda \to HW^*(L;H^{\delta})$ (compare Section \ref{Subsection The c* maps are unital ring maps}). We will not pursue this in great detail, but we will mention how this can be done by changing the setup in Section \ref{Subsection The c* maps are unital ring maps} as follows.

We assume $\omega|_{\pi_2(M,L)}=0$ so $J$-holomorphic discs bounding $L$ are constant (this will hold by Stokes' theorem  in the Liouville setup when mimicking \ref{Subsection The c* maps are unital ring maps part 2} since $\overline{L}\subset \overline{M}$ is exact). Let $\mathbb{D}=\{ z\in \C: |z|\leq 1\}$. The spiked disc $u:\C \to M$ of \ref{Subsection The c* maps are unital ring maps} which intersected the $-\nabla f$ spike at $0\in \C$ gets replaced by a punctured disc $u: \mathbb{D}\setminus \{+1\} \to M$ which intersects the spike at $z=-1\in \partial \mathbb{D}$. We pick a holomorphic strip-like end parametrization $[0,\infty) \times [0,1]$ for $\mathbb{D}\setminus \{+1\}$ near $+1 \in \partial \mathbb{D}$, and on this strip we define the interpolation $H_s$ from $H_s =0$ for $s\leq 1$ to $H_s = H^{\delta}$ for $s\geq 2$. We require that the map $u:\mathbb{D} \setminus \{+1\} \to M$ is $J$-holomorphic away from the strip-like end; it satisfies Floer's equation for $H_s$ on the strip-like end; and it satisfies the Lagrangian boundary condition $u(\partial \mathbb{D})\subset L$. The discussion in Section \ref{Subsection The c* maps are unital ring maps} can now be carried out analogously. For example, in Figure \ref{Figure PSS compositions}, instead of a sphere which we homotope to a holomorphic sphere, we have a disc which arises from gluing two punctured discs; we then homotope the Hamiltonian to zero so the disc becomes a $J$-holomorphic disc $v$ bounding $L$, which must be constant since its energy $\int v^*\omega$ is zero by $\omega|_{\pi_2(M,L)}=0$.
%
%
%
\section{Appendix 1: TQFT structure on $SH^*(M)$}
\label{Section Floer solutions}
%
%
\subsection{Model Riemann surfaces}\label{Subsection Model Riemann
Surface}
%
See Figure \ref{Figure Riemann surface}. Let $(S,j)$ be a Riemann
surface with $p\geq 1$ negative punctures and $q\geq 0$ positive
punctures, with a fixed choice of complex structure $j$ and a
fixed choice of parametrization $(-\infty,0]\times S^1$ and
$[0,\infty)\times S^1$ respectively near the negative and positive
punctures so that $j\partial_s =
\partial_t$. We call these parametrizations the \emph{cylindrical
ends}. We assume $S$ has no boundary and that away
from the ends $S$ is compact. 

There is a
contractible set of choices of complex structures $j$ on $S$
extending the $j$ on the ends ($j\partial_s =
\partial_t$), see \cite[2.2.1]{Schwarz}. We emphasize that we keep $j$ fixed (see the Remark in \ref{Subsection Compactness of Moduli Spaces}).

We also fix a one-form $\beta$ on $S$ satisfying $d\beta\leq 0$,
such that on each cylindrical end $\beta$ is some positive
constant multiple of $dt$. These constants $A_1,\ldots, A_p,
B_1,\ldots B_q>0$ will be called \emph{weights}. The reason for
using $\beta$ is that we do not have a global form $dt$ for a
general surface $S$, unless $S$ is a cylinder. The condition
$d\beta\leq 0$ will be crucial to prove an
priori energy estimate and a maximum principle (Sections \ref{Subsection Energy Appendix} and
\ref{Subsection Maximum principle for floer solns}). 

By Stokes' theorem $\sum A_a - \sum B_b = -\int_S d\beta \geq 0$ (observe this forces $p\geq 1$). We now show conversely that given $A_a,B_b>0$, this inequality is the only obstruction to constructing $\beta$.

\begin{lemma}\label{Lemma beta exists}
 If $c=\sum A_a - \sum B_b\geq 0$ then such forms $\beta$ exist. If $c=0$, then $d\beta=0$. If $c>0$, then one can ensure that $d\beta=0$ except on an arbitrarily small disc $D$ embedded in $S$.
\end{lemma}
\begin{proof}
Suppose $A_a,B_b$ satisfy $c=\sum A_a - \sum B_b=0$. Consider the exact sequence 
%
%
$
H^1(S) \stackrel{f}{\to} H^1(\partial S) \stackrel{g}{\to} H^2(S,\partial S) \to 0
$
 for the pair $(S,\partial S)$. On de Rham forms, $f[\beta] = [\beta|_{\partial S}]$ is the pull-back and $g[\alpha] = [d\widetilde{\alpha}]$ for any extension $\widetilde{\alpha}$ of $\alpha$ to $S$. We can identify $H^2(S,\partial S) \cong \R$ by integration $[\sigma] \to \int_S \sigma$. Let $\alpha$ be the $1$-form on $\partial S$ given by the data $A_a\, dt, B_b\, dt$. Under that identification, the map $g$ becomes $g[\alpha]=-\sum A_a + \sum B_b =0$. So by exactness there is a 1-form $\beta$ on $S$ with $d\beta=0$, and $[\beta|_{\partial S}] = [\alpha]$. So by adding to $\beta$ an exact form supported near the ends we can ensure that $\beta|_{\partial S}$ equals the data $A_a\, dt, B_b\, dt$ on the ends.

Now suppose $c=\sum A_a - \sum B_b > 0$. Let $\widetilde{S}$ be the Riemann surface $S$ with an additional positive puncture at a chosen point of $S$, and fix a cylindrical end parametrization near this new puncture. The above procedure applied to $\widetilde{S}$ yields a 1-form $\widetilde{\beta}$ on $\widetilde{S}$ with $d\widetilde{\beta}=0$, such that $\widetilde{\beta}$ equals $A_a\, dt, B_b\, dt$ near the original ends and equals $c\, dt$ on the new end. Finally, extend $\widetilde{\beta}$ to a form $\beta$ on $S$ by defining $\beta = h(s)\, dt$ on the new cylindrical end $[0,\infty)\times S^1$ with $h=c$ for $s\leq 1$, $h=0$ for $s\geq 2$, and $h'\leq 0$ everywhere. So $d\beta=0$ on $S$ except in the small region near the new puncture where $h'\neq 0$, and there $d\beta = h'(s)\, ds \wedge dt \leq 0$.
\end{proof}
%
%
\subsection{Floer solutions} \label{Subsection Floer solutions S surface}
%
Let $H: \overline{M} \to \R$ be a Hamiltonian linear at infinity (\ref{Subsection Hamiltonians Linear At Infty}) and assume $H\geq 0$ (required in \ref{Subsection Energy}).
A Floer solution is a smooth map $u :S \to \overline{M}$
converging to Hamiltonian orbits as $s\!\to\! \pm \infty$ on the ends,
such that $du-X\otimes \beta$ is $(j,J)$-holomorphic:
$$
(du - X\otimes \beta)^{0,1} \equiv \tfrac{1}{2}\left\{ (du - X\otimes \beta) + J \circ (du-X\otimes \beta)\circ j\right\} = 0.
$$
On a cylindrical end with weight $c$ this is Floer's equation
$\partial_s u + J (\partial_t u - cX)\!=\!0$ for the Hamiltonian $cH$.
So $u$ converges to $1$-orbits of $A_a H,B_b H$.

Let $\mathcal{M}(x_1,\ldots,x_p; y_1,\ldots, y_q;S,\beta)$ denote
the moduli space of Floer solutions which converge to Hamiltonian
orbits $x_a$ at the negative ends and $y_b$ at the positive ends.
We abbreviate it by $\mathcal{M}(x_a; y_b;S,\beta)$ and we call
$x_a,y_b$ the \emph{asymptotics}. We write
$\mathcal{M}_k(x_a;y_b;S,\beta)$ for the $k$-dimensional part of
$\mathcal{M}(x_a; y_b;S,\beta)$.
\begin{example}[Continuation cylinders]\label{Example Continuation cylinder}
Let $S$ be the cylinder $Z\!=\!\R\times S^1$ with $p\!=\!q\!=\!1$; weights
$A_1=m', B_1=m$ with $m'\geq m$; $\beta\!=\!f(s)\,dt$ with $f'(s)\!\leq\! 0$. Then $u\in
\mathcal{M}(x;y;Z,\beta)$ solves $\partial_s u + J(\partial_t u -
f(s) X)=0$, so it is a Floer continuation solution for the
homotopy $H_s=f(s)\,H$ from $m'H$ to $mH$. The monotonicity
condition $\partial_s h_s'=f'(s)\,h'\leq 0$ of \ref{Subsection
Floer continuation solutions} is equivalent to the condition
$d\beta=f'(s)\, ds\wedge dt\leq 0$ of \ref{Subsection Model
Riemann Surface}. Thus the count of $\mathcal{M}_0(x;y;Z,\beta)$
is the continuation map $SH^*(mH) \to SH^*(m'H)$.
\end{example}
%
\subsection{Energy}\label{Subsection Energy}
%
A solution $u: S \to \overline{M}$ of $(du - X\otimes \beta)^{0,1}=0$ has energy
defined by $$E(u)=\frac{1}{2} \int_S \| du-X\otimes \beta \|^2\,\textrm{vol}_S= \int_S u^*\omega-u^*(dH)\wedge \beta.$$ In \ref{Subsection Energy Appendix} we show that, since $H\geq 0$, $\omega=d\theta$ and $d\beta\leq  0$, the energy is determined a priori from the asymptotics:
$$
E(u) \leq \int_S u^*\omega - d(u^*H\, \beta)  = \!\!\!\!\sum_{\textrm{negative ends}\;
a}\!\!\!\!\!\!\!\! {\mathbb{A}}_{A_a H}(x_a) - \!\!\!\!\sum_{\textrm{positive
ends}\; b}\!\!\!\!\!\!\!\! {\mathbb{A}}_{B_b H}(y_b).
$$
%
%
%
%
\subsection{Fredholm theory}
\label{Subsection Smoothness and Fred prop}
%
Lemma \ref{Lemma
Maximum principle for Floer solns} proves that all $u\in \mathcal{M}(x_a; y_b;S,\beta)$ lie in the
compact set $R \leq \max (R(x_a),R(y_b),R_0)$ in $\overline{M}$ for $J$ of contact type for $R\geq R_0$.
So the following Lemmas don't notice that
$\overline{M}$ is non-compact, so the proofs for closed manifolds apply.
After a small generic time-dependent
perturbation of $(H,J)$ on the ends, the following hold:

\begin{theorem}\label{Theorem Smoothness and exp convergence}\cite[2.5.7]{Schwarz}. Solutions of
$(du-X\otimes \beta)^{0,1}=0$ are smooth and those of finite energy converge exponentially fast to Hamiltonian orbits at the
ends, that is near each end $|\partial_s u|\leq ce^{-\delta |s|}$
for some constants $c,\delta>0$.
\end{theorem}
\begin{theorem}\label{Theorem index of Fredholm operator}\cite[3.1.31,
3.3.11]{Schwarz} or \cite[Prop.4]{Bourgeois-Mohnke}. Let $u$ be a
Floer solution. Let $D_u$ be the linearization at $u$ of the
operator defining $(du-X\otimes
\beta)^{0,1}=0$ (see \ref{Subsection Linearization Du}), then $D_u$ is Fredholm with 
$\textstyle \mathrm{index}\, D_u = \sum |x_a| - \sum |y_b| +2n(1-g-p),$
where $g$ is the genus of $S$ $($after filling in the punctures$)$.
See Remark \ref{Remark CZ convention} for the grading conventions. See \ref{Subsection Choice of trivializations} for the index calculation.
\end{theorem}
%
%
%
%
\subsection{Smoothness of the Moduli Spaces}\label{Subsection
Smoothness of Moduli Spaces}\label{Remark cannot make H=0}
%
For closed symplectic manifolds $M$, the condition $d\beta\leq 0$ was not necessary. So the
natural approach \cite{Schwarz} was to pick any Hamiltonians
$H_a,H_b$ on the ends and any tensor $\kappa=X\otimes \beta$ which
equals $X_{H_a},X_{H_b}$ on the ends, with $\kappa=0$ away from the
ends $($so $u$ is $J$-holomorphic there: $du^{0,1}=0)$. Using this approach for closed $M$, Schwarz \cite[4.2.17,
4.2.20]{Schwarz} proved smoothness of
$\mathcal{M}(x_a;y_b;S,\beta)$ for generic $\kappa$ close to
0 away from the ends. 
Unfortunately, for
Liouville domains $M$ we cannot use this approach, since the maximum principle and the a priori energy estimate fail without $d\beta\leq 0$. So we need to reprove the smoothness result.

Just as in the Technical Remarks in \ref{Subsection Transversality and Compactness}, we need the $1$-orbits of $A_aH,B_b H$ to be nondegenerate so that they are isolated and so that $\mathcal{M}(x_a;y_b;S,\beta)$ is the zero set of a Fredholm map -- so one typically needs to make a time-dependent perturbation $H_t$ of $H$. Just as for Floer continuation solutions in \ref{Subsection Floer continuation solutions} 
we typically need to make a $C^2$-small perturbation $J_z$ of $J$ depending on $z \in S$ and supported away from the ends of $S$ to ensure the smoothness of $\mathcal{M}(x_a;y_b;S,\beta)$. In \ref{Subsection Linearization Du} we explain this perturbation. Such perturbed $J$ are called \emph{regular} (we discuss this below). 
So the precise dependence of $H,J$ on parameters is as follows. 
\begin{enumerate}
 \item $J\in C^{\infty}(S\times
\overline{M},T\overline{M})$ is $\omega$-compatible, and $H\in C^{\infty}(S \times \overline{M}, \R)$ with $H\geq 0$;

\item $H_z=H(z,\cdot)$ does not depend on $z\in S$ except possibly on the ends of $S$;

\item for large $R$, $J_z=J(z,\cdot)$ is of contact type ($dR = J^*_z\theta$) for each $z\in S$;

\item for large $R$, $H_z=h_z(R)$ only depends on $R$, and $h_z'(R)\geq 0$; 

\item on the ends of $S$: for large $|s|$ both $H,J$ can depend on $t$ but not on $s$; for large $R$, $\partial_s h_z'(R)\leq 0$ on the ends of $S$.
\end{enumerate}
\emph{Remarks. The conditions ensure that Lemma \ref{Lemma Maximum principle for Floer solns}.(2) applies. Since we can construct $SH^*$ using Hamiltonians which for $R\gg 0$ are linear of generic slope and time-independent, we can further simplify the above by taking $h_z=h$ and $J_z = J$ to be independent of $z$ for $R \gg 0$. Indeed, for $R\gg 0$ there are no $1$-orbits of $A_aH,B_bH$ (see \ref{Subsection Liouville domains Definition}) and Lemma \ref{Lemma Maximum principle for Floer solns} ensures that Floer solutions don't enter the region $R\gg 0$, so we don't need to perturb there.}

The equation $(du-X\otimes \beta)^{0,1}=0$ corresponds to $\overline{\partial}u=0$ for the operator
$$\textstyle
\overline{\partial} = \frac{1}{2} \{ (d - X\otimes \beta)+ J
\circ (d - X\otimes \beta)\circ j \}.
$$
So the moduli
space $\mathcal{M}(x_a;y_b;S,\beta)$ is the intersection of the section $\overline{\partial}$ and the zero section of a suitable Banach bundle (for details on this bundle see \cite[2.2.4]{Schwarz}):
$$\mathcal{M}(x_a;y_b;S,\beta) = \overline{\partial}\,^{-1}(0).$$
Regularity of $J$ means that $\overline{\partial}$ is transverse
to the zero section and so $\mathcal{M}(x_a;y_b;S,\beta)$ is a
smooth manifold by the implicit function theorem. Regularity of $J$ is
equivalent to the surjectivity of the linearizations $D_u$ of
$\overline{\partial}$ at Floer solutions $u$, in particular the
index of $D_u$ will then be the dimension of the kernel of $D_u$.
This kernel is the tangent space of
$\overline{\partial}\,^{-1}(0)$.

So, when
gradings are defined (\ref{Subsection Maslov index and
Conley-Zehnder index}), by Theorem \ref{Theorem index of
Fredholm operator} we have:
$$
\textstyle \textrm{dim}\, \mathcal{M}(x_a;y_b;S,\beta) = \sum |x_a| - \sum
|y_b| +2n(1-g-p).
$$
%
%
%
\subsection{Existence of regular $J$}
\label{Appendix Genericity of J}\label{Subsection Linearization Du}
%
We claim that
for generic (in the sense of Baire) $J$,
\emph{which depend on $S$}, regularity holds: the linearization
$D_u$ of $\overline{\partial}$ is surjective at Floer solutions $u$.

Let $(S,j,\beta)$ be the data described in \ref{Subsection Model
Riemann Surface} and $J=J_z,H=H_z$ the data described in \ref{Subsection
Smoothness of Moduli Spaces} (we assume that near the ends $H=H_t$ has been perturbed time-dependently to ensure that the asymptotics are nondegenerate).

By Lemma \ref{Lemma Maximum principle for Floer solns}
all Floer solutions lie in a compact $K\subset
\overline{M}$ determined by the asymptotics. We will achieve
transversality by perturbing $J$ on $K$, keeping $J$ of contact
type outside of $K$ (so the maximum principle still holds outside
$K$). 

Let $D_u: W^{1,\mathfrak{p}}(u^*T\overline{M}) \to
L^{\mathfrak{p}}(\textrm{Hom}^{0,1}(TS,u^*T\overline{M}))$ denote the linearization of $\overline{\partial}$ at $u\in \mathcal{M}(x_a;y_b;S,\beta)$ (where $\mathfrak{p}>2$). We claim that after a
generic perturbation of $J$, all $D_u$ are surjective
(equivalently: $\overline{\partial}$ is transverse to the zero section).

Consider $\overline{\partial}$ for all choices of $J$
simultaneously, where $J$ can depend on $z\in S$. Let
$\mathcal{J}$ denote the collection of all $J$, appropriately
topologized as a Banach bundle.

\emph{Technical Remarks. The precise conditions on $J$ are listed in (1), (3), (5) of Section \ref{Subsection Smoothness of Moduli Spaces}. Call $J_0$ a fixed such $J$, and fix a large $s_{\infty}>0$. Define $\mathcal{J}$ to be the set of such $J$ which agree with $J_0$ for $|s| \geq s_{\infty}$ on the ends (this is to ensure that a generic perturbation of $J_0$ within $\mathcal{J}$ will still satisfy those conditions). What regularity we demand from the $J$'s depends on how we want to topologize $\mathcal{J}$. To topologize, one first 
identifies $\mathcal{J}$ with a suitable vector subbundle of $\mathrm{End}(TM)$-valued sections over $S$ \cite[4.2.10]{Schwarz}. To make $\mathcal{J}$ into a Banach bundle, one can use Floer's $C^{\infty}_{\epsilon}$-norm \cite[4.2.6-4.2.9 and 4.2.11]{Schwarz}, alternatively one can, as in \cite[Sec.3.4]{McDuff-Salamon}, define $\mathcal{J}=\mathcal{J}^{\ell}$ using the $C^{\ell}$-norm (for large $\ell\in \N$ so that the Sard-Smale theorem applies for a given index of the Fredholm map) and afterwards one proves that smooth regular $J\in \mathcal{J}^{\ell}$ are generic also in the (non-Banach) $C^{\infty}$-topology. The approach in \cite[Sec.(8h)]{Seidel-book} and \cite[Sec.(3b)]{Abouzaid-Seidel} allows the $J\in \mathcal{J}$ to differ from $J_0$ also on the ends: it only requires the $J$ to converge to $J_0$ faster than exponentially on the ends. In this case, the perturbed $J$ may not satisfy (5) of \ref{Subsection Smoothness of Moduli Spaces}, but that is not problematic by Lemma \ref{Lemma Maximum principle for Floer solns}.(3).
}

Linearizing at a solution $\overline{\partial}_J (u)=0$:
$$
\begin{array}{c}\displaystyle
F:T_J\mathcal{J} \oplus W^{1,\mathfrak{p}}(u^*T\overline{M}) \longrightarrow
L^{\mathfrak{p}}(\textrm{Hom}^{0,1}(TS,u^*T\overline{M})),\\ \displaystyle
F(Y,\xi) = {\textstyle \frac{1}{2}}\, Y_{z,u(z)} \circ (du - X
\otimes \beta)_{z,u(z)} \circ j_z + D_u \xi,
\end{array}
$$
where $T_J\mathcal{J}$ is the tangent space (e.g. see the approach in \cite[Sec.3.4]{McDuff-Salamon}).

\begin{lemma}
If $F$ is surjective, there is a Baire second category subset
$\mathcal{J}_{reg}(S,\beta) \subset \mathcal{J}$ of
\textrm{regular} $J$ $\mathrm{(}$those for which $D_u$ is
surjective for all $u\in \mathcal{M}(x_a;y_b;S,\beta,J))$.
\end{lemma}
\begin{proof}[Sketch proof]
This is a standard trick, e.g. see \cite[3.1.5,
p.51]{McDuff-Salamon2}. One considers the universal moduli space
$\mathcal{M}$ of all Floer solutions $u\in
\mathcal{M}(x_a;y_b;S,\beta,J)$ for all $J\in \mathcal{J}$. Let
$\pi: \mathcal{M} \to \mathcal{J}$ denote the projection. Using the surjectivity of $F$, one
checks that the kernel and cokernel of $d\pi$ are isomorphic to
those of $D_u$. Thus: $d\pi$ is Fredholm since $D_u$ is Fredholm; and $d\pi$ is onto if and only if $D_u$ is onto. Thus the regular values of $\pi$ correspond to the
regular $J$. By the Sard-Smale theorem the regular values of $\pi$
in $\mathcal{J}$ are of Baire second category.
\end{proof}

\begin{lemma}\label{Lemma transversality proof}
If $du-X\otimes \beta \neq 0$ on a non-empty open subset of $S$,
then $F$ is surjective.
\end{lemma}
\begin{proof}
This is a standard argument \cite[4.2.18]{Schwarz}. By Theorem
\ref{Theorem index of Fredholm operator}, $D_u$ is Fredholm so the
image has finite codimension. Since $\textrm{Im}\, D_u \subset
\textrm{Im}\, F$, the same is true for the image of $F$. Thus it
suffices to show that the orthogonal complement to the image
$\textrm{Im}\,F$ is zero. Suppose by contradiction that there is an
$\eta \neq 0$ in the dual space,
$L^{\mathfrak{q}}(\textrm{Hom}^{0,1}(TS,u^*TM))$ for $\tfrac{1}{\mathfrak{q}}+ \tfrac{1}{\mathfrak{p}} = 1$, which is
$L^2$-perpendicular to $\textrm{Im}\, F$.

Then $\eta$ is perpendicular to $\textrm{Im}\, D_u$: $\langle
D_u \xi, \eta\rangle_{L^2}=0$ for all $\xi$. So $\eta$ is a weak
solution of $D_u^*\eta=0$, and by elliptic regularity
\cite[2.5.3]{Schwarz} $\eta$ is smooth. By the Carleman similarity
principle \cite[Cor 2.3]{Floer-Hofer-Salamon} $\eta$ can vanish at
most on a discrete set of points in $S$. Since $\eta$ is
perpendicular to $\textrm{Im}\, F$ and $\langle D_u
\xi,\eta\rangle_{L^2}=0$,
$$\textstyle
0=\, \langle F(Y,\xi),\eta\rangle_{L^2} = \frac{1}{2}\langle
Y(du-X\otimes\beta) \circ j,\eta\rangle_{L^2}  \textrm{ for all
}Y,\xi.
$$

Pick $z_0\in S$ in the open subset where $du-X\otimes \beta \neq
0$ such that $\eta_{z_0}\neq 0$. Then, by \cite[8.1]{Salamon-Zehnder}, one can pick
 a vector $Y$ at $(z_0,u(z_0))\in S\times \overline{M}$ such
that $g(Y(du-X\otimes \beta)j,\eta)\,
> 0$. Extend $Y$ locally and multiply it by a cut-off function
which depends on $z\in S$ to make $Y$ globally defined so that
$g(Y(du-X\otimes \beta)j,\eta) \geq 0$ everywhere and $>0$ near
$z_0$. This contradicts the vanishing of the above $L^2$-inner
product.

\emph{Technical Remark. In the previous Technical Remarks, we imposed that the $J\in \mathcal{J}$ all agree for $|s|\geq s_{\infty}$ on the ends, so we cannot pick $Y\neq 0$ there. So in the above argument, one needs to show that one can find a $z_0$ in the complement of the $|s|\geq s_{\infty}$ regions such that $(du-X\otimes \beta)_{z_0}\neq 0$.
 By contradiction, suppose $du-X\otimes C\, dt=0$ for $s_{\infty}-1< |s|< s_{\infty}$ (there $\beta=C\, dt$ for some $C>0$) and that no such $z_0$ existed in $|s|\geq s_{\infty}$. This equation is equivalent to $\partial_s u = 0$ and $\partial_t u - X_{C\cdot H}=0$, and it is equivalent to $v=0$ where $v(s,t)=\varphi_{C\cdot H}^{1-t}(u(s,t))$. But $v$ is a pseudo-holomorphic strip (compare Remark \ref{Remark Wrapped FH is same as Lagr FH}) so it has isolated zeros unless it is identically zero. So $v=0$ on $s_{\infty}-1< |s|< s_{\infty}$ implies $v=0$ on $|s| \geq s_{\infty}$, so $du-X\otimes \beta=0$ on $|s| \geq s_{\infty}$, contradiction. See \cite[4.2.13, 4.2.15]{Schwarz} for details on this argument.
}
\end{proof}

\begin{lemma}\label{Lemma transversality trivial solutions}
Combining the previous two Lemmas, we deduce that transversality
holds unless $du-X\otimes \beta = 0$ on all of $S$ (since being
non-zero is an open condition). By making a generic time-dependent
perturbation of $H$ at each end of $S$, we can ensure that
$du-X\otimes \beta= 0$ is impossible for all Floer solutions $u$.
\end{lemma}
\begin{proof}
After a generic time-dependent perturbation of $H$ at each end of
$S$, the asymptotics of $u$ can never be constant orbits (since
$\textrm{Crit}(H_t)$ will change in time, generically). Suppose
$du-X\otimes \beta \equiv 0$. Then $du(TS)=X\otimes \beta (TS)
\subset \textrm{span}(X)$, so the image of $u$ lands inside a
Hamiltonian orbit $v$ of $H$. When $S$ is a disc ($p=1$, $q=0$),
$u(S)\subset v(S^1)$ means that $v$ is contractible within
$v(S^1)$ and hence constant, which we ruled out. In the general
case, since we choose the time-dependent perturbations to be
different on each end, we can
ensure that $v$ is not a common orbit of the perturbed
Hamiltonians even up to rescaling.
\end{proof}
\begin{remark} If the
$A_a,B_b$ of $S$ are $\Q$-independent, we can make the same
perturbation of $H$ at each end (in the proof: $v$ has periods
$A_a,B_b$ so it would be constant). 
\end{remark}
%
%
%
%
%
%
\subsection{Compactness of the Moduli Spaces}
\label{Subsection Compactness of Moduli Spaces}
%
%
We claim that the $\mathcal{M}(x_a; y_b;S,\beta)$
have compactifications whose boundaries consist of broken Floer
solutions (Figures \ref{Figure compactness for TQFT},
\ref{Figure Breaking}). Since the complex surface $(S,j)$ is
fixed, away from the ends the energy estimate forces $C^{\infty}$-local
hence uniform convergence to a solution, thus breaking of Floer
solutions can only occur at the cylindrical ends. On those ends,
the equation turns into Floer's equation so the breaking is
analytically identical to the breaking of Floer trajectories.

\emph{Remark. We emphasize that we keep the complex structure $j$ fixed on $S$. One could in fact allow $(S,j)$ to vary in certain families (Seidel \cite[Sec.(8a)]{Seidel}), and one could also allow the surface to degenerate to a node in regions where $\beta=0$ (so the gluing theorem for $J$-holomorphic curves applies), but in this paper we will avoid these additional complications.
}

Details of the construction of compactifications for
these moduli spaces can be found in Schwarz \cite[4.3.21, 4.4.1]{Schwarz}. His work for \emph{closed} symplectic manifolds $M$ applies here because:\\[0.5mm]  1) all $u\in
\mathcal{M}(x_a; y_b;S,\beta)$ lie in the compact subset
$R \!\leq\! \max (R(x_a),R(y_b),R_0)$ of $\overline{M}$, by \ref{Subsection Maximum principle for floer solns}; \\  2) they
satisfy an a priori energy estimate by \ref{Subsection Energy}; and \\   3)
no bubbling off of $J$-holomorphic spheres occurs since
$\omega\!=\!d\theta$ is exact.
\begin{figure}[ht]
\includegraphics[scale=0.6]{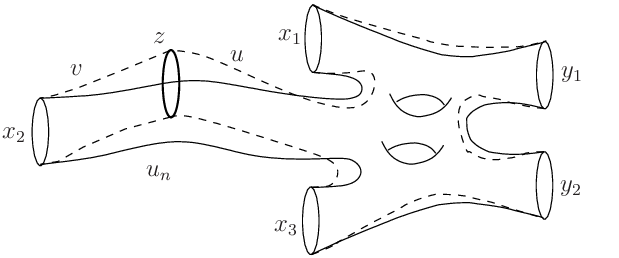}
\caption{Floer solutions $u_n\!\in\!
\mathcal{M}(x_1,x_2,x_3;y_1,y_2;S,\beta)$ converging to a broken
solution $v \# u \!\in\! \mathcal{M}(x_2,z)\! \times\!
\mathcal{M}(x_1,z,x_3;y_1,y_2;S,\beta)$.}\label{Figure compactness
for TQFT}
\end{figure}
\begin{lemma}
\label{Lemma Compactness and Gluing for TQFT} \emph{Breaking:} Suppose $u_n \in \mathcal{M}_1(x_a;y_b;S,\beta)$ has no
$C^{\infty}$-convergent subsequence. Then, passing to a
subsequence, $u_n$ converges $C^{\infty}$-locally to an isolated
Floer solution $u$ defined on the same $(S,\beta)$ but with one of
the asymptotics $w$ among $x_a,y_b$ changed to a new asymptotic $z$ at an end. On that end of $S$, there are
reparametrizations $u_n(\cdot+s_n,\cdot)\!\to\! v$ converging
$C^{\infty}$-locally to an isolated Floer trajectory for $C\cdot H$ as $s_n\!\to\!
\pm \infty$ with asymptotics $w,z$ $($where $\pm$ is the sign of that end of $S$, and $C$ is the weight for that end$)$.

\emph{Gluing theorem:} given such isolated solutions $v,u$, for large
$\lambda \!\in\! \R$ there is a smooth $1$-parameter family
$u_{\lambda} \!\in\! \mathcal{M}_1(x_a;y_b;S,\beta)$ converging to the
broken trajectory $v \# u$ or $u\# v$ in the above sense
$($respectively for breakings at negative or positive ends$)$.

\emph{Index:} the Fredholm index of
solutions is additive with respect to gluing, so the local
dimension of moduli spaces is additive (for Floer
trajectories count $\dim \widehat{\mathcal{M}}$, see \ref{Subsection Floers Equation}).
\end{lemma}
%
\subsection{Parametrized moduli spaces}
\label{Subsection Homotopy invariance}
%
Given two sets of regular data $(S,j_0,\beta_0,J_0)$ and
$(S,j_1,\beta_1,J_1)$ agreeing at the ends, let
$(S,j_{\lambda},\beta_{\lambda},J_{\lambda})$ interpolate between
them (while keeping the same data $j,\beta,J$ near the ends). This
is always possible since the choices of $j,\beta,J$ each form a
contractible set. Again one proves that for a generic choice of
the interpolation $J_{\lambda}$ the relevant operator
$\overline{\partial}$ is transverse to the zero section and thus
the moduli space
$\cup_{\lambda}\mathcal{M}(x_a;y_b;S,\beta_{\lambda},j_{\lambda},J_{\lambda})$
is smooth. The compactness results of section \ref{Subsection
Compactness of Moduli Spaces} also carry over to this setup
\cite[5.2.3]{Schwarz}. Indeed near the ends we are not varying
$j,\beta,J$ so the same breaking of solutions applies, and away
from the ends the energy estimate forces $C^{\infty}$-local hence
uniform convergence.
%
%
%
\subsection{TQFT operations}
\label{Section Operations on SH}
\label{Subsection Operations on SH(H) Definition}
%
%
Define
$$\psi_S: SC^*(B_1 H)\otimes \cdots \otimes
SC^*(B_q H) \to SC^*(A_1 H) \otimes \cdots \otimes SC^*(A_p H)$$
on the generators by counting isolated Floer solutions
$$
\psi_S(y_1 \otimes\cdots \otimes y_q) = \sum_{u\in
\mathcal{M}_0(x_1,\ldots,x_p;y_1,\ldots,y_q;S,\beta)} \epsilon_u
\; x_1\otimes \cdots \otimes  x_p,
$$
where $\epsilon_u \in \{ \pm 1 \}$ are orientation signs (Section
\ref{Appendix Coherent orientations}). Then extend $\psi_S$
linearly. When gradings are defined, degree$(\psi_S)=$
$
\sum|x_a|-\sum|y_b|=-2n(1-g-p)$ by \ref{Subsection Smoothness of
Moduli Spaces}.

\begin{theorem}\label{Theorem Homotopy of surfaces}
The $\psi_S$ are chain maps and, using the fact that we work over a field (see \ref{Subsection
Choice of coefficients}), they yield maps
$\psi_S: \otimes_{b=1}^q SH^*(B_b H)\to \otimes_{a=1}^p SH^*(A_a H)$
which are independent of the choice of the data $(\beta,j,J)$
relative to the ends.
\end{theorem}
\begin{proof}
To prove that $\psi_S$ is a chain map consider the 1-dimensional
$\mathcal{M}_1(x_a;y_b;S,\beta)$. A sequence of Floer solutions
approaching the boundary will break giving rise to an isolated
Floer trajectory at one of the ends (\ref{Subsection Compactness
of Moduli Spaces}, Figure \ref{Figure compactness for TQFT}). This
is precisely the definition of the Floer differential $\partial$
on the tensor products (\ref{Subsection Using orientation signs to
prove that TQFT maps are chain maps}). Since the oriented count of
the boundary components of a compact $1$-dimensional manifold is
zero, and after checking orientation signs in \ref{Subsection
Using orientation signs to prove that TQFT maps are chain maps},
we conclude that
$
\psi_S \circ \partial = \partial \circ \psi_S.
$
Thus $\psi_S$ descends to cohomology: $H^*(\otimes_b SC^*(B_b H)) \to H^*(\otimes_a SC^*(A_a H))$. Since we work over a field, we can apply the K\"{u}nneth theorem to obtain
$\psi_S: \otimes_b SH^*(B_b H) \to \otimes_a SH^*(A_a H).$

To prove independence of the auxiliary data $(\beta,j,J)$ consider
the $1$-dimensional part of $\bigcup_{0\leq \lambda \leq 1}
\mathcal{M}(x_a;y_b;S_{\lambda})$ as in \ref{Subsection Homotopy
invariance}, where
$S_{\lambda}=(S,\beta_{\lambda},j_{\lambda},J_{\lambda})_{0\leq
\lambda \leq 1}$ is a regular homotopy of the data $(\beta,j,J)$
relative ends (in particular, the $A_a,B_b$ are fixed). We show that $\psi_{S_0}$ and $\psi_{S_1}$ are
chain homotopic, so the claim follows. We run the usual invariance
proof of Floer homology \cite[Lemma 3.12]{Salamon}. The boundaries of the
moduli spaces are of two types: either $\lambda \to 0$ or $1$,
which respectively yield contributions to $-\psi_{S_0}$ and
$\psi_{S_1}$, or $\lambda\to \lambda_0 \in (0,1)$. In the latter
case the boundary is a broken solution consisting of a Floer
trajectory and a Floer solution in
$\mathcal{M}_{-1}$, where $\mathcal{M}_{-1}=\mathcal{M}_{-1}(x_a';y_b';S_{\lambda})$ has asymptotics $x_a',y_b'$ equal to the $x_a,y_b$ except for one asymptotics which has changed (because the Floer trajectory broke off). 

Generically
$\mathcal{M}_{-1}$ is empty since it has
virtual dimension $-1$, however at finitely many $\lambda$ there
may actually exist a solution: call these the \emph{unexpected solutions}. 

Let $K$ denote the oriented count
of the unexpected solutions. Then the broken solutions described above contribute
to $K\circ \partial$ or $\partial \circ K$ depending on whether the end where the breaking occurs is a positive or a negative end respectively (and where $\partial$ is acting on $\otimes_b SC^*(B_b H)$ or $\otimes_a SC^*(A_a H)$ respectively, see \ref{Subsection Using orientation signs to prove that TQFT maps
are chain maps}). The oriented count of the boundaries of a
compact $1$-manifold is zero, so $-\psi_{S_0}+ \psi_{S_1}+K\circ
\partial + \partial \circ K=0$. So $K$ is the required chain
homotopy.
\end{proof}
%

%
\subsection{Gluing surfaces yields compositions of
operations}\label{Subsection Gluing surfaces}
%
Two surfaces $S,S'$ as in \ref{Subsection Model Riemann
Surface} can be glued along opposite ends carrying equal weights.
Pick a gluing parameter $\lambda \gg 0$. Suppose a positive end
$[0,\infty)\!\times\! S^1$ of $S$ and a negative end
$(-\infty,0]\!\times\! S^1$ of $S'$ carry the same weight. Cut off
$(\lambda,\infty)\!\times\! S^1$ from $S$ and
$(-\infty,-\lambda)\!\times\! S^1$ from $S'$. Glue what remains of the ends
$
[0,\lambda]\times S^1 \cup [-\lambda,0]\times S^1
$
via $(\lambda,t)\!\sim\!
(-\lambda,t)$.
As the weights agree, for $\lambda \gg 0$ we get a $1$-form
$\beta_{\lambda}$ on the glued surface $S_{\lambda}$ from the
$\beta$'s on the $S,S'$ parts of $S_{\lambda}$. The result is a
$1$-family of surfaces $S_{\lambda}=S\#_{\lambda} S'$. Similarly, one can glue several pairs of ends with matching weights.

\begin{lemma}\label{Lemma Slambda trick}
Let $\mathcal{M}\subset \mathcal{M}_1(x_a;y_b;S,\beta)$ be a
connected component. Near $\partial \mathcal{M}$, $\mathcal{M}$ is
parametrized by Floer solutions $u_{\lambda}$ defined on a gluing $S_{\lambda}=$
$S \#_{\lambda} \mathrm{(}\textrm{infinite
cylinder},\beta=C\,dt\mathrm{)}$. As $\lambda\to \infty$,
$u_{\lambda}$ converges to a broken solution, broken at the gluing
end.
\end{lemma}
\begin{proof}
\begin{figure}[ht]
\includegraphics[scale=0.6]{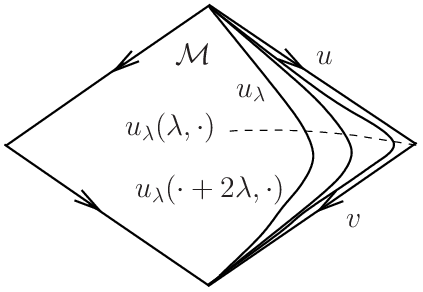}
\caption{Floer solutions $u_{\lambda}$
converging to a broken solution $u\# v$.}\label{Figure Breaking}
\end{figure}

Observe Figure \ref{Figure Breaking}. We will deal with the case
of a breaking $u\# v$ on a positive end as $\lambda\to \infty$,
the other case is analogous. By Lemma \ref{Lemma Compactness and
Gluing for TQFT}, the Floer solutions in $\mathcal{M}$ near $u\#
v$ are parametrized by a gluing parameter $\lambda\gg 0$. Indeed,
by an implicit function theorem argument, they can be described by
a gluing construction which ``glues" $u(s,t)|_{s\leq \lambda}$
with $v(s-2\lambda,t)|_{s\geq \lambda}$. In particular, the glued
solution $u_{\lambda}$ will converge $C^{\infty}$-locally to $u$,
and on the breaking end the rescaled map
$u_{\lambda}(s+2\lambda,t)$ will converge $C^{\infty}$-locally to
$v(s,t)$. The Hamiltonian orbit $v(-\infty,t)$ where the breaking
occurs is the limit of $u_{\lambda}(\lambda,t)$ (dotted line in
Figure \ref{Figure Breaking}). Let
$S_{\lambda}=S\#_{\lambda}(\R\times S^1)$ glued at the end where
the breaking occurs, extending $\beta = C\,dt$ to the cylinder.
``Extend" $u_{\lambda}$ to $S_{\lambda}$ so that it equals
$u_{\lambda}$ on the $S$-part and it equals $u_{\lambda}(s +2
\lambda,t)$ at $(s,t)\in [-\lambda,\infty)\times S^1$. This is
well-defined since we identified $(\lambda,t)\sim (-\lambda,t)$.
\end{proof}
\begin{theorem}\label{Theorem Gluing for compositions}
Let $S_{\lambda}=S\#_{\lambda}S'$ with all positive ends of $S$
glued to all negative ends of $S'$ in the given order. Then on the
chain level $\psi_{S_\lambda}=\psi_{S}\circ \psi_{S'}$ for
$\lambda\gg 0$.
\end{theorem}
\begin{remark}\label{Remark Gluing for compositions}
If we do not glue all ends, the composite $\psi_{S}\circ
\psi_{S'}$ is ill-defined. The remedy is to take disjoint unions
with cylinders $Z=(\R\times S^1,\beta=C\, dt)$ (which
induce the identity map by Lemma \ref{Lemma Chain
Homotopy}(\ref{Item Lemma Chain Homotopy constant H is identity})) so that the resulting surfaces can be fully
glued.
\end{remark}
\begin{proof}
$\psi_S\circ \psi_{S'}$ counts isolated broken Floer solutions
$u\# u'$ which broke at the ends where $S,S'$ get glued. Arguing
as in Lemmas \ref{Lemma Compactness and Gluing for TQFT} and
\ref{Lemma Slambda trick}, there is a $1$-parameter family
$u_{\lambda}$ of Floer solutions on $S_{\lambda}$ parametrizing
the Floer solutions close to $u\#u'$.

Conversely, as $\lambda\to \infty$ a $1$-parameter family of
isolated Floer solutions $u_{\lambda}$ on $S_{\lambda}$ must
converge to a broken solution counted by $\psi_S\circ \psi_{S'}$.
Indeed, the $u_{\lambda}$ have a priori bounded energy depending
on the fixed asymptotics (see \ref{Subsection Energy}), so by Theorem \ref{Theorem Smoothness and
exp convergence} they must break at the stretched ends. Since the
$u_{\lambda}$ are isolated, for dimensional reasons they cannot
break at other ends and they only break once on each stretched
end.

Since the Hamiltonian is linear at infinity, the relevant chain
complexes are finitely generated and the union of the moduli
spaces defining the operations $\psi_{S_\lambda}$, $\psi_S$,
$\psi_{S'}$ is finite. So for large $\lambda$ there is a bijection
between the moduli spaces counted by $\psi_{S_\lambda}$ and those
counted by $\psi_S\circ \psi_{S'}$. We prove in \ref{Subsection
Using orientation signs to prove that TQFT maps compose correctly}
that the solutions are counted with the same orientation signs.
Thus $\psi_{S_{\lambda}}=\psi_S\circ \psi_{S'}$ for large
$\lambda$.
\end{proof}
%
\subsection{TQFT operations on $SH^*(M)$}
\label{Subsection Operations on SH(M)}
%
\begin{theorem}\label{Theorem Operations are compatible with
limit}
There is a commutative diagram,
$$
\xymatrix{ SH^*(B_1 H)\otimes \cdots \otimes SH^*(B_q H)
\ar@{->}_-{\mathrm{continuation}}^-{\varphi}[d] \ar@{->}[r]^-{\psi_S} & SH^*(A_1 H) \otimes
\cdots \otimes SH^*(A_p H)
\ar@{->}_-{\mathrm{continuation}}^-{\varphi'}[d] \\
SH^*(B_1' H')\otimes \cdots \otimes SH^*(B_q' H')
\ar@{->}[r]^-{\psi_{S'}} & SH^*(A_1' H') \otimes \cdots \otimes
SH^*(A_p' H') }
$$
where: the surfaces $S'=S$ equal but $\beta',\beta$ may differ; for $\psi_{S'}$ we use $H'$ instead of $H$; the vertical maps are tensors of monotone continuations (so $B_b m\leq B_b' m'$, $A_a m \leq A_a' m'$ where $m,m'$ are the slopes of $H,H'$ at infinity; and 
 $\sum A_a \geq \sum B_b$, $\sum A_a'\geq \sum B_b'$ by Lemma \ref{Lemma beta exists}).
\end{theorem}
\begin{proof} Suppose first $H'=H$. Then
$\varphi = \psi_{\cup_b Z_b}, \varphi'=\psi_{\cup_a Z_a}$, where $\cup_b Z_b,\cup_a Z_a$ are disjoint
unions of continuation cylinders (Example \ref{Example Continuation cylinder}). 
By Theorem \ref{Theorem Gluing for compositions}, $\varphi'\circ \psi_S = \psi_{(\cup_a Z_a)\# S}$, and $\psi_{S'}\circ \varphi = \psi_{S'\# (\cup_b Z_b)}$ (using large gluing parameters). But the gluings $(\cup_a Z_a)\# S$ and $S'\# (\cup_b Z_b)$ consist of the same topological surface, namely $S$, and although their data $(\beta,j)$ differs, the data agrees on the ends (the weights are $A_a',B_b$). Therefore by Theorem \ref{Theorem Homotopy of surfaces} the two maps are chain homotopic, and hence  $\varphi'\circ \psi_S=\psi_{S'}\circ \varphi$ on cohomology.

Now suppose $H'\neq H$. The difficulty is that a gluing argument as above would end up with completely different Hamiltonians on the glued surface $S$. So the argument is more subtle.

Observe that by rescaling $\beta,A_a,B_b$ by $m$ and $H$ by $1/m$ we can assume $H$ has slope $1$, and similarly for $\beta',A_a',B_b',H'$. So $H,H'$ now have the same slope. Consider the diagram
$$
\xymatrix@R=12pt@C=16pt{ & \otimes_b SH^*(B_b H)
\ar@{->}^{\varphi}[d] \ar@{->}[r]^-{\psi_S} & \otimes_a SH^*(A_a H) 
\ar@{->}^{\varphi'}[d] \ar@{->}[r]^-{\mathrm{cont.}} & \otimes_a SH^*(A_a'H)\\
\otimes_b SH^*(B_bH') \ar@{->}[r]_-{\mathrm{cont.}} \ar@{<-}[ur]^-{\mathrm{cont.}}
& \otimes_b SH^*(B_b' H')
\ar@{->}[r]_-{\psi_{S'}} & \otimes_a SH^*(A_a' H')  \ar@{<-}[ur]_-{\mathrm{cont.}}   }
$$
By Lemma \ref{Lemma Chain Homotopy}(\ref{Item Lemma Chain Homotopy composing conts is cont}) the two outer triangles of continuation maps commute, and the two diagonal continuations are isomorphisms by Lemma \ref{Lemma Chain Homotopy}(\ref{Item Lemma Chain Homotopy same slopes imply iso}) since the slopes agree. By the proof in the case $H=H'$, the composites of the horizontal maps are again TQFT operations $\psi_S,\psi_{S'}$ (after modifying $\beta,\beta'$). Therefore, to prove the Theorem, we may assume that $A_a=A_a',B_b=B_b'$ and that $\varphi,\varphi'$ are isomorphisms. So it suffices to prove $\psi_S = (\varphi')^{-1}\circ \psi_{S'} \circ \varphi$.

By Theorem \ref{Theorem Gluing for compositions}, $(\varphi')^{-1}\circ \psi_{S'} \circ \varphi=\psi_{S''}$ for the gluing $S''=(\cup_a Z_a) \# S' \# (\cup_b Z_b)$ where on the cylinders $Z_a,Z_b$ we use $\beta=A_a\, dt,\beta=B_b\, dt$ and we use monotone Hamiltonians $H_z,H_{-z}$ which depend on the cylinder's coordinate $z$ and which interpolate $H,H'$ and $H',H$ respectively. Note that since $H,H'$ have the same slopes, we can ensure $\partial_s H_z=0,\partial_s H_{-z}=0$ for $R\gg 0$, so by Lemma \ref{Lemma Maximum principle for Floer solns}(\ref{Item Lemma Maximum principle ok if monotone hpy}) the maximum principle holds for $S''$.

Now observe that $S''$ and $S$ involve the same data near the ends (namely weights $A_a,B_b$) and also involve the same Hamiltonian $H$ near the ends. Therefore, a parametrized moduli space argument as in \ref{Subsection Homotopy invariance}, and mimicking the proof of Theorem \ref{Theorem Homotopy of surfaces}, shows that $\psi_{S''}$ and $\psi_S$ are chain homotopic. Thus, on cohomology, $\psi_S = (\varphi')^{-1}\circ \psi_{S'} \circ \varphi$.
\end{proof}

The diagram in Thereom \ref{Theorem Operations are compatible with limit} for $H'\!=\!H$ proves that the direct limit of the maps $\otimes_b SH^*(B_b H) \to \otimes_a SH^*(A_aH)$ as $A_a\!\to \!\infty$ is defined, yielding $\otimes_b SH^*(B_b H) \to SH^*(M)^{\otimes p}$. The resulting diagram, after letting $A_a=A_a'\to \infty$, shows that the direct limit of the maps $\otimes_b SH^*(B_b H) \to SH^*(M)^{\otimes p}$ as $B_b \to \infty$ is defined. This defines the map
$$\psi_S: SH^*(M)^{\otimes q} \to SH^*(M)^{\otimes p} \quad (p\geq 1,q\geq 0).$$ 
Theorem \ref{Theorem Homotopy of surfaces} proves that $\psi_S$ only depends on $p,q$ and the genus of $S$ (for $S$ connected), and not on $\beta,j,J$. Thereom \ref{Theorem Operations are compatible with limit} proves that $\psi_S$ does not depend on the choice of $H$. Theorem \ref{Theorem Gluing for compositions} proves that the $\psi_S$ satisfy the TQFT axioms, see \ref{Subsection TQFT axioms} (the axiom $\psi_Z=\mathrm{id}$ follows by category theory since the $\psi_Z:SH^*(A_aH) \to SH^*(A_a'H)$ define $SH^*(M)=\varinjlim SH^*(A_a H)$).
%
%
\subsection{Invariance of the TQFT structure on $SH^*(M)$}
\begin{theorem}\label{Theorem Invariance of TQFT structures}
For a symplectomorphism $\varphi: \overline{M} \cong \overline{N}$
of contact type (see \ref{Section Invariance under contactomorphisms}),
$$\varphi_*^{\otimes p} \circ \psi_{S,M} = \psi_{S,N} \circ
\varphi_*^{\otimes q},$$
where $\psi_{S,M},\psi_{S,N}$ are the $\psi_S$-operations for $SH^*(M),SH^*(N)$ respectively.

So the isomorphism
$\varphi_*: SH^*(M) \to SH^*(N)$ respects the TQFT.
\end{theorem}
\begin{proof}
Recall from \ref{Section Invariance under
contactomorphisms}, that instead of working with two manifolds $M,N$, we can just work with $M$ and two choices of data $H,J$ and $\varphi^*H,\varphi^*J$. We abbreviate $H_f=\varphi^*H$, $J_f = \varphi^*J$. One needs some caution because $H_f$ is not linear for large $R$: $H_f(R)=h(e^{-f(y)}R)$ for $f:\partial M \to \R$, using collar coordinates $(R,y)\in [1,\infty)\times \partial M$. We write $SH^*_f$ instead of $SH^*$ to remind ourselves that we use Hamiltonians of that form, and that we use almost complex structures for which the contact type condition is $J_f^*d\theta = dR_f$ for $R_f=e^{-f(y)}R$. 

By \cite[Lemma 7]{Ritter2}, after fixing a homotopy $(f_s)_{s\in \R}$ from $f$ to $0$ (independent of $s$ for large $|s|$), there is a constant $K>0$ (depending only on a $C^2$-bound for the homotopy $f_s$) such that if the weights $B',B>0$ satisfy $B'\geq KB$, then one can define the continuation map $\varphi_*:SH^*(BH) \to SH_f^*(B'H_f)$. The proof in fact shows that there is a homotopy $(H_s,J_s)$ from $(B'H_f,J_f)$ to $(BH,J)$ for which the maximum principle holds for continuation solutions. Similarly, for a homotopy $f_s$ from $0$ to $f$, and $A \geq KA'$, we get $\psi_*: SH_f^*(A'H_f) \to SH^*(AH)$.

We will call \emph{special continuation cylinders} the domains and the auxiliary data used to construct these two continuation maps. 

By \cite[Theorem 8]{Ritter2}, a direct limit over these special continuations defines mutually inverse isomorphisms $SH^*(M) \to SH_f^*(M)$ and $SH_f^*(M) \to SH^*(M)$ (and recall that $SH_f^*(M)$ is identifiable with $SH^*(N)$ by pullback via the symplectomorphism $\varphi$).
 
We now want to construct a commutative diagram of the form
$$
\xymatrix@R=16pt@C=16pt{ \otimes_b SH^*(B_b H)
\ar@{->}_-{\varphi_*^{\otimes q}}[d] \ar@{->}[r]^-{\psi_S} & \otimes_a SH^*(A_a H)
\ar@{<-}^-{\psi_*^{\otimes p}}[d] \\
\otimes_b SH_f^*(B_b' H_f)
\ar@{->}[r]^-{\psi_{S_f}} & \otimes_a SH_f^*(A_a' H_f) }
$$
In the diagram, the vertical maps are special continuations which are defined provided that $A_a \geq K A_a'$ and $B_b' \geq KB_a$, and the horizontal maps are defined provided that $\sum A_a \geq \sum B_b$ and $\sum A_a' \geq \sum B_b'$ (Lemma \ref{Lemma beta exists}).
The surfaces $S,S_f$ are the same (but the $\beta$ forms may be different), moreover $\psi_{S_f}$ is constructed using $H_f,J_f$ instead of $H,J$ (note that $f$ does not depend on $z\in S$, so Lemma \ref{Lemma Maximum principle for Floer solns} applies to $R_f,H_f,J_f$ in place of $R,H,J$). 

As discussed above, in the direct limit, the vertical maps become the isomorphisms between the $SH^*(M)$'s and the $SH^*(N)$'s, and by Theorem \ref{Theorem Operations are compatible with limit} the horizontal maps give rise to the $\psi_{S,M}$ and $\psi_{S,N}$ maps. So the claim follows if we can prove that the diagram commutes. 

The composite $\psi_*^{\otimes p} \circ \psi_{S_f} \circ \varphi_*^{\otimes q}$ corresponds to gluing special continuation cylinders onto the ends of the surface $S_f$. The gluing determines auxiliary data $(S,\beta_f',H_f',J_f')$.
We now mimic the proof of the ``independence of the auxiliary data'' in Theorem \ref{Theorem Homotopy of surfaces}. So we construct a $1$-parameter family of auxiliary data $S_{\lambda}=(S, \beta_{\lambda},H_{\lambda}, J_{\lambda})$ for $0\leq \lambda \leq 1$, interpolating between $(S,\beta,H,J)$ and $(S,\beta_f',H_f',J_f')$. We choose this interpolation so that on the collar of $\overline{M}$ and away from the ends of $S$ we have: $H_{\lambda}(R)=h(R_{\lambda f})$ and $J_{\lambda}^*d\theta = dR_{\lambda f}$, where $R_{\lambda f} =  e^{-\lambda f(y)}R$. Observe that only near the ends of $S$ the function $f_s$ depends on the coordinate $z=s+it$ of $S$, namely where the special continuation cylinders got glued. So on the ends of $S$, we ensure those same formulas for $H_{\lambda},J_{\lambda}$ hold on the collar of $\overline{M}$ with $f$ replaced by $f_s$.

\emph{Technical Remark. This last observation is important because it means we do not need to prove a new maximum principle for Floer solutions for $S_{\lambda}$ (for fixed $\lambda$): away from the ends, the maximum principle holds since $f$ does not depend on $z\in S$, and on the ends the maximum principle holds because it holds for special continuation cylinders defined using $\lambda f_s$.}

The $1$-parameter family argument in the proof of Theorem \ref{Theorem Homotopy of surfaces} then defines a chain homotopy $K$ such that $-\psi_S + \psi_*^{\otimes p} \circ \psi_{S_f} \circ \varphi_*^{\otimes q} + K \circ \partial + \partial \circ K = 0$ at the chain level, and so $\psi_S = \psi_*^{\otimes p} \circ \psi_{S_f} \circ \varphi_*^{\otimes q}$ on cohomology. Thus the above diagram commutes, as required.
\end{proof}
%
%
\section{Appendix 2: Coherent orientations}
\label{Appendix Coherent orientations}
\subsection{Coherent orientations for Floer trajectories}
A \emph{coherent orientation} for the moduli spaces
$\widehat{\mathcal{M}}(x,y)$ of \ref{Subsection Floers Equation} is a
continuous map associating to $u\in \widehat{\mathcal{M}}(x,y)$ an
orientation $\sigma(u)$ of the determinant line bundle of the operator $D_u$ (the linearization of Floer's equation),
\\[0.1mm]
\begin{tabular*}{\textwidth}{l@{\extracolsep{\fill}}cr@{\extracolsep{0pt}}} 
\strut & 
$
\textrm{Det}\, D_u \! = \! \Lambda^{\textrm{max}}
\ker D_u \! = \! \Lambda^{\textrm{max}} T_u \widehat{\mathcal{M}}(x,y),
$
 & \strut 
\end{tabular*}
\\[0.4mm]
such that the orientations glue correctly: there is an associative
gluing operation $\#$ such that $\sigma(u)\# \sigma(v)=\sigma(u\#_{\lambda}
v)$ for $\lambda\gg 0$ (here $u\#_{\lambda} v$ is any family of Floer trajectories depending on a gluing parameter $\lambda\in (\lambda_0,\infty)$ converging to the broken trajectory $u\# v$ as $\lambda \to \infty$). Coherent orientations were first
constructed by Floer-Hofer
\cite{Floer-Hofer-Coherent-Orientations}. We recall the construction below. We emphasize that coherent orientations are not unique, since the construction involves choices.
%
\subsection{Fredholm operators on trivial bundles over a cylinder}
\label{Subsection Fredholm operators on trivial bundles over a
cylinder}
%
One first constructs coherent orientations for families of
Fredholm operators of the form
\\[0.5mm]
\begin{tabular*}{\textwidth}{l@{\extracolsep{\fill}}cr@{\extracolsep{0pt}}} 
\strut & 
$
L: W^{1,\mathfrak{p}}({\mathrm{triv}}) \to L^{\mathfrak{p}}({\mathrm{triv}}),\quad Lu = \partial_s u +
J_{s,t}\partial_t u + A_{s,t}u
$
 & \strut 
\end{tabular*}
\\[0.5mm]
on the trivial vector bundle ${\mathrm{triv}}=Z\times \C^n \to Z$ over the cylinder
$Z=\overline{\R} \times S^1$ (compactified appropriately by adding
two circles $S^1_{\pm}$ at infinity), such that $A_{s,t} \in
\textrm{End}({\mathrm{triv}})$ converges to self-adjoint isomorphisms $A_{\pm
\infty,t}$ and $L \to L^{\pm}=J_{\pm \infty,t}\partial_t +
A_{\pm\infty,t}$ as $s\to \pm \infty$, and $\mathfrak{p}>2$.

Once the \emph{asymptotic operators} $L^{\pm}$ at the ends are fixed, the
family $\mathcal{O}_{\mathrm{triv}}(L^-,L^+)$ of such operators is a
contractible (indeed convex) set \cite[Prop
7]{Floer-Hofer-Coherent-Orientations}. Therefore the real line
bundle
$\textrm{Det}(\mathcal{O}_{\mathrm{triv}}(L^-,L^+)) \to \mathcal{O}_{\mathrm{triv}}(L^-,L^+)$
of all determinants
\\[0.5mm]
\begin{tabular*}{\textwidth}{l@{\extracolsep{\fill}}cr@{\extracolsep{0pt}}} 
\strut & 
$
\textrm{Det}(L)=(\Lambda^{\max}\textrm{coker}\, L)^{\vee} \otimes
\Lambda^{\max}\ker L \quad\; (\textrm{over }\R)
$
 & \strut 
\end{tabular*}
\\[0.5mm]
is trivial, and an orientation $\sigma(L)$ is a choice of
trivialization.
%
\subsection{Gluing the operators and coherent orientations}
\label{Subsection Gluing the operators and coherent orientations}
%
Operators $L,K$ which become constant respectively for
$s\!\gg\! 0,s\!\ll\! 0$, with $L^+\!=\!K^-$, can be glued
$L\# K$ to yield a natural gluing isomorphism
$\textrm{Det}\, L\# K \cong \textrm{Det}\, L \times \textrm{Det}\,
K$ \cite[Prop.9]{Floer-Hofer-Coherent-Orientations}, which we now explain.

\emph{Explanation. Suppose $L,K$ are surjective (in general, one uses a stabilization trick to reduce to this situation \cite[above Prop.9]{Floer-Hofer-Coherent-Orientations}). Then it is enough to construct a natural linear isomorphism $\psi: \ker L \oplus \ker K \to \ker (L\# K)$ and to show that $L\# K$ is surjective, since then the top exterior power of $\psi$ gives the required isomorphism of the determinants (notice the cokernels vanish). The gluing $L \# K$ actually depends on a parameter $\lambda>0$: $(L \# K)(s) = L(s + 2\lambda)$ for $s\leq -\lambda$ and $(L \# K)(s) = K(s - 2\lambda)$ for $s\geq \lambda$, so for large $\lambda$ this patches correctly with $(L\# K)(s) = L^+ = K^-$ for $s\in [-\lambda,\lambda]$. 
\\ \indent For large $\lambda$, the Fredholm index of $L \# K$ is the same as the sum of the Fredholm indices of $L,K$: this is because Fredholm indices can be expressed as a difference of Conley-Zehnder indices which only depend on the asymptotic operators \cite[Prop.6]{Floer-Hofer-Coherent-Orientations} (this is due to Salamon-Zehnder \cite[Theorem 4.1]{Salamon-Zehnder} -- see also the \emph{Explanation} in \ref{Subsection Fredholm operators on bundles E to M}). So if we can construct $\psi$ so that $L\# K$ is injective on the orthogonal complement of $\mathrm{Im}\, \psi$ then, because of the Fredholm indices, the operator $L\# K$ has zero cokernel and $\psi$ is bijective. \\ \indent Finally, we explain the construction of $\psi$.
From $(u,v)\in \ker L \oplus \ker K$ one constructs the shifted sum $(u_{2\lambda}\# v_{-2\lambda})(s) = u(s+2\lambda) + v(s-2\lambda)$. Then $\psi(u,v)\in \ker L\# K$ is defined by orthogonally projecting $u_{2\lambda}\# v_{-2\lambda}$ onto $\ker L\# K$. The proof that $L\# K$ is injective on $(\mathrm{Im}\, \psi)^{\perp}$ for large $\lambda$ is an argument by contradiction that requires some analysis \cite[Prop.9]{Floer-Hofer-Coherent-Orientations}.}

Define $\sigma(L\# K)$ by gluing
two given orientations $\sigma(L)\# \sigma(K)$ via the gluing map
\\[0.5mm]
\begin{tabular*}{\textwidth}{l@{\extracolsep{\fill}}cr@{\extracolsep{0pt}}} 
\strut & 
$
\mathcal{O}_{\mathrm{triv}}(L^-,L^+) \times \mathcal{O}_{\mathrm{triv}}(L^+,K^+) \to
\mathcal{O}_{\mathrm{triv}}(L^-,K^+).
$
 & \strut 
\end{tabular*}
\\[0.5mm]
Now construct a coherent orientation as follows. Fix
an asymptotic operator $L_0$ of the type that arises at $s\!=\!-\infty$.
Choose the orientation $1^{\vee}\otimes 1\in \R^{\vee}\otimes \R$
for the isomorphism $L\!=\!\partial_s + L_0$. So we get an
orientation on $\mathcal{O}_{\mathrm{triv}}(L_0,L_0)$. Pick any orientation of
$\mathcal{O}_{\mathrm{triv}}(L_0,L^+)$ for all $L^+\!\neq\! L_0$. Coherence requires
the orientations on
$\mathcal{O}_{\mathrm{triv}}(L_0,L^+),\mathcal{O}_{\mathrm{triv}}(L^+,L_0)$ to determine that on $\mathcal{O}_{\mathrm{triv}}(L_0,L_0)$ via gluing, so we deduce an
orientation for $\mathcal{O}_{\mathrm{triv}}(L^+,L_0)$. Similarly orientations
on $\mathcal{O}_{\mathrm{triv}}(L_0,L^-)$, $\mathcal{O}_{\mathrm{triv}}(L^-,L^+)$,
$\mathcal{O}_{\mathrm{triv}}(L^+,L_0)$ determine that on
$\mathcal{O}_{\mathrm{triv}}(L_0,L_0)$, so we deduce an orientation for
$\mathcal{O}_{\mathrm{triv}}(L^-,L^+)$. The construction is well-defined since
gluing $1^{\vee}\otimes 1$ with itself yields $1^{\vee}\otimes 1$.
%
\subsection{Fredholm operators for a bundle $E\to \overline{M}$}
\label{Subsection Fredholm operators on bundles E to M}
%
Let $E\to \overline{M}$ be any symplectic vector bundle. Consider all operators $L$ on $u^*E$ for smooth maps $u:Z\to \overline{M}$ such that
$L$ is an operator of the type above in some (and hence any) symplectic trivialization of $u^*E$. The bundles $u^*E$ have the
homotopy type of $S^1$ so $u^*E$ is trivial since
$\pi_0(Sp(2n))\!=\!0$, but two choices of trivialization may be
non-homotopic as $\pi_1(Sp(2n))\!=\!\Z$. To overcome this issue, one classifies the operators into equivalence classes $[L,u]$ determined by the following data: the asymptotics $x^{\pm}(t)=u(\pm \infty,t)$, the asymptotic $L^{\pm}$, and the Fredholm index $\mathrm{Ind}(L)$ (when deciding whether two pairs belong to the same equivalence class, we compare the $L^{\pm}$ in a common symplectic trivialization of $(x^{\pm})^*E$). 
We now explain how a choice of orientation for $L$ naturally 
determines an orientation for all operators in the class $[L,u]$ \cite[Lemmas 13/15]{Floer-Hofer-Coherent-Orientations}.

\emph{Explanation. Given equivalent $(L^u,u), (L^v,v)$ (so they share the same $L^{\pm}$ over the common $x^{\pm}$ and they have the same Fredholm index), we can pick trivializations $\phi^u\co u^*E \cong Z\times \C^n$ and $\phi^v\co v^*E\cong Z\times \C^n$ which agree over $x^{-}$. In the trivializations, the $L^u,L^v$ determine the data $J_{s,t}^u,A_{s,t}^u$ and $J_{s,t}^v,A_{s,t}^v$ in $\mathrm{End}(\C^n)$, and this data determines symplectic paths $\gamma^u_{\pm\infty},\gamma^v_{\pm\infty}$ in $\mathrm{Sp}(\C^n)$ by integrating the ODE system $\gamma_{\pm}'(t)=J_{\pm\infty,t}A_{\pm\infty,t}\gamma_{\pm}(t)$, $\gamma_{\pm}(0)=\mathrm{Id}$. The \emph{Conley-Zehnder index} \cite{Salamon-Zehnder} is an integer-valued map defined on paths in $\mathrm{Sp}(\C^n)$ joining the identity to an element which does not have eigenvalue $1$. We recall three key properties: %
\\ \indent
(1) two such paths are homotopic if and only if the indices equal; %
\\ \indent
(2) under the multiplication action of $\psi\in \pi_1(\mathrm{Sp}(\C^n),\mathrm{Id})$ on such paths the index changes by subtracting twice the Maslov index of $\psi$ (using our conventions, see Remark \ref{Remark CZ convention}); %
\\ \indent
 (3) the Fredholm index of an operator $L$ as in \ref{Subsection Fredholm operators on trivial bundles over a cylinder} is the difference of the Conley-Zehnder indices of the asymptotics: $|\gamma^u_{-\infty}|-|\gamma^u_{+\infty}|$.
\\ \indent
In our situation: $\gamma^u_{-\infty}\equiv \gamma^v_{-\infty}$ and the Fredholm indices of $L^u,L^v$ equal. So the Conley-Zehnder indices of the paths $\gamma^u_{\infty},\gamma^v_{\infty}$ equal, so $\gamma^u_{\infty},\gamma^v_{\infty}$ are homotopic, so $\phi_{\infty,t}^u(\phi_{\infty,t}^v)^{-1}$ has zero Maslov index, so $\phi_{\infty,t}^u,\phi_{\infty,t}^v$ are homotopic. So by homotopying $\phi^v$, we can assume $\phi^u,\phi^v$ agree over both ends $x^{\pm}$, and since the asymptotics $L^{\pm}$ agree, also the asymptotics $L^{\pm}_{\C^n}$ in the common trivializations over $x^{\pm}$ agree. Since $\mathcal{O}_{\mathrm{triv}}(L^-_{\C^n},L^+_{\C^n})$ is contractible, a choice of orientation for $L^u$ induces via $\phi^u,\phi^v$ an orientation for $L^v$ (and a similar homotopy argument shows that the induced orientation on $L^v$ does not depend on the choices of $\phi^u,\phi^v$).
}

Now we construct coherent orientations $\sigma([L,u])$ for all classes
$[L,u]$. For each homotopy class $[S^1,M]$, pick a representative loop $x_0$ and an asymptotic operator $L_0$. Define $\sigma(L_0,x_0)=1^{\vee}\otimes 1$ (viewing $x_0$ as an $s$-independent cylinder). Consider all $(L,u)$ such that both ends of $u$ are $x_0$ and $L^{\pm}=L_0$. Then the index $\mathrm{Ind}(L)\in 2c_1(TM)(\pi_2(M))\subset \Z$ by the Riemann-Roch theorem \cite[Appendix C]{McDuff-Salamon2}. If there is such an $(L,u)$ with non-zero index, then all such classes $[L,u]$ can be related by gluings using a class $[L_{\mathrm{min}},u_{\mathrm{min}}]$ with minimal positive index, so it is enough to pick an orientation for $[L_{\mathrm{min}},u_{\mathrm{min}}]$ to determine orientations for all such $[L,u]$. 
The remaining part of the construction of coherent orientations now proceeds as in \ref{Subsection Gluing the operators and coherent orientations}. 
%
\subsection{Definition of orientation signs for $SH^*(M)$}
\label{Subsection Definition of orientation signs}
%
Apply \ref{Subsection Fredholm operators on bundles E to M} to
$T\overline{M} \to \overline{M}$ to get a coherent orientation
$\sigma$. Then $\sigma$ is defined on all linearizations $D_u$ of
Floer's equation along a Floer trajectory $u$. An isolated Floer
trajectory $u$ has a natural orientation $\partial_s u$ determined
by its flow. Define the $\epsilon_u\!\in\!\{\pm 1\}$ of
\ref{Subsection Symplectic chain complex} by
\\[0.5mm]
\begin{tabular*}{\textwidth}{l@{\extracolsep{\fill}}cr@{\extracolsep{0pt}}} 
\strut & 
$
\sigma(u)\!=\!\epsilon_u \cdot \partial_s u.
$
 & \strut 
\end{tabular*}
\\[0.5mm]
For an isolated Floer continuation solution $v$ define the
$\epsilon_v\!\in\!\{\pm 1\}$ of \ref{Subsection Continuation Maps} by
\\[1mm]
\begin{tabular*}{\textwidth}{l@{\extracolsep{\fill}}cr@{\extracolsep{0pt}}} 
\strut & 
$
\sigma(v) = \epsilon_v \cdot (1^{\vee}\otimes 1).
$
 & \strut 
\end{tabular*}
%
%
%
%
\subsection{Using orientation signs to prove $\mathbf{d\circ d=0}$}
\label{Subsection Using orientation signs to prove dd=0}
%
Let $u\# u' \in \mathcal{M}_0(x,y) \# \mathcal{M}_0(y,z)$ be a
broken Floer trajectory lying at the boundary of a connected
component $\mathcal{M}\!\subset\! \mathcal{M}_1(x,z)$. Floer's gluing
map parametrizes a neighbourhood of this boundary by a gluing
parameter $\lambda \!\gg\! 0$ as follows. It interpolates the two maps
(\emph{gluing data})
\\[0.5mm]
\begin{tabular*}{\textwidth}{l@{\extracolsep{\fill}}cr@{\extracolsep{0pt}}} 
\strut & 
$
\textstyle 
(u_{\lambda}=u(s+ \lambda, t))|_{s\leq -1} \;\textrm{ and }\;
(u'_{-\lambda}=u'(s- \lambda, t))|_{s\geq 1}
$
 & \strut 
\end{tabular*}
\\[0.5mm]
to obtain an approximate solution to Floer's equation, and via the implicit function theorem the approximate solution uniquely determines a genuine Floer solution $u\#_{\lambda} u'\in \mathcal{M}$. As $\lambda \to \infty$, the $1$-family $u\#_{\lambda} u'$ converges to the broken trajectory $u\# u'$.

In general, given $1$-dimensional vector spaces $\R e,\R f$ oriented by the vectors $e,f$, we canonically orient $\R e \times \R f$ by $(e,0)\wedge (0,f)$, or equivalently by $(e,-f)\wedge (e,f)$. So, using the notation
$\widehat{\mathcal{M}}$ of \ref{Subsection Floers Equation}, 
$\widehat{\mathcal{M}}(x,y)\!\times\!\widehat{\mathcal{M}}(y,z)$ is canonically oriented at $(u,u')$ by the pair of basis vectors
\\[0.5mm]
\begin{tabular*}{\textwidth}{l@{\extracolsep{\fill}}cr@{\extracolsep{0pt}}} 
\strut & 
$
\textstyle ((\partial_s u, -\partial_s u'), \;(\partial_s u, \partial_s u')).
$
 & \strut 
\end{tabular*}
\\[0.5mm]
This agrees with the natural orientation of the domain $(\lambda_0,\infty)\times \R$ of $(\lambda,s)$ which parametrizes the gluing data $(u_{\lambda},u'_{-\lambda})$, since differentiating the data in
$\lambda$ and $s$ yields the two vectors $(\partial_s
u_{\lambda},-\partial_s u'_{-\lambda})$, $(\partial_s
u_{\lambda},\partial_s u'_{-\lambda})$, which agrees with the orientation defined by the above pair. 

Therefore 
$(\partial_{\lambda} (u\#_{\lambda} u'), \partial_s (u\#_{\lambda}
u'))$ is the orientation of
$T\widehat{\mathcal{M}}(x,z)_{u\#_{\lambda} u'}$ induced by the gluing map. Quotienting by
the $s$-translations, $\partial_{\lambda}
(u\#_{\lambda} u')$ is the orientation on $\mathcal{M}$
($\cong$ interval) induced by the gluing. Now $u\#_{\lambda} u'$
approaches the boundary as $\lambda$ increases, so the gluing map induces the orientation which points outward along the boundary of the $1$-manifold $\overline{\mathcal{M}}$.

Therefore, if $u \# u'$, $\widetilde{u}\#\widetilde{u}'$ are the two boundaries of $\overline{\mathcal{M}}$, then the gluing map has assigned opposite orientations to them, and we abbreviate this fact by $\partial_s u \# \partial_s u'=- \partial_s \widetilde{u}\#\partial_s\widetilde{u}'$. 

Now consider the $\sigma$-orientations. The $u \# u'$, $\widetilde{u}\#\widetilde{u}'$ belong to the same component $\mathcal{M}$, so they carry the same asymptotic operators and they have the same Fredholm index, so $\sigma(u\# u')=\sigma(\widetilde{u}\#\widetilde{u}')$. By
coherence, $\sigma(u\# u') = \sigma(u)\# \sigma(u')$ and $\sigma(\widetilde{u}\#\widetilde{u}')=\sigma(\widetilde{u})\#\sigma(\widetilde{u}')$. 

By definition of the $\epsilon$-signs, $\sigma(u)\# \sigma(u') = \epsilon_u\epsilon_{u'} \cdot \partial_s u \#  \partial_s u'$ and $\sigma(\widetilde{u})\# \sigma(\widetilde{u}') = \epsilon_{\widetilde{u}}\epsilon_{\widetilde{u}'} \cdot \partial_s \widetilde{u} \#  \partial_s \widetilde{u}'$. But we showed above that $\partial_s u \# \partial_s u'=- \partial_s \widetilde{u}\#\partial_s\widetilde{u}'$. Therefore $\epsilon_u \epsilon_{u'} +
\epsilon_{\widetilde{u}}\epsilon_{\widetilde{u}'}=0$, as required in the proof of $d\circ d=0$ in \ref{Subsection Symplectic chain complex}.
%
\subsection{Using orientation signs to prove
$\mathbf{d\circ \varphi = \varphi \circ d}$}\label{Subsection Orientation signs Floer chain maps}
%
Here $\varphi$ is a continuation map (see \ref{Subsection
Continuation Maps}) and $d\circ \varphi = \varphi \circ d$ says it is a chain map. This equation arises from an oriented count of the breakings of the $1$-dimensional moduli spaces $\mathcal{M}_1^{H_s}(x,z)$ of Floer continuations. We will now prove that this equation indeed holds for the choices of signs $\epsilon_u,\epsilon_v$ defined in \ref{Subsection Definition of orientation signs}.

 Let $u\# v$ and $v'\# u'$ be broken
continuation solutions, where $v,v'$ are Floer continuation
solutions, $u,u'$ are Floer trajectories. The orientations of $u\# v,v'\# u'$ are
defined by the orientations $\partial_s u$, $\partial_s u'$
respectively. However, the gluing data is now
\\[0.5mm]
\begin{tabular*}{\textwidth}{l@{\extracolsep{\fill}}cr@{\extracolsep{0pt}}} 
\strut & 
$
(u_{2\lambda}|_{s\leq -\lambda-1},\; v|_{s\geq -\lambda+1})
\quad \textrm{ and } \quad (v'|_{s\leq \lambda-1},\;
u'_{-2\lambda}|_{s\geq \lambda+1}).
$
 & \strut 
\end{tabular*}
\\[0.5mm]
Differentiating in $\lambda$: $(2\partial_s
u_{2\lambda},0)$, $(0,-2\partial_s u'_{-2\lambda})$. So for $\lambda\!\gg\!0$ the glued orientations are
$\partial_{\lambda}(u\#_{\lambda} v)$ (outward-pointing near the
boundary) and $-\partial_{\lambda}(v'\#_{\lambda} u')$
(inward-pointing).

Let $\mathcal{M}\subset \mathcal{M}_1^{H_s}(x,z)$ be a $1$-dimensional component whose boundaries are the broken solutions $u\# v$ and $\widetilde{u}\# \widetilde{v}$. According to the gluing above, they are both oriented in the outward-direction, so they have opposite orientations, so we abusively write $\partial_s u = - \partial_s \widetilde{u}$. 

The $\sigma$-orientations are coherent and constant on $\mathcal{M}$ so $\sigma(u)\# \sigma(v)=\sigma(u\# v) = \sigma(\widetilde{u}\# \widetilde{v}) = \sigma(\widetilde{u})\#\sigma(\widetilde{v})$. By definition, $\sigma(u)\#\sigma(v) = \epsilon_u \epsilon_v \partial_s u$, $\sigma(\widetilde{u})\#\sigma(\widetilde{v})= \epsilon_{\widetilde{u}} \epsilon_{\widetilde{v}} \partial_s \widetilde{u}$. Using $\partial_s u = - \partial_s \widetilde{u}$, we get $\epsilon_u \epsilon_v =- \epsilon_{\widetilde{u}}\epsilon_{\widetilde{v}}$. So these two breakings contribute cancelling contributions to $d \circ \varphi$. 

Similarly, if the boundaries of $\mathcal{M}$ are $v\# u$ and $\widetilde{v}\# \widetilde{u}$, we get cancelling contributions to $\varphi \circ d$. 
Finally, if the boundaries of $\mathcal{M}$ are of different type, say  $u\# v$ and $\widetilde{v}\# \widetilde{u}$, then the gluing map assigns equal orientations to them, so proceeding as above we get $\epsilon_u \epsilon_v = \epsilon_{\widetilde{v}}\epsilon_{\widetilde{u}}$. So these breakings give equal contributions to $d \circ \varphi$ and $\varphi \circ d$. This completes the proof of $d \circ \varphi =\varphi \circ d$.
%
%
\subsection{Coherent orientations for Floer solutions}
%
For TQFT operations we need coherent orientations using smooth
maps $u:S \to \overline{M}$ on a punctured Riemann surface $S$
with prescribed parametrizations on the cylindrical ends
(appropriately compactified with \emph{asymptotic circles}
$S^1_a,S^1_b$ at the ends). We mimic the construction of coherent orientations for symplectic field theory due to Bourgeois-Mohnke \cite{Bourgeois-Mohnke} which builds upon Eliashberg-Givental-Hofer \cite[1.8]{Eliashberg-Givental-Hofer}.
%
\subsection{Fredholm operators over punctured surfaces}
\label{Subsection Fredholm operators over punctured surfaces}
Let $E \to S$ be a complex vector bundle with prescribed
trivializations on the asymptotic circles. Consider operators
\\[0.5mm]
\begin{tabular*}{\textwidth}{l@{\extracolsep{\fill}}cr@{\extracolsep{0pt}}} 
\strut & 
$
L\co W^{1,\mathfrak{p}}(E)\!\to\! L^{\mathfrak{p}}(\textrm{Hom}^{0,1}(TS,E)), \quad Lu\cdot Z=\nabla_Z u + J\nabla_{jZ} u + A_z(u)\cdot Z
$
 & \strut 
\end{tabular*}
\\[0.5mm]
which restrict on the
cylindrical ends to an operator of the type in \ref{Subsection
Fredholm operators on trivial bundles over a cylinder}.
Here $Z\in TS$; $A_z \in
\textrm{Hom}(E,\textrm{Hom}^{0,1}(TS,E))$ depending on $z\in S$; and $J\in \mathrm{End}(E)$ is the complex structure of the vector bundle $E$. Denote the space of such operators by
\\[0.5mm]
\begin{tabular*}{\textwidth}{l@{\extracolsep{\fill}}cr@{\extracolsep{0pt}}} 
\strut & 
$
\mathcal{O}_{E}(L_a;L_b)=\mathcal{O}_{E}(L_1,\ldots,L_p;L_1,\ldots,L_q)
$
 & \strut 
\end{tabular*}
\\[0.5mm]
where $L_a$, $L_b$ are the asymptotic operators over the asymptotic circles respectively at
the negative and positive ends. Just as in \ref{Subsection
Fredholm operators on trivial bundles over a cylinder},
$\mathcal{O}_E(L_a;L_b)$ is a contractible space so the
determinant bundle over $\mathcal{O}_E(L_a;L_b)$ is trivial and an
orientation is a choice of trivialization.
%
\subsection{Gluings and disjoint unions}
\label{Subsection Gluings and disjoint unions}
Just as in \ref{Subsection Gluing the operators and coherent
orientations} there is a gluing operation: given $E\to S$, $E'\to
S'$ with matching trivializations over the respective punctures
$y_c$, $x_c'$ where we glue, we obtain a glued bundle $E'' \to
S''$ and a gluing map
\\[0.5mm]
\begin{tabular*}{\textwidth}{l@{\extracolsep{\fill}}cr@{\extracolsep{0pt}}} 
\strut & 
$
\#: \mathcal{O}_E(L_a;L_b,L_{-c}) \times
\mathcal{O}_{E'}(L_{c},L_{a'};L_{b'}) \to
\mathcal{O}_{E''}(L_a,L_{a'};L_b,L_{b'}),
$
 & \strut 
\end{tabular*}
\\[0.5mm]
where we always use the convention that $L_{-c}$ is an
abbreviation for the reversed ordering $(\ldots,L_{2},L_{1})$ of
the operators $(L_1,L_2,\ldots)$, so inductively in $c$ we are
gluing on the $c$-end the pair of asymptotics $L_{-c}=L_c$. This
ensures the associativity of the gluing operation $\#$ defined on
the orientations $\sigma(L)$ of $\textrm{Det}\, L$ \cite[Cor
7]{Bourgeois-Mohnke}.

The disjoint union of bundles determines a natural isomorphism
\\[0.5mm]
\begin{tabular*}{\textwidth}{l@{\extracolsep{\fill}}cr@{\extracolsep{0pt}}} 
\strut & 
$
\textrm{Det}\, L\otimes \textrm{Det}\, L' \to \textrm{Det}(L
\cup L').
$
 & \strut 
\end{tabular*}
\begin{lemma}\label{Lemma Comparing natural orientations}
 The two natural isomorphisms $\mathrm{Det}\, L\otimes \mathrm{Det}\, L' \to \mathrm{Det}(L
\cup L')$ and $\mathrm{Det}\, L'\otimes \mathrm{Det}\, L \to \mathrm{Det}(L
\cup L')$ differ by the sign $(-1)^{\mathrm{Ind}\, L \cdot \mathrm{Ind}\, L'}$.
\end{lemma}
\begin{proof}
 We first illustrate this when $L,L'$ are surjective with $1$-dimensional kernels. Suppose orientations have been chosen for their determinants: $\mathrm{Det}\, L = \ker L = \R e$, $\mathrm{Det}\, L' = \ker L' = \R e'$. Here $e,e'$ are sections of two bundles $E,E'$ over some surfaces $S,S'$, and we can naturally view them as elements in $\ker (L\cup L')$ by extending them to sections of $E\cup E' \to S \cup S'$ by defining them to be zero respectively over $S',S$. Then $\ker (L \cup L') = \R e + \R e'$ and $\mathrm{Det}\, (L\cup L') = \Lambda^2(\ker L\cup L') = \R e\wedge e'$. The two isomorphisms in the claim are respectively induced by $e\otimes e' \mapsto e\wedge e'$ and $e' \otimes e \mapsto e'\wedge e = (-1)^{1\cdot 1} e\wedge e'$, so the sign is as predicted.

In general, we first stipulate more precisely what the isomorphism preceding the claim is.
Abbreviate by $k,k',c,c'$ the dimensions of the kernels and cokernels of $L,L'$.
Given an orientation $(f_c^{\vee} \wedge \cdots \wedge f_1^{\vee}) \otimes (e_1 \wedge \cdots \wedge e_k)$ for $\mathrm{Det}(L)$, and similarly (using $f',c',e',k'$) an orientation for $\mathrm{Det}(L')$, we first extend the $f,f',e,e'$ sections over $S\cup S'$ as in the example above, and then we declare $(-1)^d (f_{c'}'^{\vee} \wedge \cdots \wedge f_1'^{\vee} \wedge f_c^{\vee} \wedge \cdots \wedge f_1^{\vee}) \otimes (e_1 \wedge \cdots \wedge e_k \wedge e_1'\wedge \cdots \wedge e_{k'}')$ to be the orientation of $\mathrm{Det}(L\cup L')$, where $d = c'(c+k)$ follows Koszul sign rules (compare \cite[Sec.(11a)]{Seidel-book} for a similar discussion). 
So the difference between the two isomorphisms in the claim, is the Koszul sign arising from switching the order of the two brackets in
\\[0.5mm]
\begin{tabular*}{\textwidth}{l@{\extracolsep{\fill}}cr@{\extracolsep{0pt}}} 
\strut & 
$
(\Lambda^{\textrm{max}}\textrm{coker}^{\vee} L \otimes \Lambda^{\textrm{max}}\ker L
)\otimes (\Lambda^{\textrm{max}}\textrm{coker}^{\vee} L' \otimes
\Lambda^{\textrm{max}}\ker L').
$
 & \strut 
\end{tabular*}
\\[0.5mm]
So the sign is:
$\textstyle (-1)^{(c+k) \cdot (c'+k')} 
=(-1)^{(k - c)\cdot (k'- c')}  =   (-1)^{\textrm{ind} L \cdot \textrm{ind} L'}.\quad\qedhere
$
\end{proof}

Denote $\sigma(L)\cup \sigma(L')$ the orientation for $L\cup L'$ induced from the orientations $\sigma(L),\sigma(L')$ via the above natural isomorphism.

If a gluing between $K$ and $L \cup L'$
only involves gluing $K$ with $L$, then \emph{distributivity}
holds:
\\[0.5mm]
\begin{tabular*}{\textwidth}{l@{\extracolsep{\fill}}cr@{\extracolsep{0pt}}} 
\strut & 
$
\sigma(K)\#(\sigma(L)\cup \sigma(L')) = (\sigma(K)\#
\sigma(L)) \cup \sigma(L').
$
 & \strut 
\end{tabular*}
%
%
%
%
\subsection{Axiomatic construction of coherent orientations}
\label{Subsection Axiomatic construction of coh orientations}
If $S$ has no punctures, then $\mathcal{O}_E(\emptyset;\emptyset)$
contains a Cauchy-Riemann operator $\overline{\partial}$, which is
$\C$-linear. So $\ker \overline{\partial}$, $\textrm{coker}\,
\overline{\partial}$ are complex, so they are canonically
oriented. So we obtain a
\emph{canonical orientation} on
$\mathcal{O}_E(\emptyset;\emptyset)$.

Now consider the standard trivial bundles $\textrm{triv}\to \C$.
Denote the operators $L$ on these bundles by $D^{\pm}$ depending
on whether infinity is a positive or negative puncture for $\C$. Define an
orientation on each $\mathcal{O}_{\textrm{triv}\to
\C}(L^-;\emptyset)$ by fixing an operator $D^-$ and picking an
orientation for it. This determines an orientation for any $D^+$
by coherence: gluing an appropriate $D^-$ forces $\sigma (D^+) \#
\sigma (D^-)$ to be the canonical orientation.

Since we established orientations for any $D_i^-\in
\mathcal{O}_{\textrm{triv}\to \C}(L_i^-;\emptyset)$, we can define
\\[1mm]
\begin{tabular*}{\textwidth}{l@{\extracolsep{\fill}}cr@{\extracolsep{0pt}}} 
\strut & 
$
\sigma(D_1^-\cup \cdots \cup D_k^-)=\sigma(D_1^-)\cup \cdots \cup
\sigma(D_k^-),
$
 & \strut 
\end{tabular*}
\\[1mm]
and similarly for $D_i^+ \in \mathcal{O}_{\textrm{triv}\to
\C}(\emptyset;L_i^+)$.

\begin{definition}\label{Definition orientations}
The orientation $\sigma(K)$ for $K\in \mathcal{O}_E(L_a;L_b)$ is
defined by capping off the punctures and requiring that we obtain
the canonical orientation:
\\[1mm]
\begin{tabular*}{\textwidth}{l@{\extracolsep{\fill}}cr@{\extracolsep{0pt}}} 
\strut & 
$
\sigma(\cup D_{-a}^+) \# \sigma(K) \# \sigma(\cup D_{-b}^-) =
\textrm{canonical},
$
 & \strut 
\end{tabular*}
\\[1mm]
with the convention that $D_{-a}$ is the reverse ordering
$(\ldots,D_2,D_1)$ of the $D_a$.
\end{definition}
%
\subsection{Orientation signs arising from gluing}
\label{Subsection Orientation signs arising from gluing}
%
We now want to find out whether a glued orientation $\sigma(L)\#
\sigma(L')$ agrees with $\sigma(L\# L')$ or not.

\begin{theorem}\cite[Thm.2]{Bourgeois-Mohnke}\label{Theorem reorder ends}
If we exchange the position of two consecutive $b$-labels $k,k+1$,
the coherent orientation of $\mathcal{O}_{E}(L_a;L_b)$ changes by
$(-1)^{\textrm{ind} \, D_k \cdot \textrm{ind} \, D_{k+1}}$, where
the $D_b \in \mathcal{O}_{\textrm{triv}\to \C}(L_b;\emptyset)$ cap
off the $b$-ends. Similarly for $a$-labels using the $D_a \in
\mathcal{O}_{\textrm{triv}\to \C}(\emptyset;L_a)$. For a general
permutation of the ends, iterate this result.
\end{theorem}
\begin{proof} For simplicity, we consider the pair of pants case. Observe the Figure:
\begin{center}
 \includegraphics[scale=0.5]{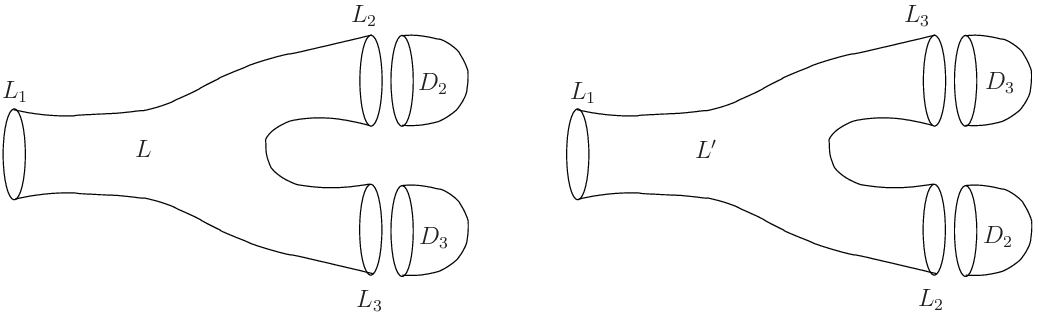}
\end{center}
We want to compare
$L\in \mathcal{O}_E(L_1;L_2,L_{3})$ and $L'\in
\mathcal{O}_E(L_1;L_3,L_2)$. We cap off the positive ends $2,3$ so that
both operators now lie in $\mathcal{O}_{E'}(L_1;\emptyset)$ so
they are equally oriented. Thus $\sigma(L\#(D_3 \cup D_2)) =
\sigma(L'\#(D_2 \cup D_3))$. So, by our axiomatic construction,
\\[0.5mm]
\begin{tabular*}{\textwidth}{l@{\extracolsep{\fill}}cr@{\extracolsep{0pt}}} 
\strut & 
$
\sigma(L)\# (\sigma(D_3)\cup \sigma(D_2)) = \sigma(L')\#
(\sigma(D_2)\cup \sigma(D_3)).
$
 & \strut 
\end{tabular*}
\\[0.5mm]
Thus $\sigma(L)$, $\sigma(L')$ are the same if and only if
$\sigma(D_3)\cup \sigma(D_2)$ and $\sigma(D_2)\cup \sigma(D_3)$
are the same. By Lemma \ref{Lemma Comparing natural orientations}, these are the same if and only if $(-1)^{\mathrm{Ind}(D_2)\cdot \mathrm{Ind}(D_3)}=+1$.
\end{proof}
\begin{theorem}[Prop.8, Prop.10,
\cite{Bourgeois-Mohnke}]\label{Theorem Gluing signs} $\sigma(L\# L')$ agrees with $\sigma(L)\#
\sigma(L')$ when all ends are glued,
so gluings of type $ \mathcal{O}_E(L_a;L_{-c}) \times
\mathcal{O}_{E'}(L_{c};L_{b'}) \to \mathcal{O}_{E''}(L_a;L_{b'})$.
For partial gluings, $\mathcal{O}_E(L_a;L_b,L_{-c}) \times
\mathcal{O}_{E'}(L_c,L_{a'};L_{b'}) \to
\mathcal{O}_{E''}(L_a,L_{a'};L_b,L_{b'})$, the orientations
$\sigma(L\# L')$ and $\sigma(L)\# \sigma(L')$ differ by
$(-1)^{\sum \textrm{ind}\, D_b \cdot \sum \textrm{ind}\, D_{a'}}$.
$($Mnemonically: we reorder $L_{b};L_{a'}$ to $L_{a'};L_b)$.
\end{theorem}
\begin{proof}[Sketch proof]
The key behind the first claim is that for complex linear
operators, gluings preserve the complex orientations. Now for
general $L,L'$, we can assume the ends labelled $a,b'$ are capped
off since they don't matter. Then attach cylinders $Z_{-c}$,
$Z'_c$ at the $c$-ends of $S,S'$ and extend $E,E'$ so that the new asymptotic operators are
complex linear. The resulting operators $L\cup Z_{-c}$, $Z'_c \cup
L'$ are now homotopic to complex linear ones, so they glue well:
$\sigma((L\cup Z_{-c}) \# (Z'_c \cup L'))=\sigma(L\cup Z_{-c}) \#
\sigma(Z'_c \cup L')$. Finally, by the associativity of $\#$, we
can move the $Z'_c$ over to the $Z_{-c}$ and ``cancel them off" in
pairs, to conclude the first claim: $\sigma(L \cup L')=\sigma(L)
\# \sigma(L')$.

For the second claim, write $\sigma_{a} = \cup_a \sigma(D_a)$.
Using distributivity and associativity of $\#$,
\\[0.5mm]
\begin{tabular*}{\textwidth}{l@{\extracolsep{\fill}}cr@{\extracolsep{0pt}}} 
\strut & 
$
(\sigma_{-a'} \cup \sigma_{-a}) \,\#\, (\sigma(L) \#
\sigma(L')) \,\#\, (\sigma_{-b'} \cup \sigma_{-b})  = 
[\sigma_{-a'} \cup (\sigma_{-a} \# \sigma(L))]  \,\#\,
 [(\sigma(L') \# \sigma_{-b'}) \cup \sigma_{-b}].
$
 & \strut 
\end{tabular*}
\\[0.5mm]
The latter is a gluing of all ends, so we can apply the first
claim provided that the ends are correctly ordered. Apply Theorem
\ref{Theorem reorder ends} iteratively to correctly reorder the
(negative) ends of the second square bracket: this gives rise to
the sign in the second claim. Except for this sign, the result of
the gluing must be the canonical complex orientation just like for
$(\sigma_{-a'} \cup \sigma_{-a}) \, \# \, (\sigma(L \# L'))\, \#\,
(\sigma_{-b'} \cup \sigma_{-b})$. Removing the caps that we added
in both cases shows $\sigma(L) \# \sigma(L')$, $\sigma(L \# L')$
differ by that sign.
\end{proof}
%
%
\subsection{Orientation signs for the TQFT}
\label{Subsection Orientation signs for the TQFT maps on SH}
%
We now show how the above axiomatic coherent orientation induces
coherent orientations for the moduli spaces
$\mathcal{M}(x_a;y_b;S,\beta)$. Pick once and for all a choice of
trivialization of $x^*T\overline{M}$ over each Hamiltonian orbit
$x$ (we discuss this further in \ref{Subsection Choice of trivializations}).
For $u\in \mathcal{M}(x_a;y_b;S,\beta,J)$, the linearization $L=D_u$
of $\overline{\partial}$ from \ref{Subsection
Smoothness of Moduli Spaces} is
an operator as in \ref{Subsection Fredholm operators over punctured surfaces} for the complex vector bundle $E=u^*T\overline{M} \to S$ (using $J$) with
asymptotic operators $(L_1^-,\ldots,L_p^-;L_1^+,\ldots,L_q^+)$ in the
chosen trivializations over the asymptotics $x_a,y_b$. Define the
orientation sign $\epsilon_u \in \{\pm 1\}$ arising for isolated $u$ in
\ref{Subsection Operations on SH(H) Definition} by
\\[1mm]
\begin{tabular*}{\textwidth}{l@{\extracolsep{\fill}}cr@{\extracolsep{0pt}}} 
\strut & 
$
\epsilon_u\cdot (1^{\vee}\otimes 1) = \sigma(u)= \sigma(\mathcal{O}_{u^*T{\overline{M}}\to
S}(L_1^-,\ldots,L_p^-;L_q^+,\ldots,L_1^+)).
$
 & \strut 
\end{tabular*}
\\[1mm]
Notice we reversed the order of the positive punctures.
Observe that the definition involves the bundle $u^*T\overline{M}$,
not just the asymptotics $L^{\pm}$. Also note that the capping off in \ref{Subsection Orientation signs arising from gluing} is an abstract construction: it does not require the asymptotics $x_a,y_b$ of $u$ to be contractible in $\overline{M}$.
\subsection{Using orientation signs to prove TQFT maps compose correctly}
\label{Subsection Using orientation signs to prove that TQFT maps
compose correctly}
In Theorem \ref{Theorem Gluing for compositions}, we claimed that
$\psi_{S}\circ \psi_{S'}=\psi_{S\#_{\lambda} S'}$ for any large enough gluing parameter
$\lambda$. Recall the proof produces a unique family of glued solutions $u \#_{\lambda} v$ which converges to the broken solution $u\# v$ as $\lambda \to \infty$. The broken solution is counted with sign $\epsilon_u \epsilon_v$ by $\psi_{S}\circ \psi_{S'}$, whereas $u \#_{\lambda} v$ is counted with sign $\epsilon_{u\#_{\lambda} v}$ by $\psi_{S\#_{\lambda} S'}$. To complete the proof we still need to show that $\epsilon_u \epsilon_v=\epsilon_{u\#_{\lambda} v}$. 
To keep the notation under control, we illustrate the proof in the case where $S=P$ and $S'=Q$ (see Figure \ref{Figure TQFT table}). So the gluing operation is
\\[0.5mm]
\begin{tabular*}{\textwidth}{l@{\extracolsep{\fill}}cr@{\extracolsep{0pt}}} 
\strut & 
$
\mathcal{O}_{u^*T{\overline{M}}\to
S}(L_1^-;L_2^+,L_1^+) \otimes \mathcal{O}_{v^*T{\overline{M}}\to
S'}(L_1'^-,L_2'^-;L_1'^+) \to \mathcal{O}_{(u\#_{\lambda} v)^*T{\overline{M}}\to
S\#_{\lambda} S'}(L_1^-;L_1'^+)
$
 & \strut 
\end{tabular*}
\\[0.5mm]
with $L_1^+ = L_1'^-$, $L_2^+=L_2'^-$.
By Theorem \ref{Theorem Gluing signs} this gluing is orientation-preserving for $\lambda\gg 0$
(the \emph{Explanation} in \ref{Subsection Gluing the operators and coherent orientations} clarifies the role of $\lambda$).
We deduce $\epsilon_u \epsilon_v = \epsilon_{u\#_{\lambda} v}$ as required.
%
%
%
%
\subsection{Using orientation signs to prove TQFT maps are chain maps}
\label{Subsection Using orientation signs to prove that TQFT maps
are chain maps}
%
If $SH^*$ is $\Z$-graded (\ref{Subsection Maslov index and
Conley-Zehnder index}), the differential $\partial$ on
$SC^*(H_1)\otimes \cdots \otimes SC^*(H_k)$ is defined by\\[1mm]$
\begin{array}{rcl}
\partial(a_1\otimes \cdots \otimes a_k)  &=&  d(a_1)
\otimes a_2 \otimes \cdots \otimes a_k + (-1)^{|a_1|} a_1 \otimes
d(a_2) \otimes a_3 \otimes \cdots \otimes a_k + \\ && +\cdots + (-1)^{|a_1|+\cdots + |a_{k-1}|} a_1 \otimes \cdots \otimes a_{k-1} \otimes d(a_k),
\end{array}
$\\[1mm]
this is the Koszul sign convention with $d$ in degree $1$. In
general, $SH^*$ is only $\Z/2$-graded, but this $\partial$ makes
sense since only the parity of $|a_1|,|a_2|,\ldots$ matter.

In Theorem \ref{Theorem Homotopy of surfaces}, we claimed $\psi_S: \otimes_b SC^*(B_b H) \to \otimes_a SC^*(A_a H)$ is a chain map: $\partial\circ \psi_S = \psi_S \circ \partial$. This equation arose from an oriented count of the broken solutions arising at the boundary of the $1$-dimensional moduli spaces of Floer solutions, but we still need to check the orientation signs. This involves two steps: (1) we need to explain how the Koszul signs in the above definition of $\partial$ arise; (2) we need to explain the minus sign in $\partial \circ \psi_S - \psi_S \circ \partial =0$.

The proof of (2) is identical to the proof in \ref{Subsection Orientation signs Floer chain maps}: replace $\mathcal{M}^{H_s}_1(x,z)$ by $\mathcal{M}_1(x_i;z_j;S,\beta)$, replace continuation solutions $v,v'$ by Floer solutions $v,v'$ (and the $s$-coordinate for $v,v'$ now refers to the $s$-coordinate on the ends of $v,v'$ where the Floer trajectories $u,u'$ broke off). The argument in \ref{Subsection Orientation signs Floer chain maps} then shows that \emph{if the $\sigma$-orientations are respected} (meaning $\sigma(u\#v)=\sigma(u)\#\sigma(v)$ and $\sigma(v'\#u')=\sigma(v')\#\sigma(u')$) then the gluing map sends the orientations $\partial_s u,\partial_s u'$ of $u\#v, v'\#u'$ to $\partial_{\lambda}(u\#_{\lambda} v), -\partial_{\lambda}(v'\#_{\lambda} u')$, which are respectively outward and inward pointing near the boundary of $\mathcal{M}$ (since they approach the boundary as $\lambda\to \infty$).

\emph{Proof of (1):} Suppose $u\#v$ is a broken Floer solution, where an isolated Floer
trajectory $u$ broke off at the first negative end of $v$. The gluing of
the linearizations $D_u\# D_v$ is
\\[0.5mm]
\begin{tabular*}{\textwidth}{l@{\extracolsep{\fill}}cr@{\extracolsep{0pt}}} 
\strut & 
$
\mathcal{O}(D_u^-;D_u^+)\times
\mathcal{O}(D_u^+,L_2^-,\ldots,L_p^-;L_q^+,\ldots,L_1^+) \to
\mathcal{O}(D_u^-,L_2^-,\ldots,L_p^-;L_q^+,\ldots,L_1^+),
$
 & \strut 
\end{tabular*}
\\[0.5mm]
so by Theorem \ref{Theorem Gluing signs},
$\sigma(u\#v)=\sigma(u)\#\sigma(v)$. Similarly, for broken Floer
solutions $v'\#u'$ when $u'$ broke off at the first positive end of $v'$, we get $\sigma(v'\#u')=\sigma(v')\#\sigma(u')$. 

Via step (2) this shows that these two breakings contribute respectively to $(d\otimes 1 \otimes \cdots \otimes 1)\circ \psi_S$ and $\psi_S \circ (d\otimes 1 \otimes \cdots \otimes 1)$ in the equation $\partial \circ \psi_S = \psi_S \circ \partial$ (here $1$ are identity maps). To prove that the breakings at the other ends also contribute correctly to  $\partial \circ \psi_S = \psi_S \circ \partial$, we reduce to the previous two cases by reordering the ends twice using Theorem \ref{Theorem reorder ends}, as follows. 

Suppose $u$ broke off at the $k$-th negative end of $v$, so $L_k^-=D_u^+$. Then Theorem \ref{Theorem reorder ends} allows us to move the $k$-th end into the first position at the cost of introducing the sign $(-1)^{[\mathrm{Ind}(D_1)+\cdots +\mathrm{Ind}(D_{k-1})]\, \mathrm{Ind}(D_k)}$ where $D_i\in \mathcal{O}_{\mathrm{triv}\to \C}(\emptyset;L^-_i)$ cap off the negative ends. After this reordering, we are in the case discussed previously where $u$ broke off at the first negative end. So gluing the cylinder with asymptotics $(D_u^{-};D_u^{+})$ onto the first negative end respects $\sigma$-orientations. By Theorem \ref{Theorem reorder ends} we can now move the first end back to its original $k$-th position, at the cost of the sign $(-1)^{[\mathrm{Ind}(D_k)-1]\, [\mathrm{Ind}(D_1)+\cdots+ \mathrm{Ind}(D_{k-1})]}$ (here we used that $u$ is an \emph{isolated} Floer trajectory, so the Fredholm indices of the caps attached to $L_k^-=D_u^+$ and to $D_u^-$ differ by $1$). This final position is the actual gluing without reorderings that we are interested in. 
Call $a_i$ the asymptotic orbits at the negative ends of the Floer solution $v$. Then by definition%
\footnote{By (\ref{Item grading of D+})-(\ref{Item grading of D-}) in the \emph{Technical remark} in \ref{Subsection Choice of trivializations}, the Conley-Zehnder grading $|L^{\pm}|\equiv \mathrm{ind}\, D^{\pm}$ modulo 2.}
 $(-1)^{|a_i|} = (-1)^{\mathrm{Ind}(D_i)}$. So the total sign caused by the reorderings is
\\[0.8mm]
\begin{tabular*}{\textwidth}{l@{\extracolsep{\fill}}cr@{\extracolsep{0pt}}} 
\strut & 
$
(-1)^{[\mathrm{Ind}(D_1)+\cdots+ \mathrm{Ind}(D_{k-1})]\, \mathrm{Ind}(D_k)}
   (-1)^{[\mathrm{Ind}(D_k)-1]\, [\mathrm{Ind}(D_1)+\cdots+ \mathrm{Ind}(D_{k-1})]} =
  (-1)^{|a_1|+\cdots + |a_{k-1}|},
$
 & \strut 
\end{tabular*}
\\[0.8mm]
which is precisely the Koszul sign for the $k$-th term in the definition of $\partial$ above. The discussion when $u'$ breaks off at a positive end of a Floer solution $v'$ in analogous. This proves (1).
%
\subsection{The choice of trivializations over the Hamiltonian orbits, the role of the canonical bundle $\mathcal{K}$, and dimension counts}
\label{Subsection Choice of trivializations}
%
Recall we chose
trivializations of $x^*T\overline{M}$ over all possible asymptotics
$x$ in \ref{Subsection Orientation signs for the TQFT maps on SH}. By construction, we actually only need to choose a homotopy class of trivializations, so we only need to choose an orientation of $x^*T\overline{M}$.

For example, suppose $c_1(M)\equiv c_1(TM,J)=0$. Then the canonical bundle $\mathcal{K}=\Lambda^{max}_{\C}T^*{\overline{M}}$ is trivial (we encountered $\mathcal{K}$ in \ref{Subsection Maslov index and Conley-Zehnder index}). A choice of trivialization of the complex line bundle $\mathcal{K}$ naturally determines a trivialization of the complex line bundles $\Lambda^{max}_{\C}(x^*T\overline{M})=x^*(\mathcal{K}^{\vee})$. So explicitly, we require that the trivialization of $x^*T\overline{M}$ induces the given trivialization of $x^*\mathcal{K}$ up to homotopy (this condition is an obstruction lying in $\pi_1(U(n))\cong \Z$). Equivalently, pick a complex volume form $\eta$, namely a non-vanishing section of $\mathcal{K}$, and require that the trivialization $x^*T\overline{M} \cong S^1 \times \C^n$ sends $\eta$ to the standard volume form of $\C^n$ (since $\pi_1(SU(n))=0$, this determines the trivialization up to homotopy).

Recall from \ref{Subsection Maslov index and Conley-Zehnder index} that the condition $c_1(M)=0$ induces a Conley-Zehnder grading on the orbits. We now explain how the grading determines the Fredholm indices by the Riemann-Roch theorem \cite[Theorem C.1.10]{McDuff-Salamon2}.

\emph{Technical Remark: the linearization $D_u$ arising in Floer theory is an elliptic first order partial differential  operator which equals a complex linear operator up to a lower order term which can be ignored when computing Fredholm indices \cite[Sec.3.1]{McDuff-Salamon}. The Riemann-Roch theorem is then applied to the complex linear part, which is a Cauchy-Riemann operator which induces a holomorphic structure on $u^*T\overline{M}$}.
\begin{enumerate}
 \item For $L\in \mathcal{O}_{u^*T\overline{M}\to Z}(L^-;L^+)$, where
$u: Z \to \overline{M}$ has asymptotics $x^{\pm}$, we can prescribe the trivialization of $u^*T\overline{M}$ in \ref{Subsection Gluing the operators and coherent orientations} to be the one which induces the given trivialization of $u^*\mathcal{K}$. The Conley-Zehnder grading described in the \emph{Explanation} in \ref{Subsection Gluing the operators and coherent orientations} and \ref{Subsection Fredholm operators on bundles E to M} agrees with the grading from \ref{Subsection Maslov index and Conley-Zehnder index}, so:
\\[0.5mm]
\begin{tabular*}{0.95\textwidth}{l@{\extracolsep{\fill}}cr@{\extracolsep{0pt}}} 
\strut & 
$
\mathbf{\mathbf{Ind}(L) = |x^-|-|x^+|}
$
 & \strut 
\end{tabular*}
\\[-4.3mm]

 \item \label{Item grading of D+} For $D^+\in \mathcal{O}_{E\to \C}(\emptyset;L^+)$, a Riemann-Roch argument \cite[Appendix C.4]{McDuff-Salamon2} shows that $\mathrm{Ind}(D^+) = n + \mu(L^+) = 2n-|L^+|$ where $n=\mathrm{rank}_{\C} E$, $\mu(L^+)$ is a Maslov index, and $|L^+|$ is the Conley-Zehnder index (see \ref{Subsection Maslov index and Conley-Zehnder index} and the Explanation in \ref{Subsection Gluing the operators and coherent orientations}). If $E=u^*T\overline{M}$ for some $u\co \C \to \overline{M}$ with asymptotic $x^+$, and $c_1(TM)|_{\pi_2(M)}=0$, then we deduce:
\\[0.5mm]
\begin{tabular*}{0.95\textwidth}{l@{\extracolsep{\fill}}cr@{\extracolsep{0pt}}} 
\strut & 
$
\mathbf{\mathbf{Ind}(D^+) = 2n - |x^+|}
$
 & \strut 
\end{tabular*}
\\[-4.3mm]

 \item \label{Item grading of D-} For $D^-\in \mathcal{O}_{E'\to \C'}(L^-;\emptyset)$, where $\C'$ is $\C$ but viewing infinity as a \emph{negative} puncture, glue an appropriate $D^+$ from (\ref{Item grading of D+}). Then
 additivity of indices and Riemann-Roch determines $\mathrm{Ind}(D^-)$ via: $\mathrm{Ind}(D^+) + \mathrm{Ind}(D^-) = \mathrm{Ind}(D^+\# D^-)= 2n + 2c_1(E\# E')[\C\# \C']$. If $E=u^*T\overline{M}$ for $u\co \C' \to \overline{M}$ with asymptotic $x^-$, and $c_1(TM)|_{\pi_2(M)}=0$, we deduce:
\\[0.5mm]
\begin{tabular*}{0.95\textwidth}{l@{\extracolsep{\fill}}cr@{\extracolsep{0pt}}} 
\strut & 
$
\mathbf{\mathbf{Ind}(D^-) = |x^-|}
$
 & \strut 
\end{tabular*}
\\[-4.3mm]

 \item For $L\in \mathcal{O}_{E\to S}(L^-_a;L^+_b)$ with $S$ of genus $g$ with no punctures, by Riemann-Roch $\mathrm{Ind}(L) = 2(1-g)n + 2c_1(E)[S]$;

\item For $L\in \mathcal{O}_{E\to S}(L^-_a;L^+_b)$ with $S$ of genus $g$ with $p$ negative and $q$ positive punctures, cap off the ends using $D^+_a,D^-_b$ of type (2),(3) respectively. Then by (4), $\mathrm{Ind}(L) + \sum \mathrm{Ind}(D^+_a) + \sum \mathrm{Ind}(D^-_b) =  2(1-g)n + 2c_1(\cup_a D^+_a \# E \# \cup_b D^-_b)[\cup_a \C \# S\# \cup_b \C']$;

\item Suppose $c_1(M)=0$. Fix a trivialization of $\mathcal{K}$. If $E=u^*T\overline{M}$ for $u: S \to \overline{M}$, for an $S$ as in (5), then pick a trivialization of $E$ inducing the given trivialization of $u^*\mathcal{K}$. Choose trivial bundles for the caps $D_a^+,D_b^-$ agreeing with the trivializations on the asymptotics $x_a,y_b$ of $u$. Then the $c_1$ term in (5) vanishes, so by (2),(3) we obtain $\mathrm{Ind}(L) + 2np - \sum |x_a| + \sum |y_b|=2(1-g)n$. We deduce the formula in Theorem \ref{Theorem index of Fredholm operator}:
\\[0.5mm]
\begin{tabular*}{0.95\textwidth}{l@{\extracolsep{\fill}}cr@{\extracolsep{0pt}}} 
\strut & 
$
\mathbf{\mathbf{Ind}(L) = \sum |x_a| - \sum |y_b| + 2n(1-g-p)}
$
 & \strut 
\end{tabular*}
\end{enumerate}
If we restrict the TQFT to genus $0$ surfaces and we restrict to contractible orbits (so we use $SH^*_0$ as in \ref{Subsection TQFT is compatible with filtrations}), the above discussion shows that we can weaken $c_1(TM)=0$ to just $c_1(TM)|_{\pi_2(M)}=0$. In this case, instead of using $\mathcal{K}$, one chooses the trivialization of $x^*T\overline{M}$ to be induced from a trivialization of $\overline{x}^*T\overline{M}$, where $\overline{x}$ is a disc in $\overline{M}$ bounding $x$. In (6) one uses such discs to cap off $u: S \to \overline{M}$ to obtain a sphere in $\overline{M}$. The choices of trivialization of $\overline{x}^*T\overline{M}$ do not affect the index, since two choices are related by the gluing of a bundle over a sphere, and by Riemann-Roch, gluing $E'\to \C P^1$ onto $E \to S$ can only change Fredholm indices 
by a value in $2c_1(TM)|_{\pi_2(M)}=0$.

In conclusion, when $c_1(TM)|_{\pi_2(M)}=0$, $SH^*_0(M)$ is a $\Z$-graded unital ring.

For cotangent bundles $\overline{M}=T^*N$ of closed \emph{oriented} manifolds $N$, there is a preferred homotopy class of trivializations for $\mathcal{K}=\Lambda^{\mathrm{max}}_{\C}TT^*N=\Lambda^{\mathrm{max}}_{\R}T^*N \otimes_{\R} \C$, since the orientation for $N$ determines a trivialization of $\Lambda^{\mathrm{max}}_{\R}T^*N$.
%
\subsection{Avoiding choices}
\label{Subsection Avoiding Choice of trivializations}
The system of coherent orientations constructed axiomatically in \ref{Subsection Axiomatic construction of coh orientations} depends on the choices of orientations for $\mathcal{O}_{\mathrm{triv}\to \C}(L^-;\emptyset)$. Different choices will give different orientation signs. To avoid making these choices, one can incorporate all choices into the symplectic chain complex, which is a constrution due to Seidel \cite[Sec.(12f)]{Seidel-book}.

For a $1$-dimensional $\R$-vector space $V$, define the
\emph{orientation space} $\mathcal{O}(V)$ to be the
$1$-dimensional $\K$-vector space generated by the two possible
orientations of $V$, subject to the relation that the two
generators sum to zero. This is functorial: any $\R$-linear map
$V\!\to\! W$ defines a natural $\K$-linear map
$\mathcal{O}(V)\!\to\! \mathcal{O}(W)$, and similarly for
multi-linear maps. 

Fix a homotopy class of trivializations for $x^*T\overline{M}$ for any Hamiltonian orbit $x$ (for example by \ref{Subsection Choice of trivializations}, when $c_1(M)=0$ a set of choices can be described in terms of a trivialization of $\mathcal{K}$). The linearization of Floer's equation $\partial_s u + J(\partial_t u - X_H)=0$ determines an asymptotic operator $L^+_x$ of the form described in \ref{Subsection Fredholm operators on trivial bundles over a cylinder} in such a trivialization of $x^*T\overline{M}$. This in turn determines, up to homotopy, an operator $D_x^{+}:W^{1,\mathfrak{p}}(\C,\C^n) \to L^{\mathfrak{p}}(\C,\C^n)$ in $\mathcal{O}_{\mathrm{triv}\to \C}(\emptyset;L_x^+)$ as in \ref{Subsection Fredholm operators over punctured surfaces} defined on the trivial bundle over $\C$ with asymptotic operator $L_x^+$.

Denote by $\mathcal{O}(x)=\mathcal{O}(V)$ the orientation space of $V=\mathrm{Det}(D^{+}_x)$. Then define
\\[1mm]
\begin{tabular*}{\textwidth}{l@{\extracolsep{\fill}}cr@{\extracolsep{0pt}}} 
\strut & 
$
SC^*(H) =\bigoplus \left\{ \mathcal{O}(x) : x \in
\mathcal{L}\overline{M},\; \dot{x}(t) = X_H(x(t)) \right\}.
$
 & \strut 
\end{tabular*}
\\[1mm]
Suppose we are given a choice of generators $X_a\in \mathcal{O}(x_a),Y_b\in \mathcal{O}(y_b)$. Let $L=D_u$ be any linearization arising from a Floer solution $u: S \to \overline{M}$ with asymptotics $x_a,y_b$. This determines asymptotic operators $L_a^-,L_b^+$ over $x_a,y_b$ in the given trivializations of $x_a^*T\overline{M},y_b^*T\overline{M}$. The caps $D_a^+\in \mathcal{O}_{\mathrm{triv}\to \C}(\emptyset;L_a^-),D_b^-\in \mathcal{O}_{\mathrm{triv}\to \C}(L_b^+;\emptyset)$ are oriented respectively by $X_a$ and by the condition $Y_b\#\mathrm{Det}(D_b^-)=\mathrm{canonical}$ (compare \ref{Subsection Axiomatic construction of coh orientations}). Call $Y_b'$ the orientation that $Y_b$ determines for $\mathrm{Det}(D_b^-)$ in this way. Then $X_a,Y_b$ uniquely determine an orientation for $L$ via the condition $\cup_{a} X_{-a} \# \mathrm{Det}(L) \# \cup_{b} Y_{-b}'=\mathrm{canonical}$ (mimicking Definition \ref{Definition orientations}). 

%
%
%
This proves that the $\mathcal{M}(x_a;y_b;S,\beta)$ are
canonically oriented \emph{relative to the ends}. 

For isolated $u\in \mathcal{M}_0(x_a;y_b;S,\beta)$ the $X_a,Y_b$ determine an orientation sign $\epsilon_{u,(X_a;Y_b)}\in \{ \pm 1 \}$ via $\mathrm{Det}(D_u)=\epsilon_{u,(X_a;Y_b)}\cdot(1^{\vee}\otimes 1)$.
So in this formalism the $\epsilon_u$ become isomorphisms
\\[0.5mm]
\begin{tabular*}{\textwidth}{l@{\extracolsep{\fill}}cr@{\extracolsep{0pt}}} 
\strut & 
$
\epsilon_u:\otimes_b \mathcal{O}(y_b) \to \otimes_a
\mathcal{O}(x_a),
$
 & \strut 
\end{tabular*}
\\[0.5mm]
extending $Y_1\otimes \cdots \otimes Y_q \mapsto \epsilon_{u,(X_a;Y_b)} X_1\otimes \cdots \otimes X_p$ multi-linearly. 
So $\epsilon_u$ no longer depends on the choice of a system of coherent orientations, since all choices are taken into account at the same time.
Then, summing over all $u\in \mathcal{M}_0(x_a;y_b;S,\beta)$, we define the TQFT operation by
\\[0.5mm]
\begin{tabular*}{\textwidth}{l@{\extracolsep{\fill}}cr@{\extracolsep{0pt}}} 
\strut & 
$
\psi_S:\otimes_b \mathcal{O}(y_b) \to \otimes_a \mathcal{O}(x_a), \;\psi_S = \sum \epsilon_u.
$
 & \strut 
\end{tabular*}
\\[0.5mm]
\indent
Similarly, we orient the moduli spaces of Floer trajectories, so for instance $u\in \mathcal{M}_0(x,y)$ defines an isomorphism
$
\epsilon_u: \mathcal{O}(y)\to \mathcal{O}(x),
$
which, if we made explicit choices of orientation generators, would be
multiplication by the sign $\epsilon_u$ of \ref{Subsection
Symplectic chain complex}. Then define the Floer differential $d: \mathcal{O}(y) \to \mathcal{O}(x)$
by the linear map $d=\sum \epsilon_u$, summing over $u \in \mathcal{M}_0(x,y)$.

For twisted symplectic cohomology, we replace the field $\K$ by
the Novikov algebra $\Lambda$, and we insert the weights
$t^{\alpha[u]}$ in front of the $\epsilon_u$ maps.
%
%
\section{Appendix 3: $SH^*(M)$ defined using a non-linear Hamiltonian}
\label{Appendix Using non-linear Hamiltonians}
\label{Subsection Using non-linear Hamiltonians}
This section is required in \ref{Subsection Abb-Sch construction}, since that uses Hamiltonians of quadratic growth.
\subsection{An alternative definition of $\mathbf{SH^*(M)}$.}
\label{Subsection An alternative definition of SH}
Let $Q:\overline{M}\to \R$ be equal to a convex function $q(R)$ for $R\geq 1$, with $q'(R)\to \infty$ as $R \to \infty$ (example: $Q=\tfrac{1}{2}R^2$). Define the action function
$$
\mathbb{A}_q(R) = -Rq'(R) + q(R).
$$

\begin{lemma}\label{Lemma non-linear hamiltonians}
There is a natural identification
$
SH^*(Q) \cong SH^*(M),
$
where $SH^*(Q)$ is defined as in \ref{Subsection Symplectic chain complex}.
\end{lemma}
\begin{proof}
By convexity of $q$, $\partial_R \mathbb{A}_q=-Rq''(R)\leq
0$, so the action $\mathbb{A}_q$ decreases on the collar. 
This convexity together with the condition $q'(R) \to \infty$ ensures\footnote{\emph{Proof: if $q'(R_a)=a$ and $q'(R_b)=b\geq a$ then $\mathbb{A}_q(R_b) = -R_b b + q(R_b) \leq -R_a b + q(R_a)$ since $\partial_R (-Rb+q(R)) = -b+q'(R)\leq 0$ for $R\in [R_a,R_b]$ since $q''\geq 0$. Now observe $-R_a b + q(R_a)\to -\infty$ as $b\to \infty$. $\qed$}} that $\mathbb{A}_q(R) \to -\infty$ as $R \to \infty$. So, using the action-restrictions from Section \ref{Section SH+}, we have a chain complex
$$
SC^*(Q;\mathbb{A}_Q\geq \mathbb{A}_q(R_m)) \equiv SC^*(H^m),
$$
where $q'(R_m)=m$ and where $H^m$ is obtained from $Q$ by changing it to be linear of slope $m$ in the region $R\geq R_m$ where $q'\geq m$. By Section \ref{Section SH+}, $SC^*(H^m) \subset SC^*(H^{m'}) \subset SC^*(Q)$ for $m\leq m'$ are inclusions of subcomplexes. Since
$SC^*(Q)$ consists of \emph{finite} linear combinations of
generators, we deduce $SH^*(Q) = \bigcup_m \, SH^*(H^m) \cong \varinjlim SH^*(H) =
SH^*(M)$. Strictly speaking, to define $SH^*(Q)$ we need to make a subtle time-dependent perturbation of $Q$ to ensure the non-degeneracy of the orbits: we do this in the Technical Remarks below.
\end{proof}
By Lemma \ref{Lemma
non-linear hamiltonians}, one could \emph{define}
$SH^*(M)=SH^*(Q)$. Call \emph{non-linear Hamiltonians} the above
 $Q$. One could drop the convexity assumption, and only assume
 $q'(R)\to \infty$ and $\mathbb{A}_q(R) \to -\infty$ as $R\to \infty$ (Lemma \ref{Lemma
non-linear hamiltonians} can still be proved by a similar argument%
\footnote{\label{Footnote Lemma nonlinear SH} Define $H^m$ as before: so $H^m$ is linear of slope $m$ at infinity, but now, due to the non-convexity of $q$, there may be some points where $H^m$ has slope $\!>\!m$. Given $m\!>\!0$, all $1$-orbits of $H^m$ have $\mathbb{A}_{H^m}\!>\!c$ for $c\!\ll\! -m$, so for a monotone homotopy $H_s$ from $Q$ to $H^m$ with $\partial_s H_s\!\leq\! 0$ (so actions decrease along continuation solutions by \ref{Subsection Floer continuation solutions}) and depending on $s$ only on $V_m=\{Q\neq H^m\}$, the continuation map sends
 $SH^*(H^m) \!\to\! SH^*(Q;\mathbb{A}_Q\!>\!c)$. Given $c<0$, for $m'\!\gg\! -c$ there are inclusions $SC^*(Q;\mathbb{A}_Q\!>\!c) \!\to\! SC^*(H^{m'})$ since the $1$-orbits of $Q$ on $V_{m'}$ will have $\mathbb{A}_{Q}\!\ll\! c$. Then observe: (1) for $m'\!\gg\! -c \!\gg\! m$, the composite $SH^*(H^m) \!\to\! SH^*(Q;\mathbb{A}_Q\!>\!c) \!\to\! SH^*(H^{m'})$ is the continuation $SH^*(H^m) \!\to\! SH^*(H^{m'})$ (using Lemma \ref{Lemma Chain Homotopy}(\ref{Item Lemma Chain Homotopy composing conts is cont})); and (2) for $-c'\!\gg\! m' \!\gg\! -c$, the composite $SC^*(Q;\mathbb{A}_Q\!>\!c) \!\to\! SC^*(H^{m'}) \!\to\! SC^*(Q;\mathbb{A}_Q\!>\!c')$ is the inclusion $SC^*(Q;\mathbb{A}_Q\!>\!c) \!\to\! SC^*(Q;\mathbb{A}_Q\!>\!c')$. The claim now follows by (1), (2), $SH^*(M)\!=\!{\displaystyle\lim_{m_j\to\infty}} SH^*(H^{m_j})$ and $SH^*(Q)\!=\!{\displaystyle\lim_{c_j\to -\infty}} SH^*(Q;\mathrm{A}_Q>c_j)$. $\qed$
}%
%
%
%
%
%
%
%
%
%
%
).
\\[1mm]
\indent \emph{Technical Remarks. As mentioned in \ref{Subsection Transversality and Compactness}, we always need the $1$-orbits to be non-degenerate. In the non-linear setup, this becomes a problem: to make the orbits non-degenerate (in particular, to remove the $S^1$-symmetry), one needs to perturb $Q:\overline{M} \to \R$
by adding a small time-dependent function $P: S^1 \times \overline{M} \to \R$, and it is not possible to make $Q+P$ depend only on $t$ and $R$ near these orbits, so the maximum principle in Lemma \ref{Lemma Maximum principle for Floer solns} may fail.
We mimic Abouzaid's approach \cite[Appendix B]{Abouzaid-IHES}: we avoid using the maximum principle, and we rather use the no escape Lemma -- the version that we need here is proved in Lemma \ref{Lemma No escape for quadratic growth}. \\ \indent
In our setup, $q$ is convex so $\partial_R \mathbb{A}_q\leq 0$. So to guarantee that Lemma \ref{Lemma No escape for quadratic growth} applies, we ensure $\mathbb{A}_q\leq 0$ for all $R\geq 1$ by requiring that $-q'(1)+q(1)\leq 0$.%
\\ \indent To make the $1$-orbits non-degenerate it suffices to make a generic $C^2$-small perturbation of $Q$ supported near the orbits.
So for Floer trajectories of $Q$ we choose $P: S^1 \times \overline{M} \to \R$ to be generic $C^2$-small and supported near the $1$-orbits of $Q$, in particular we ensure $P=0$ in certain regions where there are no $1$-orbits of $Q$: fix a sequence $R_m \to \infty$ such that $q'(R_m)$ are not Reeb periods and ensure $P=0$ for $R$ close to $R_m$. This guarantees that Lemma \ref{Lemma No escape for quadratic growth} applies to $H=Q+P$ on any $V\subset \overline{M}$ defined by $R\geq R_m$ (assuming $J$ is of contact type along $R=R_m$). So if a Floer trajectory $v: \R \times S^1 \to \overline{M}$ has asymptotics $x,y$ with $R(x)\leq R_m, R(y)\leq R_m$, then $v$ cannot enter the region $R>R_m$ since otherwise $u=v|_{v^{-1}(V)}$ would violate Lemma \ref{Lemma No escape for quadratic growth}.\\ 
\indent
Now consider Floer solutions for $Q$ on a model surface $S$ (a similar discussion holds for Floer continuations for $Q_s$, using $P: \R \times S^1 \times \overline{M} \to \R$). We choose $P: S \times \overline{M} \to \R$ to be: zero away from the ends of $S$; 
generic $C^2$-small on the ends of $S$ with $\partial_s P\leq 0$ and $P$ only depending on $t$ for $|s|\gg 0$; and $P$ is supported near the $1$-orbits of $Q$. As before, we pick a sequence $R_m \to \infty$ such that 
$A_a q'(R_m), B_b q'(R_m)$ are not Reeb periods,  where $A_a,B_b$ are the weights for $S$, and we ensure $P=0$ for $R$ near $R_m$.
So if a Floer solution $v$ has asymptotics $x_a,y_b$ with $R(x_a)\leq R_m, R(y_b)\leq R_m$, then $v$ cannot enter the region $R>R_m$ since otherwise $u=v|_{v^{-1}(V)}$ violates Lemma \ref{Lemma No escape for quadratic growth} for $H=Q+P$.\\ \indent
In conclusion, once the data $S,A_a,B_b$ is fixed, a generic perturbation $Q+P$ of $Q$ as above ensures both that the orbits are non-degenerate and that the Floer moduli spaces are contained in a compact subset of $\overline{M}$ determined by the $R$-values of the asymptotics. 
}
%
%
\subsection{Contact type $\mathbf{J}$ make transversality fail for $\mathbf{SH^*(Q)}$.}
\label{Subsection contact type J make transversality fail}
%
For Hamiltonians $H$ of linear growth we always arranged that the $1$-orbits of $H$ were in a compact set $R \leq R_0$, so if $J$ is of contact type for $R>R_0$ then no Floer trajectories enter the region $R>R_0$ and so $J$ does not need to be perturbed there to achieve transversality for the Floer moduli spaces. For non-linear Hamiltonians $Q$ as above, $1$-orbits keep appearing for larger and larger $R$, so $J$ will typically need to be perturbed everywhere on $\overline{M}$ and in all directions to achieve transversality. So we cannot impose the contact type condition $J\partial_r = \mathcal{R}$ which leaves no freedom to perturb $J$ in the plane $\textrm{span}(\partial_r,\mathcal{R})$ (recall the notation of \ref{Subsection Liouville domains Definition}). Without the contact type condition, the maximum principle in Lemma \ref{Lemma Maximum principle for Floer solns} (or the no escape Lemma \ref{Lemma no escape}) may fail, so the Floer moduli spaces may not have compactifications by broken solutions. Thus we must allow non-contact type $J$, and we need to reprove a version of the maximum principle.

We found two approaches to solve this issue which we now describe. From now on, when we say ``Floer moduli spaces'' or Floer solutions we refer generally to the moduli spaces of Floer trajectories, Floer continuation solutions, or Floer solutions.

\begin{enumerate}
                                              
\item \textbf{First approach:} we allow $J$ to be of non-contact type, provided that there is a sequence $R_m \to \infty$ such that $J$ satisfies the contact type condition along the hypersurfaces $\{R = R_m\}$. We assume that $Q=q(R)$ only depends on $R$ near $R=R_m$ (see the Technical Remarks in \ref{Subsection An alternative definition of SH}) and the slopes $q'(R_m)$ are not Reeb periods (for Floer solutions, we require this for $A_a q'(R_m), B_b q'(R_m)$, where $A_a,B_b$ are the weights). This ensures that the asymptotics for the Floer moduli spaces don't lie in $\{ R=R_m \}$.
         
\item \textbf{Second approach:} we allow a non-contact type almost complex structure $J_{\infty}$, but we require that it decays sufficiently rapidly in $R$ to a fixed $J=J_0$ which is of contact type for large $R$ (and we assume $Q=q(R)$ only depends on $R$ for $R\gg 0$ up to perturbations as in the Technical Remarks of \ref{Subsection An alternative definition of SH}). More precisely, we fix an exhausting sequence of compact sets $\overline{M}=\cup_m K_m$, with $K_m \subset \mathrm{int}(K_{m+1})$, for example $K_m=\{R \leq m \}$. Then we prove that there are reals $\delta_m>0$ such that any $\omega$-compatible almost complex structure $J_{\infty}$ satisfying $|J_{\infty}-J|<\delta_m$ on $K_m$ will satisfy a maximum principle. We also show that this leaves enough room to make generic perturbations of $J_{\infty}$ on each $K_m$ to ensure the regularity of Floer solutions.
\end{enumerate}

\emph{Remark 1. The second approach is trickier, but it is of independent interest since the assumptions can be significantly weakened as follows. Fix any time-dependent Hamiltonian $Q: S^1 \times \overline{M} \to \R$ for which in any given compact subset of $\overline{M}$ there are only finitely many $1$-orbits of $Q$ (and the orbits are non-degenerate) -- in the case of Floer solutions we assume this for $A_a Q,B_b Q$. Then we only need that $J=J_0$ is an $\omega$-compatible almost complex structure for which a no escape Lemma of the following form holds: given any compact subset $K\subset \overline{M}$, and any $\omega$-compatible almost complex structure $J'$ which equals $J$ outside of $K$, all $J'$-Floer solutions automatically lie in a compact subset of $\overline{M}$ determined by $K$ and the asymptotics (and not depending on $J'$). This generalization is useful when $\overline{M}$ does not have contact hypersurfaces at infinity (for example: negative vector bundles of rank$\,_{\C}\geq 2$, see \cite{Ritter4}).}

\begin{lemma}
 In the first approach, after a generic perturbation of $J$ which preserves the contact type condition along the hypersurfaces $\{R=R_m\}$, we can ensure compactness and transversality for Floer moduli spaces. 
\end{lemma}
\begin{proof}
By the no escape Lemma \ref{Lemma no escape}, using that $J$ is of contact type on $R=R_m$, a Floer solution $u$ with asymptotics $x_a,y_b$ must lie in $\{R\leq R_m\}$ for any $m$ satisfying  $R_m>\mathrm{max}\{ R(x_a),R(y_b) \}$. So compactness results for the Floer moduli spaces are not problematic.

For transversality, we run the proof of Lemma \ref{Lemma transversality proof}. Recall this relied on the construction of a vector $Y$ at $(z_0,u(z_0))\in S\times \overline{M}$ such that $g(Y(du-X\otimes \beta)j,\eta)>0$ at $z_0$. Since $Y$ arises from differentiating $J$, and since we impose that $J\partial_r = \mathcal{R}$ on $R=R_m$, we are forced to choose $Y(\mathrm{span}(\partial_r,\mathcal{R}))=0$ on $R=R_m$ (this is a closed condition, so we have restricted $Y$ to a Banach subbundle). So the proof of Lemma \ref{Lemma transversality proof} fails only if two conditions hold: (1) $u(z_0)\in \{ R=R_m\}$ and (2) $\mathrm{im}(du-X\otimes \beta)\subset \mathrm{span}(\partial_r,\mathcal{R})$ at $z_0$. 
Now $X=q'(R)\mathcal{R}$ near $R=R_m$, since $Q=q(R)$ there, so (2) implies $\partial_s u, \partial_t u \in \mathrm{span}(\partial_r,\mathcal{R})$ at $z_0$. Since $u$ is a Floer solution, $u$ satisfies equation (\ref{Equation Floer soln}) in Section \ref{Subsection Maximum principle for floer solns}, so $\partial_t u - X\beta_t = J(\partial_s u - X \beta_s)$. This equation implies that if $\partial_s u, \partial_t u$ both lie in $\mathrm{span}(\mathcal{R})$ at $z_0$ then in fact $\partial_t u = X \beta_t$, $\partial_s u = X\beta_s$, contradicting the assumption that $du-X\otimes \beta\neq 0$  at $z_0$. On the other hand, if $\partial_s u$ or $\partial_t u$ has a non-zero $\partial_r$-component at $z_0$, then we can pick a $z_0'$ arbitrarily close to $z_0$ such that $u(z_0')\notin \{ R= R_m\}$, so we can run the proof of Lemma \ref{Lemma transversality proof} for $z_0'$ instead of $z_0$. This proves transversality for Floer solutions (and for Floer continuation solutions).

For Floer trajectories $u: Z \to \overline{M}$ joining $x^{\pm}(t)$, the proof of transversality is slightly different since $J$ can be perturbed $t$-dependently but not $s$-dependently: the construction of the above $Y$ is done in the same way as before \cite[p.1346-1347]{Salamon-Zehnder}, but the problem is that the cut-off function that is used to extend $Y$ locally in the proof of Lemma \ref{Lemma transversality proof} must be chosen to be independent of $s$. Define the set of regular points $\mathfrak{R}(u) = \{ z_0=(s_0,t_0)\in Z: u(z_0)\neq x^{\pm}(t_0), \partial_s u(z_0) \neq 0, u(s,t_0) \neq u(z_0) \textrm{ for all }s\in \R  \}$. By \cite[Theorem 4.3]{Floer-Hofer-Salamon}, $\mathfrak{R}(u)\subset Z$ is open and dense unless $\partial_s u\equiv 0$ (if $\partial_s u\equiv 0$, then $u$ does not contribute to the Floer complex, so we ignore it). 
If $z_0\in \mathfrak{R}(u)$, one can check that there is a $t$-dependent cut-off function $\phi_t:\overline{M} \to [0,1]$ supported near $u(z_0)$ and near $t=t_0$ such that $\phi\circ u: Z \to \R$ is a cut-off function supported near $z_0$ of the type required in the proof of Lemma \ref{Lemma transversality proof}. So we can use $\phi$ to extend $Y$ locally, so the proof of Lemma \ref{Lemma transversality proof} works in our setup provided $z_0\in \mathfrak{R}(u)$ and $z_0\notin \{ R=R_m\}$ (since then $Y$ has no constraints). Finally observe that we can always choose $z_0\in \mathfrak{R}(u)$ with $u(z_0)\notin \{ R=R_m\}$: otherwise $u(\mathfrak{R}(u)) \subset \{ R=R_m\}$, but then by the density of $\mathfrak{R}(u)\subset Z$ we have $u(Z)\subset \{ R=R_m\}$, contradicting $x^{\pm}\notin \{ R=R_m \}$.
\end{proof}
\textbf{The second approach.} By \ref{Appendix Genericity of J} a generic arbitrarily small perturbation of $J$ achieves transversality for Floer solutions lying in a given compact subset $W\!\subset\! \overline{M}$. Unfortunately one cannot naively perturb $J$ inductively on larger and larger compact sets because new Floer solutions for fixed asymptotics may suddenly appear reaching into the regions where we perturbed $J$. So a more subtle argument is needed. We first introduce some notation.

A Floer moduli space is described by a finite data set $\mathcal{A}$ consisting of: a model surface $(S,j,\beta)$, weights $A_a,B_b$, and asymptotics $x_a,y_b$ ($1$-orbits of $A_a Q, B_b Q$). For example, for Floer trajectories the data $\mathcal{A}$ is the cylinder $Z$ and asymptotics $x^{\pm}$. We call such data sets $\mathcal{A}$ the \emph{auxiliary data}. 
We also allow an auxiliary data set $\mathcal{A}$ to consist of a finite collection of such data sets, and we say ``\emph{$J'$-Floer solution for $\mathcal{A}$}'' to mean an element of the Floer moduli space defined using $J'$ and one of the finitely many data sets in $\mathcal{A}$. Then define
\\[1mm]$\strut\qquad\qquad
R_{\mathcal{A}}(J')= \max \{ \max R(u): u \textrm{ is a possibly broken }J'\textrm{-Floer solution
for }{\mathcal{A}}\},
$\\[1mm]$\strut\qquad\qquad
K_{\mathcal{A}}(J') = \{R\leq R_{\mathcal{A}}(J')\}\subset \overline{M}$\\[1mm]
The assumptions on $J$ in (2) at the start of \ref{Subsection contact type J make transversality fail} ensure by the no escape Lemma that for any compact $K\subset \overline{M}$ and any $J'$ which equals $J$ outside $K$, the $R_{\mathcal{A}}(J')$ are bounded by constants which depend on $J,\mathcal{A},K$ but not on $J'$ (a similar statement holds for $J$ as in \emph{Remark 1}).

Let $\mathcal{A}_1 \subset\mathcal{A}_2 \subset \mathcal{A}_3 \subset \ldots$ be a sequence of auxiliary 
data sets. The goal is to obtain almost complex structures $J_{\infty}$ which are regular for $J_{\infty}$-Floer solutions for any data set in $\cup_m \mathcal{A}_m$ (so we ensure regularity simultaneously for any countable collection of data sets). For example, for Floer trajectories for $Q$, let $x_1,x_2,\ldots$ be a listing of the (non-degenerate) $1$-orbits of $Q$, and let $\mathcal{A}_1\subset \mathcal{A}_2\subset \ldots$ be determined by a listing of all possible choices of pairs $x^{\pm}$ taken from the $x_1,x_2,\ldots$ with each $\mathcal{A}_j$ being only a finite subset of such pairs.

\emph{Remark 2. Observe that a ``no escape Lemma'' for $J'$-Floer solutions for $\mathcal{A}$ corresponds to giving a bound for $R_{\mathcal{A}}(J')$. In the more general settings alluded to in Remark 1 above,  one may not have an $R$-coordinate, so one defines $K_{\mathcal{A}}(J')$ to be the union of all $J'$-Floer solutions for $\mathcal{A}$, and a ``no escape Lemma'' would mean giving a bound on the diameter of $K_{\mathcal{A}}(J')$. In this general setting, there may not be a priori bounds for the energy or for the Fredholm index for the Floer solution in terms of the asymptotics. In this case, part of the auxiliary data $\mathcal{A}_j$ will be a choice of bounds on the index and the energy of the solutions considered.
}

We need a brief comment about the analysis, since we will use Floer's norm \cite[p.807]{Floer-norm}.
To give the reader a gist of this norm we consider for simplicity smooth functions $[0,1] \to \R$, and we refer to \cite[p.1345]{Salamon-Zehnder} which explains how to adapt this norm to the space of $\omega$-compatible almost complex structures. For smooth functions, Floer's norm is defined as $|f| = \sum c_k |f|_{C^k}$, where $|\cdot|_{C^k}$ is the $C^k$-norm and where $c_k>0$ are constants. The subspace of $C^{\infty}$-functions of bounded Floer norm is a separable Banach space $B$. If the $c_k$ decrease sufficiently rapidly in $k$, then $B$ is%
\footnote{\label{Footnote Floer norm}
This is because $B$ contains many cut-off functions \cite[Lemma 5.1]{Floer-norm}. Let $\phi:[0,1]\to [0,1]$ be a cut-off function which equals $1$ near $0$ and $0$ near $1$, and extend $\phi$ constantly to $\R$. Then, by taking $c_k \leq 1/(k^k|\phi|_{C^k})$ for $k\geq 1$, the cut-off functions $b(x)=\phi(\|x - x_0\|^2/\delta)$ are in $B$ for any $x_0\in [0,1]$ and any $\delta>0$.
}
dense in $L^2$. This topology is stronger than the $C^{\infty}$-topology; it is dense in $C^{\infty}$ by the Stone-Weierstrass theorem; and it has a \emph{useful property}:
\begin{lemma}\label{Lemma useful property}
 A sequence $b_i\in B$ of bounded Floer norm $|b_i|\leq \alpha<\infty$ has a subsequence which converges in the $C^{\infty}$-topology (but not necessarily in the $B$-topology) to a smooth $b$ with Floer norm $|b|\leq \alpha$, so $b\in B$.
\end{lemma}
\begin{proof}
For fixed $k$, the $C^k$-norms of the $b_i$ are bounded. Since $[0,1]$ is compact, by the Arzela-Ascoli theorem a subsequence of the $b_i$ converges in $C^{k-1}$. By a diagonal argument, a subsequence of the $b_i$ converges in $C^k$ for each $k$, so it converges in $C^{\infty}$ to some $b$. So $\sum_{k=1}^{K} c_k |b|_{C^k} = {\displaystyle \lim_{i\to \infty}}\sum_{k=1}^K c_k|b_i|_{C^k}\leq \alpha$ for any finite $K$. Hence $\sum_{k=1}^{\infty} c_k |b|_{C^k}\leq \alpha$.
\end{proof}
We will write $|J'|$ for Floer's norm, and we write $|J'|_{W}$ for the Floer norm of the restriction $J'|_{\overline{W}}$ to the compactification of a bounded set $W\subset \overline{M}$, and we observe that the useful property mentioned above holds for Floer's norm provided we work on a compact subset of $\overline{M}$.
\begin{theorem}\label{Theorem Key claim}
Given a strictly increasing sequence $\varepsilon_m>0$, there exist small constants
$d_m>0$ and large constants $c_m>0$, such that: for any $\omega$-compatible almost complex structures
$J_{m}$ which satisfy\\[1mm]$\strut\quad\qquad
1) \quad J_{m}=J \;\textrm{ on }\overline{M}\setminus K_{m}$\\[1mm]$\strut\quad\qquad
2) \quad |J_{m} - J|_{K_{k}\setminus K_{k-1}}\leq d_{k}$ for all $k\in \{1,\ldots,m\}$ (defining $K_0=\emptyset$)\\[1mm]
we can guarantee that\\[1mm]$\strut\quad\qquad
3)\quad R_{\mathcal{A}_k}(J_{m})\leq R_{\mathcal{A}_k}(J) + \varepsilon_{m} + \sum_{j=1}^k c_j$,\; for $k\in \{1,\ldots,m\}$, and all $m\in \N$.\\[1mm]
We emphasize that these $J_m$ do not need to be related amongst each other.
\end{theorem}
\begin{proof}
The proof is by induction on $m$. For $m=0$ there is nothing to prove. Suppose by contradiction that the induction fails at step $m$, so no $d_m>0$ can be found. Then there is a sequence $J_i'$ such that $J_i'=J$ on $\overline{M}\setminus K_m$ satisfying 
\\[0.5mm]
\begin{tabular*}{\textwidth}{l@{\extracolsep{\fill}}cr@{\extracolsep{0pt}}}  
\strut &
$|J_i' - J|_{K_{k}\setminus K_{k-1}}\leq d_{k}$ for $k\in \{1,\ldots,m-1\}$;
& \strut 
\\
\strut &
$|J_i'-J|_{K_{m}\setminus K_{m-1}}\leq \delta_i$ with $\delta_i\to 0$ as $i\to \infty$; 
& \strut 
\\
\strut &
$R_{\mathcal{A}_k}(J_i')> R_{\mathcal{A}_k}(J) + D_k$ for some $k\in \{1,\ldots,m\}$
 & \strut 
\end{tabular*}
\\[0.5mm]
where $D_k=\varepsilon_{m} + \sum_{j=1}^{k} c_j$ (we will define $c_m$ later). By passing to a subsequence, we can assume the latter inequality holds for the same $k$ and for the same choice of auxiliary data $\mathcal{A}\in \mathcal{A}_k$ for all $i$. Thus, there is a sequence of $J_i'$-Floer solutions $u_i$ for the same auxiliary data $\mathcal{A}$ such that $u_i$ escapes $K_{\mathcal{A}_k}(J)$ by a distance of at least $D_k$.

The $J_i'$ only differ amongst each other on the compact set $K_{m}$, so the above useful property (Lemma \ref{Lemma useful property}) applies. Thus, after passing to a subsequence, the $J_i'$ converge in the $C^{\infty}$-topology to an $\omega$-compatible complex structure $J'$ with $J'=J$ on $\overline{M}\setminus K_{m-1}$ and whose Floer norms satisfy the same bounds as the $J_i'$: $|J' - J|_{K_{k}\setminus K_{k-1}}\leq d_{k}$ for $k\in\{1,\ldots,m-1\}$. So $J'$ satisfies the conditions required for $J_{m-1}$ in the claim. Therefore, by the inductive hypothesis, $R_{\mathcal{A}_k}(J')\leq R_{\mathcal{A}_k}(J) + D_k - (\varepsilon_m-\varepsilon_{m-1})$ for $k\in\{1,\ldots,m-1\}$. So $J'$-Floer solutions for $\mathcal{A}_k$ cannot escape $K_{\mathcal{A}_k}(J)$ by distance $D_k$ for $k\in\{1,\ldots,m-1\}$ since $\varepsilon_m>\varepsilon_{m-1}$.

Since the $J_i'$ differ from $J$ only on the compact set $K_m$, the $u_i$ lie in a compact subset of $\overline{M}$ independent of $i$ (this follows from the no escape Lemma \ref{Lemma No escape for quadratic growth}, or from the assumptions on $J$ in \emph{Remark 1} above). Since $J_i'\to J'$ in the $C^{\infty}$-topology and since the $u_i$ all lie in a fixed compact subset, the $u_i$ converge by Gromov compactness to a possibly broken $J'$-Floer solution $u$. But then by continuity the $J'$-Floer solution $u$ escapes $K_{\mathcal{A}_k}(J)$ by a distance of at least $D_k$ since the $u_i$ do: this contradicts the previous paragraph if $k\in \{1,\ldots, m-1\}$.

So finally suppose $k=m$, so $\mathcal{A}\in \mathcal{A}_m$. Since $J'$ differs from $J$ only on the compact set $K_{m-1}$, all $J'$-Floer solutions for $\mathcal{A}_m$ lie in some compact subset $C_m\subset \overline{M}$ depending only on $J$, $\mathcal{A}_m$, $K_{m-1}$ and not depending on $J'$ (again by the no escape Lemma or by the assumptions on $J$ in \emph{Remark 1}). Define $c_m>0$ large enough so that $\{R< R_{\mathcal{A}_m}(J)+D_m\}$ contains $C_m$. We emphasize $c_m$ depends only on $J, \mathcal{A}_m, K_{m-1}$ and not on $J'$. By construction of $c_m$, all $J'$-Floer solutions for $\mathcal{A}_m$ lie in $R< R_{\mathcal{A}_m}(J) + D_m$. But the above argument produced a $J'$-Floer solution $u$ which contradicts this inequality. So we obtained a contradiction.
\end{proof}

From now on, choose $\varepsilon_m\to \varepsilon <\infty$ in Theorem \ref{Theorem Key claim}, and note that we may assume that $d_m \to 0$ as $m\to \infty$. Also, fix small neighbourhoods $W_m \subset K_m\setminus K_{m-1}$ of $\partial K_m$.

The difficulty in constructing $J_m$ as above is that it is not clear that one can interpolate a given $J_m|_{K_m\setminus W_m}$ with $J|_{\overline{M}\setminus K_m}$ without destroying the bound $|J_m - J|_{K_m\setminus K_{m-1}}\leq d_m$.

By Footnote \ref{Footnote Floer norm}, there is a cut-off function 
\\[0.5mm]
\begin{tabular*}{\textwidth}{l@{\extracolsep{\fill}}cr@{\extracolsep{0pt}}} 
\strut & 
$
\phi_m:\overline{M}\to \R
$
 & \strut 
\end{tabular*}
\\[0.5mm]
 of bounded Floer norm which  on $K_m \setminus W_m$ equals $1$ and which equals $0$ near $\partial K_m$ and outside $K_m$. One can%
\footnote{\label{Footnote Banach algebra norm} We prove that we can ensure $|fg|\!\leq\! |f| |g|$ for functions $[0,1]\to \R$ (this generalizes to almost complex structures). It suffices to ensure $|fg|\!\leq\! C |f| |g|$ for $C>0$ independent of $f,g$ since we can then redefine $|\cdot|$ by rescaling it by $C$. By the Leibniz formula, $c_k|fg|_{C^k} \!=\! c_k \sum { \ell \choose i'} |\partial^{i'}\! f|_{C^0} |\partial^{j'}\! g|_{C^0}$ summing over $\ell \leq k$, $i'+j'=\ell$. It suffices to bound this by $C\sum_{i+j= k} c_{i} c_{j}|f|_{C^{i}} |g|_{C^{j}}$. Decomposing $|\cdot|_{C^{j}}\!=\!\sum_{j'\leq j} |\partial^{j'}\!\!\cdot|_{C^0}$ the latter sum contains $Cc_ic_j|\partial^{i'}\! f|_{C^0} |\partial^{j'}\! g|_{C^0}$ for $i= i',j'=\ell - i'\leq k-i=j$. So it suffices to ensure ${k \choose i} c_k \leq Cc_i c_{k-i}$, since ${\ell \choose i}\leq {k \choose i}$. So let $C=1/c_0$ and inductively ensure $c_k\leq Cc_i c_{k-i}/{k\choose i}$ for $1\leq i\leq k-1$. $\qed$}
 choose the constants $c_k$ which define the Floer norm so that in general $|\phi_m J'|\leq |\phi_m| |J'|$.
So if $J_m$ satisfies $|J_m-J|_{K_m\setminus K_{m-1}} \leq d_m/|\phi_m|$ (instead of $d_m$), then the interpolation $J'= (1-\phi_m)J+\phi_m J_m$ satisfies the required bound:
\\[0.5mm]
\begin{tabular*}{\textwidth}{l@{\extracolsep{\fill}}cr@{\extracolsep{0pt}}} 
\strut & 
$
|J' - J|_{W_m}=|\phi_m(J_m-J)|_{W_m}\leq |\phi_m| |J_m-J|_{W_m} \leq d_m
$
 & \strut 
\end{tabular*}
\\[0.5mm]
(the linear combination which defines $J'$ is understood to be computed in the Banach vector bundle of almost complex structures).

Using this interpolation trick, we can inductively construct $J_m$ to satisfy conditions (1) and (2) of Theorem \ref{Theorem Key claim} in such a way that $J_i=J_m$ on $K_m\setminus W_m$ for all $i\geq m$. Since $d_m \to 0$ as $m\to \infty$, these $J_m$ manifestly converge to a $J_{\infty}$ satisfying $J_{\infty}|_{K_{m-1}} = J_m|_{K_{m-1}}$. 

\begin{corollary}\label{Corollary Jinfinity limits}
Let $J_{\infty}$ be any $\omega$-compatible complex structure $J_{\infty}$ on $\overline{M}$ satisfying $$|J_{\infty}-J|_{K_k\setminus K_{k-1}}\leq d_k/|\phi_k| \textrm{ for all }k\geq 1.$$ Then a ``no escape Lemma'' holds: $R_{\mathcal{A}_k}(J_{\infty}) \leq R_{\mathcal{A}_k}(J)+\varepsilon+\sum_{j=1}^k c_j$ for all $k\geq 1$.
\end{corollary}
\begin{proof}
 Suppose by contradiction that a $J_{\infty}$-Floer solution $u$ for $\mathcal{A}_k$ escapes $K_{\mathcal{A}_k}(J)$ by at least distance $D=\varepsilon+\sum_{j=1}^k c_j$. The asymptotics of $u$ are fixed, so the image of $u$ lies in some $K_{m-1}$ for $m\gg k$. By the hypothesis, $J_{\infty}$ arises via the above interpolation construction: let $J_m=J_{\infty}$ on $K_m\setminus W_m$ and then interpolate with $J|_{\overline{M}\setminus K_m}$. So $u$ is a $J_m$-Floer solution for $\mathcal{A}_k$ since $J_m|_{K_{m-1}}=J_{\infty}|_{K_{m-1}}$. But $J_m$ satisfies the conditions in Theorem \ref{Theorem Key claim}, so $u$ cannot escape $K_{\mathcal{A}_k}(J)$ by distance $D$ since $\varepsilon>\varepsilon_m$. Contradiction.
\end{proof}

Let $B$ denote the Banach space of $\omega$-compatible almost complex structures defined on some compact subset $W\subset \overline{M}$ which have finite Floer norm. When we speak of making a \emph{generic} perturbation of $J_{\infty}$ on $W$, we mean perturbing $J_{\infty}$ on $W$ by an element of a Baire subset of $B$. We proved regularity statements for generic perturbations in Section \ref{Appendix Genericity of J} using the $C^{\ell}$-norm, from which one bootstraps to the $C^{\infty}$-topology. Those results can also be proved using the Floer norm \cite[p.1345]{Salamon-Zehnder}, however a Baire subset for the $B$-topology, although dense in $B$ and so dense in $C^{\infty}$, may not actually be $C^{\infty}$-Baire. However for genericity statements we really only care that countable intersections are still dense, and this is still true: a countable intersection of Baire subsets for the $B$-topology is Baire in $B$ and hence dense in $C^{\infty}$.

\begin{theorem}
 After a generic perturbation on each $K_m$ of the $J_{\infty}$ of Corollary \ref{Corollary Jinfinity limits}, the $J_{\infty}$ will be regular for $J_{\infty}$-Floer solutions for $\mathcal{A}_k$, for all $k\geq 1$, and it will still satisfy the no escape Lemma $R_{\mathcal{A}_k}(J_{\infty}) \leq R_{\mathcal{A}_k}(J)+\varepsilon+\sum_{j=1}^k c_j$ for all $k\geq 1$.
\end{theorem}
\begin{proof}
By the above interpolation procedure, 
$J_{\infty}$ is a limit of $J_m$ satisfying Theorem \ref{Theorem Key claim} with $J_{\infty}|_{K_{m-1}}=J_m|_{K_{m-1}}$. So a generic perturbation of $J_{\infty}$ on $K_{m}$ is precisely a generic perturbation of $J_{m+1}$ on $K_m$. By Section \ref{Appendix Genericity of J}, a generic perturbation of $J_{m+1}$ on $K_m$ ensures regularity for $J_{m+1}$-Floer solutions for $\mathcal{A}_k$ lying in $K_m$. A countable intersection of Baire subsets is Baire, so we can ensure this for all $k\geq 1$ simultaneously (with $m$ fixed). As we increase $m$, we typically need to perturb $J_{m+2},J_{m+3},\ldots$ also on $K_{m}$ so they would no longer equal $J_{m+1}|_{K_{m}}$ on $K_{m}$. However, these perturbations on $K_{m}$ again involve a countable intersection of Baire subsets, so in fact we have a Baire subset worth of choices for the value of $J_{m+1}|_{K_{m}}=J_{\infty}|_{K_m}$ whilst still satisfying Theorem \ref{Theorem Key claim}. Any $J_{\infty}$-Floer solution $u$ for $\mathcal{A}_k$ lands in some $K_m$, so $u$ is also a $J_{m+1}$-Floer solution, and so $u$ is regular since $J_{m+1}$ is regular for solutions lying in $K_m$ for any $\mathcal{A}_k$.
\end{proof}


\begin{corollary}\label{Corollary transversality via compactness
trick}
Transversality and compactness can both be achieved for
the moduli spaces of Floer solutions for
non-linear Hamiltonians, after a generic $($non-contact type$)$ perturbation $J_{\infty}$ of a given
$J$ of contact type at infinity, provided the perturbation decays to zero sufficiently fast as $R\to \infty$ (we require $|J_{\infty}-J|_{K_m\setminus K_{m-1}}\leq \delta_m$ for certain $\delta_m>0$ depending on $J,K_m$).
\end{corollary}

\subsection{TQFT structure using non-linear Hamiltonians.}
%
By the previous Section, we can define TQFT operations $\psi_S:\otimes_b SH^*(B_b Q) \to \otimes_a SH^*(A_aQ)$ for $\sum A_a \geq \sum B_b$.

By Lemma \ref{Lemma non-linear hamiltonians}, all these $SH^*(c\, Q)$ can be identified with $SH^*(M)$. So in fact, if one defined $SH^*(M)=SH^*(Q)$, one would define the TQFT operations
$$
\psi_S:SH^*(Q)^{\otimes q} \to SH^*(Q)^{\otimes p} \qquad (p\geq 1,q\geq 0),
$$
by first applying $\psi_S:SH^*(Q)^{\otimes q} \to SH^*(cQ)^{\otimes p}$, for any $c>0$ such that $q\leq cp$, and then identifying $SH^*(cQ)\cong SH^*(Q)$ via the direct limit of the continuation isomorphisms $SH^*(cH^m)\cong SH^*(H^{cm})$, using the notation $H^m$ defined in the proof of Lemma \ref{Lemma non-linear hamiltonians}.

\begin{theorem}
 The $\psi_S:SH^*(Q)^{\otimes q} \to SH^*(Q)^{\otimes p}$ are independent of $c>0$ and they define a TQFT on $SH^*(Q)$ which, via the identification $SH^*(M)=\varinjlim SH^*(H) \cong SH^*(Q)$ of Lemma \ref{Lemma non-linear hamiltonians}, agrees with the TQFT in \ref{Subsection Operations on SH(M)} constructed using linear Hamiltonians.
\end{theorem}
\begin{proof}
For $q\leq cp$, $m\leq m'$, $c\leq c'$ there is a diagram
$$
\xymatrix@C=16pt@R=12pt{
SH^*(H^{m'})^{\otimes q} \ar@{->}[r]^{\mathrm{TQFT}}_{\psi_{m',c}} & SH^*(cH^{m'})^{\otimes p} \ar@{->}[r]^{\mathrm{cont.}}_{\cong} & 
SH^*(H^{cm'})^{\otimes p}  \ar@{->}[r]^{\mathrm{cont.}} & SH^*(c'H^{m'})^{\otimes p} \ar@{->}[r]^{\mathrm{cont.}}_{\cong} & 
SH^*(H^{c'm'})^{\otimes p}  \\
SH^*(H^{m})^{\otimes q} \ar@{->}[u]^{\mathrm{cont.}}\ar@{->}[r]^{\mathrm{TQFT}}_{\psi_{m,c}} & SH^*(cH^{m})^{\otimes p} \ar@{->}[u]^{\mathrm{cont.}} \ar@{->}[r]^{\mathrm{cont.}}_{\cong} & 
SH^*(H^{cm})^{\otimes p} \ar@{->}[u]^{\mathrm{cont.}} \ar@{->}[r]^{\mathrm{cont.}} & SH^*(c'H^{m})^{\otimes p} \ar@{->}[u]^{\mathrm{cont.}} \ar@{->}[r]^{\mathrm{cont.}}_{\cong} & 
SH^*(H^{c'm})^{\otimes p} \ar@{->}[u]^{\mathrm{cont.}} 
}
$$
where $\psi_{m,c},\psi_{m',c}$ are the $\psi_S$-maps defined using the same $(S,\beta)$, but different $H^m,H^{m'}$.

The $1^{st}$ square commutes by Theorem \ref{Theorem Operations are compatible with limit}, the other squares commute by Lemma \ref{Lemma Chain Homotopy}(\ref{Item Lemma Chain Homotopy composing conts is cont}). The horizontal continuations in the $2^{nd}$ and $4^{th}$ squares are isomorphisms by Lemma \ref{Lemma Chain Homotopy}(\ref{Item Lemma Chain Homotopy same slopes imply iso}), and induce $SH^*(cQ)\cong SH^*(Q)$, $SH^*(c'Q)\cong SH^*(Q)$.

The $1$-orbits of $H^m,cH^m$ lie in $K_m=\{R\leq R_m\}$ (the complement of the region where $H^m$ has slope $m$). By the no escape Lemma \ref{Lemma no escape} the Floer solutions for $\psi_{m,c}$ lie in $K_m$. By the no escape Lemma \ref{Lemma No escape for quadratic growth}, the Floer solutions with asymptotics in $K_m$ involved in $\psi_S:SC^*(Q)^{\otimes q} \to SC^*(cQ)^{\otimes p}$ must also lie in $K_m$. By construction, $Q|_{K_m} = H^m|_{K_m}, cQ|_{K_m} = cH^m|_{K_m}$, so those two collections of Floer solutions are exactly the same.

By definition $SC^*(Q),SC^*(cQ)$ involve only finite linear combinations of $1$-orbits, and by the energy estimate \ref{Subsection Energy Appendix} there are only finitely many Floer solutions involved in $\psi(\otimes_b y_b)$ for given $1$-orbits $y_b$ of $cQ$. So, in the direct limit $m\to \infty$, the $1^{st}$ and $2^{nd}$ square yield \\[1mm]$\strut\qquad\qquad\qquad\qquad\qquad
SH^*(Q)^{\otimes q} \to SH^*(cQ)^{\otimes p} \to SH^*(Q)^{\otimes p},
$\\[1mm]
where the first map is $\psi_S = \lim\psi_{m,c}$. We now prove that this composite is independent of the choice of $c$. By Lemma \ref{Lemma Chain Homotopy}(\ref{Item Lemma Chain Homotopy composing conts is cont}) the composite of the horizontal maps of the $2^{nd}$ and $3^{rd}$ square give the continuation maps $SH^*(cH^{m'}) \to SH^*(c'H^{m'})$ and $SH^*(cH^{m}) \to SH^*(c'H^{m})$, and by Lemma \ref{Lemma Chain Homotopy}(\ref{Item Lemma Chain Homotopy does not matter which hpy}) these are equal to the continuations induced by a continuation cylinder $(Z,\beta_Z)$ with weights $c',c$ at the ends (see Example \ref{Example Continuation cylinder}). So composing the $1^{st}$ square with these continuations corresponds to attaching $(Z,\beta_Z)$ to the negative ends of $(S,\beta)$. So by Theorem \ref{Theorem Operations are compatible with limit}, the composite of the first 3 squares gives the TQFT operations $\psi_{m',c'},\psi_{m,c'}$, and these together with the $4^{th}$ square gives $SH^*(Q)^{\otimes q} \to SH^*(c'Q)^{\otimes p} \to SH^*(Q)^{\otimes p}$ in the limit. So to prove independence of $c$, it remains to show that composing the $3^{rd}$ and $4^{th}$ square induces the identity on $SH^*(Q)$ in the limit. But again by Lemma \ref{Lemma Chain Homotopy}(\ref{Item Lemma Chain Homotopy composing conts is cont}) the composite of the $3^{rd}$ and $4^{th}$ square give a square of continuation maps involving only the Hamiltonians $H^{cm'},H^{c'm'},H^{cm},H^{c'm}$, and they induce the identity on $SH^*(Q)$ in the direct limit since all these maps are in fact inclusions by construction (this is assuming $q$ is convex, but a ladder argument as in Footnote \ref{Footnote Lemma nonlinear SH} proves this also if we only assume $q'(R)\to \infty$ and $\mathbb{A}_q(R)\to -\infty$).

This very last argument also proves that a cylinder induces the identity operation on $SH^*(Q)$. Independence of the TQFT operations on the choices of $S,\beta,J$ and invariance follow from the above construction and the analogous properties for linear Hamiltonians proved in Section \ref{Section Floer solutions}. It remains to prove that gluings of surfaces $S_2\# S_1$ corresponds to compositions.

Consider the $1^{st}$ and $2^{nd}$ square above for two choices of data $(S_1,\beta_1,c_1)$ and $(S_2,\beta_2,c_2)$ with $p_1=q_2$. Composing these four squares together gives horizontal maps
\\[0.6mm]$\strut
 \xymatrix@C=12pt{SH^*\!(H^m\!)^{\otimes q_{1}} \ar@{->}[r]^-{\mathrm{\psi_{S_1}}} & SH^*\!(c_1H^m\!)^{\otimes p_{1}} \ar@{->}[r]^{\mathrm{cont.}} & SH^*\!(H^{c_1m}\!)^{\otimes p_{1}} \ar@{->}[r]^-{\psi_{S_2}} & SH^*\!(c_2H^{c_1m}\!)^{\otimes p_{2}} \ar@{->}[r]^{\mathrm{cont.}} & SH^*\!(H^{c_2c_1m}\!)^{\otimes p_{2}}
}
$\\[1mm]
Here $\psi_{S_2}$ is defined using the data $(S_2,\beta_2,H^{c_1m})$ with weights $c_2,1$ respectively at negative and positive ends. Without changing the moduli spaces, we can instead use the data $(S_2,c_1\beta_2,\tfrac{1}{c_1}H^{c_1m})$ using weights $c_2c_1,c_1$ respectively (this will ensure that $\beta_1,\beta_2$ have matching weights when we later glue). The $2^{nd}$ map involves cylinders $(Z,\beta_Z)$ with $\beta_Z=c_1\, dt$ which use a Hamiltonian $H_z$ depending on $z\in Z$ which equals $\tfrac{1}{c_1}H^{c_1m}$ and $H^m$ near the two ends (these Hamiltonians both have slope $m$ at infinity, so we can pick $H_z$ independent of $z$ for $R\gg 0$). Let $\psi_0$ denote the composite of the first 3 maps above (the no escape Lemma still holds by  Remark \ref{Remark no escape Lemma comments}(\ref{Item no escape for z dependent H}) since $\partial_s H_z=0$ for $R\gg 0$). By Theorem \ref{Theorem Gluing for compositions}, this corresponds to gluing $S_2\# (\sqcup_{p_{1}} Z) \# S_1$. Now consider a different gluing: $(\sqcup_{p_{2}} Z') \# S_2 \# S_1$ obtained by using the Hamiltonian $H^{m}$ on both $S_2,S_1$ but on the cylinder $Z'$ we use $H_z$ and $\beta_Z=c_1c_2\,dt$. These two gluings $S_2\# (\sqcup_{p_{1}} Z) \# S_1$ and $(\sqcup_{p_{2}} Z') \# S_2 \# S_1$ share the same asymptotic data (in particular, the Hamiltonians agree at the asymptotics). The same proof as in Theorem \ref{Theorem Homotopy of surfaces}, using parametrized moduli spaces, shows that the two gluings induce chain homotopic maps. By Theorem \ref{Theorem Gluing for compositions}, $\psi_{ (\sqcup Z') \# S_2 \# S_1} = \psi_{\sqcup Z'} \circ \psi_{ S_2 \# S_1}$. Thus the above horizontal map equals
\\[0.6mm]$\strut
\xymatrix{SH^*(H^m)^{\otimes q_{1}} \ar@{->}[r]^-{\mathrm{\psi_{S_2\# S_1}}} & SH^*(c_2c_1H^{m})^{\otimes p_{2}} \ar@{->}[r]^-{\mathrm{\psi_{\sqcup Z'}}} & SH^*(c_2H^{c_1m})^{\otimes p_{2}} \ar@{->}[r]^-{\mathrm{cont.}}  &  SH^*(H^{c_2c_1m})^{\otimes p_{2}}
}
$\\[1mm]
By Lemma \ref{Lemma Chain Homotopy}(\ref{Item Lemma Chain Homotopy composing conts is cont}), the composite of the last two maps is the continuation $SH^*(c_2c_1H^m)^{\otimes p_{2}} \to SH^*(H^{c_2c_1m})^{\otimes p_{2}}$. As before, this argument is compatible with the continuations which change $H^m$ to $H^{m'}$, so in the direct limit this proves that $\psi_{S_2} \circ \psi_{S_1} = \psi_{S_2\# S_1}$ on $SH^*(Q)$.
\end{proof}
\noindent\emph{Remark.} Similarly, in the wrapped case, $HW^*(L)\cong HW^*(L;Q)$ respects the TQFT.
%
\subsection{The TQFT for $\mathbf{T^*N}$ using Levi-Civita $\mathbf{J}$ instead of contact type $\mathbf{J}$}
\label{Subsection levicivita J vs contact type J} 
%
For $\overline{M}=T^*N$, Abbondandolo-Schwarz \cite{Abbondandolo-Schwarz,Abbondandolo-Schwarz2} use an $\omega$-compatible almost complex structure $J_{\mathrm{LC}}$ on $\overline{M}=T^*N$ which is induced by a Levi-Civita connection for $N$ \cite[Sec.1.5]{Abbondandolo-Schwarz}, and these are not of contact type at infinity. We now prove that their Floer cohomology $SH^*(Q;J_{\mathrm{LC}})$ agrees with our $SH^*(Q;J)$ for $Q$ with $q(R)$ of quadratic growth in $R$, and that the products agree.\\[1mm]
\noindent \textbf{Claim 1.} \emph{There is an isomorphism $SH^*(Q;J_{\mathrm{LC}})\cong SH^*(Q;J)$ respecting action-filtrations.}\\
\emph{Proof.}
The Floer complex admits a filtration by action (Section \ref{Section SH+}), and recall $\mathbb{A}_Q$ is bounded above since $\mathbb{A}_q(R)\to -\infty$ as $R\to \infty$. By Lemma \ref{Lemma non-linear hamiltonians}, it suffices to do the following: \begin{enumerate}
                                                                                                                                                                                                                                                         \item build filtered-isomorphisms $\varphi_c: SH^*(Q;J_{\mathrm{LC}};\mathbb{A}_Q\geq c) \cong SH^*(Q;J;\mathbb{A}_Q\geq c)$;

\item show that the $\varphi_c$ are compatible with decreasing the action bound $c\in \R$.
                                                                                                                                                                                                                                                        \end{enumerate} 

 Our $SH^*(Q;J;\mathbb{A}_Q\geq c)$ are invariant under deforming $J$ on a compact subset, provided $J$ is of contact type at infinity. So we can pick $J$ to equal $J_{\mathrm{LC}}$ on any large compact $K\subset \overline{M}$. Abbondandolo-Schwarz \cite[Theorem 1.14]{Abbondandolo-Schwarz} proved $L^{\infty}$-estimates which imply that $J_{\mathrm{LC}}$-Floer trajectories whose asymptotics satisfy $\mathbb{A}_Q\geq c$ do not have enough energy to escape a compact subset of $\overline{M}$ determined by $c$, and we pick $K$ so that $K$ contains this subset. Thus, the same statement holds for $J$-Floer trajectories since $J|_K=J_{\mathrm{LC}}|_K$. So $\varphi_c$ in (1) is an identification for this $J$. Then (2) follows because the inclusions $SC^*(Q;J;\mathbb{A}_Q\geq c) \to SC^*(Q;J;\mathbb{A}_Q\geq c')$, for $c\geq c'$, are invariant on cohomology under deforming $J$ on compact subsets. $\qed$\\[1mm]
\textbf{Claim 2.} \emph{The isomorphism $SH^*(Q;J_{\mathrm{LC}})\cong SH^*(Q;J)$ respects the product structures.}\\
\emph{Proof.} This follows by the same argument, since Abbondandolo-Schwarz proved $L^{\infty}$-estimates also for $J_{\mathrm{LC}}$-Floer solutions defined on a suitable pair-of-pants surface \cite[Prop.6.2]{Abbondandolo-Schwarz2}. $\qed$\\[1mm]
\textbf{Claim 3.} \emph{Theorem \ref{Theorem Abbondandolo-Schwarz product theorem} holds for $SH^*(Q;J)$.}\\
\emph{Proof.} By Abbondandolo-Schwarz \cite[Thm.3.1]{Abbondandolo-Schwarz} $SH^*(Q;J_{\mathrm{LC}})\!\cong\! H_{n-*}(\mathcal{L}N)$ respects the action filtrations (Remark \ref{Remark AS iso respects filtration}). By Claim 1,  $SH^*(Q;J_{\mathrm{LC}})\!\cong\! SH^*(Q;J)$ respects the filtration. $\qed$
%
%
\section{Appendix 4: Energy, maximum principle, and no escape Lemma}
\label{Appendix Maximum principle and no escape lemma}
%
%
\subsection{Energy of Floer solutions}\label{Subsection Energy Appendix}
%
Define the energy of a map $u:S \to \overline{M}$ by
$$\textstyle
E(u)=\frac{1}{2} \int_S \| du-X\otimes \beta \|^2\,
\textrm{vol}_S.
$$
\emph{Explanation. Let $Y\in \textrm{Hom}(TS,u^*T\overline{M})$. In a local holomorphic coordinate $s+it$ for $(S,j)$: $\mathrm{vol}_S = ds \wedge dt$; $Y=Y_s ds + Y_t dt$ where $Y_s=Y(\partial_s)$, $Y_t=Y(\partial_t)$; and $\|Y\|^2= |Y_s|_J^2 + |Y_t|_J^2$ where $|\cdot|^2_J=g_J(\cdot,\cdot) = \omega(\cdot,J\cdot)$ on $T\overline{M}$. Fact: $\|Y\|^2\,\mathrm{vol}_S$ is independent of the choice of $s+it$.}

\textbf{Example.} For a cylinder $S=Z$ with $\beta=dt$,
$E(u)=\frac{1}{2}\int (|
\partial_s u |^2 + | \partial_t u - X |^2)\, ds\wedge dt$, so for a Floer trajectory $u$ we get the usual energy $E(u)=\int | \partial_s u |^2\, ds\wedge dt$.

Recall $Y^{0,1}=\tfrac{1}{2}(Y + J\circ Y \circ j)$, so $Y^{0,1}=0$ implies $Y\circ j = J\circ Y$, so locally $Y_t = JY_s$ since $j\partial_s = \partial_t$, therefore $\tfrac{1}{2}\|Y\|^2\mathrm{vol}_S= \omega(Y_s,Y_t)\, ds\wedge dt$.

For $Y=du-X\otimes \beta$, decompose $\beta=\beta_s\, ds + \beta_t\, dt$,
then $(du - X\otimes \beta)^{0,1}=0$ implies
$$
\begin{array}{lll}
\frac{1}{2}\| du - X\otimes \beta \|^2\mathrm{vol}_S & = & \omega(\partial_s u-X
\beta_s, \partial_t u - X\beta_t)\, ds\wedge dt \\ & = & [\,\omega(\partial_s u,
\partial_t u) +dH(\partial_t u)\beta_s - dH(\partial_s u)\beta_t\,]\, ds\wedge dt 
\\ & = & u^*\omega-u^*(dH)\wedge \beta.
\end{array}
$$

We can rewrite $-u^*(dH)\wedge \beta = -d(u^*(H)\,\beta)+u^*(H)\, d\beta$. Assuming
$u$ is a Floer solution for weights $A_a,B_b$, with $H \geq 0$, $d\beta\leq 0$, by Stokes'
theorem we obtain the energy estimate:
$$
\begin{array}{lll}
E(u) & = & \int_S u^*\omega-u^*(dH)\wedge \beta \\
& \leq & \int_S u^*\omega - d(u^*(H)\, \beta) \qquad (\textrm{since } u^*(H)\,d\beta\leq 0)\\
& = & \int_S d(u^*\theta - u^*(H)\, \beta)  \,\qquad (\textrm{since } \omega=d\theta)\\
& = & \displaystyle \!\!\!\!\sum_{\textrm{negative ends}\;
a}\!\!\!\! {\mathbb{A}}_{A_a H}(x_a) - \!\!\!\!\sum_{\textrm{positive
ends}\; b}\!\!\!\! {\mathbb{A}}_{B_b H}(y_b).
\end{array}
$$
%
%
\subsection{Energy of wrapped solutions}\label{Subsection Energy Appendix Wrapped}
Recall from \ref{Subsection Wrapped TQFT structure} that wrapped solutions $u:S \to \overline{M}$ solve $(du-X\otimes \beta)^{0,1}=0$ with $u(\partial S) \subset \overline{L}$, where $\beta|_{\partial S}=0$, $d\beta\leq 0$ and $\theta|_{\overline{L}}=df$. The same proof as in \ref{Subsection Energy Appendix} shows that wrapped solutions with asymptotic chords $x_a,y_b$ satisfy the a priori estimate $E(u)\leq \sum \mathbb{A}_{A_a H}(x_a) - \sum \mathbb{A}_{B_b H}(y_b)$ (recall $\mathbb{A}_H$, defined in \ref{Subsection Action functional in wrapped case}, involves $f$).
\subsection{Maximum principle for Floer solutions}
\label{Subsection Maximum principle for floer solns}
%
By the discussion in \ref{Subsection Energy Appendix}, the equation $Y^{0,1}=0$ for $Y=du-X\otimes \beta$ in a local holomorphic coordinate $s+it$ for $S$ corresponds to:
\begin{equation}\label{Equation Floer soln}
\left\{\begin{array}{ll}
\partial_t u = X \beta_t + J\partial_s u - JX \beta_s \\
\partial_s u = X \beta_s - J \partial_t u + JX \beta_t
\end{array}\right.
\end{equation}

\begin{lemma}\label{Lemma Maximum principle for Floer solns}

For $R\geq R_0$ assume: $J$ is of contact type and $H=h(R)$ only depends on $R$ with $h'\geq 0$. Then for any
local Floer solution $u$ landing in $R\geq R_0$, the coordinate
$R\circ u$ cannot have local maxima unless it is constant. So global Floer solutions with asymptotics $x_a,y_b$ must
lie in the region $R\leq \max \{ R(x_a), R(y_b),R_0\}$. The result also holds in the following situations:
\begin{enumerate}
\item if $H$ is time-dependent near the ends $($where $\beta$ is a
multiple of $dt)$;
\item\label{Item Lemma Maximum principle ok if monotone hpy} if $h=h_z(R)$ depends on
$z=s+it$ near the ends $($where $\beta$ is a
multiple of $dt)$, provided we assume the monotonicity condition: $\partial_s
h_z'\leq 0$ for $R\geq R_0$;
\item if $J=J_z$ depends on $z\in S$, assuming the contact condition $dR=J_z^*\theta$ for $R\geq R_0$.
\end{enumerate}
\end{lemma}
\begin{proof}
Let $\rho(s,t)$ denote the $R$-coordinate of $u(s,t)$. Using the contact condition $\theta =
-dR \circ J$,
$$
\begin{array}{rcl}
\partial_s \rho & = & dR(\partial_s u) = dR(X \beta_s - J \partial_t u + JX
\beta_t)
\\ & = & \theta(\partial_t u) - \theta(X)\beta_t
\\[1mm] \partial_t \rho & = & dR(\partial_t u) = dR(X \beta_t + J\partial_s u - JX \beta_s)
\\ & = & -\theta(\partial_s u) + \theta(X)\beta_s
\\[1mm] d^c \rho& = & d\rho\circ j =  (\partial_t \rho)\, ds - (\partial_s \rho) \, dt
\\ & = & -\theta(\partial_s u)\,ds -\theta(\partial_t u)\,dt + \theta(X)\beta_s
ds+ \theta(X)\beta_t dt
\\ & = & -u^*\theta + \theta(X)\beta
\\[1mm] -dd^c \rho & = & u^*\omega - d(\theta(X)\beta)
\\ & = & \frac{1}{2}\| du - X\otimes \beta \|^2\, ds\wedge dt + u^*(dH)\wedge
\beta- d(\theta(X)\beta).
\end{array}
$$

By \ref{Subsection Liouville domains Definition}, $X = h'(R) \mathcal{R}$ where $\mathcal{R}$ is the
Reeb vector field, so $\theta(X)= R h'(R)$. Thus,
$$
u^*(dH)\wedge \beta- d(\theta(X)\beta) = h'(\rho)d\rho \wedge \beta -
d(\rho h'(\rho)\beta) =  - \rho h''(\rho) d\rho\wedge \beta -\rho h'(\rho)d\beta.
$$

Thus $\Delta \rho \, ds\wedge dt = -dd^c \rho \geq - \rho h''(\rho)
d\rho\wedge \beta - \rho h'(\rho)d\beta$, so
$$\Delta \rho + (\textrm{first order terms in }\rho) \geq -
\frac{\rho h'(\rho)d\beta}{ds\wedge dt}
$$
Since $d\beta\leq 0$, $h' \geq 0$, $\rho\geq 0$, the right side is
$\geq 0$. The claim follows by the maximum principle \cite[Sec.6.4]{Evans}
for the elliptic operator $L=\Delta + R h''(R) \beta_t \partial_s
- R h''(R) \beta_s
\partial_t$ since $Lu\geq 0$. 

(1) If $H=h_t$ on an end where $\beta=C\,dt$ $(C>0)$, the extra term
$-\rho(\partial_t h_t') dt \wedge \beta = 0$. 

(2) If $H=h_z$ the extra
term $-\rho (\partial_s h_z') C \geq 0$ provided $\partial_s
h_z'\leq 0$. 

(3) A $z$-dependence for $J=J_z$ does not affect the proof.
\end{proof}
%
%
\subsection{Maximum principle for the Lagrangian setting}
\label{Subsection Maximum principle for the Lagrangian setting}
Recall from \ref{Subsection Wrapped TQFT structure} that wrapped solutions $u:S \to \overline{M}$ solve $(du-X\otimes \beta)^{0,1}=0$ with $u(\partial S) \subset \overline{L}$, where $\beta|_{\partial S}=0$ and $d\beta\leq 0$.

\begin{lemma}\label{Lemma Maximum principle for wrapped solns}
Under the assumptions of Lemma \ref{Lemma Maximum principle for Floer solns}, also 
any local wrapped solution $u$ landing in $R\geq R_0\geq 1$ cannot admit a local maximum of 
$R\circ u$ unless $R\circ u$ is constant, so global wrapped solutions with asymptotics $x_a,y_b$ lie in the region $R\leq \max \{ R(x_a),R(y_b),R_0\}$. 
\end{lemma}
\begin{proof} The proof is identical to Lemma \ref{Lemma Maximum principle for
Floer solns}, except maxima might occur on $\partial S$. But then
by Hopf's Lemma \cite[Sec.6.4.2]{Evans} the derivative $\partial_t
\rho$ of $\rho=R\circ u$ in the outward normal direction
$\partial_t$ on $\partial S$ would be strictly positive at those maxima, which is false:\\[1mm]$\strut\qquad\qquad\qquad
\partial_t \rho = dR(\partial_t u) =
dR(X \beta_t + J\partial_s u - JX \beta_s) = -\theta(\partial_s
u)=0
$\\[1mm]
using $dR(X)=0$, $dR\circ J = - \theta$, the pull-back $\beta|_{\partial
S}=0$ so $\beta_s=0$ on $\partial S$, and finally the pull-back
$\theta|_{\overline{L}\setminus L}=0$ on the collar and $\partial_s u \in
T\overline{L}$ on $\partial S$ (since $\partial_s = -j\partial_t$
is tangent to $\partial S$).
\end{proof}
%
\subsection{No escape lemma}
\label{Subsection No escape lemma}
%
In Corollary \ref{Corollary No escape} we need the following Lemma. This result is analogous to Abouzaid-Seidel \cite[Lemma
7.2]{Abouzaid-Seidel}, adapted to our setup.

Let $(V,d\theta)$ be a non-compact symplectic manifold with boundary $\partial
V$ of \emph{negative contact type}: the Liouville field $Z$, defined by
$d\theta(Z,\cdot)=\theta$, points strictly inwards.

Near $\partial V$
define the coordinate $R=R_0 e^r$ parametrizing the time $r$ flow of
$Z$ starting from $\partial V=\{R=R_0\}$ (compare \ref{Subsection Liouville
domains Definition}).

Assume $J$ is of contact type along $\partial
V$, meaning $J^*\theta = dR$ holds at points of $\partial V$. Suppose $H:V\to [0,\infty)$ only depends on $R$ near $\partial V$, say $H=h(R)$. Define $X$ by $d\theta(\cdot,X)=dH$ and define $\mathcal{R}=JZ$. Then by the contact condition one easily obtains that $X=h'(R)\mathcal{R}$ near $\partial V$.

Suppose $h(R)=mR$ is linear near
$\partial V$ (see Lemma \ref{Lemma No escape for quadratic growth} for the general case). 
Let
$S$ be a compact Riemann surface with boundary and let $\beta$ be a $1$-form
on $S$ with $d\beta\leq 0$.

\begin{lemma}\label{Lemma no escape}
Any solution $u: S \to V$ of $(du-X\otimes \beta)^{0,1}=0$, with
$u(\partial S) \subset \partial V$, must map entirely into
$\partial V$ and must solve $du=X\otimes \beta$. 
\end{lemma}
\begin{proof} We use a trick we learnt from Mohammed
Abouzaid: we show $E(u) \leq 0$, so $E(u)=0$, so
$du=X\otimes \beta$, so $du$ lands in the span of $X=h'\mathcal{R}
\subset T\partial V$, thus $u(S)\subset
\partial V$ as required. 
$$
\begin{array}{llll}
E(u) & = & \int_S u^*d\theta - u^*(dH)\wedge \beta & \textrm{(Section \ref{Subsection Energy Appendix} using $\omega=d\theta$)}\\
& \leq & \int_{S} d(u^*\theta)-d(u^*H\beta) & \textrm{(}d\beta\leq 0\textrm{ and } H\geq 0\textrm{)}\\
& = & \int_{\partial S} u^*\theta-(u^*H)\beta & \textrm{(Stokes' Theorem)}\\
& = & \int_{\partial S} u^*\theta-\theta(X)\beta & \textrm{(on } u(\partial S)\subset \partial V: H=mR=\theta(h'(R)\mathcal{R})=\theta(X) \textrm{)}\\
& = & \int_{\partial S} \theta(du-X \otimes \beta) & \\
& = & \int_{\partial S} -\theta J(du-X \otimes \beta)j &
\textrm{(since }(du-X\otimes \beta)^{0,1}=0 \textrm{)}\\
& = & \int_{\partial S} -dR(du-X\otimes\beta)j &
\textrm{($J$ is of contact type along }u(\partial S)\subset \partial V\textrm{)}\\
& = & \int_{\partial S} -dR(du)j & \textrm{(}dR\textrm{ vanishes
on }X=h'(R)\mathcal{R}\textrm{ on }u(\partial S)\subset \partial V\textrm{)}
\end{array}
$$
Let $\hat{n} =$ outward normal direction along $\partial S\subset
S$. Then $(\hat{n},j\hat{n})$ is an oriented frame, so $\partial
S$ is oriented by $j\hat{n}$. Now $dR(du)j(j\hat{n})=d(R\circ
u)\cdot (-\hat{n})\geq 0$ since in the inward direction
$-\hat{n}$, $R\circ u$ can only increase since $\partial V$
minimizes $R$. So $E(u)\leq 0$.
%
\end{proof}
\begin{remarks}\label{Remark no escape Lemma comments}\strut 
\begin{enumerate}
 \item One can allow $u$ to have negative/positive punctures, provided that the asymptotics $x_a,y_b$ satisfy $\sum {\mathbb{A}}_{A_a H}(x_a) - \sum {\mathbb{A}}_{B_b H}(y_b)\leq 0$ $($this is the new contribution in the above proof to the integral $\int_{\partial S} u^*\theta -(u^*H)\beta)$.
 
\item \label{Item no escape for z dependent H} If $H=H_z$ depends on $z\in S$ 
on a subset of $S$ parametrized by a holomorphic coordinate $z=s+it$ for a subset of values of $(s,t)\in \R \times S^1$ where $\beta=c\, dt$ for some $c>0$, then on that subset we require that: $\partial_s H\leq 0$ and near $\partial V$ we require $H=m_sR$ for slopes $m_s>0$ depending on $s$. So in the above proof, $u^*(dH)\wedge \beta=d(u^*H)\wedge \beta - c\,\partial_s (u^*H)\, ds\wedge dt$, and this last term has the correct sign needed to prove $E(u)\leq 0$. Also $H=m_s R= \theta(X_{H_z})$ on $\partial V$ and hence on $u(\partial S)\subset \partial V$.

In the setting of Lemma \ref{Lemma No escape for quadratic growth}, if $H=H_z$ on a subset as above, then we require $\partial_s H\leq 0$ and near $\partial V$ we require $H=q_z(R)$ with $\mathbb{A}_{q_z}(R_0)\leq 0$ and $\partial_s \mathbb{A}_{q_z}(R_0)\geq 0$. This ensures that $-d(\mathbb{A}_{q_z}(R_0)\beta)=-c\partial_s \mathbb{A}_{q_z}(R_0)\, ds\wedge dt \leq 0$ on that subset, so in the last line of the proof of Lemma \ref{Lemma No escape for quadratic growth} we still obtain $-\int_S d(\mathbb{A}_{q_z}(R_0)\beta) \leq 0$.

\item If $J$ is of contact type on all of $V$ then, by the proof of Lemma \ref{Lemma Maximum principle for Floer solns}, $\Delta R\, ds \wedge dt = -dd^c R = u^*d\theta - d(\theta(X)\beta)$. So the second half of the above proof is Green's formula \cite[Appendix C.2]{Evans} for $S$: $\int_{\partial S} u^*\theta - \theta(X)\beta = \int_S \Delta R \, ds \wedge dt = \int_{\partial S} \frac{\partial R}{\partial \hat{n}} dS \leq 0$.
\end{enumerate}
\end{remarks}
%
\subsection{No escape lemma for non-linear $H$}
\label{Subsection No escape lemma for nonlinear H}
%
We need the following Lemma in the Technical Remarks in \ref{Subsection An alternative definition of SH}. For $q: \R \to \R$, define the action function $\mathbb{A}_q(R) = -Rq'(R) + q(R)$.
\begin{lemma}\label{Lemma No escape for quadratic growth}
In Lemma \ref{Lemma no escape}, suppose $H: V \to [0,\infty)$ has the form $H=q(R)$ near $\partial V=\{ R = R_0 \}$ instead of $H=mR$. If $\mathbb{A}_q(R_0)\leq 0$ then the conclusion of the Lemma holds.
\end{lemma}
\begin{proof}
 It suffices to show $\int_{\partial S} u^*\theta - (u^*H) \beta \leq \int_{\partial S} u^*\theta - \theta(X)\beta$ since this was the only equality in the proof of Lemma \ref{Lemma no escape} that used the assumption $H=mR$ near $\partial V$.

So we need $\int_{\partial S} (\theta(X) - u^*H)\beta \leq 0$. On $u(\partial S) \subset \partial V$, $\theta(X) = \theta(q'(R)\mathcal{R}) = R_0 q'(R_0)$ and $u^*H=q(R_0)$. So, using Stokes' theorem, $d\beta \leq 0$ and $\mathbb{A}_q(R_0)\leq 0$:
$$ \textstyle
\int_{\partial S} (\theta(X) - u^*H)\beta = \int_{\partial S} (R_0 q'(R_0) - q(R_0))\beta = 
-\mathbb{A}_q(R_0) \int_{\partial S} \beta = -\mathbb{A}_q(R_0) \int_S d\beta \leq 0. \qedhere
$$
\end{proof}
%
%
\subsection{No escape lemma for the Lagrangian setting}
\label{Subsection No escape lemma for the Lagrangian setting}
%
To study the Viterbo restriction in the wrapped setup in \ref{Subsection No escape for wrapped solutions}, we need the analogue of Lemma \ref{Lemma no escape} when the map $u: S \to V$ has Lagrangian boundary conditions, which is due to Abouzaid-Seidel \cite[Lemma 7.2]{Abouzaid-Seidel}.

We use the same notation as in \ref{Subsection No escape lemma}: $(V,d\theta)$, $R$, $\partial V=\{R=R_0\}$, $H$, $J$, $S$, $\beta$. The novelty is that the Riemann surface $S$ has corners, and the corners divide the boundary of $S$ into pieces which are closed intervals or circles. These pieces are labeled by the two letters $b,l$ (because later our maps $u: S \to V$ will be required to land in the boundary $\partial V$ on the $b$-pieces and in a Lagrangian $L$ on the $l$-pieces), and two pieces intersecting at a corner must carry different labels. Abbreviate the decomposition by $\partial S = \partial_b S \cup \partial_l S$. We require $\partial_b S \neq \emptyset$.

Let $L\subset V$ be an exact Lagrangian, say $\theta|_L=df$, such that $\theta|_L$ vanishes to infinite order on the boundary $\partial L = L \cap \partial V$. In addition to $d\beta \leq 0$, we also require that the pull-back $\beta|_{\partial_l S}=0$. Finally, we need to strengthen the exactness assumption on $L$: \boxed{\emph{$f|_{\partial L} = 0$}}

\begin{lemma}\label{Lemma no escape with Lag bdry condns}
Any solution $u: S \to V$ of $(du-X\otimes \beta)^{0,1}=0$, with
$u(\partial_b S) \subset \partial V$ and $u(\partial_l S) \subset L$,  must map entirely into
$\partial V$ and must solve $du=X\otimes \beta$.
\end{lemma}
\begin{proof} We run the proof of Lemma \ref{Lemma no escape}. Thus $E(u) \leq \int_{\partial S} u^*\theta - (u^*H)\beta$. 
The new contributions arise along the pieces $\gamma \subset \partial_l S$. Since $(u^*H)\beta|_{\gamma}=0$, only the $u^*\theta|_{L}=u^*df$ integrand contributes. By Stokes' theorem, $\int_{\gamma} u^*\theta=0$ for circles $\gamma$. For intervals $\gamma$, $\int_{\gamma} u^*\theta=f(z)-f(z')$ for corners $z,z' \in \partial_l S \cap \partial_b S$ so this also vanishes by the assumption $f|_{\partial L} = 0$.
\end{proof}
%
%

\end{document}